\documentclass[11pt]{article} 

\usepackage{amsmath,amssymb,amsthm}
\usepackage[utf8]{inputenc} 
\usepackage{epic}
\usepackage{eepic}
\usepackage{rotating}
\usepackage{xcolor}

\usepackage{geometry} 
\numberwithin{equation}{section}

\usepackage{graphicx} 
\usepackage{graphics}
\usepackage[font=sf, labelfont={sf,bf}, margin=1.2cm]{caption}

\makeatletter
\newcommand{\oset}[2]{%
  {\mathop{#2}\limits^{\vbox to -.5\ex@{\kern-\tw@\ex@
   \hbox{\scriptsize #1}\vss}}}}
\makeatother

\newtheorem{theorem}{Theorem}[section]

\newtheorem{definition}[theorem]{Definition}
\newtheorem{proposition}[theorem]{Proposition}
\newtheorem{corollary}[theorem]{Corollary}
\newtheorem{lemma}[theorem]{Lemma}

\usepackage{booktabs} 
\usepackage{array} 
\usepackage{paralist} 
\usepackage{verbatim} 
\usepackage{subfig}

\usepackage{fancyhdr} 
\usepackage{sectsty}
\allsectionsfont{\sffamily\mdseries\upshape}

\usepackage[nottoc,notlof,notlot]{tocbibind} 
\usepackage[titles,subfigure]{tocloft}

\title{Lacunarity, Kakeya-type sets and directional maximal operators}

\author{\textsc{Edward Kroc and Malabika Pramanik}}

\begin{document}  
\maketitle
{\allowdisplaybreaks

\begin{abstract}
We develop a notion of finite order lacunarity for direction sets in $\mathbb R^{d+1}$. Given a direction set $\Omega$ that is sublacunary according to this definition, we construct random examples of Euclidean sets that contain unit line segments with directions from $\Omega$ and enjoy analytical features similar to those of traditional Kakeya sets of infinitesimal Lebesgue measure. This generalizes to higher dimensions a planar result due to Bateman~\cite{Bateman}. Combined with earlier work of Alfonseca \cite{Alfonseca}, Bateman \cite{Bateman}, Parcet and Rogers \cite{ParcetRogers}, this notion of lacunarity and Kakeya-type sets also yields a characterization in all dimensions for directional maximal operators to be $L^p$-bounded. 
\end{abstract}
\renewcommand{\thefootnote}{\fnsymbol{footnote}} 
\renewcommand{\thefootnote}{\arabic{footnote}}

\tableofcontents

\section{Introduction}

\subsection{Background}\label{background}

This paper is concerned with a generalization of the classical Euclidean Kakeya set, also called a Besicovitch set.  In $\mathbb{R}^{d+1}$, a Kakeya set is one that contains a unit line segment in every direction.  Here, we are concerned with sets that contain a unit line segment in every direction of a given subset of directions $\Omega$.  For certain ``large enough" subsets $\Omega$, the geometric and analytical structure of these sets is remarkably similar to that of traditional Kakeya sets.  The bulk of this paper is devoted to making this idea precise. 

In the study of Kakeya sets and for quantification purposes, it is often convenient to work with a $\delta$-neighborhood of the set rather than the set itself, where $\delta$ is an infinitesimal positive constant. This neighborhood is therefore a set of small but positive Lebesgue measure built of roughly $\delta^{-d}$ thin tubes of unit length and spherical cross-section of radius $\delta$. For thickenings of Kakeya sets resulting from many concrete classical constructions, these constituent tubes enjoy certain structural properties that have proved to be of considerable analytical and geometric significance \cite[Chapter 10]{SteinHA}, \cite{{KatzTao}, {KatzLabaTao}}. The present article focuses on one of them (see Definition \ref{Kakeya-type set} below). We study this property in a context similar but not identical to classical Kakeya sets, investigate its applications in analytical problems of independent interest, and obtain a characterization of direction sets $\Omega$ for which such structure can hold. 

\begin{definition}\label{Kakeya-type set}
	Fix a set of directions $\Omega\subseteq \mathbb R^{d+1}$.  We say a cylindrical tube is oriented in direction $\omega \in \Omega$ if the principal axis of the cylinder is parallel to $\omega$.  If for some fixed constant $A_0 \geq 1$ and any choice of integer $N\geq 1$, there exist 
\begin{enumerate}[-]
\item a number $0 < \delta_N \ll 1$, $\delta_N \searrow 0$ as $N \nearrow \infty$, and 
\item a collection of tubes $\{P_t^{(N)}\}$ with orientations in $\Omega$, length at least 1 and cross-sectional radius at most $\delta_N$ 
\end{enumerate} 
obeying  
\begin{equation}\label{Kakeya-type condition}
\lim_{N\rightarrow\infty}\frac {|E^*_N(A_0)|}{|E_N|} = \infty, \quad \text{ with }\quad E_N := \bigcup_t P_t^{(N)},\quad E_N^*(A_0) := \bigcup_t A_0P_t^{(N)},
\end{equation}
then we say that $\Omega$ \textit{admits Kakeya-type sets}.  Here, $|\cdot|$ denotes $(d+1)$-dimensional Lebesgue measure, and $A_0P_t^{(N)}$ denotes the tube with the same centre, orientation and cross-sectional radius as $P_t^{(N)}$ but $A_0$ times its length. The tubes that constitute $E_N$ may have variable dimensions subject to the restrictions mentioned above. We refer to $\{E_N : N \geq 1 \}$ as sets of Kakeya type. 
\end{definition}

The inadmissibility of Kakeya-type sets is related to the boundedness on Lebesgue spaces of certain maximal averages widely studied in harmonic analysis. Given a set of directions $\Omega\subseteq \mathbb R^{d+1}$, we consider the directional maximal operator $D_{\Omega}$ defined by
\begin{equation}\label{max op}
	D_{\Omega}f(x) := \sup_{\omega\in\Omega}\sup_{h>0}\frac 1{2h}\int_{-h}^{h} |f(x+\omega t)|dt,
\end{equation}
where $f:\mathbb{R}^{d+1}\rightarrow\mathbb{R}$ is a function that is locally integrable along lines.  We also consider the Kakeya-Nikodym maximal operator $M_{\Omega}$ defined by
\begin{equation}\label{Kakeya max op}
	M_{\Omega}f(x) := \sup_{\omega\in\Omega}\sup_{\substack{P\ni x\\ P\parallel\omega}} \frac 1{|P|} \int_P |f(y)|dy,
\end{equation}
for any locally integrable function $f:\mathbb{R}^{d+1}\rightarrow\mathbb{R}$. The inner supremum in the definition \eqref{Kakeya max op} above is taken over all cylindrical tubes $P$ that contain the point $x$ and are oriented in the direction $\omega$.  The tubes are taken to be of arbitrary length $\ell$ and have circular cross-section of arbitrary radius $r$, with $r \leq \ell$. 

If $\Omega$ is a set with nonempty interior, then due to the existence of Kakeya sets with $(d+1)$-dimensional Lebesgue measure zero~\cite{Besicovitch}, $D_{\Omega}$ and $M_{\Omega}$ are unbounded as operators on $L^p(\mathbb{R}^{d+1})$ for $p\in [1,\infty)$.  More generally, if $\Omega$ admits Kakeya-type sets, then both these operators are unbounded on $L^p(\mathbb{R}^{d+1})$ for $p \in [1,\infty)$. Indeed, a standard argument shows that for any tube $P$ of length $\ell$ oriented along a unit vector $\omega$, 
\[\frac{1}{2A_0 \ell} \int_{-A_0\ell}^{A_0 \ell} 1_{P}(x + t\omega) \, dt \geq \frac{\text{length of } P}{2A_0 \ell} = \frac{1}{2A_0} \quad \text{ for all } x \in A_0P. \] 
If $P$ is chosen to be one of the tubes that constitute the set $E_N$ defined as in \eqref{Kakeya-type condition},  the inequality above implies that $M_{\Omega}1_{E_N}(x) \geq D_{\Omega}1_{E_N}(x) \geq c_0 = (2A_0)^{-1} > 0$ for any $x \in E_N^{\ast}(A_0)$. Hence
\begin{equation} \label{Kakeya sets and D} ||M_{\Omega}||_{p \rightarrow p} \geq ||D_{\Omega}||_{p \rightarrow p} \geq \frac{c_0||1_{E_{N}^{\ast}(A_0)}||_p}{||1_{E_N}||_p} \geq c_0 \left(\frac{|E_N^{\ast}(A_0)|}{|E_N|} \right)^{\frac{1}{p}}.\end{equation} 
If $\Omega$ admits Kakeya-type sets, Definition \ref{Kakeya-type set} ensures that the sets $E_N$ can be chosen so that the right hand side approaches infinity as $N \rightarrow \infty$ for $p \in [1, \infty)$. This establishes the claimed unboundedness of both $M_{\Omega}$ and $D_{\Omega}$.


\subsection{Results}\label{results}

All results to date suggest that the direction set $\Omega$ will admit Kakeya-type sets if it is of suitably large size.  Of course, the notion of ``size" here begs clarification.  For us, a direction set will be small when it is sufficiently \textit{lacunary}.  Precise definitions of lacunarity needed in this paper are deferred to Section \ref{lacunarity section}, but the general idea is easy to describe. In one dimension, a relatively compact set $\{a_i\}$ is lacunary of order 1 if there is $a \in \mathbb R$ and some positive $\lambda <1$ such that $|a_{i+1} - a|\leq \lambda |a_i-a|$ for all $i$. Such a set has traditionally been referred to as a lacunary sequence with lacunarity constant (at most) $\lambda$.  A lacunary set of order 2 consists of a single (first-level) lacunary sequence $\{a_i\}$, along with a collection of disjoint (second-level) lacunary sequences; a second-level sequence is squeezed between two adjacent elements of $\{a_i\}$. The lacunarity constants of all sequences are uniformly bounded by some positive $\lambda < 1$. Roughly speaking, a set on the real line is lacunary of finite order if there is a decomposition of the real line by points of a lacunary sequence such that the restriction of the set to each of the resulting subintervals is lacunary of lower order. All lacunarity constants implicit in the definition are assumed to be uniformly bounded away from unity. 
A set is then said to be \textit{sublacunary} if it does not admit a finite cover by lacunary sets of finite order.

\begin{figure}[h!]
\setlength{\unitlength}{0.8mm}
\begin{picture}(0,0)(-40,100)

	\put(0,0){\special{sh 0.99}\ellipse{3}{3}}
	\dottedline{1}(72,9)(72,1)
	\allinethickness{.08cm}\dottedline{2}(0,-1)(100,-1)
	\path(0,0)(40,85)
	\path(0,0)(80,50)
	\path(0,0)(98,15)

	\put(105,13){\Large\shortstack{$\lambda^{3}$}}
	\put(87,50){\Large\shortstack{$\lambda^{2}$}}
	\put(45,82){\Large\shortstack{$\lambda$}}

	\allinethickness{0.03cm}\path(0,0)(15,95)
	\path(0,0)(25,92)
	\path(0,0)(32,89)
	\dottedline{1}(26,70)(30,68)

	\put(12,98){\tiny\shortstack{$\lambda+\gamma^{k_1}$}}
	\put(24,94){\tiny\shortstack{$\lambda+\gamma^{k_1+1}$}}
	\put(31,90){\tiny\shortstack{$\lambda+\gamma^{k_1+2}$}}

	\path(0,0)(59,72)
	\path(0,0)(67,65)
	\path(0,0)(73,59)
	\dottedline{1}(57,45)(61,40)

	\put(59,74){\tiny\shortstack{$\lambda^{2}+\gamma^{k_2}$}}
	\put(68,66){\tiny\shortstack{$\lambda^{2}+\gamma^{k_2+1}$}}
	\put(74,59){\tiny\shortstack{$\lambda^{2}+\gamma^{k_2+2}$}}

	\path(0,0)(88,39)
	\path(0,0)(93,30)
	\path(0,0)(96,23)
	\dottedline{1}(73,16)(74,13)

	\put(90,39){\tiny\shortstack{$\lambda^{3}+\gamma^{k_3}$}}
	\put(95,30){\tiny\shortstack{$\lambda^{3}+\gamma^{k_3+1}$}}
	\put(98,23){\tiny\shortstack{$\lambda^{3}+\gamma^{k_3+2}$}}


\end{picture}
\vspace{8.5cm}
\caption{\label{Fig:lacunary set of order 2} A direction set in the plane, represented as a collection of unit vectors, with parameters $0 < \gamma < \lambda < 1/2$. The set of angles made by these vectors with the positive horizontal axis is $\{(\lambda^{j}+\gamma^{k}) : k\geq j\}$, which is lacunary of order 2.}
\end{figure}
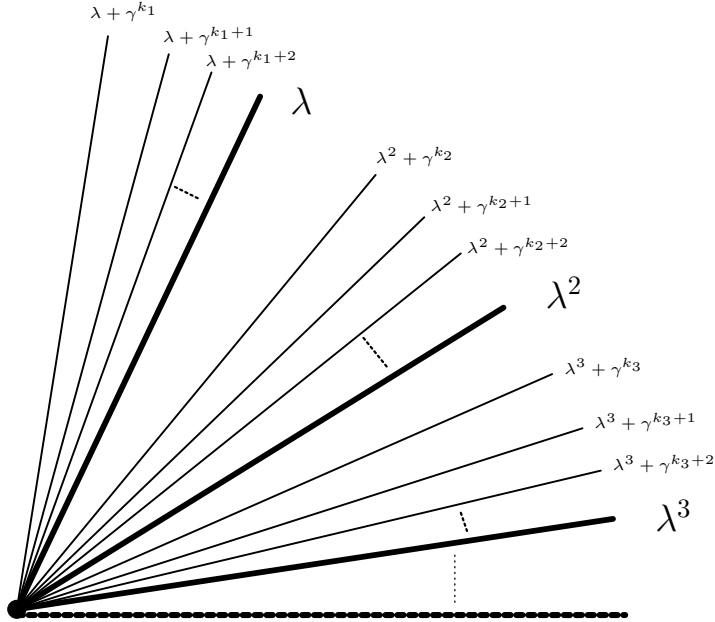
 In higher dimensions, the idea of lacunarity is not immediately clear.  For $d\geq 2$, Nagel, Stein, and Wainger \cite{NagelSteinWainger} considered lacunary sets of the form $\Omega = \{(\theta_j^{m_1},\ldots, \theta_j^{m_d}) : j \geq 1 \}$, where $0<m_1<\cdots <m_d$ are fixed constants and $\{ \theta_j \}$ is a lacunary sequence with lacunarity constant $0<\lambda<1$, i.e., $0<\theta_{j+1}\leq\lambda \theta_j$.  For such direction sets $\Omega$, they showed that the operator $D_{\Omega}$ is bounded on all $L^p(\mathbb R^d)$, $1<p\leq\infty$. A two-dimensional extension of this result was obtained by Sj\"ogren and Sj\"olin~\cite{SjogrenSjolin}, where they formulated a generalized notion of lacunarity that was to prove the basis of a body of subsequent work. Carbery \cite{Carbery} considered coordinate-wise lacunary sets of the form $\Omega = \{(r^{k_1},\ldots,r^{k_d}) : k_1,\ldots,k_d\in\mathbb{Z}^+\}$ for some $0<r<1$, and showed that the corresponding directional maximal operator is bounded on all $L^p(\mathbb R^d)$, $1<p\leq\infty$.  Very recently, Parcet and Rogers \cite{ParcetRogers} have generalized an almost-orthogonality result of Alfonseca \cite{Alfonseca}, building on previous work of Alfonseca, Soria, and Vargas~\cite{AlfonsecaSoriaVargas}, Carbery \cite{Carbery}, Nagel, Stein, and Wainger~\cite{NagelSteinWainger}, to recover these results and to extend the $L^p$-boundedness of $D_{\Omega}$, $1<p\leq\infty$, to sets $\Omega$ that are lacunary of finite order in a broader sense. 

On the other hand, direction sets $\Omega$ that are sufficiently sublacunary have been shown to admit Kakeya-type sets, and thus lead to unbounded directional maximal operators, per the argument at the end of Section \ref{background}.  There is an extensive body of work in this direction \cite{{DuoandikoetxeaVargas}, {Vargas}, {Katz}, {BatemanKatz}}, authored in part by Duoandikoetxea, Vargas, Bateman and Katz. A fundamental and representative example, examined by Bateman and Katz \cite{BatemanKatz}, is a direction set in the plane where the slopes belong to the standard middle-third Cantor set. Combining the aforementioned positive results with strategies developed to treat the reverse direction, the $L^p$-boundedness of $D_{\Omega}$ and $M_{\Omega}$ has been completely characterized in the plane by Bateman \cite{Bateman} and remains one of the most definitive results in the subject.  In higher dimensions, the authors~\cite{KrocPramanik} have recently constructed Kakeya-type sets over certain Cantor-type subsets of a curve on the sphere $\mathbb S^d$, $d\geq 2$.  What seems clear is that any attempt to bridge the gap between these negative and positive results in general dimensions will require a precise and appropriate notion of finite order lacunarity.  We will provide our definition in Definition~\ref{definition of lacunary direction sets}.  For now, we state our main results with the precise terminology deferred until Section \ref{lacunary direction set section}.

\begin{theorem}\label{MainThm1}
Let $d \geq 2$. If the direction set $\Omega \subseteq \mathbb R^{d+1}$ is sublacunary in the sense of Definition \ref{definition of lacunary direction sets}, then $\Omega$ admits Kakeya-type sets. 
\end{theorem}

Combining this result with others from the literature (most notably \cite{{Bateman}, {Alfonseca}, {ParcetRogers}}), we obtain the following necessary and sufficient condition for Kakeya-type sets to exist. 
\begin{theorem}\label{MainThm2}
	For any dimension $d \geq 1$, the following are equivalent:
	\begin{enumerate}
	\item[(1)] The direction set $\Omega \subseteq \mathbb R^{d+1}$ is sublacunary in the sense of Definition \ref{definition of lacunary direction sets}.
	\item[(2)] The set of directions $\Omega$ admits Kakeya-type sets in the sense of Definition \ref{Kakeya-type set}.
	\item[(3)] The maximal operators $D_{\Omega}$ and $M_{\Omega}$ defined in \eqref{max op} and \eqref{Kakeya max op} are unbounded on $L^p(\mathbb{R}^{d+1})$ for every $p\in(1,\infty)$.
	\end{enumerate}
\end{theorem}

To clarify, the implication (3) $\implies$ (1) for $d = 1$ is in \cite{Alfonseca}, expanding on the work started in ~\cite{{NagelSteinWainger},  {SjogrenSjolin}, {Carbery}, {AlfonsecaSoriaVargas}}. For $d \geq 2$, this is due to \cite{ParcetRogers}, as we will see in Theorem \ref{boundedness theorem for directional max op}. The proof of (1) $\implies$ (2) is the content of \cite{Bateman} for $d = 1$ and of Theorem \ref{MainThm1} for $d \geq 2$. The implication (2) $\implies$ (3) is established in the argument presented in the paragraph of \eqref{Kakeya sets and D} in all dimensions. 

Some of the implications above are known to admit stronger variants. For instance, (2) implies (3) even when $p=1$, as the argument leading to \eqref{Kakeya sets and D} shows. Further, it is not necessary to know that the operator $D_{\Omega}$ is unbounded on all $L^p(\mathbb{R}^{d+1})$, $p\in(1,\infty)$, in order to conclude that $\Omega$ is sublacunary.  We will prove in Section \ref{Appendix section} that the weaker requirement
\[ {\textit{(3') The maximal operator $D_{\Omega}$ is unbounded on $L^p(\mathbb R^{d+1})$ for some $p\in(1,\infty)$}}, \] suffices to establish (1). Thus $D_{\Omega}$ enjoys an interesting dichotomy in that it is either bounded on all or none of the Lebesgue spaces $L^p$ with $p \in (1, \infty)$.  

\subsection{Structure of the proofs and layout of the paper} \label{layout section} 

The paper is divided into ten main sections, not counting the introduction.  In Section 2 we define (admissible) finite order lacunarity and sublacunarity, consider several instructive and famous examples of such sets, and prove the implication (3) $\Rightarrow$ (1) of Theorem~\ref{MainThm2}.  

Section 3 begins the program of proving the implication (1) $\Rightarrow$ (2) of Theorem~\ref{MainThm2}; i.e., of constructing Kakeya-type sets over sublacunary direction sets.  We begin by reviewing the necessary literature on trees and how they can be used to encode subsets of Euclidean space.  The so-called \textit{splitting number} of a tree, as defined in \cite{Bateman}, is then shown to be the critical concept that allows us to recast the notion of (admissible) finite order lacunarity of a set into an equivalent and more tractable form for the purposes of our proof.  We use this language of trees in Section 4 to extract a convenient subset of an arbitrary sublacunary direction set, denoted by $\Omega_N$.  Section 5 explores the geometry of the intersection of two tubes and the implications of this geometry for the structure of trees encoding the sets of orientations and  positions of a given collection of thin $\delta$-tubes.  

\begin{figure}[h!]
\setlength{\unitlength}{0.8mm}
\begin{picture}(-50,0)(-7,13)

	\put(0,0){\special{sh 0.01}\ellipse{10}{10}}
	\put(-2.5,-1.5){\shortstack{$\S 2$}}
	\path(20,-7)(60,-7)(60,7)(20,7)(20,-7)
	\put(27,1){\shortstack{Theorem 1.3:}}
	\put(30,-4){\shortstack{(3) $\Rightarrow$ (1)}}

	\put(0,-20){\special{sh 0.01}\ellipse{10}{10}}
	\put(-2.5,-21.5){\shortstack{$\S 3$}}
	\put(30,-20){\special{sh 0.01}\ellipse{10}{10}}
	\put(27.5,-21.5){\shortstack{$\S 4$}}
	\put(30,-40){\special{sh 0.01}\ellipse{10}{10}}
	\put(27.5,-41.5){\shortstack{$\S 5$}}
	\put(60,-30){\special{sh 0.01}\ellipse{10}{10}}
	\put(57.5,-31.5){\shortstack{$\S 6$}}
	\put(110,-50){\special{sh 0.01}\ellipse{10}{10}}
	\put(106.5,-51.5){\shortstack{$\S 10$}}
	\put(90,-30){\special{sh 0.01}\ellipse{10}{10}}
	\put(87.5,-31.5){\shortstack{$\S 9$}}
	\put(90,-10){\special{sh 0.01}\ellipse{10}{10}}
	\put(87.5,-11.5){\shortstack{$\S 8$}}
	\put(90,10){\special{sh 0.01}\ellipse{10}{10}}
	\put(87.5,8.5){\shortstack{$\S 7$}}
	\put(140,-30){\special{sh 0.01}\ellipse{10}{10}}
	\put(136.5,-31.5){\shortstack{$\S 11$}}
	\path(130,-7)(170,-7)(170,7)(130,7)(130,-7)
	\put(137,1){\shortstack{Theorem 1.3:}}
	\put(140,-4){\shortstack{(1) $\Rightarrow$ (2)}}

	\thicklines
	\path(6,0)(19,0)
	\path(17,1)(19,0)(17,-1)
	\path(0,-6)(0,-14)
	\path(-1,-12)(0,-14)(1,-12)
	\path(6,-20)(24,-20)
	\path(22,-19)(24,-20)(22,-21)
	\path(36,-20)(55,-27)
	\path(52.5,-27.5)(55,-27)(53.5,-25)
	\path(30,-26)(30,-34)
	\path(29,-32)(30,-34)(31,-32)
	\path(36,-40)(55,-33)
	\path(52.5,-32.5)(55,-33)(53.5,-35)
	\path(64,-25.5)(85,7)
	\path(82.25,5.75)(85,7)(85,4)
	\path(65.5,-28)(84.5,-12)
	\path(81.5,-12.5)(84.5,-12)(83.5,-15)
	\path(66,-30)(84,-30)
	\path(82,-31)(84,-30)(82,-29)
	\path(65.5,-32)(104,-50)
	\path(102.5,-47.5)(104,-50)(101,-50.5)	
	\path(96,9)(129,0)
	\path(127.5,2)(129,0)(126.5,-1)
	\path(96,-11)(134.5,-27)
	\path(133,-25)(134.5,-27)(132,-27.5)
	\dottedline{0.75}(90,-16)(90,-24)
	\path(89,-22)(90,-24)(91,-22)
	\dottedline{1}(93,-15)(107,-44.5)
	\path(104.5,-43)(107,-44.5)(107.5,-41.5)
	\path(96,-30)(134,-30)
	\path(132,-29)(134,-30)(132,-31)
	\path(116,-48)(134.5,-33.5)
	\path(133.5,-36)(134.5,-33.5)(131.5,-34)
	\path(142,-24.5)(150,-8)
	\path(151,-11)(150,-8)(147,-9.5)

\end{picture}
\vspace{5.5cm}
\caption{\label{Fig: outline of paper} Diagram illustrating the approximate dependence structure between sections in this paper with respect to the proof of Theorem~\ref{MainThm2}.  Dotted arrows indicate a dependence in terms of definitions and notation only.}
\end{figure}
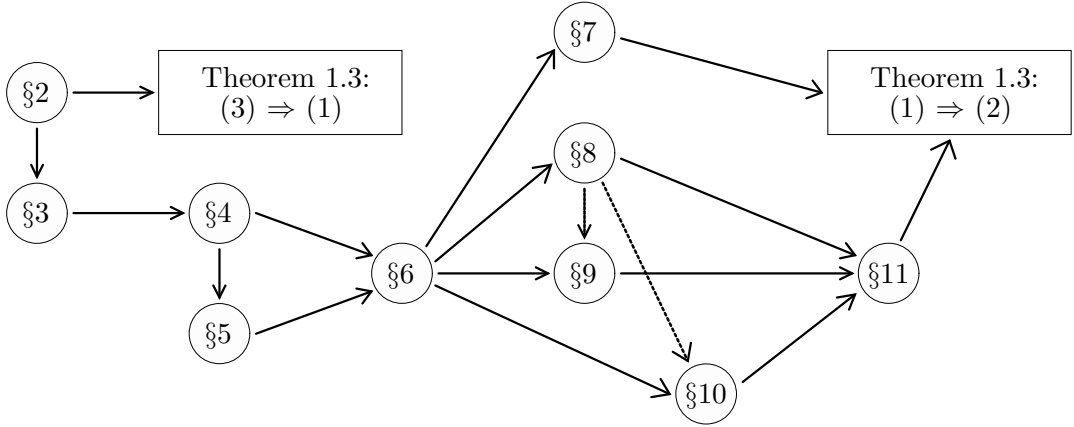

Section 6 combines results from the previous two sections to describe the actual mechanism we use to assign slopes in $\Omega_N$ to $\delta$-tubes affixed to a prescribed set of points in Euclidean space.  Here, we also reformulate the implication (1) $\Rightarrow$ (2) of Theorem~\ref{MainThm2} in terms of quantitative upper and lower bounds on the sizes of a typical Kakeya-type set $E_N$ and its principal dilate $E^*_N(A_0)$ as described in Definition~\ref{Kakeya-type set} (see Proposition \ref{main theorem reformulated}).  From here, the paper splits into more or less two disjoint expositions, each one charged with establishing one of these two probabilistic and quantitative bounds.

In Section 7 we prove the quantitative upper bound previously prescribed using an argument similar to~\cite{Bateman}.  Sections 8, 9, 10, and 11 combine to establish the corresponding lower bound.  Section 11 details the actual estimation, utilizing all the smaller pieces developed in Sections 8, 9, and 10.  These three sections revolve around a central theme of ideas, notably the structure imposed on the position and slope trees of a collection of two, three, or four $\delta$-tubes, certain pairs of which are required to intersect at a given location in space.

The framework of this paper is the same as in \cite{{Bateman}, {BatemanKatz}}, and bears the closest resemblance to \cite{KrocPramanik}. In particular, given a sublacunary direction set, our goal is to construct a family of tubes, all of which originate from (or said to be rooted in) the hyperplane $\{0\} \times [0,1)^d$, after an appropriate coordinate transformation. For a given root position, a slope from $\Omega_N$ is assigned to it using a random mechanism that preserves heights and lineages of both source (root) and image (slope) within their respective trees. The quantitative lower and upper bounds \eqref{random Kakeya lower bound} and \eqref{random Kakeya upper bound} formulated in Proposition \ref{main theorem reformulated} ensure that the random set thus constructed is of Kakeya-type with positive probability. Of the two bounds  \eqref{random Kakeya lower bound} and \eqref{random Kakeya upper bound}, the first is the most significant contribution of this paper. More precisely, the issue is the following. A large lower bound on a union of tubes follows if they do not have significant pairwise overlap among themselves, i.e. if the total size of pairwise intersections is small. In dimension two, a good upper bound on this intersection size was available uniformly in every sticky slope assignment. The counting argument that provided this bound is not transferable to general dimensions, but it is still possible to obtain the desired bound with large probability. A probabilistic statement similar to but not as strong as \eqref{random Kakeya lower bound} can be derived relatively easily via an estimate on the first moment of the total size of random pairwise intersections. Unfortunately, this is still not sharp enough to yield the disparity in the sizes of the tubes and their translated counterparts necessary to claim the existence of a Kakeya-type set. To strengthen the bound, we need a second moment estimate on the pairwise intersections. Both moment estimates share some common features; for instance, 
\begin{enumerate}[-]
\item Euclidean distance relations between roots and slopes of two intersecting tubes,
\item interplay of the above with the relative positions of the roots and slopes within the respective trees that they live in, which affects the slope assignments.  
\end{enumerate}  
However, the technicalities are far greater for the second moment compared to the first, requiring a study of pairwise intersections among three or four tubes, not just two. The above-mentioned points appear in a somewhat simplied form in \cite{KrocPramanik}, where the authors considered a special case of a direction set $\Omega$ obtained as a Cantor-type subset of a curve. There the direction tree had a richer structure, and as a consequence geometric and probabilistic estimates were simpler. The generality of this paper involved in handling arbitrary sublacunary direction sets gives rise to substantial technical challenges, necessitating the analysis carried out in Sections 8-11.     

\section{Finite order lacunarity} \label{lacunarity section}

The concept of finite order lacunarity is ubiquitous, and recognized as fundamental in the study of planar Kakeya-type sets and associated directional maximal operators. It is no surprise that it continues to play a similar central role in this article. The existing literature on the subject embodies several different notions of Euclidean lacunarity both in single and general dimensions, see in particular \cite{Bateman, Carbery, ParcetRogers, SjogrenSjolin}. The present section is devoted to a discussion of the definitions to be used in the remainder of the paper. The concepts introduced here will be revisited in Section \ref{subsection : Lacunarity on Trees}, using the language of trees. The interplay of these two perspectives is essential to the proofs of Theorems \ref{MainThm1} and \ref{MainThm2}.  

\subsection{Lacunarity on the real line}
\begin{definition}[Lacunary sequence] \label{defn: Lacunary sequence in R}
Let $A = \{a_1,a_2,\ldots\}$ be an infinite sequence of points contained in a compact subset of $\mathbb{R}$. Given a constant $0 < \lambda < 1$, we say that $A$ is a {\em{lacunary sequence}} converging to $\alpha$ with {\em{constant of lacunarity}} at most $\lambda$, if
\begin{equation}\label{lacunarity constant}
	|a_{j+1}-\alpha| \leq \lambda |a_j-\alpha| \quad \text{ for all } j \geq 1.\nonumber
\end{equation}
\end{definition} 
\begin{definition}[Lacunary sets] \label{defn: Lacunary sets}
In $\mathbb{R}$, a {\em{lacunary set of order 0}} is a set of cardinality at most 1, i.e., either empty or a singleton.  Recursively, given a constant $0<\lambda<1$ and an integer $N \geq 1$, we say that a relatively compact subset $U$ of $\mathbb{R}$ is a {\em{lacunary set of order at most $N$}} with lacunarity constant at most $\lambda$, and write $U \in \Lambda(N;\lambda)$, if there exists a lacunary sequence $A$ with lacunarity constant $\leq \lambda$ with the following properties:
\begin{enumerate}[-]
\item $U \cap [\sup(A), \infty) = \emptyset$, $U \cap (-\infty, \inf(A)] = \emptyset$, 
\item For any two elements $a, b \in A$, $a < b$ such that $(a, b) \cap A = \emptyset$, the set $U \cap [a, b) \in \Lambda(N-1, \lambda)$.   
\end{enumerate} 
The {\em{order of lacunarity}} of $U$ is exactly $N$ if $U \in \Lambda(N;\lambda) \setminus \Lambda(N-1;\lambda)$.  A lacunary sequence $A$ obeying the conditions above will be called a {\em{special sequence}} and its limit will be termed a {\em{special point}} for $U$. 
\end{definition} 
For any fixed $N$ and $\lambda$, the class $\Lambda(N;\lambda)$ is closed under containment, scalar addition and multiplication; these properties, summarized in the following lemma, are easy to verify and left to the reader. 
\begin{lemma}\label{lacunarity under linear operations}
Let $U \in \Lambda(N;\lambda)$. Then 
\begin{enumerate}[(i)]
\item $V \in \Lambda(N;\lambda)$ for any $V \subseteq U$. 
\item $c_1U + c_2 \in \Lambda(N;\lambda)$ for any $c_1 \ne 0$, $c_2 \in \mathbb R$. 
\end{enumerate} 
\end{lemma} 
The sets of interest to us are those that are generated by finite unions of sets of the form described in Definition \ref{defn: Lacunary sets}. 
\begin{definition}[Admissible lacunarity of finite order and sublacunarity] \label{defn: Admissible finite order lacunarity} 
We say that a relatively compact set $U \subseteq \mathbb{R}$ is an {\em{admissible lacunary set of finite order}} if there exist a constant $0<\lambda<1$ and integers $1\leq N_1,N_2<\infty$ such that $U$ can be covered by $N_1$ lacunary sets of order $\leq N_2$, each with lacunarity constant $\leq \lambda$.  If $U$ does not satisfy this criterion, we call it {\em{sublacunary}}.
\end{definition} 
\subsubsection{Examples} \label{1d lacunary examples section}
\begin{enumerate}[(a)]
\item A standard example of a lacunary set of order 1 and lacunarity constant $\lambda \in (0,1)$ is $U = \{\lambda^{j} : j \geq 1\}$, or any nontrivial subsequence thereof. Indeed $U$ is itself a lacunary sequence, and hence its own special sequence. 

A general lacunary set of order 1 need not always be a lacunary sequence. For example $\{2^{-2j} \pm 4^{-2j} : j \geq 1\}$ is lacunary of order 1 relative to the special sequence $\{ 2^{-j} : j \geq 1\}$. Despite this, lacunary sequences are in a sense representative of the class $\Lambda(1;\lambda)$, since any set in $\Lambda(1;\lambda)$ can be written as the union of at most four lacunary sequences with lacunarity constant $\leq \lambda$. By Lemma \ref{lacunarity under linear operations}, the set $\{a \lambda^j + b : j \geq 1 \}$ is lacunary of order at most 1 for any unit vector $(a,b)$. 
\item In general, given an integer $k \geq 1$ and constants $M_1 \leq M_2 \leq \cdots \leq M_k$ with $M_1 \geq \max(2, k-1)$, the set \[ U = \left\{ M_1^{-j_1} +  M_2^{-j_2} +  \cdots + M_k^{-j_k} : 0 \leq j_1 \leq j_2 \leq \cdots \leq j_k \right\} \] is lacunary of order $k$ and has lacunarity constant $\leq M_1^{-1}$. The special sequence can be chosen to be $A = \{ M_1^{-j} : j \geq 1\}$.
\item \label{dyadic rationals example} A set that is dense in some nontrivial interval, however small, is sublacunary. For example, dyadic rationals of the form $\{\frac{k}{2^m} : 0 \leq k < 2^m \}$ for a fixed $m$ can be written as a finite union of lacunary sequences with a given  lacunarity $\lambda$, but the number of sequences in the union grows without bound as $m \rightarrow \infty$. By Lemma \ref{lacunarity under linear operations},  a set that contains an affine copy of $\{\frac{k}{2^m} : 0 \leq k < 2^m \}$ for every $m$ is sublacunary.
\item The set $U = \{ 2^{-j} + 3^{-k} : j, k \geq 0\}$ can be covered by a finite union of sets in $\Lambda(2;\frac{1}{2})$. For instance the two subsets of $U$ where $k \ln 3 \leq (j-1) \ln 2$ and $k \ln 3 \geq j \ln 2$ respectively are each lacunary of order 2, with $\{ 3^{-k} \}$ and $\{ 2^{-j} \}$ being their respective special sequences. The complement, where $(j-1) \ln 2 \leq k \ln 3 \leq j \ln 2$, contains at most one $k$ per $j$, and is a finite union of lacunary sets of order 1.   
\item A slight variation of the above example: $\{2^{-j} + (q_j - 2^{-j})3^{-k} : j, k \geq 0 \}$, where $\{q_j\}$ is an enumeration of the rationals in $[\frac{9}{10}, 1]$, leads to a very different conclusion. This set contains $\{ q_j\}$, and is hence sublacunary, even though the set may be viewed as a special sequence $\{ 2^{-j} \}$ with collections of lacunary sequences converging to every point of it. This example illustrates the relevance of the requirement that the lower order components of $\Lambda(N;\lambda)$ lie in disjoint intervals of $\mathbb R$.   
\item Given any $0 < \lambda < 1$ and $m > 0$, there is a constant $C = C(\lambda, m)$ such that for any unit vector $(a,b)$, the set $U_{a,b} = \{ a \lambda^{j} + b \lambda^{mj} : j \geq 1\}$ can be covered by $C$ sets in $\Lambda(1;\lambda)$. We leave the verification of this to the reader, but will provide a general statement along these lines in Section \ref{Euclidean examples section}, see example (\ref{NSW example}).  
\item \label{Carberyd2} Given any $0 < \lambda < 1$, $m \in \mathbb Q \cap (0, \infty)$, there is a constant $C = C(\lambda, m)$ such that for any unit vector $(a,b)$, the set \[U_{a,b} = \{u_{jk}= a \lambda^j + b \lambda^{mk} : j, k \geq 1 \}\] can be covered by at most $C$ lacunary sets of order at most 2. 
This is clear for $(a,b) = (1,0)$ or $(0,1)$, with the order of lacunarity being 1. For $ab \ne 0$, there are four possibilities concerning the signs of $a$ and $b$. We deal with $a > 0$ and $b < 0$, the treatment of which is representative of the general case. The set $U_{a,b}$ is decomposed into three parts:
\begin{align*}
V_{a,b} &= \bigl\{ u_{jk} \in U : a\lambda^{j} + b \lambda^{mk} \geq a \lambda^{j+1} \bigr\},\\  W_{a,b} &= \bigl\{ u_{jk} \in U : a\lambda^{j} + b \lambda^{mk} < b \lambda^{m(k+1)} \bigr\}, \\  Z_{a,b} &= U_{a,b} \setminus \bigl[V_{a,b} \cup W_{a,b} \bigr].    
\end{align*} 
Then for every fixed $j$, the set $V_{a,b} \cap [a \lambda^{j+1}, a \lambda^{j})$ is an increasing lacunary sequence with constant $\leq \lambda^m$, converging to $a \lambda^j$. An analogous conclusion holds for $W_{a,b} \cap [b \lambda^{mk}, b \lambda^{m(k+1)})$. Thus $V_{a,b}$ and $W_{a,b}$ are both lacunary of order 2, with their special sequences being $A = \{a \lambda^j \}$ and $A = \{ b \lambda^{mk}\}$ respectively. For $u_{jk} \in Z_{a,b}$, the indices $j$ and $k$ obey the inequality \[ -\frac{a}{b}(1 - \lambda) < \lambda^{mk-j} \leq -\frac{a}{b} (1 - \lambda^m)^{-1}. \] Since $m$ is rational, the values of $mk-j$ range over rationals of a fixed demonimator (same as that of $m$). The inequality above therefore permits at most $C$ solutions of $mk-j$, the constant $C$ depending on $\lambda$ and $m$, but independent of $(a,b)$. Thus $Z_{a,b}$ is covered by a $C$-fold union of subsets, each consisting of elements $u_{jk} = \lambda^{j}(a + b \lambda^{mk-j})$ for which $mk-j$ is held fixed at one of these solutions. Each such set is lacunary of order 1 with lacunarity $\leq \lambda$. 
\end{enumerate}

\subsubsection{Non-closure of finite order lacunarity under algebraic sums} \label{counterexample section}
An important aspect of the class of admissible lacunary sets of finite order is that it is not closed under set-algebraic operations, as we establish in the example furnished below. This feature, perhaps initially counterintuitive, is the main inspiration for the definition of higher dimensional lacunarity provided in the next subsection. 

{\em{Example: }} Let $N_j \nearrow \infty$ be a fast growing sequence, and $M_j = 2^{m_j}$ a slower growing one, so that \begin{equation} \label{M_j and N_j}M_j < N_j - N_{j-1}. \end{equation} For instance, $N_j = 2^{j^2}$ and $M_j = 2^j$ will do. For every $j \geq 1$ and $1 \leq k \leq M_j = 2^{m_j}$, set $q_{jk} = 2^{-N_j} (1 + k2^{-m_j})$, and define 
\[ U_j = \{2^{-N_j + k} + q_{jk} : 1 \leq k \leq M_j \}, \quad U = \bigcup_{j=1}^{\infty} U_j, \quad V = \{-2^{-j} : j \geq 1 \}.  \]    
An element of $U_j$ of the form $2^{-N_j + k} + q_{jk}$ lies in the dyadic interval $[2^{-N_j + k}, 2^{-N_j  + k + 1})$, and for a given $k$, is the only element of $U_j$ in this interval. Further, $U_j \subseteq [2^{-N_j + 1}, 2^{-N_j + M_j + 1})$, hence by the relation \eqref{M_j and N_j}, $U_j \cap U_{j'} = \emptyset$ if $j \ne j'$. Thus $U \in \Lambda(1;\frac{1}{2})$, since for any $i \geq 1$, the set $U \cap [2^{-i}, 2^{-i+1})$ is either empty or a single point. Clearly $V$ is a lacunary sequence, hence  $V \in \Lambda(1;\frac{1}{2})$ as well, being its own special sequence. On the other hand, 
\[ U + V \supseteq \bigcup_{j=1}^{\infty} \bigl\{q_{jk} : 1 \leq k \leq M_j \bigr\}.\]
In other words, $U+V$ contains an affine copy of the dyadic rationals of the form $\{k2^{-m_j} : 1 \leq k \leq 2^{m_j} \}$ in $[0,1]$, for every $j$. As discussed in example (\ref{dyadic rationals example}) in Section \ref{1d lacunary examples section}, $U+V$ is sublacunary.  

The counterexample above illustrates the sensitivity of lacunarity on ambient coordinates, and precludes a higher dimensional generalization of this notion that relies on componentwise extension. For instance, the two-dimensional set $U \times V$ (with $U$, $V$ as above) has lacunary coordinate projections in the current system of coordinates, but there are other directions of projection, for instance the line of unit slope, along which the projection of this set is much more dense. 

\subsection{Finite order lacunarity in general dimensions} \label{lacunarity general dimensions section}
Let $\mathbb V$ be a $d$-dimensional affine subspace of an Euclidean space $\mathbb R^n$, $n \geq d$. Given a base point $\mathbf a$ of $\mathbb V$ and an orthonormal basis $\mathcal{B} = \{\mathbf{v}_1,\ldots,\mathbf{v}_{d}\}$ of the linear subspace $\mathbb V - a$, 
we define the projection maps
\begin{equation}\label{projection map}
	\pi_{j} = \pi_j[\mathbf a, \mathcal B]: \mathbb V \rightarrow\mathbb{R},\quad\text{ via }\quad x = \mathbf a + \sum_{j=1}^{d} x_j \mathbf v_j \rightarrow x_j = \pi_{j}(x), \quad 1\leq j\leq d.
\end{equation}
\begin{definition}[Admissible lacunarity and sublacunarity of Euclidean sets] \label{defn : Finite order lacunarity general dimensions} 
Let $U$ be a relatively compact subset of $\mathbb V$.
	\begin{enumerate}
		\item[(i)] We say that the set $U$ is {\em{admissible lacunary of order at most $N$ (as an Euclidean subset of $\mathbb V$)}} with lacunarity constant at most $\lambda < 1$ if there exists an integer $R \geq 1$ satisfying the following property: for any choice of basis $\mathcal{B}$ and base point $\mathbf a$, and each $1\leq j\leq d$, the projected set \[\pi_j(U) = \{\pi_j(x) : x \in U \} \subseteq \mathbb R \] can be covered by $R$ members of $\Lambda(N;\lambda)$, with the class $\Lambda(N;\lambda)$ as described in Definition \ref{defn: Admissible finite order lacunarity}. The projection $\pi_j$ depends on $\mathbf a$ and $\mathcal B$ via \eqref{projection map}. The collection of sets $U$ that obey these conditions for a given choice of $N, \lambda$ and $R$ will be denoted by $\Lambda_d(N, \lambda, R; \mathbb V)$. 
		\item[(ii)] The set $U$ is called {\em{sublacunary in $\mathbb V$}} if it is not admissible lacunary of finite order; i.e., if for any $\lambda < 1$ and integers $N, R \geq 1$ there exists a choice of basis $\mathcal B$ and an index $1 \leq j \leq d$ such that $\pi_{j}(U)$ cannot be covered by any $R$-fold union of one-dimensional lacunary sets of order at most $N$ and lacunarity constant at most $\lambda$. 
	\end{enumerate}
\end{definition} 
{\em{Remarks:}} 
\begin{enumerate}[-]
\item An equivalent formulation of the definition of $U \in \Lambda_d(N, R, \lambda; \mathbb V)$ is that for any line $L$ in $\mathbb V$ (and indeed in $\mathbb R^{d+1}$ as we will soon see in Lemma \ref{n and d}), the projection of $U$ onto $L$ is coverable by at most $R$ sets in $\Lambda(N;\lambda)$. 
\item We ask the reader to verify that the choice of base point in $\mathbb V$ is not important in this definition, since $\pi_j[\mathbf a, \mathcal B](U)$ is a translate of $\pi_j[\mathbf a', \mathcal B](U)$ for any $\mathbf a, \mathbf a' \in \mathbb V$. Thus $\pi_j[\mathbf a, \mathcal B](U) \in \Lambda(N;\lambda)$ if and only if $\pi_j[\mathbf a', \mathcal B](U) \in \Lambda(N;\lambda)$. 
\item The definition is also invariant under rotation in $\mathbb R^n$; if $O$ is an orthogonal transformation of $\mathbb R^n$, then $U \in \Lambda_d(N, \lambda, R; \mathbb V)$ if and only if $O(U) \in \Lambda_d(N, \lambda, R; O(\mathbb V))$. 
\item The choice of rotation $\mathcal B$ within $\mathbb V$ is however critical. It is not possible to have necessary and sufficient implications like the ones above for two arbitrary choices of bases $\mathcal B$ and $\mathcal B'$. We provide examples below. Henceforth, we will refer to the choice of a pair $\varphi = (\mathbf a, \mathcal B)$ as a system of coordinates, with the main focus on $\mathcal B$. 
\end{enumerate}
Before proceeding to examples, we check the definition for consistency if $U$ is a subset of several affine subspaces. 
\begin{lemma} \label{n and d}
Let $U \subseteq \mathbb V$ be as above. Then for any choice of $N, R, \lambda$, the set $U \in \Lambda_d(N, R, \lambda; \mathbb V)$ if and only if $U \in \Lambda_n(N, R, \lambda; \mathbb R^n)$. 
\end{lemma}  
\begin{proof}
The ``if'' implication is clear, so we consider the converse. Without loss of generality, we may choose $\mathbb V = \{1 \} \times \mathbb R^{n-1}$. Given any unit vector $\omega = (\omega_1, \cdots, \omega_n) \in \mathbb R^n$ with $0 < |\omega_1| < 1$, let $\mathbb L$ denote the line through the origin in $\mathbb R^n$ pointing in the direction of $\omega$. Let $\mathbb L'$ denote the projection of $\mathbb L$ on $\mathbb V$, so that $\mathbb L' = \{e_1 + s \omega' : s \in \mathbb R \}$, where $e_1$ is the first canonical basis vector in $\mathbb R^n$, and $\omega' = (0, \omega_2, \cdots, \omega_n)$. The desired conclusion follows from the claim that
\begin{equation} \label{pi and pi'}
\text{ the sets } \pi(U) \text{ and }  \pi'(U) \text{ are affine copies of each other,}   
\end{equation}  
where $\pi(U)$ and $\pi'(U)$ denote the scalar projections onto $\mathbb L$ and $\mathbb L'$, measured from the origin and $(1, 0, \cdots, 0)$ respectively. Indeed, Lemma \ref{lacunarity under linear operations} then permits us to extend known lacunarity features of the former directly to the latter.  

To establish \eqref{pi and pi'}, it suffices to note that for any $x \in \mathbb R^n$, 
\[ \pi(x) = (x \cdot \omega) \omega, \quad \text{ and } \quad \pi'(x) = \frac{x \cdot \omega'}{|\omega'|^2} \omega'. \]
The choice of $\mathbb V$, $\omega$ and $\omega'$ yield the relations $(x - y) \cdot \omega = (x - y) \cdot \omega'$ for any $x, y \in U$, hence the above expressions imply that 
\[ |\pi(x) - \pi(y)| = |\omega'| |\pi'(x) - \pi'(y)|,\]
which is the desired conclusion.  
\end{proof} 
\subsubsection{Examples of admissible lacunary and sublacunary sets in $\mathbb R^d$} \label{Euclidean examples section}
\begin{enumerate}[(a)]
\item \label{NSW example} A set of the form considered by Nagel, Stein and Wainger \cite{NagelSteinWainger}, such as 
\begin{equation} \label{NSW sets}  U = \{\gamma(\theta_j) : j \geq 1 \}, \quad \text{ where } \quad \gamma(t) = (t^{m_1},\cdots, t^{m_d}) \end{equation}  is admissible lacunary of order 1. Here $0<m_1<\cdots <m_d$ are fixed constants, and $0<\theta_{j+1}\leq\lambda \theta_j$, for some $0<\lambda<1$ and all $j$. Critical to this verification are the following two properties of $U$ appearing in \cite[Lemma 4]{NagelSteinWainger}: 
\begin{enumerate}[-]
\item There is a constant $C_1 = C_1(m_1, \cdots, m_d)$ obeying the following requirement. For any unit vector $\xi = (\xi_1, \cdots, \xi_d)$ in $\mathbb R^d$, the set $\mathbb N$ of positive integers can be decomposed into $C_1$ disjoint consecutive intervals $\{ \mathbb N_s \}$; for every $s$, there exists  $r(s) \in \{1, \cdots, d\}$ such that 
\begin{equation}  \label{C_1 pieces} \max_{1 \leq r \leq d} |\theta_j^{m_r} \xi_r| = \bigl|\theta_j^{m_{r(s)}} \xi_{r(s)} \bigr| \quad \text{ for all } j \in \mathbb N_s.\end{equation}  
The composition of $\mathbb N_s$ depends on $\xi$. 
\item Further for any $c > 0$, there is a constant $C_2 = C_2(c, m_1, \cdots, m_d)$ independent of $\xi$ and $\mathbb N_s$  so that 
\begin{equation} \label{generic ineq in N_s} \max_{\begin{subarray}{c}r \in \{1, \cdots, d\} \\ r \ne r(s) \end{subarray}} \bigl| \theta_j^{m_r} \xi_r \bigr| < c \bigl| \theta_j^{m_{r(s)}} \xi_{r(s)}\bigr|  \end{equation} 
for all but $C_2$ integers $j \in \mathbb N_s$. 
\end{enumerate}
Assuming these two facts, the claim of lacunarity is established as follows. Using the definition of $\mathbb N_s$ in (\ref{C_1 pieces}), the set $U$ can be decomposed into $C_1$ pieces $U_s$, where $U_s = \{ \gamma(\theta_j) : j \in \mathbb N_s \}$. Fix a constant $R$ such that $2d \lambda^{m_1R-1} < 1$. If $j'>j$ are two integers in $\mathbb N_s$ that are at least $R$-separated and for both of which \eqref{generic ineq in N_s} holds with $c = \frac{1}{2d}$, then 
\begin{equation} \label{NSW example steps}
\begin{aligned}  \bigl| \sum_{r=1}^{d} \xi_r \theta_{j'}^{m_r} \bigr| \leq d \bigl| \xi_{r(s)} \theta_{j'}^{m_{r(s)}} \bigr| &\leq d \bigl(\lambda^{j'-j})^{m_{r(s)}} \bigl| \xi_{r(s)} \theta_{j}^{m_{r(s)}} \bigr| \\ &\leq 2 d \bigl(\lambda^R)^{m_{1}} \bigl| \sum_{r=1}^{d} \xi_r \theta_{j}^{m_r} \bigr| < \lambda \bigl| \sum_{r=1}^{d} \xi_r \theta_{j}^{m_r} \bigr|.  \end{aligned} \end{equation} 
Thus each $U_s$ is the union of at most $R$ lacunary sequences of lacunarity $< \lambda$, together with the $C_2$ points where \eqref{generic ineq in N_s} fails.    
\item \label{Carbery example} A set of the form considered by Carbery \cite{Carbery}, i.e., 
\begin{equation} \label{Carbery sets}  U = \{ \Gamma_{\mathbf k} = (\lambda^{k_1}, \cdots, \lambda^{k_d}) : \mathbf k = (k_1, \cdots, k_d) \in \mathbb N^d\} \end{equation} 
is admissible lacunary of order $d$. We prove this by induction on $d$. The initializing step for $d=2$ has been covered in example (\ref{Carberyd2}) of Section \ref{1d lacunary examples section}. For a general $d$ and after splitting $U$ into $d!$ pieces, we may assume that $k_1 \leq k_2 \leq \cdots \leq k_d$. Given any unit vector $\xi = (\xi_1, \cdots, \xi_d) \in \mathbb R^d$, we write 
\begin{align*} U &= \bigcup_{s=1}^{d} U_s \quad \text{ with } \quad U_s = \{ \Gamma_{\mathbf k} : \mathbf k \in \mathbb N_s^d\}, \text{ where } \\ 
 \mathbb N_s^d &= \bigl\{\mathbf k \in \mathbb N^d:  \bigl| \lambda^{k_s} \xi_s\bigr| = \max_{1 \leq r \leq d} \bigl|\lambda^{k_r} \xi_r \bigr| \bigr\}.  
\end{align*}  
Depending on the signs of $\lambda^{k_s} \xi_s$ and $\Gamma_k \cdot \xi - \lambda^{k_s} \xi_s$, each $\mathbb N_s^d$ can be decomposed into four parts. Their treatments are similar with trivial adjustments, so we focus on the subset of $\mathbb N_s^d$ where 
\[ \lambda^{k_s} \xi_s > 0 \quad \text{ and } \quad \sum_{r \ne s} \lambda^{k_r} \xi_r \geq 0, \]
continuing to call this subset $\mathbb N_s^d$ to ease notational burden. One last splitting is needed; for a constant $A$ to be specified shortly, we write  
\[ \mathbb N_s^d = \mathbb N_{s,1}^d \cup \mathbb N_{s,2}^d, \text{ where } \mathbb N_{s,1}^d = \{\mathbf k \in \mathbb N_s^d : \lambda^{k_s} \xi_s > A |\lambda^{k_r} \xi_r| \text{ for all } r \ne s \}. \]

For $\mathbf k \in \mathbb N_{s,1}^d$, 
\begin{equation} \label{on N_{s1}}\lambda^{k_s} \xi_s \leq \Gamma_{\mathbf k} \cdot \xi < \lambda^{k_s} \xi_s \bigl(1 + dA^{-1} \bigr) < \lambda^{k_s-1} \xi_s,\end{equation}
where the last inequality follows for a suitable choice of $A$. We argue that $\{ \xi_s \lambda^{k_s} : k_s \geq 1 \}$ may be viewed as a special sequence for $\{ \Gamma_{\mathbf k} \cdot \xi : \mathbf k \in \mathbb N_{s,1}^d \}$. Indeed, if $k_s$ is fixed, then \eqref{on N_{s1}} shows that 
\begin{align*} \{ \Gamma_{\mathbf r} \cdot \xi : \mathbf r \in \mathbb N_{s,1}^d \} \cap [\xi_s \lambda^{k_s}, \xi_s \lambda^{k_s-1}) &= \{ \Gamma_{\mathbf r} \cdot \xi : \mathbf r \in \mathbb N_{s,1}^d, \; r_s = k_s \} \\ &\subseteq \xi_s \lambda^{k_s} + \Bigl\{\sum_{r \ne s} \lambda^{k_r} \xi_r : k_r \in \mathbb N, \; r \ne s  \Bigr\}. \end{align*}
By the induction hypothesis, there is a constant $R$ independent of $\xi$ such that the set on the right hand side above is coverable by at most $R$ sets in $\Lambda(d-1;\lambda)$.  Hence $\{ \Gamma_{\mathbf k} : \mathbf k \in \mathbb N_{s,1}^d \}$ is admissible lacunary of order $d$. 

We turn to the complementary set $\mathbb N_{s,2}^d$. After decomposing $\mathbb N_{s,2}^d$ into $(d-1)$ subsets, we may fix an index $\ell$ such that 
\begin{equation}  \label{on N_{s2}}
|\lambda^{k_{\ell}} \xi_{\ell}| \leq \lambda^{k_s} \xi_{s} \leq  A |\lambda^{k_{\ell}} \xi_{\ell}| 
\end{equation}  
on $\mathbb N_{s,2}^d$. Without loss of generality let $\ell \geq s$. The number of possible values of $k_{\ell}- k_s$ obeying \eqref{on N_{s2}} is at most a fixed constant $C$ depending on $A$ (hence $\lambda$ and $d$), but independent of $\xi$. Thus $\mathbb N_{s,2}^d$ may be written as the $C$-fold union of subsets indexed by $c$, where the subset identified by $c$ contains all $\mathbf k \in \mathbb N_{s,2}^d$ with the property that $k_{\ell} - k_s = c \geq 0$. For $\mathbf k$ in such a subset, 
\[  \Gamma_{\mathbf k} \cdot \xi = (\xi_s + \lambda^{c} \xi_{\ell}) \lambda^{k_s} + \sum_{r \ne \ell, s} \lambda^{k_r} \xi_r. \]
Since the number of summands in the linear combination above is $(d-1)$, the induction hypothesis dictates that $\{ \Gamma_{\mathbf k}  : \mathbf k \in \mathbb N_{s,2}^d \}$ is admissible lacunary of order $(d-1)$, completing the proof.         
\item A curve in $\mathbb R^{d}$ is sublacunary. So is a Cantor-like subset of it as considered in \cite{KrocPramanik}.
\item \label{UV example} If $U$ and $V$ are the lacunary sets of order 1 constructed in Section \ref{counterexample section}, the set $U \times V$ is sublacunary. Indeed, after a rotation of angle $\frac{\pi}{4}$  one of the coordinate projections turns out to be a constant multiple of $U+V$.  We have seen in Section \ref{counterexample section} that this last set is sublacunary on $\mathbb R$.  
\end{enumerate}
\subsection{Finite order lacunarity for direction sets} \label{lacunary direction set section}
Given two sets $\Omega_1, \Omega_2 \subseteq \mathbb R^{d+1} \setminus \{0\}$, we say that $\Omega_1 \sim \Omega_2$ if
\[ \left\{ \frac{\omega}{|\omega|} : \omega \in \Omega_1 \right\} = \left\{\frac{\omega}{|\omega|} : \omega \in \Omega_2 \right\}. \]  
The binary relation $\sim$ is clearly an equivalence relation among sets in $\mathbb R^{d+1} \setminus \{0\}$. An equivalence class of $\sim$ is, by definition, a {\em{direction set}}. By a slight abuse of nomenclature, we will refer to a set $\Omega \subseteq \mathbb R^{d+1} \setminus \{0\}$ as a direction set to mean the equivalence class of $\sim$ that contains $\Omega$. Clearly the maximal operators $D_{\Omega}$ and $M_{\Omega}$, as well as the admittance of Kakeya-type sets (as in Definition \ref{Kakeya-type set}), remain unchanged for all members of this equivalence class. 

Certain modifications are necessary to extend the notion of lacunarity from Euclidean sets to direction sets, in view of the latter's scale invariance. Given a direction set $\Omega \subseteq \mathbb R^{d+1} \setminus \{0\}$,  we denote by $\mathcal C_{\Omega}$ the cone generated by this set of directions, namely 
\begin{equation}  \mathcal C_{\Omega} := \{r \omega : r > 0, \; \omega \in \Omega \}. \label{direction cone} \end{equation} 

\begin{definition} \label{definition of lacunary direction sets}
Let $\Omega \subseteq \mathbb R^{d+1} \setminus \{0\}$ be a direction set, with $\mathcal C_{\Omega}$ as in \eqref{direction cone}.  
\begin{enumerate}[(i)]
\item Given an integer $N$ and a positive constant $\lambda < 1$, we say that $\Omega$ is admissible lacunary as a direction set with order at most $N$ and lacunarity at most $\lambda$ if there exists an integer $R$ such that $U \in \Lambda_d(N, \lambda, R;\mathbb V)$ in the sense of Definition \ref{defn : Finite order lacunarity general dimensions}, for every hyperplane $\mathbb V$ at unit distance from the origin and every relatively compact subset $U$ of $\mathcal C_{\Omega} \cap  \mathbb V$. 
\item A direction set $\Omega \subseteq \Omega_0$ failing this property is termed a sublacunary direction set. Thus $\Omega$ is sublacunary as a direction set if for any choice of integers $N, R$ and positive constant $\lambda < 1$ there is a tangential hyperplane $\mathbb V$ of the unit sphere, a relatively compact subset $U$ of $\mathcal C_{\Omega} \cap \mathbb V$ and a line $L$ in $\mathbb V$ such that the projection of $U$ along $L$ cannot be covered by any $R$-fold union of sets in $\Lambda(N;\lambda)$. 
\end{enumerate} 
\end{definition}
\subsubsection{Examples of admissible lacunary and sublacunary direction sets}
\begin{enumerate}[(a)]
\item A direction set $\Omega$ of the form considered by Nagel, Stein and Wainger \cite{NagelSteinWainger},
\[ \Omega = \{ u_j = (\gamma(\theta_j), 1) : j \geq J \} \]  is admissible lacunary of order 1. Here the function $\gamma$ and the sequence $\theta_j$ are as described in example (\ref{NSW example}) of Section \ref{Euclidean examples section}. Thus $\Omega$ is paramterized by the positive constants $m_1 < m_2 < \cdots < m_d$. We set $m_{d+1} = 0$. To verify the claim, we choose $\mathbb V = \{ x \in \mathbb R^{d+1}: x \cdot \eta = 1\}$ for some unit vector $\eta$, so that   
\[ \mathcal C_{\Omega} \cap \mathbb V = \left\{ v_j = \frac{u_j}{ u_j \cdot \eta} : u_j \in \Omega \right\}.  \]
Fix a unit vector $\omega = (\omega', \omega_{d+1}) \in \mathbb R^{d+1}$, and let $\pi_{\omega}$ denote the scalar projection onto $\omega$; i.e., $\pi_{\omega}(v) = v \cdot \omega$. As required by Definitions \ref{definition of lacunary direction sets} and \ref{defn : Finite order lacunarity general dimensions} and in view of Lemma \ref{n and d}, we aim to show that that there is a large constant $R$ (independent of $\mathbb V$) for which any relatively compact subset of $\pi_{\omega}(\mathcal C_{\Omega} \cap \mathbb V)$ can be covered by $R$ members of $\Lambda(1;\lambda)$. By the property \eqref{C_1 pieces} of $\Omega$, we first decompose the integers into a bounded number $C_1$ of disjoint intervals ($C_1$ independent of $\omega$ and $\eta$), on each of which there exists an index $1 \leq r \leq d+1$ such that \begin{equation} \max_{1 \leq i \leq d+1} |\theta_j^{m_i} \eta_i | = |\theta_j^{m_r} \eta_r|. \label{def N_r}\end{equation} 
Let us denote by $\mathbb N_r[\eta]$ one of the subintervals for which  \eqref{def N_r} holds. For $j \in \mathbb N_r[\eta]$,     
\begin{equation}  \label{second round} \pi_{\omega}(v_{j}) - \frac{\omega_{r}}{\eta_{r}} = \frac{\xi \cdot u_j}{\eta_{r} \; (\eta \cdot u_j)}, \quad \text{ where } \quad \xi = (\xi_1, \cdots, \xi_{d+1}) \in \mathbb R^{d+1}\end{equation} 
with $\xi_k = \omega_k \eta_r - \omega_r \eta_k,$ so that $\xi_r = 0$. Our goal is to show that for $j \in \mathbb N_r[\eta]$, the sequence on the right hand side above can be covered by an $R$-fold union of lacunary sequences converging to 0. 

Using \eqref{C_1 pieces} again, we decompose $\mathbb N_r[\eta]$ into at most $C_1$ pieces, of the form $\mathbb N_{rs}[\eta, \xi] = \mathbb N_r[\eta] \cap \mathbb N_s[\xi]$. Since $\xi_r = 0$, we conclude that $\mathbb N_r[\xi] = \emptyset$; hence $s \ne r$. By property \eqref{generic ineq in N_s}, for every $c > 0$, there are at most a bounded number $C_2 = C_2(c)$ indices $j \in \mathbb N_{rs}[\eta, \xi]$ for which at least one of the inequalities 
\begin{equation} \label{generic 2} \max_{i \ne r} |\theta_j^{m_i} \eta_i| < c |\theta_j^{m_r} \eta_r|, \qquad  \max_{i \ne s} |\theta_j^{m_i} \xi_i| < c |\theta_j^{m_s} \xi_s|\end{equation}  fails. 

First suppose $s > r$. Choosing two integers $j,j' \in \mathbb N_{rs}[\eta, \xi]$ with $j' - j \geq R$ for both of which the constraints in \eqref{generic 2} hold, we follow the steps laid out in \eqref{NSW example steps}, obtaining from \eqref{second round}
\begin{align*}
\left[\left| \pi_{\omega}(v_{j'}) - \frac{\omega_{r}}{\eta_{r}} \right|\right] \left[ \left| \pi_{\omega}(v_{j}) - \frac{\omega_{r}}{\eta_{r}} \right| \right]^{-1} &= \frac{\xi \cdot u_{j'}}{\xi \cdot u_{j}} \cdot \frac{\eta \cdot u_j}{\eta \cdot u_{j'}} \\ &\leq \left[\frac{d |\xi_s| \theta_{j'}^{m_s}}{\frac{1}{2}|\xi_{s}| \theta_j^{m_s}}\right] \cdot \left[ \frac{d |\eta_r| \theta_j^{m_r}}{\frac{1}{2} |\eta_r|\theta_{j'}^{m_r}} \right] \\ &\leq 4d^2 \left(\frac{\theta_{j'}}{\theta_j}\right)^{m_s - m_r} \leq 4d^2 \lambda^{R(m_s-m_r)}. 
\end{align*}  
If $R$ is selected large enough to satisfy $4d^2 \lambda^{R(m_s - m_r)} <  \lambda$, then for $j \in \mathbb N_{rs}[\eta, \xi]$ the sequence on the right hand side of \eqref{second round} can be covered by the union of $R$ lacunary sequences converging to zero, excluding the $C_2$ points where \eqref{generic 2} fails. For $s < r$, the same calculation above can be replicated for $j' < j$ with $j' - j < -R$. Thus in this case the sequence in \eqref{second round} grows as $j$ increases, and hence has to be finite by the assumption of relative compactness. Nonetheless, this finite sequence is still coverable by a lacunary sequence going to zero, this time in reverse order of $j$. In either event, we have decomposed the set $\{ \pi_{\omega}(v_j) : j \in \mathbb N_{rs}[\eta, \xi]\}$ into $R$ lacunary sequences of lacunarity $\lambda$, proving the claim. 
\item A direction set of the type studied in \cite{Carbery}, namely 
\[ \Omega = \{ (\Gamma_{\mathbf k}, 1) : 0 \leq k_1 \leq k_2 \leq \cdots \leq k_d \}, \] 
(with $\Gamma_{\mathbf k}$ as in \eqref{Carbery sets}) is admissible lacunary of order $d$. This is proved along lines similar to the example above, using methods already explained in examples (\ref{Carberyd2}) and (\ref{Carbery example}) of Section \ref{1d lacunary examples section} and \ref{Euclidean examples section} respectively; we omit the details here.  
\item A curve in $\mathbb R^{d+1}$ is sublacunary as a direction set.    
\item For sets $U$, $V$ as constructed in Section \ref{counterexample section}, the direction set $\Omega = \{1\} \times U \times V$ is sublacunary, since $U \times V$ is sublacunary as an Euclidean set (see example (\ref{UV example}) in Section \ref{Euclidean examples section}). 
\item Let $\{ q_{\ell} : \ell \geq 1\}$ be an enumeration of the rationals on any nontrivial interval, say on $[\frac{1}{2}, \frac{2}{3}]$. A direction set of the type considered by Parcet and Rogers \cite[Example 1 on page 4]{ParcetRogers}, such as 
\[ \Omega = \{(q_{\ell}2^{-\ell}, 2^{-\ell}, 1) : \ell \geq 1\} \] is sublacunary, even though the one-dimensional coordinate projections in the current coordinate system are lacunary of order at most 1. Choosing $\mathbb V = \{ x_2 = 1 \}$, we find that \[ \mathcal C_{\Omega} \cap \mathbb V = \{(q_{\ell}, 1, 2^{\ell}) : \ell \geq 1 \}. \]
The order of lacunarity of the $x_1$-projection grows without bound as we choose increasingly large compact subsets of $\mathcal C_{\Omega} \cap \mathbb V$.      
\item We also mention another example considered by Parcet and Rogers \cite[Example 2 on page 4]{ParcetRogers}. Given the canonical orthonormal basis $\{e_1, e_2, e_3 \}$ of $\mathbb R^3$, let us fix another orthonormal basis $\{ e_1, e_2', e_3' \}$ with span$\{ e_2, e_3\} = \text{span}\{ e_2', e_3'\}$ and $e_3'$ lying in the first quadrant determined by $e_2$ and $e_3$.  The direction set under consideration is $\Omega = \{u_{\ell} : \ell \geq 1 \}$, where $u_{\ell}$ is a sequence of vectors satisfying $u_{\ell} \cdot e_2' = q_{\ell} u_{\ell} \cdot e_1$ for some enumeration of rationals $\{ q_{\ell} \}$ in an interval. The last condition does not completely specify $u_{\ell}$, hence the direction set so defined is not unique (further restrictions are imposed in \cite{ParcetRogers}), but regardless of any subsequent choice $\Omega$ is sublacunary. Choosing $\mathbb V = \{x_1 =1\}$, we observe that   
\[ \mathcal C_{\Omega} \cap \mathbb V = \left\{ \frac{u_{\ell}}{u_{\ell} \cdot e_1} : \ell \geq 1\right\}. \]
Projecting $\mathcal C_{\Omega} \cap \mathbb V$ in the direction $e_2'$, we find that the projected set is $\{q_{\ell} : \ell \geq 1 \}$, which is not lacunary of finite order.   
\end{enumerate}
\subsection{Boundedness of directional maximal operators} \label{Appendix section}
To confirm that our definition of directional lacunarity of finite order agrees with similar notions existing in the literature, we are able to use the result of Parcet and Rogers \cite{ParcetRogers} to establish the $L^p$-boundedness of directional and Kakeya-Nikodym maximal operators associated to such direction sets $\Omega \subseteq \mathbb R^{d+1}$. Incidentally, this also proves the implication ``(3) $\implies$ (1) " in Theorem \ref{MainThm2}. Let us recall from \eqref{max op} and \eqref{Kakeya max op} the relevant definitions.  
\begin{theorem} \label{boundedness theorem for directional max op}
Given positive integers $N, R$, a positive constant $\lambda < 1$ and any exponent $p \in (1, \infty]$, there exists a positive finite constant $C_p = C_p (N, \lambda, R)$ with the following property. Any admissible lacunary direction set $\Omega \subseteq \mathbb R^{d+1}$ of finite order that obeys Definition \ref{definition of lacunary direction sets}(i) with the prescribed values of $N$, $\lambda$ and $R$ also satisfies 
\begin{equation} \label{M and D bounded} ||M_{\Omega}||_{p \rightarrow p} \leq C_p \quad \text{ and } \quad ||D_{\Omega}||_{p \rightarrow p} \leq C_p. \end{equation}  
\end{theorem}
\begin{proof}
We first argue that the boundedness of $D_{\Omega}$ on any $L^p(\mathbb R^{d+1})$ implies the same for $M_{\Omega}$. Without loss of generality, we may assume that $\Omega \subseteq (-\epsilon, \epsilon)^d \times \{1\}$ for some small constant $\epsilon > 0$. Let us define for any $x \in \mathbb R^{d+1}$ the vectors \[ v_j(x) = x_{d+1}e_j - x_j e_{d+1}, \quad 1 \leq j \leq d, \] 
where $\{e_1, \cdots e_{d+1} \}$ denotes the canonical orthonormal basis in $\mathbb R^{d+1}$.  For $\omega = (\omega_1, \cdots, \omega_d, 1) \in \Omega$, the collection $\{v_1(\omega), \cdots, v_d(\omega) \}$ spans $\omega^{\perp}$. Then 
\begin{align*}
M_{\Omega}f(x) &\leq C_d \sup_{\omega \in \Omega} \sup_{0 < r \leq h} \frac{1}{r^d h} \int_{\begin{subarray}{c}  |t| \leq h \\ |s| \leq r \end{subarray}} \bigl| f \bigl( x - t \omega - \sum_{j=1}^{d} s_j v_j(\omega) \bigr)\bigr| \, dt \, ds \\ 
&\leq C_d \sup_{\omega \in \Omega} \sup_{r > 0} \frac{1}{r^d} \int D_{\Omega} f\bigl(x - \sum_{j=1}^{d} s_j v_j(\omega) \bigr) \, ds \\ 
&\leq C_d \sup_{\omega \in \Omega} \sup_{r > 0} \frac{1}{r^{d-1}} \int D_{\Omega_1} \circ D_{\Omega} f\bigl(x - \sum_{j=2}^{d} s_j v_j(\omega) \bigr) \, ds_2 \cdots, ds_d \\ 
&\leq \cdots \leq C_d D_{\Omega_d} \circ D_{\Omega_{d-1}} \circ \cdots \circ D_{\Omega_1} \circ D_{\Omega} f(x),  
\end{align*}  
where $\Omega_j = \{v_j(\omega) : \omega \in \Omega \}$, $1 \leq j \leq d$. The relation \[ \frac{v_j(\omega) \cdot \xi}{v_j(\omega) \cdot \eta} = \frac{\omega \cdot v_j(\xi)}{\omega \cdot v_j(\eta)} \qquad \text{ for all } \xi, \eta \in \mathbb R \] implies that if $\Omega$ is admissible lacunary of order at most $N$ as a direction set, then so is $\Omega_j$ for every $j$. Thus a bound on the $L^p$ operator norm of $M_{\Omega}$ would follow if the second conclusion (for directional maximal operators) in \eqref{M and D bounded} is known to hold for all such direction sets. We will henceforth concentrate only on $D_{\Omega}$, with $\Omega$ being admissible lacunary of finite order. 

As mentioned before, the $L^p$-boundedness of $D_{\Omega}$ is a restatement of the main result in \cite{ParcetRogers}; we merely supply the connecting details. The proof is by induction. The quantity that forms the basis for induction is related but not identical to the order of lacunarity of the direction set as prescribed in Definition \ref{definition of lacunary direction sets}. To set up the induction parameter and hypothesis, we need a few preparatory steps. Without loss of generality and by a generic rotation if necessary, we may assume that $\Omega$, which is admissible lacunary of order $N = N(\Omega)$ as a direction set, is contained in a fixed small spherical cap in the first orthant that stays away from the coordinate hyperplanes. For each index $1 \leq j \leq d+1$, we set $\mathbb V_j := \{x \in \mathbb R^{d+1} : x_j = 1\}$ and define 
\[ \Theta_j(\Omega) = \mathcal C_{\Omega} \cap \mathbb V_j = \left\{\left(\frac{\omega_1}{\omega_{j}}, \cdots, \frac{\omega_{d+1}}{\omega_j} \right) : \omega = (\omega_1, \cdots, \omega_{d+1}) \in \Omega \right\}. \]
Then $\Theta_j(\Omega) \in \Lambda_d(N,\lambda, R;\mathbb V_j)$, according to Definition \ref{definition of lacunary direction sets}. Appealing to Definition \ref{defn : Finite order lacunarity general dimensions}, let $N_{kj} = N_{kj}(\Omega) \leq N$ and $R_{kj} = R_{kj}(\Omega) \leq R$ be the smallest non-negative integers such that \[\pi_k[\Theta_j(\Omega)] = \left\{ \frac{\omega_k}{\omega_j} : \omega \in \Omega \right\}, \qquad k \ne  j \]  is coverable by at most $R_{kj}$ members of $\Lambda(N_{kj}; \lambda)$. Here $\pi_k$ denotes the projection onto the $k$th coordinate axis in the ambient coordinate system. By decomposing $\Omega$ into at most $Rd^2$ pieces if necessary, we may assume that $R_{kj}=1$ for all $k \neq j$. Set 
\[ \Sigma = \left\{(k, j) : 1 \leq k < j \leq d+1 \right\} \qquad \text{ and } \qquad \Sigma^{\ast} = \Sigma^{\ast}(\Omega) = \left\{(k, j) \in \Sigma : N_{kj} \geq 1 \right\}. \] For a generic rotation mentioned at the beginning of this proof, $\Sigma = \Sigma^{\ast}$.

The induction is based on 
\begin{equation}  \label{def N_0} N_0(\Omega) = N(\Omega) + \sum_{\sigma = (k, j)\in \Sigma^{\ast}} N_{kj}(\Omega).   \end{equation}      
The induction hypothesis is that the second inequality in \eqref{M and D bounded} holds for all $\Omega$ with $N_0(\Omega) \leq N_0$. The initializing step $N_0 = 0$ follows from the one-dimensional Hardy-Littlewood maximal theorem. For a direction set $\Omega$ with $N_0(\Omega) = N_0$ and any $\sigma = (k, j) \in \Sigma^{\ast}$, let $\{\theta_{\sigma, i} : i \geq 1 \}$ be a lacunary (without loss of generality decreasing) sequence with lacunarity constant $\leq \lambda$ that serves as a special sequence for $\pi_k \left[\Theta_j(\Omega) \right]$ (see Definition \ref{defn: Lacunary sequence in R}). As in \cite{ParcetRogers}, we set  
\[ \Omega_{\sigma, i} = \left\{\omega \in \Omega : \theta_{\sigma, i+1} < \frac{\omega_k}{\omega_j} \leq \theta_{\sigma, i} \right\},  \]
and observe that $\Omega_{\sigma, i} \subseteq \Omega$ is admissible lacunary of order at most $N$ as a direction set with the same parameters $R$ and $\lambda$ as before. In particular, $N(\Omega_{\sigma, i}) \leq N(\Omega)$, and $N_{k'j'}(\Omega_{\sigma, i}) \leq N_{k'j'}(\Omega)$ for all $(k',j') \in \Sigma$.  The result of \cite{ParcetRogers} states that 
\[ ||D_{\Omega}||_{p \rightarrow p} \leq C \sup_{\sigma \in \Sigma^{\ast}} \sup_{i \geq 1} ||D_{\Omega_{\sigma, i}}||_{p \rightarrow p}.\] 
(In fact, \cite{ParcetRogers} addresses the generic and nontrivial case of $\Sigma^{\ast} = \Sigma$, but the proof goes through with trivial modifications after a reduction to lower dimensions even when $\Sigma^{\ast} \subsetneq \Sigma$, i.e., if $N_{kj}=0$ for certain pairs $(k,j) \in \Sigma$). From the definition of $\Omega_{\sigma, i}$, we conclude that \[\pi_k \left[\Theta_j(\Omega_{\sigma, i}) \right] = \pi_k \left[ \Theta_j(\Omega)\right] \cap (\theta_{\sigma, i+1}, \theta_{\sigma, i}] \in \Lambda(N_{kj}-1, \lambda)\] for any $\sigma = (k,j) \in \Sigma^{\ast}$; hence $N_{kj}(\Omega_{\sigma, i}) \leq N_{kj}-1$.  It now follows from \eqref{def N_0} that
   \[ N_0(\Omega_{\sigma, i}) \leq N_0(\Omega) - 1, \] 
allowing us to carry the induction forward. 
\end{proof} 


\section{Rooted, labelled trees}

As in \cite{BatemanKatz}, \cite{Bateman} and \cite{KrocPramanik}, the language of rooted, labelled trees continues to be the vehicle of choice for construction of Kakeya-type sets.  We recall the basic terminology of trees and state the relevant facts in Sections~\ref{trees} and \ref{tree encoding} below, referring the reader to our previous work~\cite{KrocPramanik} for a more detailed discussion of some of the stated results, and to \cite{LyonsPeres} for a comprehensive treatise on the subject.

\subsection{The terminology of trees}\label{trees}

A \textit{tree} is defined to be a connected undirected graph with no cycles.  A \textit{rooted, labelled tree} $\mathcal{T}$ is one whose vertex set is a nonempty collection of finite sequences of nonnegative integers such that if $\langle i_1,\ldots,i_n\rangle\in \mathcal{T}$, then
\begin{enumerate}
	\item[(i.)] for any $k$, $0\leq k\leq n$, $\langle i_1,\ldots,i_k\rangle\in \mathcal{T}$, where $k=0$ corresponds to the empty sequence, and
	\item[(ii.)] for every $j\in \{0,1,\ldots,i_n\}$, we have $\langle i_1,\ldots,i_{n-1},j\rangle\in \mathcal{T}$.  
\end{enumerate}
We say that $\langle i_1,\ldots, i_{n-1}\rangle$ is the \textit{parent} of $\langle i_1,\ldots,i_{n-1},j\rangle$ and that $\langle i_1,\ldots,i_{n-1},j\rangle$ is the $(j+1)th$ \textit{child} of $\langle i_1,\ldots,i_{n-1}\rangle$.  A parent-child pair is an edge, and a sequence of connected edges $(e_1, e_2, \ldots)$ is a {\em{ray}}, where by convention we require that the child vertex of $e_i$ agree with the parent vertex of $e_{i+1}$ for all $i \geq 1$.  The empty sequence $\emptyset$ is the designated \textit{root} of the tree $\mathcal{T}$ and all vertices of the form $\langle i_1 \rangle\in\mathcal{T}$ are children of this root.  We let $\partial \mathcal{T}$ denote the collection of all rays in $\mathcal{T}$ of maximal (possibly infinite) length.  For a fixed vertex $v\in\mathcal{T}$, we also define the \textit{subtree of $\mathcal{T}$ generated by the vertex} $v$ to be the maximal subtree of $\mathcal{T}$ with $v$ as the root.

The \textit{height} of the tree is taken to be the supremum of the lengths of all the sequences in the tree.  Further, we define the height $h(\cdot)$ of a vertex to be the length of its identifying sequence.  If the height of a vertex $v$ is equal to $k$, we say that $v$ is a $k$\textit{th generation vertex} of the tree.  The height of the root is always taken to be zero.  If $u$ and $v$ are two vertices in $\mathcal{T}$ that lie along the same ray, with $h(u) > h(v)$, then we say $u$ is a \textit{descendant} of $v$ (or that $v$ is an \textit{ancestor} of $u$), and we write $u\subset v$.  The \textit{youngest common ancestor} of $u$ and $v$, denoted by $D(u,v)$, is the vertex of maximal height that any ray passing through $u$ has in common with any ray passing through $v$.

If $\mathcal{T}$ is a tree and $n\in\mathbb{Z}^+$, the \textit{truncation} of $\mathcal{T}$ to height $n$, denoted $\mathcal{T}_n$, is the subtree of $\mathcal{T}$ consisting of all vertices with height no more than $n$.  A tree is called \textit{locally finite} if its truncation to every level is finite; i.e. consists of finitely many vertices.  All of our trees will have this property.  In the remainder of this article, when we speak of a \textit{tree} we will always mean a \textit{locally finite, rooted, labelled tree}.

The following definition will be very important for us later.
\begin{definition}\label{D:stickiness}
	Let $\mathcal{T}$ and $\mathcal{T}'$ be two trees with equal (possibly infinite) heights.  A map $\sigma: \mathcal{T}\rightarrow \mathcal{T}'$ is called {\bf sticky} if
\begin{enumerate}
	\item[$\bullet$] for all $v\in \mathcal{T}$, $h(v) = h(\sigma(v))$, and
	\item[$\bullet$] $u\subset v$ implies $\sigma(u)\subset\sigma(v)$ for all $u,v\in \mathcal{T}$.
\end{enumerate}
We often say that $\sigma$ is sticky if it preserves heights and lineages.
\end{definition}
A one-to-one and onto sticky map between two trees, when it exists, is said to be an \textit{isomorphism} and the two trees are said to be \textit{isomorphic}.  Two isomorphic trees can and will be treated as essentially identical objects.


\subsection{Encoding bounded subsets of Euclidean space by trees}\label{tree encoding section}

The language of rooted, labelled trees is especially convenient for representing bounded sets in Euclidean spaces.  This connection is well-studied in the literature.  We refer the interested reader to \cite{LyonsPeres} for more information.

Fix any integer $M\geq 2$.  For any nonnegative integer $i$ and positive integer $k$ such that $i<M^k$, there exists a unique representation
\begin{equation}\label{M-adic representation}
	i = i_1M^{k-1} + i_2M^{k-2} + \cdots + i_{k-1}M + i_k,
\end{equation}
where the integers $i_1,\ldots,i_k$ take values in $\mathbb{Z}_M := \{0,1,\ldots,M-1\}$. These are the digits of the $M$-adic expansion of $i$. An easy consequence of \eqref{M-adic representation} is that there is a one-to-one correspondence between $M$-adic rationals in $[0,1)$ of the form $i/M^k$ and finite integer sequences $\langle i_1,\ldots, i_k\rangle$ of length $k$ with $i_j\in\mathbb{Z}_M$ for each $j$.  More generally, for any $\mathbf{i} = (j_1,\cdots,j_d) \in\mathbb{Z}^d$ such that $\mathbf{i}\cdot M^{-k}\in [0,1)^d$, we can apply \eqref{M-adic representation} to each component of $\mathbf{i}$ to obtain 
\begin{equation} \label{i expansion} \frac {\mathbf{i}}{M^k} = \frac{1}{M^k}(j_1, \cdots, j_d) = \frac {\mathbf{i}_1}{M} + \frac {\mathbf{i}_2}{M^2} + \cdots + \frac {\mathbf{i}_k}{M^k}, \end{equation}  with $\mathbf{i}_j\in\mathbb{Z}_M^d$ for all $j$.  In this way, we identify $\mathbf{i}$ with $\langle \mathbf{i}_1,\ldots,\mathbf{i}_k\rangle$. Let $\phi : \mathbb{Z}_M^d \rightarrow \{0,1,\ldots, M^d-1\}$ be an enumeration of $\mathbb{Z}_M^d$.  We refer to 
\begin{equation}\label{tree encoding}
	\mathcal{T}([0,1)^d;M,\phi) = \left\{\langle \phi(\mathbf{i}_1),\ldots,\phi(\mathbf{i}_k)\rangle : k\geq 0,\ \mathbf{i}_j \in\mathbb{Z}_M^d\right\}
\end{equation}
as the {\em{full $M$-adic tree in dimension $d$}}. Every vertex of the full $M$-adic tree has exactly $M^d$ children; therefore there are exactly $M^{kd}$ vertices of the $k$th generation. For our purposes, it will suffice to fix $\phi$ to be the lexicographic ordering, and so we will omit the notation for $\phi$ in \eqref{tree encoding}, writing simply, and with a slight abuse of notation, 
\begin{equation}\label{better tree encoding}
	\mathcal{T}([0,1)^d;M) = \left\{\langle \mathbf{i}_1,\ldots,\mathbf{i}_k\rangle : k\geq 0,\ \mathbf{i}_j \in\mathbb{Z}_M^d\right\}.
\end{equation}
We will refer to the tree in \eqref{better tree encoding} by the notation $\mathcal{T}([0,1)^d)$ once the base $M$ has been fixed.

Each vertex $v = \langle \mathbf{i}_1,\ldots,\mathbf{i}_k\rangle$ of $\mathcal{T}([0,1)^d;M)$ at height $k$ represents the unique $M$-adic cube in $[0,1)^d$ of sidelength $M^{-k}$, containing $\mathbf{i}\cdot M^{-k}$, of the form 
\begin{equation} \label{cube} 
Q = \left[\frac {j_1}{M^k},\frac {j_1+1}{M^k}\right)\times\cdots\times \left[\frac {j_d}{M^k},\frac {j_d+1}{M^k}\right).
\end{equation}   
Here $\langle \mathbf i_1, \cdots, \mathbf i_k \rangle$ is related to $(j_1, \cdots, j_d)$ by \eqref{i expansion}. Consequently, any $x\in[0,1)^d$ can be realized as the intersection of a nested sequence of $M$-adic cubes.  Thus, we view the tree in \eqref{better tree encoding} as an encoding of the set $[0,1)^d$ with respect to base $M$.  Any subset $E\subseteq[0,1)^d$ then corresponds to a subtree $\mathcal T(E;M)$ of $\mathcal{T}([0,1)^d;M)$.  The vertices on the tree $\mathcal T(E;M)$ represent $M$-adic cubes that have nontrivial intersection with $E$. As a result, an infinite ray in $\mathcal T(E;M)$ identifies a point in $E$ or its closure. Needless to say, the tree representation of the set $E$ is coordinate-sensitive. Indeed trees representing the same set in two systems of coordinates may possess widely different features - a property that we will need to take into account shortly. 

In light of the discussion above and for simplicity, we will henceforth identify the vertex $v = \langle \mathbf {i}_1, \mathbf {i}_2, \cdots, \mathbf i_k \rangle  \in \mathcal T([0,1)^d)$ with the corresponding cube $Q$ as in \eqref{cube} lying on $[0,1)^d$. With this understanding, the notation $u \subset v$ stands both for set containment as well as tree ancestry.


\subsection{The splitting number of a tree}

There are many ways to quantify the ``size" or ``spread'' of a tree (see~\cite{LyonsPeres}). Of these, the concept of a \textit{splitting number} proved to be the most relevant in the planar characterization of directions that admit Kakeya-type sets~\cite{Bateman}. Not surprisingly, it will turn out to be equally important for us. One of its applications is the explicit restatement of finite order lacunarity of a set $\Omega$ in terms of the structure of the tree encoding $\Omega$.  We define the notion of splitting number below, then collect some fundamental results about this quantity that will allow us to prove Theorem~\ref{MainThm1}, which is also the first forward implication in Theorem~\ref{MainThm2}.

We say that a vertex $v\in\mathcal{T}$ \textit{splits in} $\mathcal{T}$ if it has at least two children in $\mathcal{T}$.  When it is clear to which tree we are referring, we will just say that $v$ \textit{splits}, and we will call $v$ a \textit{splitting vertex}.  Define split$_{\mathcal{T}}(R)$, the \textit{splitting number of a ray} $R$ \textit{in} $\mathcal{T}$ to be the number of splitting vertices in $\mathcal{T}$ along that ray.  The \textit{splitting number of a vertex} $v$ \textit{with respect to a tree} $\mathcal{T}$ is defined to be 
\begin{equation}\label{splitting vertex}
\text{split}_{\mathcal{T}}(v) := \max_{\mathcal{S}_v\subseteq\mathcal{T}}\ \min_{R_v\in \partial \mathcal{S}_v}\text{split}_{\mathcal{S}_v}(R_v),
\end{equation}
where the maximum is taken over all subtrees $\mathcal{S}_v\subseteq\mathcal{T}$ rooted at $v$, and the minimum is taken over all rays $R_v$ in $\mathcal{S}_v$ that originate at the vertex $v$.  Finally, the \textit{splitting number of the tree} $\mathcal{T}$ is defined as 
\begin{equation}\label{tree split}
\text{split}(\mathcal{T}) := \max_{v\in\mathcal{T}}\ \text{split}_{\mathcal{T}}(v).
\end{equation}
\subsubsection{Examples} \label{splitting number examples}
\begin{enumerate}[(a)]
\item If $\Omega = \{ 2^{-j} : j \geq 1\}$, then split$(\mathcal T(\Omega;2)) = 1$. 
\item If $\Omega_m = \{ \frac{k}{2^m} : 0 \leq k < 2^m\}$, then split$(\mathcal T(\Omega_m;2)) = m$. As a result, the tree depicting all dyadic rationals has infinite splitting number. 
\item \label{U V splitting number} Let $U$ and $V$ be the sets constructed in Section \ref{counterexample section}. Then split$(\mathcal T(U \times V);2) = 2$, while split$(\mathcal T(\varphi(U \times V);2)) = \infty$ for the coordinate transformation $\varphi(u,v) = (u+v, u-v)$.   
\end{enumerate} 
\subsubsection{Preliminary facts about splitting numbers}
Our first result about splitting numbers (of vertices) says that they are monotone nonincreasing in lineages.
\begin{lemma}\label{monotonicity}
	Let $u,v\in\mathcal{T}$ with $u\subseteq v$.  Then split$_{\mathcal{T}}(u) \leq \text{split}_{\mathcal{T}}(v).$
\end{lemma}

\begin{proof}
Let $\mathcal{S}_u$ be a subtree of $\mathcal{T}$ rooted at $u$.  Define $\mathcal{S}_{v \rightarrow u}$ to be the union of the tree $\mathcal{S}_u$ with the path in $\mathcal{T}$ connecting $v$ to $u$.  This is a subtree of $\mathcal{T}$ rooted at $v$.  Since $v$ does not split in $\mathcal{S}_{v \rightarrow u}$ and there are no splitting vertices in $\mathcal{S}_{v \rightarrow u}$ between $v$ and $u$, we find that for any ray $R$ in $\mathcal S_u$,  \begin{equation} \label{split equality} \text{split}_{\mathcal S_u}(R) = \text{split}_{\mathcal S_{v \rightarrow u}}(R_v), \end{equation} 
where $R_v$ is the ray in $\mathcal S_{v \rightarrow u}$ rooted at $v$ obtained by extending $R$ to $v$.  Conversely, if $R_v$ is a ray in $\mathcal S_{v \rightarrow u}$, then \eqref{split equality} holds for $R = R_v \cap \mathcal S_u$.  Maximizing over all subtrees $ \mathcal{S}\subseteq\mathcal{T}$ rooted at $u$, we have that 
\begin{align*}
	\text{split}_{\mathcal{T}}(u) &= \max_{\mathcal{S}_u\subseteq\mathcal{T}}\min_{R\in \partial \mathcal{S}_u}\text{split}_{\mathcal S_{u}}(R)\\
	&= \max_{\mathcal{S}_{v \rightarrow u}\subseteq\mathcal{T}}\min_{R_v\in \partial \mathcal{S}_{ v \rightarrow u}}\text{split}_{\mathcal S_{v \rightarrow u}}(R_v)\\
	&\leq \text{split}_{\mathcal{T}}(v).
\end{align*}
The last inequality is a consequence of \eqref{splitting vertex}, since the class of subtrees of the form $\mathcal S_{v \rightarrow u}$ is a subcollection of trees rooted at $v$. 
\end{proof}

An immediate consequence of Lemma~\ref{monotonicity} is that split$(\mathcal{T}) = \text{split}_{\mathcal{T}}(v_0)$, where $v_0$ is the root of $\mathcal{T}$.  Our next result says that splitting numbers of trees are also monotonic in an appropriate sense.

\begin{lemma}\label{monotone trees}
	Let $\mathcal{S}\subseteq\mathcal{T}$.  Then split$(\mathcal{S})\leq$ split$(\mathcal{T})$.
\end{lemma}

\begin{proof}
	By Lemma~\ref{monotonicity}, split$(\mathcal{S}) = \text{split}_{\mathcal{S}}(v_0)$, where $v_0$ is the root of $\mathcal{S}$.  Since $v_0\in\mathcal{S}\subseteq\mathcal{T}$ and any subtree of $\mathcal S$ is also a subtree of $\mathcal T$, we find that 
\begin{align*}
\text{split}_{\mathcal S}(v_0) &= \max_{\mathcal S_{v_0} \subseteq \mathcal S} \min_{R_{v_0}\in \partial \mathcal{S}_{v_0}}\text{split}_{S_{v_0}}(R_{v_0})\\ 
	&\leq \max_{\mathcal S_{v_0} \subseteq \mathcal T} \min_{R_{v_0}\in \partial \mathcal{S}_{v_0}}\text{split}_{S_{v_0}}(R_{v_0}) \\
	&\leq  \text{split}_{\mathcal T}(v_0)\\
	& \leq \text{split}(\mathcal T),
\end{align*}
where the last two inequalities are implied by  \eqref{splitting vertex} and \eqref{tree split} respectively. Lemma~\ref{monotone trees} follows.
\end{proof}

A feature of trees with finite splitting number, originally observed in \cite [Lemma 5]{Bateman}, is that all vertices with largest split occur along a ray. This specialized ray will turn out to be critical in the detection of lacunary limits. 
\begin{lemma}\label{SplitsOnARay}
Let $\mathcal T$ be a tree with split$(\mathcal T) = N$. Then there exists a ray $R$ in $\mathcal T$ (of finite or infinite length) such that a vertex $v$ lies on $R$ if and only if split$_{\mathcal T}(v) = N$, provided the latter collection contains more than one element. 
\end{lemma} 
\begin{proof} We prove by contradiction. Suppose there are two vertices $u,v\in\mathcal{T}$ with split$_{\mathcal T}(u) = \text{split}_{\mathcal T}(v) = N$, $u \not\subseteq v$, $v \not\subseteq u$. Then their youngest common ancestor $D(u,v)$ is neither $u$ nor $v$.  By Lemma~\ref{monotonicity}, we know that $\text{split}_{\mathcal{T}}(D(u,v)) \geq N$.  Since $u\neq v$, the vertex $D(u,v)$ is actually a splitting vertex.  Therefore, $\text{split}_{\mathcal{T}}(D(u,v)) \geq N+1$.  But this contradicts the requirement that $\text{split}(\mathcal{T}) = N$, establishing our claim.
\end{proof}

\subsubsection{A reformulation of Theorem \ref{MainThm1}}
The dichotomy between trees with finite versus infinite splitting number will prove to be our main distinction of interest.  Roughly speaking, a tree that has infinite splitting number in some coordinate system must encode a ``large" subset of Euclidean space, the threshold of size being determined by sublacunarity. However, as we have seen in example (\ref{U V splitting number}) of Section \ref{splitting number examples},  the splitting number of a tree encoding a set is sensitive to the coordinates used to represent the set. More strongly, even the finiteness of the splitting number could be affected by the choice.  This consideration features prominently in the restatement of Theorem \ref{MainThm1} that we are about to set down.

Our proof of Theorem \ref{MainThm1} will follow a two-step route. 
\begin{proposition}\label{SublacunaryInfiniteSplit}
	Fix a dimension $d \geq 2$ and an integer $M \geq 2$. If a direction set $\Omega\subseteq \mathbb R^{d+1} \setminus \{0\}$ is sublacunary (in the sense of Definition \ref{definition of lacunary direction sets}), then \begin{equation}  \label{infty conclusion} \sup_{\mathbb V} \sup_{W_{\Omega}} \sup_{\varphi} \text{split}(\mathcal{T}(\varphi(W_{\Omega});M)) = \infty. \end{equation} 
Here $\mathbb V$ ranges over the collection of all hyperplanes at unit distance from the origin. For a fixed $\mathbb V$, the set $W_{\Omega}$ ranges over all relatively compact subsets of $\mathcal C_{\Omega} \cap \mathbb V$, and the innermost supremum is taken over all coordinate choices $\varphi = (\mathbf a, \mathcal B)$ on $\mathbb V$, where $\mathbf a \in \mathbb V$ is the point closest to the origin and $\mathcal B = \{\mathbf v_1, \cdots, \mathbf v_d \}$ is any orthonormal basis of $\mathbb V - \mathbf a$. In other words, $\varphi$ represents a rotation in $\mathbb V$ centred at $\mathbf a$, with \[\varphi(\mathcal C_\Omega \cap \mathbb V) = \bigl\{(x_1, \cdots, x_d) : x = \mathbf a + \sum_{j=1}^d x_j \mathbf v_j \in \mathcal C_\Omega \cap \mathbb V \bigr\}.\] 

Thus for every $N \geq 1$, there exists a hyperplane $\mathbb V_N$, a relatively compact subset $W_N$ of $\mathcal C_{\Omega} \cap \mathbb V_N$, and a coordinate  system $\varphi_N$ on $\mathbb V_N$ such that \begin{equation} \label{split > N} \text{split}(\mathcal{T}(\varphi_N(W_N);M)) > N. \end{equation}    
\end{proposition}

\begin{proposition}\label{InfiniteSplitKakeya}
	If a direction set $\Omega$ obeys \eqref{infty conclusion} for some $M\geq 2$, then $\Omega$ admits Kakeya-type sets.
\end{proposition}

Proposition~\ref{InfiniteSplitKakeya} will be the subject of the main body of our paper (Sections~\ref{UpperBoundSection} --~\ref{LowerBoundSection}).  We prove Proposition~\ref{SublacunaryInfiniteSplit} in Section~\ref{subsection : Lacunarity on Trees} below.

We end this section with a natural question: how does the splitting number of a tree $\mathcal T(\Omega;M)$ change if $\Omega$ is re-encoded as a tree with respect to a different base? It is not difficult to see that the number itself is not invariant under change of base. For example, if $\Omega = \{\frac{k}{4^N} : 0 \leq k < 4^N \}$ for some integer $N \geq 1$, then split$(\mathcal T(\Omega; 2)) = 2N$, whereas split$(\mathcal T(\Omega;4)) = N$.  On the other hand,  no consistent notion of ``size" of a set should be dependent on the choice of base we use to encode that set. The appropriate base-invariant concept here turns out not to be the {\em{value}} of the quantity in \eqref{infty conclusion}, but whether it is {\em{finite}} or not. Indeed for any two choices of base integers $M, M' \geq 2$, the corresponding expressions in \eqref{infty conclusion} are either both finite or both infinite. We do not need this stronger conclusion, but observe that Proposition \ref{SublacunaryInfiniteSplit} combined with Theorem \ref{MainThm2} gives an aposteriori proof of this fact. This all serves to remind the reader that in the Kakeya-type construction, the choice of base used to encode the direction set as a tree is purely utilitarian and non-central to the proof.


\subsection{Lacunarity on trees} \label{subsection : Lacunarity on Trees}

A distinctive feature in the planar characterization of Kakeya-type sets \cite{Bateman} is the observation that the lacunarity of a set is reflected in the structure of its tree.  Following the ideas developed there, we recast the concept of finite order lacunarity of a one-dimensional set using the structure of the splitting vertices of its tree. This provides a tool of convenience in the proof of  Proposition \ref{SublacunaryInfiniteSplit}, the main objective of this section. 
\begin{lemma}\label{TreeLacunaryTraditional}
	For any $M\geq 2$, $N \geq 1$, there is a constant $C = C(N,M)$ with the following property.  If a relatively compact set $U \subseteq \mathbb R$ is such that split$(\mathcal T(U;M)) = N$, then $U$ can be covered by the $C$-fold union of sets in $\Lambda(N;M^{-1})$ as described in Definition \ref{defn: Lacunary sets}.
\end{lemma}
The proof of this lemma will be presented later in this section. Assuming this, the proof of the proposition is completed as follows. 

\begin{proof}[\textit{Proof of Proposition~\ref{SublacunaryInfiniteSplit}}] We prove the contrapositive, starting with the assumption that 
\begin{equation} \label{contrapositive assumption} \sup_{\mathbb V} \sup_{W_{\Omega}} \sup_{\varphi} \text{split}(\mathcal{T}(\varphi(W_{\Omega} ;M))=N <\infty. \end{equation}  
Fix an arbitrary coordinate system $\varphi = (\mathbf a, \mathcal B)$ of $\mathbb V$ and let $\pi_j$ denote the projection maps defined in \eqref{projection map} with respect to this choice. For the remainder of this proof, we will assume that $\mathbb V$ is represented in these coordinates, so that $\pi_j$ may be thought of as the coordinate projections. Let $W = W_{\Omega}$ be an arbitrary relatively compact subset of $\mathcal C_{\Omega} \cap \mathbb V$. Since the tree encoding a set matches that of its closure, we may suppose without loss of generality that $W = W_{\Omega}$ is compact in $\mathbb V$. 

For any $1\leq j\leq d+1$, we create a subset $W_j \subseteq W$ that contains for every $x_j \in \pi_{j}(W)$ a unique point $x \in W$ with $\pi_j(x) = x_j$. For concreteness, $x$ could be chosen to be minimal in $\pi_j^{-1}(x_j) \cap W$ with respect to the lexicographic ordering. In other words, $\pi_j$ restricted to $W_j$ is a bijection onto $\pi_j(W)$. We claim that
\begin{equation} \label{pi_j and Omega_j}
\text{split}(\mathcal T(\pi_j(W);M)) \leq \text{split}(\mathcal{T}(W_j;M)).
\end{equation} 
Assuming this for the moment, we obtain from the hypothesis \eqref{contrapositive assumption} and Lemma \ref{monotone trees} that $\text{split}(\mathcal{T}(\pi_j (W);M)\leq \text{split}(\mathcal{T}(W;M)) \leq N$.  Applying Lemma~\ref{TreeLacunaryTraditional} to $U = \pi_j(W)$, we see that there is a constant $C$ (uniform in $\mathbb V$, $\varphi$, $j$ and $W$) such that the projections $\pi_j (W)$ can be covered by the $C$-fold union of one-dimensional lacunary sets of order $\leq N$ and lacunarity $\leq M^{-1}$.  Thus, $W = W_{\Omega}$ is admissible lacunary of order at most $N$ according to Definition \ref{defn : Finite order lacunarity general dimensions}. Hence $\Omega$ is admissible lacunary of finite order as a direction set by Definition \ref{definition of lacunary direction sets}. 

It remains to establish \eqref{pi_j and Omega_j}. Any infinite ray $R = R(x_j)$ in $\partial \mathcal T(\pi_j(W);M)$ corresponds to a point $x_j \in \pi_j(W)$. Let $R^{\ast} = R^{\ast}(x) \in \partial \mathcal T(W_j;M)$ denote the ray that represents $\pi_j^{-1}(x_j) = x$. This establishes a bijection between the collection of rays in the two trees. Let $v_0$ and $v_0^{\ast}$ denote the roots of the trees $\mathcal T(\pi_j(W);M)$ and $\mathcal T(W;M)$ respectively, so that $\pi_j(v_0^{\ast}) = v_0$. If $\mathcal S$ is a subtree of $\mathcal T(\pi_j(W);M)$ rooted at $v_0$, let us denote by $\mathcal S^{\ast}$ the subtree of  $\mathcal T(W_j;M)$ rooted at $v_0^{\ast}$ generated by all rays $R^{\ast}$ such that $R$ is a ray of $\mathcal S$.  It is clear that if a vertex $v$ on $R(x_j)$ splits in $\mathcal S$, then there are two points $x_j \ne x_j'$ in $\pi_j(W)$ lying in distinct children of $v$. This implies that $x = \pi_j^{-1}(x_j)$ and $x' = \pi_j^{-1}(x_j')$ lie in distinct children of $v^{\ast}$, which denotes the vertex of height $h(v)$ on $R^{\ast}(x)$.  This makes $v^{\ast}$ a splitting vertex of  $\mathcal S^{\ast}$. Thus every splitting vertex of $\mathcal S$ lying on $R$ generates a splitting vertex of $\mathcal S^{\ast}$ lying on $R^{\ast}$ at the same height.  As a result,  split$_{\mathcal S}(R) \leq \text{split}_{\mathcal S^{\ast}}(R^{\ast})$. Combining these facts with the definition of the splitting number of a tree, we obtain  
\begin{align*}
\text{split}(\mathcal T(\pi_j(W);M)) &= \max_{\mathcal S} \min_{R \in \partial \mathcal S} \text{split}_{\mathcal S}(R) \\ &\leq \max_{\mathcal S^{\ast}} \min_{R^{\ast} \in \partial \mathcal S^{\ast}} \text{split}_{\mathcal S^{\ast}}(R^{\ast}) \\ &\leq \text{split}(\mathcal T(W_j;M)). 
\end{align*}  
In view of Lemma \ref{monotonicity}, the maxima in the first and second lines above are taken over all subtrees $\mathcal S$ and $\mathcal S^{\ast}$ rooted at $v_0$ and $v_0^{\ast}$ respectively.  This completes the proof of  \eqref{pi_j and Omega_j} and hence of Proposition~\ref{SublacunaryInfiniteSplit}. 
\end{proof}

We now turn to the proof of the lemma on which the argument above was predicated. 

\begin{proof}[Proof of Lemma \ref{TreeLacunaryTraditional}]
We apply induction on $N$. The base case $N = 1$ will be treated momentarily in Lemma \ref{base case}. Proceeding to the induction step, 
let $R^{\ast}$ denote an infinite ray of the tree $\mathcal T = \mathcal T(U;M)$ that contains all the vertices $\{v^{\ast} : \text{split}_{\mathcal T}(v^{\ast}) = N  \}$. The existence of such a ray has been established in Lemma \ref{SplitsOnARay}.  For every vertex $v$ in $\mathcal T(U;M)$ which does not lie on $R^{\ast}$ but whose parent does, we define a set $U_{v}$ as follows: $\mathcal T_v = \mathcal T(U_v;M)$, where $\mathcal T_v$ denotes the maximal subtree of $\mathcal T$ rooted at $v$. The definition of the ray $R^{\ast}$ dictates that each $U_{v}$ has the property that split$(\mathcal T(U_v;M)) \leq N-1$. By the induction hypothesis, there exists a constant $C = C(N-1, M)$ such that each $U_v$ is covered by the $C$-fold union of sets in $\Lambda(N-1;M^{-1})$. The set $U$ can therefore be covered by the $C$-fold union of sets $U^{[i]}$, where each $U^{[i]}$ shares a tree structure similar to $U$: it contains the point identified by $R^{\ast}$, with the additional feature that now $U_v^{[i]} \in \Lambda(N-1;M^{-1})$ for every $v \in \mathcal V^{[i]}$, where \[\mathcal V^{[i]} := \bigl\{ v \in \mathcal T(U^{[i]}; M) : v  \notin R^{\ast}\text{ but parent of $v$ is in } R^{\ast} \bigr\}. \]

For every vertex $v \in \mathcal V^{[i]}$, let $a_v$ denote the left hand endpoint of the $M$-adic interval represented by $v$. The tree encoding the collection of points $A = \{ a_v : v \in \mathcal V^{[i]} \}$ contains the ray $R^{\ast}$; indeed the only splitting vertices of $\mathcal T(A;M)$ lie on $R^{\ast}$. Therefore split$(\mathcal T(A;M)) =1$. Hence, by Lemma \ref{base case}, $A$ is at most a $C$-fold union of monotone lacunary sequences with lacunarity $M^{-1}$, each converging  to the point identifying $R^{\ast}$. Let us continue to denote by $A$ one such monotone (say decreasing) sequence. If $a=a_v$ and $b$ are two successive elements of this sequence with $a < b$, then $U^{[i]} \cap [a,b) = U_v^{[i]}$, which is in $\Lambda(N-1;M^{-1})$. Thus $U^{[i]}$ is in $\Lambda(N;M^{-1})$ according to Definition \ref{defn: Lacunary sets}, completing the proof.
\end{proof} 

\begin{lemma} \label{base case}
Fix $M \geq 2$, and let $A \subseteq \mathbb R$ be a relatively compact set with the property that split$(\mathcal T(A;M)) = 1$. Then $A$ can be written as the union of at most $6M$ lacunary sequences (defined in Definition \ref{defn: Lacunary sequence in R}) each with lacunarity constant $\leq M^{-1}$. 
\end{lemma}
\begin{proof}
The argument here closely follows the line of reasoning in \cite[Remark 2, page 60]{Bateman}. By Lemma \ref{SplitsOnARay}, there is a ray $R^{\ast}$ in $\mathcal T(A;M)$ of infinite length such that all the splitting vertices of $\mathcal T(A;M)$ lie on it. The ray $R^{\ast}$ uniquely identifies a point in $\mathbb R$, say $a^{\ast} = \alpha(R^{\ast})$. Any ray that is not $R^{\ast}$ but is rooted at a vertex of $R^{\ast}$ is therefore non-splitting. Thus for every $j = 0, 1, 2, \cdots$ there exists at most $M-1$ rays $R_j$ in $\mathcal T(A;M)$ whose $M$-adic distance from $R^{\ast}$ is $j$. In other words, if $a_j = \alpha(R_j)$ is the point in $A$ identified by $R_j$, then there are at most $M-1$ distinct points $a_j \ne a^{\ast}$ such that 
\begin{equation} \label{Madic distance}h(D(a^{\ast}, a_j)) =  h(D(\alpha(R^{\ast}), \alpha(R_j) )) = j.  \end{equation}  

We define two subsets $A_{\pm}$ of $A$, containing respectively points $a \geq a^{\ast}$ and $a \leq a^{\ast}$. This decomposes $\mathcal T(A;M)$ into two subtrees $\mathcal T(A_{\pm};M)$. Let us focus on $\mathcal T(A_{+};M)$, the treatment for the other tree being identical. We decompose $\mathcal T(A_{+};M)$ as the union of at most $M$ trees $\mathcal T(A_{i+};M)$, $i \in \mathbb Z_M$, constructed as follows. The tree  $\mathcal T(A_{i+};M)$ contains the ray $R^{\ast}$, and for every vertex $v$ in $R^{\ast}$ the ray in $\mathcal T(A_{+};M)$, if any, descended from the $i$th child of $v$. In view of the discussion in the preceding paragraph, if there exists an integer $j$ for which a ray $R_j$ in $\mathcal T(A_{i+};M)$ obeys \eqref{Madic distance}, then such a ray must be unique. 

We now fix $i \in \mathbb Z_M$ and proceed to cover $A_{i+}$ by a threefold union of lacunary sequences converging to $a^{\ast}$. Let $\{ n_1 < n_2 < \cdots \}$ be the subsequence of integers with the property that $R_j \in \mathcal T(A_{i+};M)$ if and only if $j = n_k$ for some $k$. The important observation is that if $n_{k+2}$ is a member of this subsequence, then 
\begin{equation} \label{distance calculation} 
a_{n_k} - a^{\ast} \geq \frac{1}{M^{n_{k+2}}}. 
\end{equation}  
We will return to the proof of this statement in a moment, but a consequence of it and \eqref{Madic distance} is that for any $k \geq 0$ and fixed $\ell = 0,1,2$, 
\[ a_{n_{3(k+1)+\ell}} - a^{\ast} \leq M^{-n_{3k+3 + \ell}} = M^{-n_{3k+3 + \ell} + n_{3k+2 + \ell}} M^{-n_{3k+2+\ell}} \leq M^{-1} (a_{n_{3k+\ell}} - a^{\ast}). \]
Thus for every fixed $\ell = 0,1,2$, the sequence $\mathfrak A_{\ell} = \{ a_{n_{3k + \ell}} : k \geq 0 \}$ is covered by a lacunary sequence with constant $\leq M^{-1}$ converging to $a^{\ast}$. Since $A_{i+}$ is the union of $\{\mathfrak A_\ell : \ell = 0, 1,2\}$, the result follows.  

It remains to settle \eqref{distance calculation}, which is best explained by Figure \ref{Fig:lacunary sequence separation}.
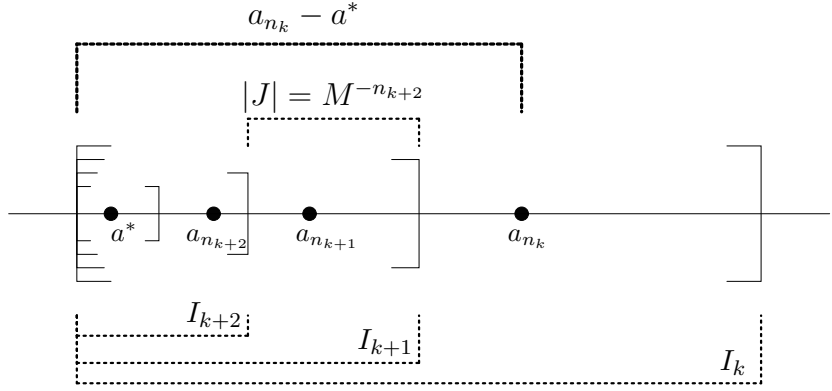
\begin{figure}[h!]
\setlength{\unitlength}{0.9mm}
\begin{picture}(0,0)(-30,20)

        \put(5,0){\special{sh 0.99}\ellipse{2}{2}}
        \put(20,0){\special{sh 0.99}\ellipse{2}{2}}
        \put(34,0){\special{sh 0.99}\ellipse{2}{2}}
        \put(65,0){\special{sh 0.99}\ellipse{2}{2}}
        \path(-10,0)(110,0)
        \path(5,10)(0,10)(0,-10)(5,-10)        \path(4,8)(0,8)(0,-8)(4,-8)        \path(3,6)(0,6)(0,-6)(3,-6)
        \path(2,4)(0,4)(0,-4)(2,-4)

        \path(95,10)(100,10)(100,-10)(95,-10)
        \path(46,8)(50,8)(50,-8)(46,-8)
        \path(22,6)(25,6)(25,-6)(22,-6)        \path(10,4)(12,4)(12,-4)(10,-4)        \put(5,-4){\small\shortstack{$a^*$}}
        \put(16,-4){\small\shortstack{$a_{n_{k+2}}$}}
        \put(32,-4){\small\shortstack{$a_{n_{k+1}}$}}
        \put(63,-4){\small\shortstack{$a_{n_{k}}$}}
        \put(24,16){\large\shortstack{$|J| = M^{-n_{k+2}}$}}
        \put(25,28){\large\shortstack{$a_{n_k}-a^*$}}
        \put(16,-16){\shortstack{$I_{k+2}$}}
        \put(41,-20){\shortstack{$I_{k+1}$}}
        \put(94,-23){\shortstack{$I_k$}}
\allinethickness{.3mm}\dottedline{1}(25,10)(25,14)(50,14)(50,10)
        \allinethickness{.4mm}\dottedline{1}(0,15)(0,25)(65,25)(65,15)  
        \allinethickness{.3mm}\dottedline{1}(0,-15)(0,-25)(100,-25)(100,-15)
        \allinethickness{.3mm}\dottedline{1}(0,-15)(0,-22)(50,-22)(50,-15)
        \allinethickness{.3mm}\dottedline{1}(0,-15)(0,-18)(25,-18)(25,-15)


\end{picture}
\vspace{4.5cm}
\caption{\label{Fig:lacunary sequence separation} A figure explaining inequality \eqref{distance calculation} when $M=2$ and $n_k = k$.}
\end{figure}
If $I_j$ is the $M$-adic interval of length $M^{-n_j}$ containing $a^{\ast}$, then $I_{k+2}$ cannot share a right endpoint with $I_{k+1}$, since this would prevent the existence of a point $a_{n_{k+1}} \geq a^{\ast}$ obeying \eqref{Madic distance} with $j = n_{k+1}$. Thus $a^{\ast}$ (in $I_{k+2}$) and $a_{n_k}$ (which is to the right of $I_{k+1}$) must lie on opposite sides of $J$, the rightmost $M$-adic subinterval of length $M^{-n_{k+2}}$ in $I_{k+1}$. This implies $a_{n_k} - a^{\ast} \geq |J|$, which is the conclusion of \eqref{distance calculation}.  
\end{proof} 

\section{Pruning of the slope tree} \label{Slope tree pruning section}

We now fix a base integer $M \geq 2$ and a sublacunary direction set $\Omega \subseteq \mathbb R^{d+1}$ (obeying the conclusion of Proposition \ref{SublacunaryInfiniteSplit}), and turn our attention to the proof of Proposition~\ref{InfiniteSplitKakeya}. We will also fix an absolute constant $C_0 \geq 1$, which will remain unchanged for the rest of the proof, and whose value will be specified later ($C_0=10$ will do). Given any integer $N$ however large, Proposition~\ref{SublacunaryInfiniteSplit} (see \eqref{split > N}) supplies a hyperplane $\mathbb V_N$ at unit distance from the origin, a coordinate system $\varphi_N$ on $\mathbb V_N$, and a relatively compact subset $W_N \subseteq \mathcal C_{\Omega} \cap \mathbb V_N$ for which split$(\mathcal T(\varphi_N(W_N);M)) > (N+1)(2C_0+1)^d$. The choice of $N$, and hence $\mathbb V_N$, $W_N$ and $\varphi_N$ will stay fixed during the analysis in Sections \ref{general facts about tube families}-\ref{LowerBoundSection}. The existence of Kakeya-type sets, which is the goal of Proposition \ref{InfiniteSplitKakeya}, relies on the ability to conduct this analysis for arbitrarily large $N$. The constant $C_0$, on the other hand, does not change with $N$. 

Without loss of generality we will assume that $\mathbb V_N = \{1\} \times \mathbb R^d$ and that $\varphi_N$ is the ambient coordinate system in $\mathbb V_N$ (and hence in all hyperplanes parallel to $\mathbb V_N$). The use of $\varphi_N$ will be dropped in the sequel, and we will simply write split$(\mathcal T(W_N;M)) > (N+1)(2C_0+1)^d$. We will also assume that $W_N \subseteq \{1\} \times [0,1)^d$; indeed if $W_N \subseteq \{1\} \times [0, M^L)^d$ for some large $L$, then we scale by a factor of $M^{-L}$ in directions perpendicular to $e_1 = (1, 0, \cdots, 0)$, leaving the direction $e_1$ unchanged. The tree corresponding to the scaled version of $W_N$ has the same splitting number as the original tree. Further, a union $E_N$ of tubes pointing in the scaled directions can be rescaled back to tubes with orientations in $W_N$, with the ratio $|E_N^{\ast}|/|E_N|$ (as explained in \eqref{Kakeya-type condition}) unchanged. From this point onwards, our direction set will be an appropriately chosen subset of $W_N \subseteq \{1\} \times [0,1)^d$ for a fixed $N$.  We rename $W_N$ as $\Omega$, since this will not cause any confusion in the sequel. 

An important preparatory step in the construction of Kakeya-type sets is the extraction of a subset of the direction set $\Omega$, whose representative tree with respect to base $M$ enjoys special structural properties, in terms of $M$-adic and Euclidean distance between certain vertices. The essential features of this trimming process and the modified direction set are summarized  below in the main result of this section. 

\begin{proposition}\label{pruning stage 1}
	Let $M \geq 2$ be a base integer, $C_0 \geq 1$ a fixed constant, and $N \gg 1$ a large parameter as described above. Let $\Omega \subseteq \{1\} \times [0,1)^d$ be a direction set obeying the hypothesis split$(\mathcal T(\Omega;M)) > (N+1)(2C_0+1)^d$. Then there exist 
\begin{itemize} 
\item a finite subset $\Omega_N \subseteq \Omega$ of cardinality $2^N$, and 
\item an integer $J=J(\Omega, N) \geq N$
\end{itemize} 
 such that the following properties hold for the tree $\mathcal T_J(\Omega_N;M)$ of height $J$ encoding $\Omega_N$:
	\begin{enumerate}[(i)]
	\item \label{N split per ray} Every ray in $\mathcal T_J(\Omega_N;M)$ splits exactly $N$ times. 
	\item \label{two children per split} Every splitting vertex in $\mathcal T_J(\Omega_N;M)$ has exactly two children.
        \item \label{Euclidean separation} Let $v$ be any splitting vertex of $\mathcal T_J(\Omega_N;M)$ and let $w_1, w_2$ be its two children as specified by part (\ref{two children per split}). If $v_i \subseteq w_i$ denotes the first splitting descendant of $w_i$ for $i=1,2$, then the Euclidean distance between the cubes $v_1$ and $v_2$ is at least $C_0 M^{-h}$, where $h = \min \{h(v_i): i=1,2\}$.  
	\end{enumerate}
The integer $J$ can be chosen to ensure that the following additional condition is met: 
\begin{enumerate}[(iv)]
\item \label{Defn of J} $C_0 M^{-J} \leq \min\{ |\omega - \omega'| : \omega \ne \omega', \; \omega, \omega' \in \Omega_N \}$. 
\end{enumerate}
\end{proposition}
The pruning process leading to the outcome claimed in the proposition is based on an iterative algorithm. The building block of the iteration is contained in Lemma \ref{Pruning building block} below, with Lemma \ref{nonadjacency} supplying an easy but necessary intermediate step. 
\begin{lemma} \label{nonadjacency}
Fix integers $r \geq 0$ and $C_0 \geq 1$. A collection of cubes of cardinality $\geq (2C_0 + 1)^d + 1$ consisting of $M$-adic cubes of sidelength $M^{-r}$ and must contain at least two cubes whose Euclidean separation is $\geq C_0M^{-r}$.  
\end{lemma} 
\begin{proof}
We first treat the case $r=0$. The cube $Q_0 = [0, 2C_0+1)^d$ contains exactly $(2C_0+1)^d$ subcubes of unit sidelength with vertices in $\mathbb Z^d$.  The central subcube $Q$ maintains a minimum distance of $C_0$ from the boundary of $Q_0$. Rephrasing this after a translation, any cube $Q$ with vertices in $\mathbb Z^d$ and of sidelength 1 admits at most $(2C_0+1)^d$ similar cubes whose distance from itself is $\leq C_0$. The case of a general $r \geq 0$ follows by scaling $Q_0$ by a factor of $M^{-r}$.    
\end{proof}
\begin{lemma} \label{Pruning building block}
Fix a constant integer $C_0 \geq 1$, an integer $N_0 \geq (2C_0+1)^d$ and a vertex $v_0$ of the full $M^d$-adic tree $\mathcal T(\{1\} \times [0,1)^d; M)$. Let $\mathcal T_{[0]}$ rooted at $v_0$ be a subtree with the property that every ray in $\mathcal T_{[0]}$ splits at least $N_0$ times. Then there exist an integer $k \geq 1$ and a subtree $\mathcal T_{[1]}$ of $\mathcal T_{[0]}$ rooted at $v_0$ and of height $k$ such that: 
\begin{enumerate}[(i)]
\item \label{nonadjacency - two descendants} The root $v_0$ has exactly two descendants $v_1$ and $v_2$ of height $k$ in $\mathcal T_1$. 
\item \label{nonadjacency - Euclidean separation} The Euclidean separation between the cubes $v_1$ and $v_2$ is given by dist$(v_1, v_2) \geq C_0 M^{-k}$. 
\item \label{nonadjacency - remaining splits} If $\mathcal T_{[0]}(v_i)$ is the maximal subtree of $\mathcal T_{[0]}$ rooted at $v_i$ then each ray in $\mathcal T_{[0]}(v_i)$ splits at least $N_0 - (2C_0+1)^d$ times. 
\end{enumerate}    
\end{lemma}
\begin{proof}
Each ray in $\mathcal T_{[0]}$ splits at least $N_0$ times, so there exists a generation in this tree consisting of at least $2^{N_0}$ vertices. Since $2^{N_0}\gg (2C_0 + 1)^d$, let us define $k$ to be the smallest height in $\mathcal T_{[0]}$ such that the number of vertices at that height exceeds $(2C_0+1)^d$. By Lemma \ref{nonadjacency}, there exist vertices $v_1$ and $v_2$ at height $k$ such that dist$(v_1, v_2) \geq C_0 M^{-k}$. The subtree $\mathcal T_1$ of height $k$ rooted at $v_0$ and generated by $v_1$, $v_2$ clearly obeys conditions (\ref{nonadjacency - two descendants}) and (\ref{nonadjacency - Euclidean separation}) stated in Lemma \ref{Pruning building block}. To complete the proof, let us recall that the number of elements of $\mathcal T_{[0]}$ at height $k-1$ is $\leq (2C_0+1)^d$. Thus any ray of $\mathcal T_{[0]}$ rooted at $v_0$ contains at most $(2C_0 + 1)^d-1$ splitting vertices of height $\leq k-2$, since each splitting vertex of height $\leq k-2$ gives rise to at least one new element (different among themselves and distinct from the terminating vertex of the ray) at height $k-1$. Since every ray of $\mathcal T_{[0]}$ contained at least $N_0$ splitting vertices to begin with, at most $(2C_0+1)^d$ of which may be lost by height $k-1$, we are left with at least $N_0 - (2C_0 + 1)^d$ splitting vertices per ray rooted at $v_i$, which is the conclusion claimed in (\ref{nonadjacency - remaining splits}). 
\end{proof} 

\begin{figure}[h!]
\setlength{\unitlength}{0.7mm}
\begin{picture}(-50,0)(-145,10)

	\path(-70,10)(-70,-30)(-110,-30)(-110,10)(-70,10)
	\path(-90,10)(-90,-30)
	\path(-70,-10)(-110,-10)
	\put(-91,13){\Large\shortstack{$v$}}
	\setlength\fboxsep{0pt}
	\put(-110,-10){\colorbox{gray!40}{\framebox(20,20){\Large$a$}}}
	\put(-90,-10){\colorbox{gray!40}{\framebox(20,20){\Large$b$}}}
	\put(-90,-32){\vector(0,-1){10}}

	\path(-60,-45)(-60,-75)(-120,-75)(-120,-45)(-60,-45)
	\path(-90,-45)(-90,-75)
	\path(-75,-45)(-75,-75)
	\path(-105,-45)(-105,-75)
	\path(-60,-60)(-120,-60)
	\put(-128,-62){\Large\shortstack{$a$}}
	\put(-55,-62){\Large\shortstack{$b$}}
	\put(-105,-60){\colorbox{gray!40}{\framebox(15,15){\Large$a_1$}}}
	\put(-90,-60){\colorbox{gray!40}{\framebox(15,15){\Large$b_1$}}}
	\put(-105,-75){\colorbox{gray!40}{\framebox(15,15){\Large$a_2$}}}
	\put(-90,-75){\colorbox{gray!40}{\framebox(15,15){\Large$b_2$}}}
	\put(-90,-77){\vector(0,-1){10}}

	\path(-70,-90)(-70,-130)(-110,-130)(-110,-90)(-70,-90)
	\path(-90,-90)(-90,-130)
	\path(-80,-90)(-80,-130)
	\path(-100,-90)(-100,-130)
	\path(-70,-110)(-110,-110)
	\path(-70,-100)(-110,-100)
	\path(-70,-120)(-110,-120)
	\put(-120,-102){\Large\shortstack{$a_1$}}
	\put(-65,-102){\Large\shortstack{$b_1$}}
	\put(-120,-122){\Large\shortstack{$a_2$}}
	\put(-65,-122){\Large\shortstack{$b_2$}}
	\put(-90,-100){\colorbox{gray!40}{\framebox(10,10){}}}
	\put(-100,-100){\colorbox{gray!70}{\framebox(10,10){\Large$v_1$}}}
	\put(-90,-110){\colorbox{gray!40}{\framebox(10,10){}}}
	\put(-100,-110){\colorbox{gray!40}{\framebox(10,10){}}}
	\put(-90,-120){\colorbox{gray!40}{\framebox(10,10){}}}
	\put(-100,-120){\colorbox{gray!40}{\framebox(10,10){}}}
	\put(-90,-130){\colorbox{gray!70}{\framebox(10,10){\Large$v_2$}}}
	\put(-100,-130){\colorbox{gray!40}{\framebox(10,10){}}}


	\put(0,0){\special{sh 0.99}\ellipse{2}{2}}
	\put(-15,-45){\special{sh 0.99}\ellipse{2}{2}}
	\put(15,-45){\special{sh 0.99}\ellipse{2}{2}}
	\put(-25,-85){\special{sh 0.99}\ellipse{2}{2}}
	\put(-8,-85){\special{sh 0.99}\ellipse{2}{2}}
	\put(8,-85){\special{sh 0.99}\ellipse{2}{2}}
	\put(25,-85){\special{sh 0.99}\ellipse{2}{2}}
	\put(30,-125){\special{sh 0.99}\ellipse{2}{2}}
	\put(-30,-125){\special{sh 0.99}\ellipse{2}{2}}

	\path(-15,-45)(0,0)(15,-45)(25,-85)
	\path(-25,-85)(-15,-45)(-8,-85)
	\path(30,-125)(25,-85)
	\path(-30,-125)(-25,-85)
	\path(15,-45)(8,-85)

	\put(-2,3){\Large\shortstack{$v$}}
	\put(-20,-42){\Large\shortstack{$a$}}
	\put(16,-42){\Large\shortstack{$b$}}
	\put(28,-85){\Large\shortstack{$b_2$}}
	\put(10,-90){\Large\shortstack{$b_1$}}
	\put(-12,-90){\Large\shortstack{$a_2$}}
	\put(-33,-85){\Large\shortstack{$a_1$}}
	\put(-38,-127){\Large\shortstack{$v_1$}}
	\put(33,-127){\Large\shortstack{$v_2$}}


	\allinethickness{.55mm}\path(-70,10)(-70,-30)(-110,-30)(-110,10)(-70,10)
	\path(-120,-45)(-120,-75)(-90,-75)(-90,-45)(-120,-45)
	\path(-60,-45)(-60,-75)(-90,-75)(-90,-45)(-60,-45)
	\path(-70,-90)(-70,-110)(-90,-110)(-90,-90)(-70,-90)
	\path(-110,-90)(-110,-110)(-90,-110)(-90,-90)(-110,-90)
	\path(-70,-130)(-70,-110)(-90,-110)(-90,-130)(-70,-130)
	\path(-110,-130)(-110,-110)(-90,-110)(-90,-130)(-110,-130)

\end{picture}
\vspace{10cm}
\caption{\label{Fig:pruning separation} An illustration of the procedure generating the forced Euclidean separation between the descendants $v_1$ and $v_2$ of $v\in\mathcal{T}$, in $\mathbb{R}^2$ when $M=2$.}  
\end{figure}
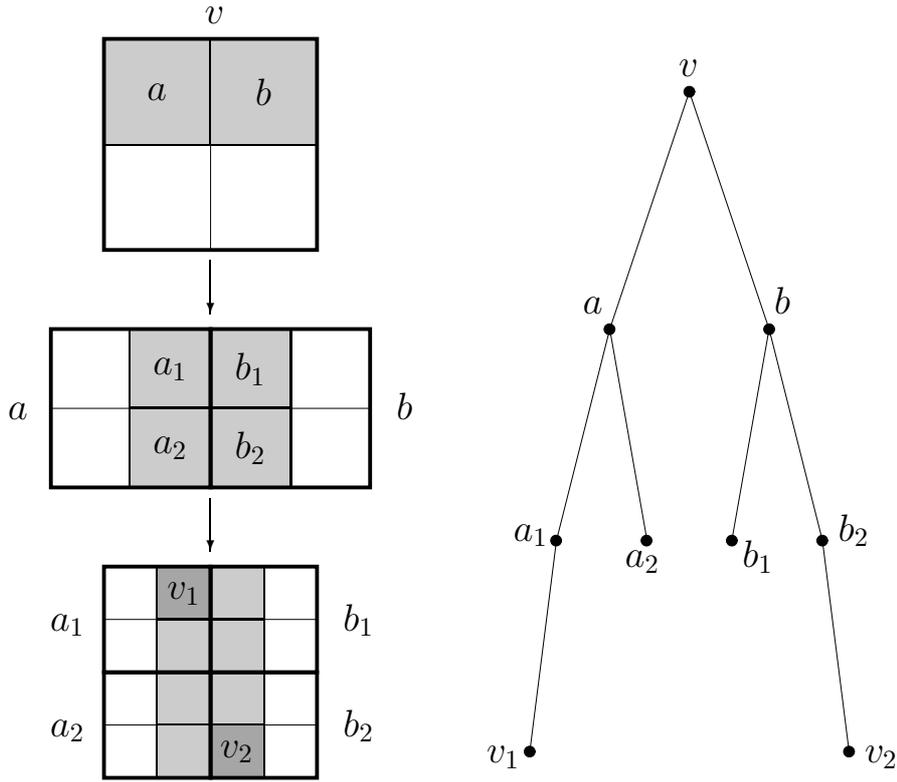

With the preliminary steps out of the way, we are ready to prove the main proposition. 
\begin{proof}[Proof of Proposition \ref{pruning stage 1}]
We know that split$(\mathcal T(\Omega;M)) > (N+1)(2C_0+1)^d$. Given any $N \geq 1$, we can therefore fix a subtree $\overline{\mathcal T}$ of $\mathcal T(\Omega;M)$ of infinite height in which every ray splits at least $(N+1)(2C_0 + 1)^d$ times. The pruning is executed on the subtree $\overline{\mathcal T}$ as follows. 

In the first step we apply Lemma \ref{Pruning building block} with \[ \mathcal T_{[0]} = \overline{\mathcal T}, \quad v_0 = \{1\} \times [0,1)^d \quad \text{ and } \quad N_0 = (N+1)(2C_0 + 1)^d. \] This yields a subtree $\mathcal T_{[1]}$ rooted at $\{1\} \times [0,1)^d$ of height $i_0$ consisting of two vertices $w_1$ and $w_2$ at the bottom-most level with dist$(w_1, w_2) \geq C_0 M^{-i_0}$. Every ray in $\mathcal T_{[1]}$ splits exactly once. Let us denote by $\overline{\mathcal T}(w_i)$ the maximal subtree of $\overline{\mathcal T}$ rooted at $w_i$. By Lemma \ref{Pruning building block}  any ray of $\overline{\mathcal T}(w_i)$ splits at least $N(2C_0+1)^d$ times. Set $\mathbb W_1 := \{ w_1, w_2 \}$.

At the second step we invoke Lemma \ref{Pruning building block} twice, resetting the parameters in that lemma to be \[ \mathcal T_{[0]} = \overline{\mathcal T}(w_i) , \quad v_0 = w_i, \quad  N_0 = N(2C_0+1)^d \] for $i=1,2$ respectively, and obtaining two subtrees as a consequence. Appending these two newly pruned subtrees of $\overline{\mathcal T}(w_i)$ to $\mathcal T_{[1]}$ from the previous step, we arrive at a tree $\mathcal T_{[2]}$ rooted at $\{ 1\} \times [0,1)^d$ of finite height but with rays of possibly variable length, in which every ray splits exactly twice.  If $v_1$ and $v_2$ are the first two splitting descendants of $\{1\} \times [0,1)^d$ in this tree, then $v_i \subseteq w_i$. Hence \[ \text{dist}(v_1, v_2) \geq \text{dist}(w_1, w_2) \geq C_0M^{-i_0} \geq C_0 M^{-h} \quad \text{ where } \quad h = \min_{i=1,2}h(v_i) \geq i_0, \] verifying the requirements (\ref{N split per ray})-(\ref{Euclidean separation}) of Proposition \ref{pruning stage 1} for $N=2$. Let us denote by $\mathbb W_2$ the collection of four vertices of maximal lineage in $\mathcal T_{[2]}$ obtained at the conclusion of this step. For any $w \in \mathbb W_2$, every ray of the tree $\overline{\mathcal T}(w)$ (defined as before as the maximal subtree of $\overline{\mathcal T}$ rooted at $w$) contains at least $(N-1)(2C_0+1)^d$ splitting vertices. Further $\mathbb W_2$ can be written as \[ \mathbb W_2 =  \bigcup \{ \mathbb W_2(w') : w' \in \mathbb W_1\}, \]  where $\mathbb W_2(w')$ consists of the two vertices in $\mathbb W_2$ descended from $w'$. For fixed $w' \in \mathbb W_1$, Lemma \ref{Pruning building block} asserts that the vertices $v, v'$ in $\mathbb W_2(w')$ have the same height $i_{w'}$, with dist$(v,v') \geq C_0 M^{-i_{w'}}$. 

In general at the end of the $k$th step we have a tree $\mathcal T_{[k]}$ of finite height, but with rays of potentially variable length, obeying the requirements (\ref{N split per ray})-(\ref{Euclidean separation}) for $N=k$. The collection of vertices of highest lineage in $\mathcal T_{[k]}$ is termed $\mathbb W_k$. We have that $\#(\mathbb W_k) = 2^k$. The collection $\mathbb W_k$ can be decomposed as \[ \mathbb W_k = \bigcup \{\mathbb W_k(w') : w' \in \mathbb W_{k-1} \}, \quad \text{ where } \quad \mathbb W_k(w') = \{ w_1(w'), w_2(w') \} \] consists of the two descendants of $w'$ that lie in $\mathbb W_k$. Lemma \ref{Pruning building block} ensures that \begin{align}  h(w_1(w')) = h(w_2(w')) &=: i_{w'} > h(w'),\text{ and that } \nonumber \\ \text{dist}(w_1(w'), w_2(w')) &\geq C_0 M^{-i_{w'}}.\label{pruning induction} \end{align}  Any ray in $\mathcal T_{[k]}$ splits exactly $k$ times, and for any $w \in \mathbb W_k$ each ray of $\overline{\mathcal T}(w)$ splits at least $(N-k+1) (2C_0+1)^d$ times. 

In the $(k+1)$th step, Lemma \ref{Pruning building block} is applied $2^k$ times in succession. In each application, the values of $\mathcal T_{[0]}$, $v_0$, $N_0$ are reset to \[ \mathcal T_{[0]} = \overline{\mathcal T}(w), \quad v_0 = w, \quad N_0 = (N-k+1)(2C_0+1)^d \] respectively for some $w \in \mathbb W_k$. The resulting tree $\mathcal T_{[k+1]}$, obtained by appending the $2^k$ newly constructed trees to $\mathcal T_{[k]}$ at the appropriate roots, clearly obeys (\ref{two children per split}) and also (\ref{N split per ray}) with $N = k+1$. Part (\ref{Euclidean separation}) only needs to be verified for the splitting vertices  $v_1(w')$ and $v_2(w')$ descended from $w' \in \mathbb W_{k-1}$, since the splitting vertices of older generations have been dealt with in previous steps.  But \[v_i(w') \subseteq w_i(w') \text{ for } i=1,2, \; \text{ so } \; \min_{i=1,2} h(v_i(w')) =h \geq i_{w'}. \] Combining this with \eqref{pruning induction}, we obtain \[ \text{dist}(v_1(w'), v_2(w')) \geq \text{dist}(w_1(w'), w_2(w')) \geq C_0 M^{-i_{w'}} \geq C_0 M^{-h}.\] 

In view of the number of splitting vertices per ray in the original subtree $\overline{\mathcal T}$, the process described above can be continued at least $N$ steps. The tree $\mathcal T_{[N]}$ of finite height but variable ray lengths obtained at the conclusion of the $N$th step satisfies the conditions (\ref{N split per ray})-(\ref{Euclidean separation}).  We pick from every vertex of maximal lineage in $\mathcal T_{[N]}$ exactly one point of $\Omega$, calling the resulting collection of $2^N$ chosen points $\Omega_N$.  Set $\delta := \min\{|\omega - \omega'| : \omega, \omega' \in \Omega_N, \; \omega \ne \omega' \} > 0$. The rays in $\mathcal T_{[N]}$ are now extended as rays representing the points in $\Omega_N$ (and hence without introducing any further splits) to a uniform height $J$ that satisfies $M^{-J} \leq C_0^{-1} \delta$, thereby meeting the criterion in part (iv).   
\end{proof} 

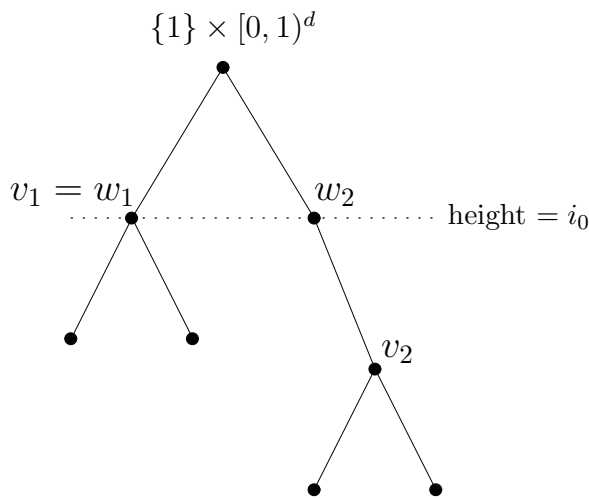
\begin{figure}[h]
\setlength{\unitlength}{0.8mm}
\begin{picture}(-50,0)(-90,10)

        \put(0,0){\special{sh 0.99}\ellipse{2}{2}}
        \put(-15,-25){\special{sh 0.99}\ellipse{2}{2}}
        \put(15,-25){\special{sh 0.99}\ellipse{2}{2}}
        \put(-25,-45){\special{sh 0.99}\ellipse{2}{2}}
        \put(-5,-45){\special{sh 0.99}\ellipse{2}{2}}
        \put(25,-50){\special{sh 0.99}\ellipse{2}{2}}
        \put(35,-70){\special{sh 0.99}\ellipse{2}{2}}
        \put(15,-70){\special{sh 0.99}\ellipse{2}{2}}

        \path(-15,-25)(0,0)(15,-25)(25,-50)
        \dottedline{2}(-25,-25)(35,-25)
        \path(-25,-45)(-15,-25)(-5,-45)        \path(35,-70)(25,-50)(15,-70)        \put(-12,5){\large\shortstack{$\{1\}\times[0,1)^d$}}
        \put(-35,-22){\Large\shortstack{$v_1 = w_1$}}
        \put(15,-22){\Large\shortstack{$w_2$}}
        \put(26,-48){\Large\shortstack{$v_2$}}
        \put(37,-26){\shortstack{height $= i_0$}}


\end{picture}
\vspace{6.5cm}
\caption{\label{Fig:second step} An illustration of a pruned tree at the second step of pruning.}
\end{figure}

\subsection{Splitting and basic slope cubes} 
The pruned slope tree $\mathcal T_J(\Omega_N;M)$ produced by Proposition \ref{pruning stage 1} looks like an elongated version of the full binary tree of height $N$. Rays in this tree may have long segments with no splits. However only the splitting vertices of $\mathcal T_J(\Omega_N;M)$ and certain other vertices related to these are of central importance to the subsequent analysis. With this in mind and to aid in quantification later on, we introduce the class of splitting vertices  
\begin{align} 
\mathcal G = \mathcal G(\Omega_N) &:= \bigcup_{j=1}^N \mathcal G_j(\Omega_N), \text{ where for every $1 \leq j \leq N$} \label{splitting vertex collection} \\ 
\label{generation j splitting vertices}
\mathcal G_j(\Omega_N) &:= \left\{\gamma \, : \begin{aligned} &\text{ there exists $v \in \Omega_N$ such that $\gamma$ is the $j$th splitting} \\  &\text{ vertex on the ray identifying $v$ in $\mathcal T_J(\Omega_N;M)$} \end{aligned} \right\}. 
\end{align} 
The vertices in $\mathcal G_j(\Omega_N)$ will be termed {\em{the $j$th splitting vertices}}. As dictated by the pruning mechanism, such vertices $\gamma$ may occur at different heights of the tree $\mathcal T_J(\Omega_N;M)$, and hence could represent $M$-adic cubes of varying sizes. Thus the index $j$, which encodes the number of splitting vertices on the ray leading up to and including $\gamma$, should not be confused with the height of $\gamma$ in $\mathcal T_J(\Omega_N;M)$. Given $\gamma \in \mathcal G(\Omega_N)$, we write 
\begin{equation} \label{defn nu} 
\nu(\gamma) = j \quad \text{ if } \gamma \in \mathcal G_j(\Omega_N),
\end{equation} 
and refer to $\nu(\gamma)$ as the {\em{splitting index}} of $\gamma$. Indeed $N-\nu(\gamma)$ is the splitting number of $\gamma$ with respect to $\mathcal T_J(\Omega_N;M)$, defined as in \eqref{splitting vertex}. Note that $\mathcal G_1(\Omega_N)$ consists of a single element, namely the unique splitting vertex of $\mathcal T(\Omega_N;M)$ of minimal height.  In general $\#(\mathcal G_j(\Omega_N)) = 2^{j-1}$, i.e., there are exactly $2^{j-1}$ splitting vertices of index $j$. We declare $\mathcal G_{N+1}(\Omega_N) \equiv \Omega_N$. 

Another related quantity of importance is the one mentioned in part (\ref{Euclidean separation}) of Proposition \ref{pruning stage 1}. In view of its ubiquitous occurrence in the sequel, we set up the following notation. For $\gamma \in \mathcal G_j(\Omega_N)$, $1 \leq j \leq N-1$, 
\begin{equation} \label{height of youngest splitting child}
\lambda(\gamma) = \lambda_j(\gamma) := \min \{h(\gamma') \; : \; \gamma' \subsetneq \gamma, \; \gamma' \in \mathcal G_{j+1}(\Omega_N) \}.
\end{equation} 
Thus $\lambda_j(\gamma) > h(\gamma)$ is the height of the first splitting vertex of index $(j+1)$ descended from $\gamma$. We refer to  an element of $\{\lambda_j(\gamma) : \gamma \in \mathcal G_j(\Omega_N) \}$ as a {\em{$j$th fundamental height}} of $\Omega_N$. There could be at most $2^{j-1}$ such heights. The collection of all fundamental heights will be denoted by $\mathcal R$; it will play a vital role in the remainder of the article, specifically in the random construction outlined in Section \ref{random construction section}. The two descendants of $\gamma \in \mathcal G_j(\Omega_N)$ at height $\lambda_j(\gamma)$, at least one (but not necessarily both) of which corresponds to a $(j+1)$th splitting vertex, are called the {\em{$j$th basic slope cubes}}. The entirety of $j$th basic slope cubes as $\gamma$ ranges over $\mathcal G_j(\Omega_N)$ is termed $\mathcal H_j(\Omega_N)$. More precisely,  
\begin{equation} \label{jth basic slope cubes} 
\mathcal H_j(\Omega_N) := \left\{ \theta \, : \begin{aligned} &\text{ there exists } \omega \in \Omega_N \text{ and } \gamma_j \in \mathcal G_j(\Omega_N)  \\ &\text{ such that } \omega \in \theta \subsetneq \gamma_j \text{ and } h(\theta) = \lambda(\gamma_j)  
\end{aligned} \right\}. 
\end{equation} 
Note that every $j$th basic slope cube $\theta$ is either itself a $(j+1)$th splitting vertex $\gamma_{j+1} \in \mathcal G_{j+1}(\Omega_N)$, or uniquely identifies such a vertex in the sense that there exists a non-splitting ray in the slope tree rooted at $\theta$ that terminates at $\gamma_{j+1}$. In either event, we say that {\em{$\gamma_{j+1} \in \mathcal G_{j+1}(\Omega_N)$ is identified by $\theta \in \mathcal H_j(\Omega_N)$}}. Since every $\gamma \in \mathcal G_j(\Omega_N)$ contributes exactly two cubes to $\mathcal H_j(\Omega_N)$, it follows that $\#(\mathcal H_j(\Omega_N)) = 2^{j}$. We declare $\mathcal H_0(\Omega_N) = \mathcal G_1(\Omega_N)$ and  $\mathcal H_N(\Omega_N) = \Omega_N$. \\

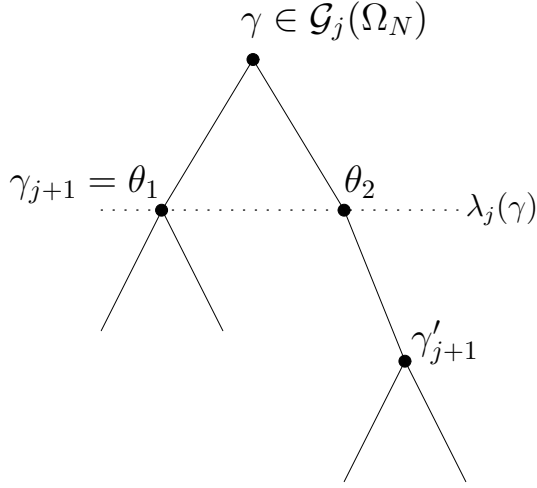
\begin{figure}[h!]
\setlength{\unitlength}{0.8mm}
\begin{picture}(-50,0)(-90,10)

        \put(0,0){\special{sh 0.99}\ellipse{2}{2}}
        \put(-15,-25){\special{sh 0.99}\ellipse{2}{2}}
        \put(15,-25){\special{sh 0.99}\ellipse{2}{2}}
        \put(25,-50){\special{sh 0.99}\ellipse{2}{2}}

        \path(-15,-25)(0,0)(15,-25)(25,-50)
        \dottedline{2}(-25,-25)(35,-25)
        \path(-25,-45)(-15,-25)(-5,-45)
        \path(35,-70)(25,-50)(15,-70)

        \put(-2,5){\Large\shortstack{$\gamma\in\mathcal G_j(\Omega_N)$}}
        \put(-40,-22){\Large\shortstack{$\gamma_{j+1} = \theta_1$}}
        \put(15,-22){\Large\shortstack{$\theta_2$}}
        \put(26,-48){\Large\shortstack{$\gamma_{j+1}'$}}
        \put(35,-26){\large\shortstack{$\lambda_j(\gamma)$}}

\end{picture}
\vspace{6.5cm}
\caption{\label{Fig:basic slope cube tree} Two basic slope cubes $\theta_1,\theta_2\in\mathcal H_j(\Omega_N)$ and their parent vertex $\gamma\in\mathcal G_j(\Omega_N)$.  Notice that $\gamma_{j+1} = \theta_1$ and $\gamma_{j+1}'$ are both members of $\mathcal G_{j+1}(\Omega_N)$.}\end{figure}

One of the important features of the pruning mechanism outlined in Proposition \ref{pruning stage 1} is an Euclidean separation condition between the two $j$th basic slope cubes descended from a common splitting vertex $\gamma_j \in \mathcal G_j(\Omega_N)$. The following implication of this condition will be convenient for later use.   
\begin{corollary} \label{Two distances comparable} 
Given a splitting vertex $\gamma$ of $\mathcal T_J(\Omega_N)$, define 
\begin{align} 
\rho_\gamma &:= \sup\{|a-b| : a \in \gamma_1 \cap \Omega_N, \; b \in \gamma_2\cap \Omega_N \}, \label{sup distance}  \\
\delta_\gamma &:= \inf\{|a-b| : a \in \gamma_1 \cap \Omega_N, \; b \in \gamma_2\cap \Omega_N \}, \label{inf distance}
\end{align} where $\gamma_1$ and $\gamma_2$ are the two children of $\gamma$ in $\mathcal T_J(\Omega_N;M)$.  Then, the two quantities $\rho_\gamma$ and $\delta_\gamma$, both of which are trivially bounded by diam$(\gamma) = \sqrt{d} M^{-h(\gamma)}$, are comparable, i.e., $\delta_\gamma \leq \rho_\gamma \leq (1 + 2 \sqrt{d}C_0^{-1}) \delta_\gamma$.  
\end{corollary} 
\begin{proof}
Using part (\ref{Euclidean separation}) of Proposition \ref{pruning stage 1} and the notation set up in \eqref{height of youngest splitting child}, we observe that $\gamma_i \cap \Omega_N \subseteq v_i$ where $v_i$ is the first splitting descendant of $\gamma_i$, so that $\delta_\gamma \geq C_0 M^{-\lambda(\gamma)}$. Let $a_i$, $b_i$ be points in the closures of $\gamma_i \cap \Omega_N$, $i=1,2$ such that $\delta_\gamma = |a_1 - b_1|$, $\rho_\gamma = |a_2-b_2|$.  Then 
\begin{align*}  \rho_\gamma = |a_2-b_2| &\leq |a_1 - b_1| + |a_2 - a_1| + |b_2-b_1| \\ &\leq |a_1-b_1| + \text{diam}(v_1) + \text{diam}(v_2) \\ &\leq |a_1 - b_1| + 2 \sqrt{d} M^{-\lambda(\gamma)} \\ &\leq \delta_\gamma + 2 \sqrt{d} C_0^{-1} \delta_\gamma \leq C_2 \delta_\gamma, \end{align*}  
where the third inequality above follows from the fact that $v_i$ is either itself a cube of sidelength $M^{-\lambda(\gamma)}$ or is contained in one.    
\end{proof} 
\subsection{Binary representation of $\Omega_N$} \label{binary slope tree}
The classes of basic slope cubes $\mathcal H_j(\Omega_N)$ allow us to represent each element in $\Omega_N$ in terms of a unique $N$-long binary sequence as follows. Since every splitting vertex of $\mathcal T_J(\Omega_N)$ has exactly two children, one of them must be larger (or older) than the other in the lexicographic ordering. Let us agree to call the older (respectively younger) child of a vertex $v$ its $0$th (respectively 1st) offspring. For $1 \leq j \leq N$, we define a bijective map $\Psi_j: \{0,1\}^j \rightarrow \mathcal H_{j}(\Omega_N)$ inductively as follows. For $j=1$, 
\begin{equation} \label{defn Psi_1} \Psi_1(i) := \Biggl\{\begin{aligned}
&\text{ the unique element of $\mathcal H_1(\Omega_N)$} \\ &\text{ descended from the $i$th child of $\gamma_1$ }, 
\end{aligned}\end{equation} 
where $i=0,1,$ and $\gamma_1$ is the single element in $\mathcal H_0(\Omega_N) = \mathcal G_1(\Omega_N)$. In general if $\Psi_{j}$ has been defined, then for $\bar{\epsilon} \in \{0,1\}^{j}$ and $i=0,1$, we set 
\begin{equation} 
\Psi_{j+1}(\bar{\epsilon}, i) := \Biggl\{\begin{aligned}
&\text{ the unique element of $\mathcal H_{j+1}(\Omega_N)$ } \\ &\text{ descended from the $i$th child of $\gamma_{j+1}$,} \end{aligned} \label{defn Psi_j} \end{equation} 
where $\gamma_{j+1}$ is the unique element of $\mathcal G_{j+1}(\Omega_N)$ identified by $\Psi_{j}(\bar{\epsilon})$.

The map $\Psi_N$ provides the claimed bijection of $\{0,1\}^N$ onto $\Omega_N$. In fact, the discussion above yields the following stronger conclusion, the verification of which is straightforward and left to the reader.
\begin{proposition} \label{Splitting tree lemma}
Let $\mathcal H_j(\Omega_N)$ be as in \eqref{jth basic slope cubes}.
\begin{enumerate}[(i)]
\item The collection of vertices \begin{equation} \label{Splitting tree} \mathcal H(\Omega_N) := \bigcup_{j=1}^{N}
\left\{ (\theta_1, \cdots, \theta_j) \Bigl| \begin{aligned} &\exists \; \omega \in \Omega_N, \text{ such that } \omega \in \theta_k, \\ &\theta_k \in \mathcal H_k(\Omega_N), \; 1 \leq k \leq j  \end{aligned} \right\} \bigcup \{\gamma_1\}  
\end{equation} 
is a tree rooted at $\gamma_1 \in \mathcal H_0(\Omega_N)$ of height $N$, in which $(\theta_1, \cdots, \theta_j, \theta_{j+1})$ is a vertex of height $(j+1)$ and a child of $(\theta_1, \cdots, \theta_j)$. Every element $\theta_j \in \mathcal H_j(\Omega_N)$ identifies a vertex $(\theta_1, \cdots, \theta_j)$ of the $j$th generation in this tree. 
\item Let $\mathcal B_N$ denote the full binary tree of height $N$, namely the tree $\mathcal T_N([0,1);2)$. The map $\Psi : \mathcal B_N \rightarrow \mathcal H(\Omega_N)$ defined by
\begin{equation} \label{splitting vertex binary tree isomorphism}
\begin{aligned} 
\Psi(\emptyset) &= \text{ the unique element $\gamma_1 \in \mathcal H_0(\Omega_N)$}, \\ \Psi(\bar{\epsilon}) &= \; \Psi_{j}(\bar{\epsilon}) \quad \text{ if } \quad \bar{\epsilon} \in \{0,1\}^j, \; 1 \leq j \leq N,
\end{aligned}
\end{equation}  
with $\Psi_j$ as in \eqref{defn Psi_j} is a tree isomorphism in the sense of Definition \ref{D:stickiness}.       
\end{enumerate} 
\end{proposition} 
Although we will not need to use it, an analogous argument shows that the class of splitting vertices $\mathcal G(\Omega_N)$ is isomorphic to $\mathcal B_{N-1}$.


\section{Families of intersecting tubes} \label{general facts about tube families}
The finite set of directions $\Omega_N$ created in Proposition \ref{pruning stage 1} forms the basis of the construction of Kakeya-type sets. Indeed the sets of interest that verify the conclusion of Theorem \ref{MainThm1} will be the union of a family of tubes, with each tube assigned a slope from $\Omega_N$. Each tube is based on a suitably fine subcube of the $d$-dimensional unit cube $\{0\} \times [0,1)^d$, hereafter referred to as the {\em{root hyperplane}}. The tree depicting the root hyperplane, more precisely the full $M$-adic tree of dimension $d$ and height $J$ will be termed the {\em{root tree}}.  For $0 \leq k \leq J$, let $\mathcal Q(k)$ be the collection of all vertices of height $k$ in the root tree, i.e.,
\begin{equation} \label{defn Q_k}
\mathcal Q(k) := \left\{Q \; : \; Q \in \mathcal T(\{0\} \times [0,1)^d;M), \; h(Q) = k \right\}.
\end{equation} 
Geometrically, and in view of the discussion in Section \ref{tree encoding}, a member $Q$ of $\mathcal Q(k)$ is an $M$-adic cube of sidelength $M^{-k}$ of the form 
\begin{equation} \label{root cubes}
Q = \{0\} \times \prod_{\ell=1}^{d} \left[ \frac{j_{\ell}}{M^{k}}, \frac{j_{\ell+1}}{M^{k}}\right), \; \text{ where } \; (j_1, j_2, \cdots j_{d}) \in \{ 0,1, \cdots, M^{k}-1\}^d, 
\end{equation} 
so that $\#(\mathcal Q(k)) = M^{kd}$. In view of the above, and for the purpose of distinguishing vertices of the root and the slope trees, a vertex in the root tree is termed a {\em{spatial cube}}. For reasons to be made clear in a moment, an element of $\mathcal Q(J)$ (i.e., a youngest vertex of the root tree) is of added significance and will be called a {\em{root cube}}.  Given a fixed constant $A_0 \geq 1$, and for $t \in \mathcal Q(J)$, $\omega \in \Omega_N$, we define a {\em{tube rooted at $t$ with orientation $\omega$}} to be the set   
\begin{equation} \label{tube defn}
P_{t, \omega} := \widetilde{Q}_t + [0, 10 A_0]\omega = \bigl\{s + r\omega : s \in \widetilde{Q}_t, \; 0 \leq r \leq 10A_0  \bigr\}. 
\end{equation} 
Here $\widetilde{Q}_t$ denotes the $c_d$-dilate of the cube $t$; i.e., the cube with the same centre as $t$ but with $c_d$ times its sidelength, for a small positive constant $c_d$ soon to be specified in Corollary \ref{which is bigger corollary}. For instance, the choice $c_d = d^{-2d}$ will suffice. Thus $P_{t, \omega}$ is essentially a $(d+1)$-dimensional cylinder of constant length and with cubical cross-section of sidelength $c_d M^{-J}$ perpendicular to the $x_1$-axis. An algorithm $\sigma$ that assigns to every root $t \in \mathcal Q(J)$ a slope $\sigma(t) \in \Omega_N$ produces, according to the prescription \eqref{tube defn}, a family of tubes of cardinality $M^{Jd}$, and a corresponding set
\begin{equation} \label{generic sticky Kakeya}
\mathbb K (\sigma) = \mathbb K(\sigma;N, J) := \bigcup \{ P_{t, \sigma(t)} : t \in \mathcal Q(J) \}. 
\end{equation}  
While this definition is quite general, in our applications the slope assignment map $\sigma$ will always be chosen to be sticky in the sense of Definition \ref{D:stickiness} and as a mapping between the trees representing roots and slopes respectively; specifically, 
\[ \sigma : \mathcal T_J(\{0\} \times [0,1)^d;M) \rightarrow \mathcal T_J(\Omega_N;M). \]  
Random sticky slope assignment algorithms will be prescribed in the next section, but for now we record some properties of arbitrary tubes and features of general sets of the form $\mathbb K(\sigma)$.  

\subsection{Intersection of two tubes} 

\begin{lemma} \label{intersection criterion lemma}
For $v, v' \in \Omega_N$ and $t, t' \in \mathcal Q(J)$, $t \ne t'$, let $P_{t, v}$ and $P_{t', v'}$ be the tubes defined as in (\ref{tube defn}). If there exists $x = (x_1, \cdots, x_{d+1}) \in P_{t, v} \cap P_{t', v'}$, then the inequality
 \begin{equation} \label{intersection criterion inequality} \bigl| \text{cen}(t') - \text{cen}(t) + x_1(v'-v) \bigr| \leq 2 c_d \sqrt{d} M^{-J} \end{equation} 
holds, where $\text{cen}(t)$ denotes the centre of the cube $t$.  
\end{lemma}
\begin{proof}
If $x \in P_{t, v} \cap P_{t', v'}$, then there exist points $y \in \widetilde{Q}_t$, $y' \in \widetilde{Q}_{t'}$ such that $x = y + x_1v = y' + x_1 v'$, i.e., $x_1(v'-v) = y - y'$. The inequality (\ref{intersection criterion inequality}) follows since both $|y - \text{cen}(t)|$ and $|y'- \text{cen}(t')|$ are bounded above by $c_d \sqrt{d} M^{-J}$.   
\end{proof} 
\begin{corollary} \label{which is bigger corollary}
If the constant $c_d$ is chosen sufficiently small, then under the hypotheses of Lemma \ref{intersection criterion lemma}, 
\begin{equation} \label{wich is bigger} 
|x_1||v-v'| \geq \frac{1}{2} M^{-J}. 
\end{equation} 
\end{corollary}
\begin{proof}
Since $t \ne t'$, we know that $|\text{cen}(t') - \text{cen}(t)| \geq M^{-J}$. The inequality in \eqref{intersection criterion inequality} therefore implies that 
\[ |x_1| |v-v'| \geq |\text{cen}(t) - \text{cen}(t')| - 2 c_d \sqrt{d} M^{-J} \geq (1 - 2 c_d \sqrt{d}) M^{-J} \geq \frac{1}{2} M^{-J},\]
provided $c_d$ is chosen to satisfy $2 c_d \sqrt{d} \leq \frac{1}{2}$.  
\end{proof}  
Lemma \ref{intersection criterion lemma} provides an intersection criterion for two tubes in the form of an algebro-geometric inequality. We will also need to know the size of this intersection. This estimate is by now standard in the literature, dating back to the work of C\'ordoba \cite{Cordoba}. The result below is easily verifiable, but the reader may consult \cite[Lemma 10.3.6, p.~374]{Grafakos} as a reference.   
\begin{lemma} \label{intersection size lemma} 
If $P_{t,v}$ and $P_{t',v'}$ are any two intersecting tubes of the form \eqref{tube defn}, then 
\[ |P_{t,v} \cap P_{t',v'}| \leq \frac{C_d M^{-J(d+1)}}{M^{-J} + |v-v'|},\]
where $C_d$ is a dimension-dependent constant.   
\end{lemma}

\subsection{Tubes and a point} \label{tubes and point section}
A crucial component of the proof of Proposition \ref{InfiniteSplitKakeya}, amplified in Section \ref{UpperBoundSection}, is to identify when a given point $x$ belongs to a union of tubes of the form \eqref{generic sticky Kakeya}. In our applications, the set $\mathbb K(\sigma)$ in \eqref{generic sticky Kakeya} will be probabilistically generated by random sticky maps, and we will need to estimate the likelihood of such an inclusion. But many major ingredients of the argument pertain to general sets $\mathbb K(\sigma)$ generated by an arbitrary sticky $\sigma$. We discuss these features here. 

\begin{lemma} \label{leading to Poss(x)}
Let $x \in \mathbb R^{d+1}$, $A_0 \leq x_1 \leq 10 A_0$. If the parameter $C_0$ used in the pruning of the slope tree $\mathcal T(\Omega;M)$ (see Proposition \ref{pruning stage 1}) is chosen sufficiently large relative to the constant $A_0$ in \eqref{tube defn}, then the following property holds:  for any $t \in \mathcal Q(J)$, there exists at most one $v(t) \in \Omega_N$ such that $x \in P_{t, v(t)}$. 
\end{lemma}
\begin{proof} 
If there exist slopes $v, v' \in \Omega_N$ such that $x \in P_{t,v} \cap P_{t,v'}$, then the points $x - x_1v$ and $x - x_1v'$ must both lie in $t$. In other words, 
\[ |x_1(v-v')| = |(x-x_1v)-(x-x_1v')|\leq \sqrt{d}M^{-J}.  \]
Since $x_1 \geq A_0$, this implies that $|v-v'| \leq A_0^{-1} \sqrt{d} M^{-J}$, which is $\leq \frac{C_0}{2} M^{-J}$ for a choice of $C_0$ sufficiently large. Comparing with part (iv) of Proposition \ref{pruning stage 1}, we find this is possible in $\Omega_N$ only if $v = v'$. 
\end{proof} 
The lemma above motivates the following definition: for $x \in \mathbb R^{d+1}$ with $A_0 \leq x_1 \leq 10 A_0$,  
\begin{equation} \label{defn Poss(x)}
\text{Poss}(x) := \left\{t \in \mathcal Q(J) : \begin{aligned}  &\text{ there exists } v(t) = v(t;x) \in \Omega_N \\ &\text{ such that } x \in P_{t, v(t)} \end{aligned} \right\}.  
\end{equation} 
\begin{lemma} \label{Poss(x) inclusion}
The set Poss$(x)$ introduced in \eqref{defn Poss(x)} can also be characterized as follows:
\begin{equation} \text{Poss}(x) = \{ t \in \mathcal Q(J) : t \cap (x-x_1 \Omega_N) \ne \emptyset \}. \label{Poss(x) another description} 
\end{equation}  
Thus Poss$(x)$ is contained in an $O(M^{-J})$-neighborhood of an affine copy of $\Omega_N$ in the root hyperplane $\{0\} \times [0,1)^d$.   
\end{lemma} 
\begin{proof}
If $t \in \text{Poss}(x)$,  it follows from the definition \eqref{tube defn} of a tube and the description \eqref{defn Poss(x)} of Poss$(x)$ that $x - x_1 v(t) \in t$ for some $v(t) \in \Omega_N$. Thus the left hand side of \eqref{Poss(x) another description} is contained in the right hand side. Conversely, if $x-x_1 v \in t$ for some $v \in \Omega_N$, then $x \in t + x_1 v \subseteq P_{t,v}$. This means that $t \in \text{Poss}(x)$, and the result follows.   
\end{proof}
The mapping \begin{equation} \label{defn v(t)} v : \text{Poss}(x) \rightarrow \Omega_N \quad \text{ which sends } \quad t \mapsto v(t) \quad \text{ with } \quad x \in P_{t,v(t)} \end{equation}   
is uniquely defined by Lemma \ref{leading to Poss(x)}. It captures for every $t \in \text{Poss}(x)$ the ``correct slope'' that ensures that a tube rooted at $t$ with that slope contains $x$. A purely deterministic object driven by $\Omega_N$, this map has a certain structure that is critical to the subsequent analysis.  To formalize this property, let us recall the definitions of $\mathcal G_j(\Omega_N)$ and $\mathcal H_j(\Omega_N)$ from \eqref{generation j splitting vertices} and \eqref{jth basic slope cubes}. We denote for every $\omega \in \Omega_N$ and $1 \leq j \leq N$, 
\begin{equation} \label{defn of eta_j}
\eta_j(\omega) :=  h(\theta) \quad \text{ where } \quad \omega \subseteq \theta \in \mathcal H_j(\Omega_N).
\end{equation}
In other words, $\eta_j(\omega)$ is the height of the $j$th basic slope cube on the ray identifying $\omega$ in $\mathcal T_J(\Omega_N;M)$. We note that $\eta_N(\omega) \equiv J$ for all $\omega \in \mathcal H_N(\Omega_N) = \Omega_N$.  

The quantity $\eta_j$ is used to define the following objects:   
\begin{align}  
\mathcal N_x &:= \left\{ \Phi_j(t) \; : \; t \in \text{Poss}(x), \; 0 \leq j \leq N \right\}, \label{defn N} \\  \mathcal M_x &:= \left\{ \Theta_j(t) \, : \, t \in \text{Poss}(x), \; 0 \leq j \leq N \right\},  \quad \text{ where } \label{defn M} \\ \Phi_j(t) &:= \begin{cases} \{0\} \times [0,1)^d &\text{ for } j = 0 \\  \bigl(Q_1^{\ast}(t), \cdots, Q_j^{\ast}(t)\bigr) &\text{ for } j \geq 1, \quad \text{ and } \end{cases} \label{defn Phi_j} \\
\Theta_j(t) &:= \begin{cases} \{1\} \times [0,1)^d &\text{ for } j = 0 \\  \bigl(\theta_1(t), \cdots, \theta_j(t)\bigr) &\text{ for } j \geq 1. \end{cases} \label{defn Theta_j}
\end{align} 
Here for $j \geq 1$, the cube $Q_j^{\ast}(t)$ is a cube in the root hyperplane containing $t$. In contrast, $\theta_j(t)$ is a vertex in $\mathcal H_j(\Omega_N)$, hence a cube in $\{ 1\} \times [0,1)^d$, containing the point $v(t) \in \Omega_N$. Furthermore, both cubes are located at the same height in their respective trees and obey the defining properties \begin{equation} \label{s_j and theta_j} t \subseteq Q_j^{\ast}(t), \quad v(t) \in \theta_j(t), \quad \text{ and } \quad h(Q_j^{\ast}(t)) = h(\theta_j(t)) = \eta_j(v(t)).\end{equation}  

We pause briefly to clarify the definitions \eqref{defn Phi_j} and \eqref{defn Theta_j} (see Figure~\ref{Fig:pull-back process}). Given any $t \in \text{Poss}(x)$, we pick on the ray identifying $t$ the vertices that lie at the same height as the basic slope cubes of $v(t)$. The entries of the vector $\Phi_j(t)$ are the first $j$ chosen vertices on this ray. On the other hand, $\Theta_j(t)$ consists of the first $j$ basic slope cubes containing $v(t)$.  The vectors $\Phi_{N}(t)$ and $\Theta_{N}(t)$ identify $t$ and $v(t)$ respectively. For reasons to emerge shortly in Lemma \ref{N to M sticky map}, we view the collection $\mathcal N_x$ as a tree, in which $\Phi_j(t)$ is a vertex of height $j$, and $\Phi_{j+1}(t)$ is a child of $\Phi_j(t)$. As we have already noted, the set Poss$(x)$, and hence the youngest generation of $\mathcal N_x$, contains all possible roots that could support tubes with directions in $\Omega_N$ containing $x$. For an arbitrary $\sigma$, it is therefore natural to phrase a necessary criterion for the inclusion $x \in \mathbb K(\sigma)$ in terms of $\mathcal N_x$. For this reason we choose to call $\mathcal N_x$ the {\em{reference tree}}, and its defining cubes $Q_j^{\ast}(t)$ as {\em{reference cubes}}. The collection $\mathcal M_x$ should be thought of as the ``image'' of $\mathcal N_x$ on the slope side, and hence a tree as well, with $\Theta_j(t)$ being a vertex of the $j$th generation and the parent of $\Theta_{j+1}(t)$.  In fact, $\mathcal M_x$ is a subtree of $\mathcal H(\Omega_N)$ defined as in \eqref{Splitting tree}. In view of Proposition \ref{Splitting tree lemma}, any vertex $\Theta_j(t)$ of height $j \geq 1$ in $\mathcal M_x$ is identified with the $j$-long binary sequence $\Psi^{-1}(\Theta_j(t))$.  

\begin{figure}[h!]
\setlength{\unitlength}{0.8mm}
\begin{picture}(-50,0)(-106,20)

	\put(-50,0){{\special{sh 0.99}\ellipse{2}{2}}}
	\dottedline{1}(-50,0)(-60,-7)
	\dottedline{1}(-50,0)(-40,-7)
	\dottedline{1}(-50,0)(-46,-8)
	\put(-55,-15){{\special{sh 0.99}\ellipse{2}{2}}}
	\dottedline{1}(-55,-15)(-65,-22)
	\dottedline{1}(-55,-15)(-45,-22)
	\dottedline{1}(-55,-15)(-52,-23)
	\put(-57,-25){{\special{sh 0.99}\ellipse{2}{2}}}
	\dottedline{1}(-57,-25)(-67,-32)
	\dottedline{1}(-57,-25)(-47,-32)
	\dottedline{1}(-57,-25)(-60,-33)
	\put(-55,-35){{\special{sh 0.99}\ellipse{2}{2}}}
	\dottedline{1}(-55,-35)(-65,-42)
	\dottedline{1}(-55,-35)(-45,-42)
	\dottedline{1}(-55,-35)(-58,-43)
	\put(-52,-45){{\special{sh 0.99}\ellipse{2}{2}}}
	\dottedline{1}(-52,-45)(-65,-52)
	\dottedline{1}(-52,-45)(-54,-53)
	\dottedline{1}(-52,-45)(-58,-53)
	\put(-50,-55){{\special{sh 0.99}\ellipse{2}{2}}}
	\dottedline{1}(-50,-55)(-60,-62)
	\dottedline{1}(-50,-55)(-42,-62)
	\dottedline{1}(-50,-55)(-55,-63)
	\put(-49,-65){{\special{sh 0.99}\ellipse{2}{2}}}
	\dottedline{1}(-49,-65)(-59,-72)
	\dottedline{1}(-49,-65)(-40,-72)
	\dottedline{1}(-49,-65)(-45,-73)
	\put(-50,-75){{\special{sh 0.99}\ellipse{3}{3}}}
	\path(-55,-15)(-57,-25)(-55,-35)(-52,-45)(-50,-55)(-49,-65)(-50,-75)

	\path(-52,2)(-48,2)(-48,-2)(-52,-2)(-52,2)
	\put(-60,0){\Large\shortstack{$\Phi_j$}}
	\path(-59,-23)(-55,-23)(-55,-27)(-59,-27)(-59,-23)
	\put(-74,-25){\Large\shortstack{$\Phi_{n-1}$}}	
	\path(-52,-53)(-48,-53)(-48,-57)(-52,-57)(-52,-53)
	\put(-63,-57){\Large\shortstack{$\Phi_n$}}
	\put(-52,-84){\Large\shortstack{$t$}}


	\put(20,0){{\special{sh 0.99}\ellipse{2}{2}}}
	\put(15,-25){{\special{sh 0.99}\ellipse{2}{2}}}
	\put(13,-55){{\special{sh 0.99}\ellipse{2}{2}}}
	\put(17,-75){{\special{sh 0.99}\ellipse{3}{3}}}
	\path(15,-25)(13,-55)(17,-75)

	\put(17,0){\vector(-1,0){62}}
	\put(12,-25){\vector(-1,0){64}}
	\put(10,-55){\vector(-1,0){54}}

	\put(25,-2){\Large\shortstack{$\Theta_j$}}
	\put(19,-27){\Large\shortstack{$\Theta_{n-1}$}}
	\put(18,-56){\Large\shortstack{$\Theta_{n}$}}
	\put(13,-84){\Large\shortstack{$v(t)$}}

	\allinethickness{.4mm}\dottedline{2}(-55,-15)(-50,0)(-47,15)
	\dottedline{2}(15,-25)(20,0)(20,15)
	\allinethickness{.25mm}\dottedline{1}(20,0)(25,-10)
	\dottedline{1}(15,-25)(8,-35)
	\dottedline{1}(13,-55)(25,-65)


\end{picture}
\vspace{8.5cm}
\caption{\label{Fig:pull-back process} The pull-back mechanism used to define $\mathcal{N}_x$, for $M= d=2$.}  
\end{figure}
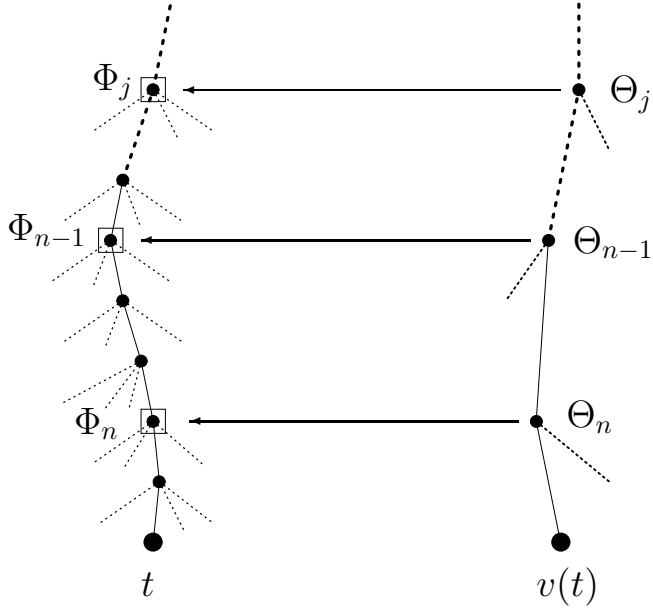

Given the constraints of our pruning mechanism in Proposition \ref{pruning stage 1}, the ``correct slope''  map $t \mapsto v(t)$ need not be sticky as a mapping from $\mathcal T_J(\text{Poss}(x);M)$ to $\mathcal T_J(\Omega_N;M)$. It does however possess a weak variant of the stickiness property that we specify in the next lemma. As we will see in Lemma \ref{N to M sticky map}, this milder substitute is able to achieve two goals that are of fundamental relevance to this study. First, it assigns a tree structure to $\mathcal N_x$ and $\mathcal M_x$. Second, it is strong enough to lift $v$ as a sticky map from $\mathcal N_x \rightarrow \mathcal M_x$.
\begin{lemma} \label{v(t) weak sticky}
There is a sufficiently large choice of the parameter $C_0$ in Proposition \ref{pruning stage 1} for which the following conclusion holds. Let $x \in \mathbb R^{d+1}$ with $A_0 \leq x_1 \leq 10A_0$, $t, t' \in \text{Poss}(x)$ and $u = D(t,t')$. Set $w=D(v(t), v(t'))$, so that $w \in \mathcal G(\Omega_N)$, the class of splitting vertices defined in \eqref{splitting vertex collection}. Then
\begin{equation}  h(u) < \lambda(w), \label{weak sticky} \end{equation}  
with $\lambda$ defined as in \eqref{height of youngest splitting child}. 
\end{lemma} 
{\em{Remark:}} If $v$ defined in \eqref{defn v(t)} was indeed a sticky map, one would have access to the inequality $h(u) \leq h(w)$. We know however that $\lambda(w) > h(w)$, and hence \eqref{weak sticky} should be viewed as a weak version of stickiness.  
\begin{proof}
 If $x \in P_{t, v(t)} \cap P_{t', v(t')}$, then by the inequality \eqref{intersection criterion inequality} in Lemma \ref{intersection criterion lemma},
\begin{align} A_0 |v(t)-v(t')| &\leq |x_1||v(t)-v(t')| \nonumber \\  &\leq |\text{cen}(t') - \text{cen}(t)| + 2\sqrt{d} M^{-J} \nonumber \\ &\leq 2 \sqrt{d} M^{-h(u)} + 2\sqrt{d} M^{-J} \leq 4 \sqrt{d} M^{-h(u)}, \nonumber \\
\text{ and thus } |v(t) -v(t')| &\leq 4 \sqrt{d} A_0^{-1} M^{-h(u)}. \label{upper bound on difference} \end{align}  
On the other hand, $v(t)$ and $v(t')$ each lie in distinct children of $w$, which must be a splitting vertex of $\mathcal T_J(\Omega_N;M)$. If $w \in \mathcal G_j(\Omega_N)$ and if $\gamma$, $\gamma'$ denote the $(j+1)$th splitting vertices descended from $w$, then each of $\gamma$ and $\gamma'$ contains exactly one of $v(t)$ and $v(t')$.  By Proposition \ref{pruning stage 1}(\ref{Euclidean separation}), 
\begin{equation} \label{lower bound on difference}
|v(t) - v(t')| \geq \text{dist}(\gamma, \gamma') \geq C_0 M^{-\lambda_j(w)}.   
\end{equation} 
Combining \eqref{upper bound on difference} and \eqref{lower bound on difference} we obtain 
\[ C_0 M^{-\lambda_j(w)} \leq 4 \sqrt{d} A_0^{-1} M^{-h(u)}.\]
If the constant $C_0$ is chosen larger than $4\sqrt{d} A_0^{-1}$, then the inequality above implies \eqref{weak sticky}, as claimed.  
\end{proof}
\begin{lemma}\label{N to M sticky map}
The collection of vertex tuples $\mathcal N_x$, $\mathcal M_x$ defined in \eqref{defn N}, \eqref{defn M} are well-defined as trees rooted at $\{0\} \times [0,1)^d$ and $\{1\} \times [0,1)^d$ respectively, with the ancestry relation as described in the discussion leading up to Lemma \ref{v(t) weak sticky}. More precisely, the map $v$ defined in \eqref{defn v(t)} meets the following consistency requirements:
\begin{enumerate}[(i)]
\item Let $t, t' \in \text{Poss}(x)$, $u = D(t,t')$. If the index $j$ satisfies $\eta_j(v(t)) \leq h(u)$ then we also have $\eta_j(v(t')) \leq h(u)$, in which case $\Phi_j(t) = \Phi_j(t')$ and $\Theta_j(t) = \Theta_j(t')$. \label{consistency}
\item The map from $\mathcal N_x \rightarrow \mathcal M_x$ that sends $\Phi_j(t) \mapsto \Theta_j(t)$ is well-defined and sticky. \label{sticky map lift}
\end{enumerate}
\end{lemma} 
\begin{proof} 
Let $\gamma_j(t) \in \mathcal G_j(\Omega_N)$ denote the $j$th splitting vertex on the ray identifying $v(t)$. Then $\eta_j(v(t)) = \lambda_j(\gamma_j(t))$.  If $\eta_j(v(t)) = \lambda_j(\gamma_j(t)) \leq h(u)$, then Lemma \ref{v(t) weak sticky} implies that $\eta_j(v(t)) = \lambda_j(\gamma_j(t)) < \lambda(w)$, where $w = D(v(t), v(t'))$. Unravelling the implication of this inequality, we see that the height of the first splitting descendant of $\gamma_j(t)$ is strictly smaller than the corresponding quantity for $w$. Since both $\gamma_j(t)$ and $w$ are splitting vertices lying on the ray of $v(t)$, this means that $\gamma_j(t)$ is an ancestor of $w$ of strictly lesser height. In other words, $w \subseteq \gamma_{j+1}(t)$. Since the rays for $v(t)$ and $v(t')$ agree up to and including height $h(w)$, we conclude that their first $(j+1)$ splitting vertices are identical; i.e., 
\begin{equation}
\gamma_{k}(t) = \gamma_{k}(t') \text{ for } k \leq j+1. 
\end{equation} 
Hence $\eta_k(v(t)) = \lambda_k(\gamma_k(t)) = \lambda_k(\gamma_k(t')) = \eta_k(v(t'))$ for all such $k$, implying one of the desired conclusions in part (\ref{consistency}). Since \[ h(w) \geq h(\gamma_{j+1}(t)) = h(\gamma_{j+1}(t')) \geq \eta_j(v(t)) = \eta_j(v(t')), \] the vectors $v(t)$ and $v(t')$ must agree at height $\eta_j$. Thus $\Theta_j(t) = \Theta_j(t')$. Of course if $\eta_j(v(t)) = \eta_j(v(t')) \leq h(u)$, then $\Phi_j(t) = \Phi_j(t')$. This completes the proof of the first part of the lemma.

Part (\ref{sticky map lift}) is essentially a restatement of the result in part (\ref{consistency}). To ascertain that the map is well-defined we choose $t, t' \in \text{Poss}(x)$ with $u= D(t,t')$ and $\Phi_j(t) = \Phi_j(t')$ and aim to show that $\Theta_j(t) = \Theta_j(t')$. The hypothesis $\Phi_j(t) = \Phi_j(t')$ implies that $\eta_j(v(t)) = \eta_j(v(t')) \leq h(u)$, and part (\ref{consistency}) implies that the images match. Stickiness is a by-product of the definitions. 
\end{proof}
Lemma \ref{N to M sticky map} permits the unambiguous assignment of an ``ideal image'' (namely an edge in $\mathcal M_x$) to every edge of the tree $\mathcal N_x$, in the following sense: if every edge in the ray leading up to $\Phi_N(t)$ receives its ideal image, then $x \in P_{t, v(t)}$. To make this quantitatively precise, let us define the {\em{reference slope function}} $\kappa$ as follows: for every edge $e$ in $\mathcal N_x$ joining the vertices $\Phi_j(t)$ to $\Phi_{j+1}(t)$, we define a binary counter $\kappa(e)$ through the defining equation
\begin{equation} \label{defn overlay function} 
\Psi^{-1} \circ \Theta_{j+1}(t) = (\Psi^{-1} \circ \Theta_j(t), \kappa(e))
\end{equation} 
where $\Psi$ is the tree isomorphism defined in Proposition \ref{Splitting tree lemma}. In other words, $\kappa(e)$ is zero (respectively one) if and only if the ray identifying $\Theta_{j+1}(t)$ in $\mathcal T_{J}(\Omega_N)$ passes through the $0$th (respectively $1$st) child of the $(j+1)$th splitting vertex identified by $\Theta_j(t)$.
\begin{corollary}\label{kappa makes sense}
The reference slope function $\kappa$ described in \eqref{defn overlay function} is well-defined, and assigns to each edge of $\mathcal N_x$ a unique value of 0 or 1.  
\end{corollary} 
\begin{proof}
If there exist $t \ne t'$ in Poss$(x)$ such that the terminating vertex of $e$ could be represented either as $\Phi_{j+1}(t)$ or as $\Phi_{j+1}(t')$, then Lemma \ref{N to M sticky map} guarantees that $\Theta_{k}(t) = \Theta_{k}(t')$ for all $k \leq j+1$, proving that $\kappa(e)$ given by \eqref{defn overlay function} is a well-defined function on the edge set of $\mathcal N_x$. 
\end{proof} 
The reader may find it helpful to visualize the edges of the reference tree $\mathcal N_x$ with an overlay of model binary values assigned by $\kappa$, against which any other slope assignment will be tested. This intuition is made precise in the following lemma. Given a fixed point $x$ and a union of tubes $\mathbb K(\sigma)$ of the form \eqref{generic sticky Kakeya} generated by a sticky slope map $\sigma$, the result offers a criterion governed by the reference slope function $\kappa$ for verifying whether $x \in \mathbb K(\sigma)$. Indeed for such $\sigma$, we can define $\mathcal N_x(\sigma)$ and $\mathcal M_x(\sigma)$ akin to \eqref{defn N} and \eqref{defn M}, but using the given slope map $t \mapsto \sigma(t)$ instead of the naturally generated $v$ given by \eqref{defn v(t)}. More precisely, we set  
\begin{align}
\mathcal N_x(\sigma) & :=  \{\Phi_j(t ; \sigma) \, : \, t \in \text{Poss}(x), 0 \leq j \leq N\}, \label{defn N omega} \\   
\mathcal M_x(\sigma) & :=  \{\Theta_j(t; \sigma) \, : \, t \in \text{Poss}(x), \; 0 \leq j \leq N\}, \label{defn M omega} 
\end{align}
where for $j \geq 1$, both $\Phi_j(t; \sigma)$ and $\Theta_j(t; \sigma)$ are $j$-long vectors whose $i$th components are $M$-adic cubes of identical size, containing $t$ in the root hyperplane and $\sigma(t)$ in the slope tree respectively.  For $\Theta_j(t;\sigma)$, the $i$th entry is required to lie in $\mathcal H_{i}(\Omega_N)$, which uniquely specifies both vectors. In light of the preceding results in this section, it is not surprising that the collections \eqref{defn N omega} and \eqref{defn M omega} are trees and that $\sigma$ extends to a map between these trees. 
\begin{lemma} \label{percolation preparation lemma} 
The following conclusions hold:
\begin{enumerate}[(i)]
\item \label{auxiliary slope map} The collections $\mathcal N_x(\sigma)$ and $\mathcal M_x(\sigma)$ as in \eqref{defn N omega} and \eqref{defn M omega} are well-defined as trees rooted respectively at $\{ 0\} \times [0,1)^d$ and $\{ 1\} \times [0,1)^d$. The tuples $\Phi_j(t;\sigma)$ and $\Theta_j(t;\sigma)$ are deemed vertices of generation $j$, and parents of $\Phi_{j+1}(t;\sigma)$ and $\Theta_{j+1}(t;\sigma)$ respectively. The map $\Phi_j(t ; \sigma) \mapsto \Theta_j(t; \sigma)$ from $\mathcal N_x(\sigma)  \rightarrow \mathcal M_x(\sigma)$ is well-defined and sticky.
\item \label{analogue of kappa for omega} If $e$ denotes the edge connecting $\Phi_j(t; \sigma)$ and $\Phi_{j+1}(t ; \sigma)$ in $\mathcal N_x(\sigma)$, then the quantity $\iota_{\sigma}(e)$ defined by 
\begin{equation} \label{defn nonoverlay function} \Psi^{-1} \circ \Theta_{j+1}(t;\sigma) = (\Psi^{-1} \circ \Theta_j(t;\sigma), \iota_{\sigma}(e)) \end{equation}  
gives rise to a well-defined binary function on the edge set of $\mathcal N_x(\sigma)$. 
\item \label{inclusion criterion} If $x \in \mathbb K(\sigma)$, then there exists $t \in \text{Poss}(x)$ such that
$\Theta_N(t; \sigma) = \Theta_N(t)$. In particular, this implies that 
\begin{equation} \label{comparing Phi}
\Phi_j(t;\sigma) = \Phi_j(t) \quad \text{ for all } \quad 1 \leq j \leq N,  
\end{equation}  
and hence that $\mathcal N_x$ and $\mathcal N_x(\sigma)$ share a common ray $R$ identifying $t$ with the property 
\begin{equation} \label{equality on each edge}
\iota_{\sigma}(e) = \kappa(e) \quad \text{ for every edge $e$ in $R$.}
\end{equation} 
\end{enumerate}  
\end{lemma} 
\begin{proof}
Despite the obvious similarity of the statement with that of Lemma \ref{N to M sticky map}, the distinction in the proofs should be noted. The assumed stickiness of $\sigma$ simplifies the proof of part (\ref{auxiliary slope map}), compared to that of Lemma \ref{N to M sticky map}, where $v$ was not known to be sticky. Indeed if $t \ne t'$ are such that $\Phi_j(t;\sigma) = \Phi_j(t';\sigma)$, then as before $\eta_j(\sigma(t)) = \eta_j(\sigma(t')) \leq h(u)$.  By stickiness of $\sigma$, we have $h(u) \leq h(z)$, where $z = D(\sigma(t), \sigma(t'))$. This implies that the rays identifying $\sigma(t)$ and $\sigma(t')$ agree up to and including height $\eta_j(\sigma(t))$, i.e., $\Theta_j(t;\sigma) = \Theta_j(t';\sigma)$. Part (\ref{analogue of kappa for omega}) is an easy consequence of part (\ref{auxiliary slope map}) and follows exactly the same way as Corollary \ref{kappa makes sense} was deduced from Lemma \ref{N to M sticky map}. Finally, if $x \in \mathbb K(\sigma)$, then there is some $t \in \text{Poss}(x)$ such that $\sigma(t) = v(t)$. Since the chain of basic slope cubes containing any $v \in \Omega_N$ is unique, this implies that $\Theta_N(t;\sigma) = \Theta_N(t)$, and hence $\Phi_j(t;\sigma) = \Phi_j(t)$ for all $1 \leq j \leq N$. The last equality says that $t$ is identified by the same sequence of vertices and hence the same ray in both $\mathcal N_x$ and $\mathcal N_x(\sigma)$.  
If $e_1, e_2, \cdots e_N$ are the successive edges in this ray, with $e_{j+1}$ connecting $\Phi_j(t)$ with $\Phi_{j+1}(t)$, then a consequence of the definitions \eqref{defn overlay function}, \eqref{defn nonoverlay function} of $\kappa$ and $\iota_{\sigma}$ is that 
\[ (\iota_{\sigma}(e_1), \cdots, \iota_{\sigma}(e_N)) = \Psi^{-1} \circ \sigma(t)  = \Psi^{-1} \circ v(t) = (\kappa(e_1), \cdots, \kappa(e_N)),\] 
where $\Psi$ is the tree isomorphism defined in Proposition \ref{Splitting tree lemma}, and part (\ref{inclusion criterion}) follows.   
\end{proof} 
We end this section with a bound on the number of vertices of the reference tree at a given height, a result that will be useful for probability computations later. In view of the characterization \eqref{Poss(x) another description}  of Poss$(x)$ given in Lemma \ref{Poss(x) inclusion}, and our construction of $\Omega_N$, this is intuitively clear.  
\begin{lemma} \label{vertex counting in N_x}
There exists a positive constant $C$ depending on $d$ and $A_0$ but uniform in $x \in [A_0, A_0 + 1] \times \mathbb R^d$ such that the number of vertices of height $j$ in $\mathcal N_x$  is bounded above by $C 2^j$. 
\end{lemma} 
\begin{proof}
Let $n_j(x)$ denote the number of vertices of height $j$ in $\mathcal N_x$. In view of the relations \eqref{defn N} and \eqref{s_j and theta_j} defining $\mathcal N_x$, the cardinality $n_j(x)$ equals the number of spatial cubes in the collection
\begin{equation}
\{Q_j^{\ast}(t) : \; t \in \text{Poss}(x), \; t \subseteq Q_j^{\ast}(t), \; h(Q_j^{\ast}(t)) = \eta_j(v(t))  \}, \label{alt N_x} 
\end{equation} 
so we proceed to count the number of such cubes $Q_j^{\ast}(t)$. Let us recall from the definition \eqref{defn Poss(x)} of Poss$(x)$ that $x \in P_{t, v(t)}$. This implies that $x - x_1 v(t) \in t$, and hence for $\theta_j(t)$ as in \eqref{s_j and theta_j},
\begin{align*}
|\text{cen}(Q_j^{\ast}(t)) -& x + x_1 \text{cen}(\theta_j(t))| \\ &\leq |\text{cen}(Q_j^{\ast}(t)) - \text{cen}(t)| + |x_1||\text{cen}(\theta_j(t))  - v(t)| + |\text{cen}(t) -x + x_1 v(t)| \\
&\leq \sqrt{d} M^{-\eta_j(v(t))} + (A_0+1) \sqrt{d} M^{-\eta_j(v(t))} + \sqrt{d} M^{-J} \\ &\leq 4A_0 \sqrt{d} M^{-\eta_j(v(t))}.  
\end{align*} 
Let us unravel the geometric implications of the inequality above. For a given $\theta_j(t)$ containing $v(t)$, there are at most a constant number $C(d,A_0)$ of $M$-adic cubes of sidelength same as $\theta_j(t)$ (hence candidates for $Q_j^{\ast}(t)$) whose centres are within distance $4A_0 \sqrt{d} M^{-\eta_j(v(t))}$ of $x - x_1 \text{cen}(\theta_j(t))$. On the other hand, each $\theta_j(t) \in \mathcal H_j(\Omega_N)$, and hence the total number of possible $\theta_j(t)$ as $t$ ranges over Poss$(x)$ is at most $\#(\mathcal H_j(\Omega_N)) = 2^j$, by Proposition \ref{Splitting tree lemma}. Since $n_j(x)$ is the cardinality of the collection in \eqref{alt N_x}, the observations above lead to the bound $n_j(x) = O(2^j)$ as claimed.   
\end{proof} 


\section{Random construction of Kakeya-type sets} \label{random construction section} 
Motivated by the generalities laid out in the previous section, specifically Lemmas \ref{N to M sticky map} and \ref{percolation preparation lemma}, we now proceed to describe a randomized algorithm for generating a class of sticky slope assignments $\sigma$. Let us recall the class $\mathcal R$ of fundamental heights of $\Omega_N$ defined in \eqref{height of youngest splitting child} and the discussion thereafter.    

We start with a collection of independent and identically distributed Bernoulli$(\frac{1}{2})$ random variables 
\begin{equation} \label{Bernoulli warehouse} 
\mathbb X := \left\{X_Q : Q \in \mathcal Q(k), \; k \in \mathcal R \right\}, 
\end{equation} 
with $\mathcal Q(k)$ defined as in \eqref{defn Q_k}. The collection $\mathbb X$ therefore assigns, for every fundamental height $k$ an independent binary random variable to every $M$-adic cube of sidelength $M^{-k}$ in the root hyperplane. We use $\mathbb X$ as the randomization source for our construction.

Let $h_0$ denote the height of the single element $\theta_0 \in \mathcal G_1(\Omega_N) = \mathcal H_0(\Omega_N)$, in other words, the first splitting vertex of $\mathcal T_J(\Omega_N;M)$. We define $\sigma(Q_0) \equiv \theta_0$ for all $Q_0 \in \mathcal Q(h_0)$. At the first step of the randomization process, each $Q_0 \in \mathcal Q(h_0)$ is decomposed into subcubes $Q_1$ of sidelength $M^{-h_1}$ where $h_1 = \lambda_1(\theta_0) > h_0$.  We call these subcubes the {\em{first basic spatial cubes}}. Each first basic spatial cube $Q_1$ receives from the Bernoulli collection $\mathbb X$ defined in \eqref{Bernoulli warehouse} a value of $X_{Q_1}$, which is either zero or one. Recalling from \eqref{defn Psi_1} that \[ \Psi_1(X_{Q_1}) \in \mathcal H_1(\Omega_N), \quad \text{ and that } \quad h(\Psi_1(X_{Q_1})) = h_1, \]  we define 
\[ \sigma(Q_1) = \sigma_{\mathbb X}(Q_1) = \Psi_1(X_{Q_1}) \]
for any first basic spatial cube $Q_1$. Each element of $\mathcal H_1(\Omega_N)$, and hence each $\sigma(Q_1)$, is either a second splitting vertex of $\Omega_N$ or the identifier of one. If the root cube $Q_1$ already maps into a second splitting vertex under $\sigma$, no further action is needed for it in step one.  Now, suppose there exists $\gamma \in \mathcal G_2(\Omega_N)$ such that $h(\gamma) > h_1$ (there could be at most one such $\gamma$). Then for any cube $Q_1$ for which $\Psi_1(X_{Q_1}) $ is the unique ancestor of $\gamma$ at height $h_1$, we decompose $Q_1$ into subcubes $Q_1'$ of sidelength $M^{-h(\gamma)}$ and set $\sigma(Q_1') = \gamma$ for all such $Q_1' \subsetneq Q_1$. Thus, at the end of the first step, 
\begin{enumerate}[(a)]
\item we have obtained  a partition of the root hyperplane into first basic spatial cubes, and randomly assigned each such cube a first basic slope cube in $\mathcal H_1(\Omega_N)$ of the same height, namely $\lambda_1(\theta_0) = h_1$.
\item If the vertices in $\mathcal G_2(\Omega_N)$ occur at different heights, then predicated on the random assignment in part (a) certain first basic spatial cubes could subdivide further to generate a different partition of the root hyperplane, say $\{\mathcal Q_1(\gamma) : \gamma \in \mathcal G_2(\Omega_N)\}$. Each cube $Q_1' \in \mathcal Q_1(\gamma)$ is of height $h(\gamma)$ and is mapped to $\gamma$. We will refer to $Q_1'$ as {\em{a spatial cube of second splitting height}}. Thus a first basic spatial cube is either itself a spatial cube of second splitting height, or is uniformly partitioned into a disjoint union of such cubes.  
\end{enumerate} 

\begin{figure}[h!]
\setlength{\unitlength}{0.8mm}
\begin{picture}(-50,0)(-40,10)

	\put(0,0){\special{sh 0.99}\ellipse{2}{2}}
	\put(-15,-25){\special{sh 0.99}\ellipse{2}{2}}
	\put(15,-25){\special{sh 0.99}\ellipse{2}{2}}
	\put(25,-50){\special{sh 0.99}\ellipse{2}{2}}

	\path(-15,-25)(0,0)(15,-25)(25,-50)
	\path(-25,-45)(-15,-25)(-5,-45)
	\path(35,-70)(25,-50)(15,-70)

	\put(-2,5){\Large\shortstack{$\gamma\in\mathcal G_j(\Omega_N)$}}
	\put(-42,-22){\Large\shortstack{$\gamma_{j+1} = \theta_1$}}
	\put(-15,-25){\circle{4}}
	\put(17,-22){\Large\shortstack{$\theta_2$}}
	\path(13,-27)(13,-23)(17,-23)(17,-27)(13,-27)
	\put(26,-48){\Large\shortstack{$\widetilde{\gamma}_{j+1}$}}


	\put(92,2){\Large\shortstack{$Q_{j-1}'\in\mathcal{Q}_{j-1}(\gamma)$}}
	\put(130,-27){\Large\shortstack{$Q_j$}}
	\put(113,-52){\Large\shortstack{$Q_j'$}}

	\put(90,0){\special{sh 0.99}\ellipse{2}{2}}
	\path(90,0)(93,-7)
	\put(93,-7){\special{sh 0.99}\ellipse{2}{2}}
	\dottedline{2}(93,-7)(98,-15)
	\path(90,0)(87,-7)
	\put(87,-7){\special{sh 0.99}\ellipse{2}{2}}
	\dottedline{2}(87,-7)(91,-15)
	\path(90,0)(98,-7)
	\put(98,-7){\special{sh 0.99}\ellipse{2}{2}}
	\dottedline{1}(98,-7)(103,-12)
	\dottedline{1}(103,-12)(115,-25)
	\dottedline{1}(103,-12)(120,-25)
	\dottedline{1}(103,-12)(125,-25)
	\path(90,0)(82,-7)
	\put(82,-7){\special{sh 0.99}\ellipse{2}{2}}
	\dottedline{1}(82,-7)(78,-12)
	\dottedline{1}(78,-12)(70,-25)
	\dottedline{1}(78,-12)(75,-25)
	\dottedline{1}(78,-12)(80,-25)
	\dottedline{1}(78,-12)(85,-25)
	\dottedline{1}(78,-12)(90,-25)
	\dottedline{1}(78,-12)(95,-25)
	\dottedline{1}(78,-12)(100,-25)

	\put(70,-25){\special{sh 0.99}\ellipse{2}{2}}
	\put(70,-25){\circle{4}}
	\put(75,-25){\special{sh 0.99}\ellipse{2}{2}}
	\path(73,-23)(73,-27)(77,-27)(77,-23)(73,-23)
	\put(80,-25){\special{sh 0.99}\ellipse{2}{2}}
	\put(80,-25){\circle{4}}
	\put(85,-25){\special{sh 0.99}\ellipse{2}{2}}
	\path(83,-23)(83,-27)(87,-27)(87,-23)(83,-23)

	\put(90,-25){\special{sh 0.99}\ellipse{2}{2}}
	\path(92,-23)(92,-27)(88,-27)(88,-23)(92,-23)
	\put(95,-25){\special{sh 0.99}\ellipse{2}{2}}
	\path(93,-23)(93,-27)(97,-27)(97,-23)(93,-23)
	\put(100,-25){\special{sh 0.99}\ellipse{2}{2}}
	\put(100,-25){\circle{4}}

	\dottedline{2}(104,-25)(111,-25)
	\put(115,-25){\special{sh 0.99}\ellipse{2}{2}}
	\path(113,-23)(113,-27)(117,-27)(117,-23)(113,-23)
	\put(120,-25){\special{sh 0.99}\ellipse{2}{2}}
	\put(120,-25){\circle{4}}
	\put(125,-25){\special{sh 0.99}\ellipse{2}{2}}
	\path(123,-23)(123,-27)(127,-27)(127,-23)(123,-23)

	\dottedline{2}(75,-25)(75,-40)
	\dottedline{2}(85,-25)(85,-40)
	\dottedline{2}(95,-25)(95,-40)
	\dottedline{2}(115,-25)(115,-40)
	\dottedline{2}(125,-25)(125,-40)

	\dottedline{1}(90,-25)(90,-45)
	\dottedline{1}(90,-45)(75,-50)
	\dottedline{1}(90,-45)(80,-50)
	\dottedline{1}(90,-45)(85,-50)
	\dottedline{1}(90,-45)(90,-50)
	\dottedline{1}(90,-45)(95,-50)
	\dottedline{1}(90,-45)(100,-50)
	\dottedline{1}(90,-45)(105,-50)
	\dottedline{1}(90,-45)(110,-50)
	\put(75,-50){\special{sh 0.99}\ellipse{2}{2}}
	\put(80,-50){\special{sh 0.99}\ellipse{2}{2}}
	\put(85,-50){\special{sh 0.99}\ellipse{2}{2}}
	\put(90,-50){\special{sh 0.99}\ellipse{2}{2}}
	\put(95,-50){\special{sh 0.99}\ellipse{2}{2}}
	\put(100,-50){\special{sh 0.99}\ellipse{2}{2}}
	\put(105,-50){\special{sh 0.99}\ellipse{2}{2}}
	\put(110,-50){\special{sh 0.99}\ellipse{2}{2}}

\end{picture}
\vspace{6.5cm}
\caption{\label{Fig:basic root cube assignment} A pictorial representation of the basic slope and root cubes and a typical slope assignment.  
Vertices $Q_j$ for which $X_{Q_j} = 0$ are indicated by a circle and assigned $\theta_1$; others are indicated by squares and assigned $\theta_2$. For the squared vertices, a further slope assignment is made at a finer level. } 
\end{figure}
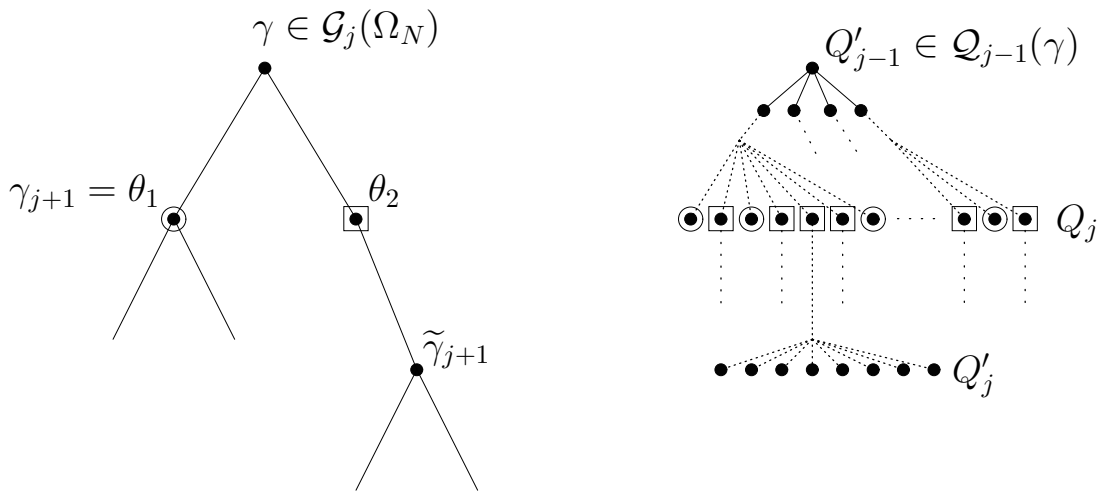

In general, the $j$th step of the construction generates a random and possibly non-uniform partition of the root hyperplane into spatial cubes $Q_j'$ of $(j+1)$th splitting height. Each $Q_j'$ is the terminal member of a descending chain
\begin{equation} \label{chain of basic and splitting cubes}
Q_{j}' \subseteq Q_j \subsetneq Q_{j-1}' \subseteq Q_{j-1} \subsetneq \cdots \subsetneq Q_1' \subseteq Q_1,  
\end{equation} 
 where for every $k \leq j$,  $Q_k$ is a $k$th basic spatial cube, and $Q_k'$ is a spatial cube of $(k+1)$th splitting height. 
Each $Q_k$ is mapped by $\sigma$ to a $k$th basic slope cube in $\mathcal H_k(\Omega_N)$, whereas $Q_k'$ is mapped to a splitting vertex in $\mathcal G_{k+1}(\Omega_N)$. All such assignments preserve heights and lineages; in other words, for a sequence of cubes as in \eqref{chain of basic and splitting cubes}, 
\begin{equation} \label{sigma preserves heights and lineages}
\begin{aligned}
\sigma(Q_j') &\subseteq \sigma(Q_j) \subsetneq \sigma(Q_{j-1}') \subseteq \cdots \subsetneq \sigma(Q_1') \subseteq \sigma(Q_1), \\ &h(\sigma(Q_j)) = h(Q_j), \quad h(\sigma(Q_j')) = h(Q_j'). 
\end{aligned}
\end{equation}  
The spatial cubes at $(j+1)$th splitting height can therefore be classified as follows: 
\begin{equation} \label{jth step outcome}
\mathcal Q_j(\gamma) := \{ Q_j' : \sigma(Q_j') = \gamma \},  \qquad \gamma \in \mathcal G_{j+1}(\Omega_N). 
\end{equation}  

At the $(j+1)$th step each $Q_j'$ from the collection $\mathcal Q_j(\gamma)$ is decomposed into subcubes $Q_{j+1}$ of height $\lambda_{j+1}(\gamma) > h(\gamma)$. These are the $(j+1)$th {\em{basic spatial cubes}}. Each spatial cube $Q_{j+1}$ is assigned the binary value $X_{Q_{j+1}}$ from the Bernoulli collection $\mathbb X$ in \eqref{Bernoulli warehouse}.  Combined with the random assignments that the basic ancestors of $Q_{j+1}$ have received, this produces an image of $Q_{j+1}$ under $\sigma$:
\begin{equation} \label{random variable assignment}
\sigma_{\mathbb X}(Q_{j+1}) := \Psi_{j+1}(X_{Q_1}, \cdots, X_{Q_{j+1}}) \in \mathcal H_{j+1}(\Omega_N), \qquad Q_{j+1} \subsetneq \cdots \subsetneq Q_1. 
\end{equation} 
Each $\sigma(Q_{j+1})$ is the unique identifier of some $\gamma \in \mathcal G_{j+2}(\Omega_N)$. We decompose $Q_{j+1}$ into subcubes $Q'_{j+1}$ of height $h(\gamma)$ (in some cases no further decomposition may be needed) and set $\sigma(Q_{j+1}') = \gamma$. This results in a newer and finer partition of the root hyperplane into spatial cubes $Q_{j+1}'$ of $(j+1)$th splitting height, producing an analogue of \eqref{jth step outcome} for the $(j+2)$th step and allowing us to carry the induction forward. 

Continuing the procedure described above for $N$ steps, we obtain a decomposition of the root hyperplane into a family of basic cubes of order $N$, each of which is of sidelength $M^{-J}$, and hence is by definition a root cube. Every such cube $t= Q_{N}(t)$ is contained in a unique chain of basic spatial cubes of lower order: 
\begin{equation} \label{chain of basic cubes containing t} t = Q_{N}(t) \subsetneq Q_{N-1}(t) \subsetneq \cdots Q_2(t) \subsetneq Q_1(t) 
\end{equation}  and is assigned a slope $\sigma_{\mathbb X}(t) = \Psi_N(X_{Q_1}, \cdots, X_{Q_N})$ in $\mathcal H_N(\Omega_N) = \Omega_N$. We will shortly expand on further structural properties of the slope map $t \mapsto \sigma_{\mathbb X}(t)$, but first observe that it gives rise to a random set 
\begin{equation} \label{random sticky Kakeya}
K_N(\mathbb X) := K(\sigma_{\mathbb X}; N, J)
\end{equation}    
according to the prescription \eqref{generic sticky Kakeya}.

\subsection{Features of the construction}
We pause briefly to summarize the important features of the construction above:
\begin{enumerate}[-]
\item Randomization only occurs for  cubes in the root hyperplane that correspond to the fundamental heights, though all cubes of a given fundamental height need not receive a random assignment. 
\item The only cubes that receive a random binary assignment from $\mathbb X$ are by definition the basic spatial cubes. {\em{Unlike the basic slope cubes that constitute $\mathcal H_j(\Omega_N)$, a basic spatial cube $Q_j$ is a random quantity}}. For instance, the size of a $j$th basic spatial cube $Q_j$ always ranges in the set $\{h(\theta) : \theta \in \mathcal H_j(\Omega_N) \} \subseteq \mathcal R$, but the exact value of the size depends on the binary assignment $X_{Q_1}, \cdots, X_{Q_{j-1}}$ received by its basic ancestors. Similarly, a spatial cube $Q_j'$ of $j$th splitting height is random, though of course a splitting vertex in $\mathcal G_j(\Omega_N)$ is not. 
\item On the other hand, the random variable $X_{Q_j}$ that a basic cube $Q_j$ {\em{receives}} is independent of all random variables used in previous or concurrent steps of the process, by virtue of our choice of \eqref{Bernoulli warehouse}. In other words, 
\begin{equation} \label{independence}
\text{The collection of random variables } \left\{X_{Q_j} : Q_j \text{ basic} \right\} \text{ is independent.} 
\end{equation}  
This fact is vital in computing slope assignment probabilities in Sections \ref{UpperBoundSection} and \ref{probability estimation section}.  
\item Thus far, $\sigma$ has been prescribed only for basic cubes and their subcubes of splitting heights. Having achieved this, it is not difficult to extend $\sigma$ as a sticky map between the root tree and the slope tree. We address this in the next lemma. 
\end{enumerate}
\begin{lemma} \label{sigma sticky lemma} 
For every realization of $\mathbb X$, there exists a sticky map 
\[ \sigma_{\mathbb X} : \mathcal T_J(\{0\} \times [0,1)^d;M) \rightarrow \mathcal T_J(\Omega_N;M)\]
that agrees with the slope assignment algorithm prescribed in \eqref{random variable assignment}.  
 \end{lemma} 
\begin{proof}
For every $1 \leq j \leq N$, the root hyperplane is partitioned into $j$th basic spatial cubes. Any $M$-adic cube $Q$ is therefore either a basic spatial cube or contained in one. Thus there exists for every $Q$ a unique index $\bar{j} = \bar{j}_{\mathbb X}(Q)$ and a nested sequence of $k$th basic spatial cubes $Q_k$ such that 
\[ Q_{N} \subsetneq \cdots Q_{\bar{j} + 1}\subsetneq Q \subseteq Q_{\bar{j}} \subsetneq Q_{\bar{j}-1} \subsetneq \cdots \subsetneq Q_1. \]
Recalling that $\sigma(Q_k)$ has been defined for all $1 \leq k \leq N$ obeying the requirements \eqref{sigma preserves heights and lineages} of preserving height and lineage, we set \[ \sigma_{\mathbb X}(Q) := \left\{ \begin{aligned} &\text{ unique vertex in $\mathcal T_J(\Omega_N;M)$ of height } \\ &\text{ $h(Q)$ on the ray identifying $\sigma(Q_{\bar{j} + 1})$} \end{aligned} \right\}. \] 
Then $\sigma_{\mathbb X}$ is well-defined, sticky, and consistent with the prescriptions made in \eqref{random variable assignment}. 
\end{proof} 
\subsection{Theorem \ref{MainThm1} revisited}
We will now invest our efforts into proving that with positive probability the sets $K_N(\mathbb X)$ just created in \eqref{random sticky Kakeya} are of Kakeya type. 
\begin{proposition}\label{main theorem reformulated}
There exist positive absolute constants $c = c(d,M)$ and $C = C(d,M)$ obeying the property described below. For every $N \geq 1$ and $\Omega_N$ as in Proposition \ref{pruning stage 1}, the random set $K_N(\mathbb X)$ defined in \eqref{random sticky Kakeya} satisfies   the following inequalities: 
\begin{align} 
\text{Pr}\Bigl( \Bigl\{ \mathbb X :  |K_N(\mathbb X) \cap [0,1] \times \mathbb R^d| \geq c\frac{\sqrt{\log N}}{N} \Bigr\} \Bigr) &\geq \frac{3}{4}, \label{random Kakeya lower bound} \\ 
\mathbb E_{\mathbb X} \bigl|K_N(\mathbb X) \cap [A_0, A_0+1] \times \mathbb R^d \bigr| &\leq \frac{C}{N}.  \label{random Kakeya upper bound} 
\end{align} 
\end{proposition} 
The proof of the proposition will occupy the remainder of the paper, with the estimates \eqref{random Kakeya upper bound} and \eqref{random Kakeya lower bound} established in Sections \ref{UpperBoundSection} and \ref{LowerBoundSection} respectively. Before launching into them, let us observe that these two estimates combine to generate the Kakeya-type set whose existence is claimed in Theorem \ref{MainThm1} and subsequently reformulated in Proposition \ref{InfiniteSplitKakeya}.    
\begin{corollary}
Given Proposition \ref{main theorem reformulated}, the statement of Proposition \ref{InfiniteSplitKakeya} follows. Specifically, for every $N \geq 1$ there exists a realization of $\mathbb X$ for which the union of tubes defined by 
\begin{equation}  E_N  := K_N(\mathbb X) \cap [A_0, A_0+1] \times \mathbb R^d \quad \text{ obeys } \quad \frac{|E_N^{\ast}(2A_0+1)|}{|E_N|} \underset{N \rightarrow \infty}{\longrightarrow} \infty. \label{what is E_N} \end{equation} 
In other words, $\Omega$ admits Kakeya-type sets. 
\end{corollary} 
\begin{proof} 
The proof is identical to that of \cite [Proposition 2.1]{KrocPramanik}, so we briefly sketch the outline. The bound \eqref{random Kakeya upper bound} on the expected value implies that the estimate
\begin{equation} \label{random Kakeya upper bound 2}
\bigl|K_N(\mathbb X) \cap [A_0, A_0+1] \times \mathbb R^d \bigr| \leq \frac{4C}{N}
\end{equation} 
holds with probability at least $\frac{3}{4}$, by Markov's inequality. Combined with \eqref{random Kakeya lower bound}, this lets us conclude that both 
\eqref{random Kakeya upper bound 2} and 
\begin{equation} \label{random Kakeya lower bound 2} |K_N(\mathbb X) \cap [0,1] \times \mathbb R^d| \geq c\frac{\sqrt{\log N}}{N} \end{equation}  
must hold with probability at least $\frac{1}{2}$. Since $E_N^{\ast}(2A_0+1) \supseteq K_N(\mathbb X) \cap [0,1] \times \mathbb R^d$ for $E_N$ defined as in \eqref{what is E_N}, any $K_N(\mathbb X)$ obeying both \eqref{random Kakeya upper bound 2} and \eqref{random Kakeya lower bound 2} yields \[ \frac{|E_N^{\ast}(2A_0+1)|}{|E_N|} \geq c \sqrt{\log N} \rightarrow \infty,\] as claimed.   
\end{proof} 

\section{Proof of the upper bound \eqref{random Kakeya upper bound}} \label{UpperBoundSection}

\begin{proposition} \label{percolation proposition}
There exists a positive constant $C$ possibly depending on $d$ and $M$ but uniform in $x \in [A_0, A_0 +1] \times \mathbb R^d$ such that the probability Pr$(x) := \text{Pr}(x \in K_N(\mathbb X))$ obeys the estimate
\begin{equation} \label{bound on Pr(x)} \text{Pr}(x) \leq \frac{C}{N}. \end{equation} 
As a consequence, \eqref{random Kakeya upper bound} holds.  
\end{proposition} 
\begin{proof}
The proof of \eqref{bound on Pr(x)} is a consequence of the three lemmas stated and proved below in this section. In Lemma \ref{Relating Pr to survival} and following the direction laid out in \cite{{BatemanKatz}, {Bateman}}, we establish that Pr$(x)$ is bounded above by the probability that the reference tree $\mathcal N_x$ survives a Bernoulli$(\frac{1}{2})$ percolation, as described in the appendix (Section \ref{percolation on trees section}). The details of the specific percolation criterion that permit this correspondence are described in Lemma \ref{percolation well-defined}. Using general facts about percolation collected in Section \ref{percolation on trees section} and information on $\mathcal N_x$ observed in Section \ref{general facts about tube families}, we compute in Lemma \ref{computing survival probability} a bound on the survival probability that is uniform in $x$ to obtain the claimed estimate \eqref{bound on Pr(x)}.   

Given \eqref{bound on Pr(x)}, the upper bound in \eqref{random Kakeya upper bound} follows easily. Since $\Omega_N \subseteq \{1\} \times [0,1)^d$, any tube, and hence $K_N(\mathbb X)$, is contained in the compact set $[0, 10 A_0]^{d+1}$. Thus 
\[ K_N(\mathbb X) \cap [A_0, A_0+1] \times \mathbb R^d= K_N(\mathbb X) \cap [A_0, A_0+1] \times [0, 10A_0]^d,\] 
and hence 
\begin{align*}
\mathbb E_{\mathbb X} \bigl|K_N(\mathbb X) \cap [A_0, A_0+1] \times \mathbb R^d \bigr| &= \mathbb E_{\mathbb X} \int_{[A_0, A_0+1] \times [0, 10A_0]^d} 1_{K_N(\mathbb X)}(x) \, dx \\ &= \int_{[A_0, A_0+1] \times [0, 10A_0]^d} \mathbb E_{\mathbb X}\bigl(1_{K_N(\mathbb X)}(x) \bigr) \, dx \\ & = \int_{[A_0, A_0+1] \times [0, 10A_0]^d} \text{Pr}(x) \, dx \\ &\leq \frac{C}{N}, 
\end{align*}
completing the proof. 
\end{proof} 
Much of the groundwork for Lemma \ref{Relating Pr to survival} has already been established in Section \ref{tubes and point section}. In particular, let us recall the definition of the reference tree $\mathcal N_x$ and reference cubes $Q_j^{\ast}(t)$ from \eqref{defn N}, \eqref{defn Phi_j}, and \eqref{s_j and theta_j}.  We will also need the reference slope function $\kappa$ as in \eqref{defn overlay function} defined on the edges of $\mathcal N_x$. Motivated by Lemma \ref{percolation preparation lemma}(\ref{inclusion criterion}), we define a random variable for each edge of $\mathcal N_x$:  
\begin{equation} \label{percolation criterion}
Y_e = Y_e(\mathbb X) := \begin{cases} 1 &\text{ if } X_{Q_{j+1}^{\ast}(t)} = \kappa(e), \\ 0 &\text{ otherwise, }\end{cases} 
\end{equation} 
where as usual $e$ denotes the edge in $\mathcal N_x$ joining $\Phi_j(t)$ and $\Phi_{j+1}(t)$. As described in Section \ref{percolation on trees section}, we use $Y_e$ to determine whether to retain or to remove the edge $e$ in $\mathcal N_x$, the value zero corresponding to removal. We emphasize that a reference cube $Q_{j+1}^{\ast}(t)$ is a deterministic vertex of the tree representing the root hyperplane, and need not in general coincide with the $(j+1)$th basic spatial cube $Q_{j+1}(t)$ described in \eqref{chain of basic cubes containing t}. The important point, as we will see in Lemma \ref{Relating Pr to survival}, is that if $x \in K_N(\mathbb X)$, then these two cubes do match for some $t$ and for every $j$.   
\begin{lemma} \label{percolation well-defined}
The retention-removal criterion described in \eqref{percolation criterion} gives rise to a well-defined Bernoulli$(\frac{1}{2})$ percolation on $\mathcal N_x$. 
\end{lemma}
\begin{proof} 
Since $Q_{j+1}^{\ast}(t)$ identifies the terminating vertex of the edge $e$, any two representations $\Phi_{j+1}(t) = \Phi_{j+1}(t')$ of this vertex gives rise to $Q_{j+1}^{\ast}(t) = Q_{j+1}^{\ast}(t')$. So $X_{Q_{j+1}^{\ast}(t)}$ is consistently defined on the edges. We have already seen in Corollary \ref{kappa makes sense} that $\kappa$ is a well-defined function on the edge set of $\mathcal N_x$, hence so is $Y_e$. The probability that $Y_e$ equals one is clearly $1/2$ since it is given by the Bernoulli$(\frac{1}{2})$ random variable $X_{Q_{j+1}^{\ast}(t)}$. Finally, any two distinct edges $e$ and $e'$ must have distinct terminating vertices, and therefore end in distinct reference cubes. The random variable assignments for such cubes are independent by our assumption on $\mathbb X$. Hence the events of retention and removal are independent for different edges, and the result follows.   
\end{proof}  
\begin{lemma} \label{Relating Pr to survival} 
Let $x$ be a point in $[A_0, A_0+1] \times \mathbb R^d$. If $x \in K_N(\mathbb X)$, then there is at least one ray of full length in $\mathcal N_x$ all of whose edges are retained after the percolation described by $Y_e(\mathbb X)$. As a result, the probability $\text{Pr}(x)$ defined in Proposition \ref{percolation proposition} admits the bound
\begin{equation}  \text{Pr}(x) \leq p^{\ast}(x),
\label{Pr(x) bounded above by survival probabilities} \end{equation}  
where $p^{\ast}(x)$ denotes the survival probability of $\mathcal N_x$ under the Bernoulli$(\frac{1}{2})$ percolation given in \eqref{percolation criterion}.   
\end{lemma} 
\begin{proof}
If $x \in K_N(\mathbb X)$, then by Lemma \ref{percolation preparation lemma}(\ref{inclusion criterion}) there exists $t \in \text{Poss}(x)$ such that the ray identifying $t$ is common to $\mathcal N_x$ and $\mathcal N_x(\sigma_{\mathbb X})$. Restating \eqref{comparing Phi}, this means that $\Phi_j(t) = \Phi_j(t;\sigma)$ for all $1 \leq j \leq N$. But the left hand side of the preceding equality identifies the (deterministic) $j$th reference cube containing $t$, whereas the right hand side represents the (random) $j$th basic spatial cube containing $t$. In other words, we find that $Q_j(t ; \mathbb X) = Q_{j}^{\ast}(t)$ for all $1 \leq j \leq N$, and hence 
\[ \iota_{\sigma}(e) = X_{Q_{j+1}(t;\mathbb X)} = X_{Q_{j+1}^{\ast}(t)}. \]
Combined with \eqref{percolation criterion} and \eqref{equality on each edge}, this implies the existence of an entire ray in $\mathcal N_x$ (namely the one identifying $t$) that survives the percolation given by $Y_e$.  Summarizing, we obtain that 
\[ \{ \mathbb X : x \in K_N(\mathbb X)\} \subseteq 
\left\{ \mathbb X : \begin{aligned} &\mathcal N_x \text{ survives the Bernoulli$\bigl(1/2 \bigr)$} \\ &\text{percolation dictated by $Y_e(\mathbb X)$} \end{aligned} \right\}, \] 
from which \eqref{Pr(x) bounded above by survival probabilities} follows.   
\end{proof}
\begin{lemma} \label{computing survival probability} 
There is a positive constant $C$ that is uniform in $x \in [A_0, A_0+1] \times \mathbb R^d$ such that the survival probability $p^{\ast}(x)$ of $\mathcal N_x$ under Bernoulli$(\frac{1}{2})$ percolation is $\leq \frac{C}{N}$. 
\end{lemma} 
\begin{proof}
In view of Corollary \ref{survival probability reduced}, $p^{\ast}(x)$ is bounded above by 
\[ \left[ \sum_{j=1}^{N} \frac{2^j}{n_j(x)} \right]^{-1} \text{ where } n_j(x) = \text{ number of vertices in $\mathcal N_x$ of height $j$}.  \] 
But Lemma \ref{vertex counting in N_x} gives that $n_j(x) \leq C2^j$, which leads to the stated bound. 
\end{proof} 

\section{Probability estimates for slope assignments} \label{probability estimation section} 
We now turn to \eqref{random Kakeya lower bound}, where we need to establish that with high probability, the volume of space close to the root hyperplane is much more widely populated by the random set $K_N(\mathbb X)$ than away from it. As indicated in Section \ref{layout section}, the proof requires detailed knowledge of the probability that a given subset of root cubes receives prescribed slope assignments. We establish the necessary probabilistic estimates in this section for easy reference in the proof of \eqref{random Kakeya lower bound}, which is presented in Section \ref{LowerBoundSection}. 
\subsection{A general rule}
To get started, let us recall from Section \ref{random construction section} that a slope assigned to a root is not completely arbitrary and has to obey the requirement of stickiness. The definition below, introduced to avoid vacuous root-slope combinations, draws attention to this constraint.  
\begin{definition}
Let $A$ be a collection of root cubes and $\Gamma_A = \{\alpha(t) : t \in A \} \subseteq \Omega_N$ a collection of slopes indexed by $A$. We say that the collection of root-slope pairs 
\begin{equation} \label{sticky admissible collection}  \left\{(t, \alpha(t))\, : \, t \in A \subseteq \mathcal Q(J), \; \alpha(t) \in \Gamma_A \subseteq \Omega_N \right\} \end{equation} 
is sticky-admissible if there exists a realization of $\mathbb X$ as in \eqref{Bernoulli warehouse} for which the sticky map $\sigma_{\mathbb X}$ described in Section \ref{random construction section} has the property that 
\begin{equation} \sigma_{\mathbb X}(t) = \alpha(t) \quad \text{ for all } t \in A. \label{sticky admissible event} \end{equation}   
\end{definition}
Given a sticky-admissible collection \eqref{sticky admissible collection}, we first prescribe a general algorithm for computing the probability of the event \eqref{sticky admissible event}. Preparatory to stating the result, let us define two collections consisting of tuples of vertices from the root tree and the slope tree respectively:
\begin{align}
\mathbb N (A;\alpha) &:= \{ \Phi_j(t; \alpha) : t \in A, \; 0 \leq j \leq N\},  \label{defn N(A, alpha)}\\   
\mathbb M(A;\alpha) &:= \{ \Theta_j(t; \alpha) : t \in A, \; 0 \leq j \leq N \}. \label{defn M(A, alpha)} 
\end{align}
These objects are analogous to the trees \eqref{defn N omega} and \eqref{defn M omega} introduced earlier, with the usual interpretation of $\Phi_j(t;\alpha)$ and $\Theta_j(t; \alpha)$ following those definitions. Namely, for $j \geq 1$, the element $\Theta_j(t;\alpha)$ is a vector  with $j$ entries, whose $i$th component represents the $i$th basic slope cube in $\mathcal T_J(\Omega_N)$ containing $\alpha(t)$. The vector $\Phi_j(t;\alpha)$ is also a $j$-long sequence. Its $i$th entry represents the unique cube containing $t$ located at the same height as the $i$th entry of $\Theta_j(t; \alpha)$. This common height is $\eta_i(\alpha(t))$ defined as in \eqref{defn of eta_j}. Not surprisingly, for a choice $A$ and $\alpha$ that gives rise to a sticky-admissible collection \eqref{sticky admissible collection}, the collections $\mathbb N(A;\alpha)$ and $\mathbb M(A;\alpha)$ are indeed trees (with the $0$th generations removed) that contain the information required for computing the probability of the event \eqref{sticky admissible event}. This is the content of Lemma \ref{general probability lemma} below, which forms the computational framework for all the probability estimates in this section. 
\begin{lemma} \label{general probability lemma}
Let $A \subseteq \mathcal Q(J)$ and $\Gamma_A = \{\alpha(t) : t \in A \} \subseteq \Omega_N$ be sets for which the collection given in \eqref{sticky admissible collection} is sticky-admissible. Then the following conclusions hold. 
\begin{enumerate}[(i)]
\item The collections $\mathbb N(A;\alpha)$ and $\mathbb M(A;\alpha)$ defined in \eqref{defn N(A, alpha)} and \eqref{defn M(A, alpha)} are well-defined trees in which  $\Phi_j(t;\alpha)$ and $\Theta_j(t;\alpha)$ are deemed vertices of height $j$, and parents of $\Phi_{j+1}(t;\alpha)$ and $\Theta_{j+1}(t;\alpha)$ respectively. 
\item If $n(A;\alpha)$ denotes the total number of vertices in $\mathbb N(A;\alpha)$ not counting the root,  then 
\begin{equation}  \text{Pr}(\sigma_{\mathbb X}(t) = \alpha(t) \text{ for all } t \in A) = 2^{-n(A;\alpha)}. \label{general probability equation}  \end{equation} 
\end{enumerate} 
\end{lemma} 
\begin{proof}
The proof of the first claim bears a close resemblance with that of Lemma \ref{percolation preparation lemma}. To check that the trees are well-defined, we pick root cubes $t \ne t'$ with $u = D(t,t')$ and aim to show that 
\[ \Phi_j(t;\alpha) = \Phi_j(t', \alpha) \quad \text{ and } \quad \Theta_j(t;\alpha) = \Theta_j(t', \alpha)\]
for all $j$ such that $\eta_j(\alpha(t)) \leq h(u)$. But the collection \eqref{sticky admissible collection} is sticky-admissible by hypothesis, hence there is a sticky map $\sigma$ such that \eqref{sticky admissible event} holds. By the property of stickiness,
\[ h(D(\alpha(t), \alpha(t'))) = h(D(\sigma(t), \sigma(t'))) \geq h(u).\] 
Since $\alpha(t)$ and $\alpha(t')$ agree up to height $h(D(\alpha(t), \alpha(t')))$, it follows that $\Theta_j(t;\alpha)$ and $\Theta_{j}(t';\alpha)$ must match for $\eta_j(\alpha(t)) \leq h(u)$. This in turn implies that $\Phi_j(t;\alpha) = \Phi_j(t';\alpha)$.   

We turn now to the proof of \eqref{general probability equation}. Let us write
\begin{equation}  \Phi_j(t;\alpha) = \bigl(Q_1^{\ast}(t;\alpha), \cdots, Q_j^{\ast}(t;\alpha) \bigr) \quad \text{ and } \quad \Theta_j(t;\alpha) = \bigl( \theta_1(t;\alpha), \cdots, \theta_j(t;\alpha)\bigr). \label{Phi_j and Theta_j for alpha}\end{equation}  
In order to describe the event of interest, we need to recall from \eqref{chain of basic cubes containing t} the definition of basic spatial cubes $Q_j(t)$ containing $t$, their role in the random construction as explained in Section \ref{random construction section}, and also the definition of the maps $\Psi_j$ and $\Psi$ from \eqref{defn Psi_j} and Proposition \ref{Splitting tree lemma}.  Putting these together we find that 
\begin{align}
\bigl\{ \sigma_{\mathbb X}(t) = \alpha(t) &\text{ for all } t \in A \bigr\} \nonumber \\ &= \Bigl\{ \sigma_{\mathbb X}(Q_j(t)) = \theta_j(t;\alpha) \text{ for all $1 \leq j \leq N$ and all $t \in A$}\Bigr\} \nonumber \\ &= \Bigl\{ \Psi_N(X_{Q_1(t)}, \cdots, X_{Q_N(t)}) = \alpha(t) \text{ for all $t \in A$}\Bigr\} \nonumber \\&= \bigcap_{j=1}^N \bigcap_{t \in A}\Bigl\{X_{Q_j(t)} = \pi_j \circ \Psi^{-1} \circ \alpha(t) \Bigr\} \nonumber \\ &= \bigcap_{ j =1}^N \bigcap_{t \in A}\Bigl\{ Q_j(t) = Q_j^{\ast}(t;\alpha) \text{ and } X_{Q_j^{\ast}(t;\alpha)} = \pi_j \circ \Psi^{-1} \circ \alpha(t) \Bigr\}.    \label{final event}
\end{align} 
Here $\pi_j$ denotes the projection onto the $j$th component of an input sequence.  In the first two steps of the string of equations above, we have used the definition \eqref{random variable assignment} of $\sigma$ and its stickiness as ensured by Lemma \ref{sigma sticky lemma}. To justify the last step we observe that $Q_1(t) = Q_1^{\ast}(t;\alpha)$ is non-random; further if it is given that $Q_{\ell}(t) = Q_{\ell}^{\ast}(t;\alpha)$ for all $\ell \leq j$, then the additional requirement   
\[  X_{Q_j(t)} = \pi_j \circ \Psi^{-1} \circ \alpha(t) \quad \text{ implies } \quad Q_{j+1}(t) = Q_{j+1}^{\ast}(t;\alpha),\]
leading to the conclusion in \eqref{final event}. By virtue of our assumption of sticky-admissibility, the event described above is of positive probability; in particular the value assignment to the random variables in $\mathbb X$ as prescribed in \eqref{final event} is consistent, i.e., for $t \ne t'$, \[  \pi_j \circ \Psi^{-1} \circ \alpha(t) = \pi_j \circ \Psi^{-1} \circ \alpha(t') \quad \text{ whenever } \quad Q_j^{\ast}(t;\alpha) = Q_j^{\ast}(t';\alpha).\] In view of our assumption \eqref{Bernoulli warehouse} on the distribution of $\mathbb X$, the probability of the event in \eqref{final event} is half raised to a power that equals the number of distinct cubes in the collection $\{ Q_j^{\ast}(t;\alpha); 1 \leq j \leq N, \; t \in A\}$, in other words $n(A;\alpha)$. 
\end{proof} 
\subsection{Root configurations} \label{root configurations subsection}
Application of Lemma \ref{general probability lemma} requires explicit knowledge of the structure of the trees $\mathbb N(A;\alpha)$ and $\mathbb M(A;\alpha)$, from which $n(A;\alpha)$ can be computed. These objects depend in turn on the trees depicting $A$ and $\Gamma_A$. We now proceed to compute $n(A;\alpha)$ in some simple situations where $\#(A) \leq 4$. On one hand, the small size of $A$ permits the classification of possible root configurations into relatively few categories, each of which gives rise to a specific $n(A;\alpha)$. On the other hand, these cases cover all the probabilistic estimates that we will need in Section \ref{LowerBoundSection}. 

While each root configuration requires distinct consideration, it is recommended that the first-time reader focus on the cases when $\#(A) = 2$, and when $\#(A) = 4$ with the four roots in what we call a type 1 configuration (see Definition \ref{four point type definition}).  These cases contain many of the main ideas needed to push through the proof of the lower bound on the size of a typical $K_N(\mathbb{X})$ claimed in \eqref{random Kakeya lower bound}, Proposition \ref{main theorem reformulated}.  A thorough treatment of all distinct cases when $\#(A)\leq 4$ is needed to completely establish Proposition~\ref{main theorem reformulated}, but focusing on the two recommended cases should make the arguments far easier to absorb upon a first reading.  When $\#(A)=2$ in particular, the reader may focus attention on Lemmas~\ref{two point lemma}, \ref{C count lemma}, \ref{splitting vertex summation lemma} and~\ref{root tree summation lemma}, and the application of these lemmas in the proof of Proposition~\ref{first moment proposition}.  The treatment of the case of four distinct roots in type 1 configuration has been carried out on Lemmas~\ref{four point type 1 lemma}, \ref{E_{41} size lemma}, \ref{splitting vertex summation lemma} and~\ref{root tree summation lemma}, with the application of these lemmas occurring in the proof of Proposition~\ref{second moment proposition}, for which this is the generic case.

\subsection{Notation} Throughout this section the following notation will be used, in conjunction with the terminology of root hyperplane, root tree and root cube already set up in Section \ref{general facts about tube families}, page \pageref{defn Q_k}. Since any vertex $\Phi_j(t;\alpha)$ in $\mathbb N(A;\alpha)$ is uniquely identified by its last component $Q_j^{\ast}(t;\alpha)$ defined as in \eqref{Phi_j and Theta_j for alpha}, we write \begin{equation} \label{congruency notation}\Phi_j(t;\alpha) \cong Q_j^{\ast}(t;\alpha), \end{equation}  often opting to describe the left hand side by the right. In particular if $j=N$, then $\Phi_N(t;\alpha) \cong Q_N^{\ast}(t;\alpha) = t$, in which case the latter notation is used instead of the (more cumbersome) former. 

Given a vertex $u$ in the root tree, a vertex $\omega \in \mathcal T_J(\Omega_N)$ and a positive integer $k$ that is no larger than either $h(\omega)$ or $h(u)$, we also define  
 \begin{align}
\theta(\omega,k) &:= \text{ the basic slope cube containing $\omega$ of maximal height $\leq k$}, \text{ and } \label{defn theta} \\ 
\mu(\omega,k) &:= \text{ $j$ if } \theta(\omega,k) \in \mathcal H_j(\Omega_N) \nonumber \\  
&\;= \text{ number of basic slope cubes of height $\leq k$ that contain $\omega$}, \quad \text{ and } \label{defn mu} \\
Q_u[\omega,k] &:= \text{ ancestor of $u$ in the root tree at height $\mu(\omega,k)$.} \label{defn Q_u}
\end{align}
Figure \ref{Fig: theta and mu quantities} on page \pageref{Fig: theta and mu quantities} depicts these quantities. If $\omega' \subseteq \omega$ and/or $u' \subseteq u$, then it follows from the definitions above that 
\begin{align*} \theta(\omega,k) &= \theta(\omega',k), \quad \mu(\omega, k) = \mu(\omega',k),  \text{ and } \\ Q_{u}[\omega, k] &= Q_{u'}[\omega, k] = Q_{u}[\omega', k] = Q_{u'}[\omega',k]. \end{align*}  
These facts will be frequently used in the sequel without further reference. 
\vskip0.2cm
\begin{figure}[h!] 
\setlength{\unitlength}{0.8mm}
\begin{picture}(-50,0)(-75,10)

        \put(0,0){\special{sh 0.99}\ellipse{2}{2}}
        \put(-15,-25){\special{sh 0.99}\ellipse{2}{2}}
        \put(-15,-25){\circle{4}}
        \put(15,-25){\special{sh 0.99}\ellipse{2}{2}}
        \put(15,-25){\circle{4}}
        \put(25,-45){\special{sh 0.99}\ellipse{2}{2}}
        \put(45,-100){\special{sh 0.99}\ellipse{2}{2}}
        \put(35,-60){\special{sh 0.99}\ellipse{2}{2}}
        \put(35,-60){\circle{4}}        
        \put(10,-70){\special{sh 0.99}\ellipse{2}{2}}   
        \put(16,-60){\special{sh 0.99}\ellipse{2}{2}}
        \put(16,-60){\circle{4}}

        \path(-15,-25)(0,0)(15,-25)(25,-45)
        \dottedline{2}(-25,-50)(45,-50)
        \dottedline{2}(-25,-35)(35,-35)
        \dottedline{2}(-5,-65)(55,-65)
        \path(-25,-45)(-15,-25)(-5,-45)
        \path(35,-60)(25,-45)(10,-70)        
        \path(30,-75)(35,-60)(45,-75)        
        \path(5,-80)(10,-70)(15,-80)        
        \dashline{2}(45,-75)(45,-100)

\put(-2,5){\Large\shortstack{$\omega_j\in\mathcal {G}_{j}(\Omega_N)$}}
        \put(15,-22){\Large\shortstack{$\theta(\omega,k_1) = \theta(\omega,k_2)\in\mathcal {H}_j(\Omega_N)$}
}
        \put(37,-58){\Large\shortstack{$\theta(\omega,k_3)\in\mathcal {H}_{j+1}(\Omega_N)$}}        
        \put(-55,-51){\large\shortstack{height $=k_2$}}
        \put(-55,-36){\large\shortstack{height $=k_1$}}
        \put(-35,-66){\large\shortstack{height $=k_3$}}
        \put(43,-107){\Large\shortstack{$\omega\in\mathcal{T}_J(\Omega_N)$}}


\end{picture}
\vspace{9.5cm}
\caption{\label{Fig: theta and mu quantities} Given $\omega\in \Omega_N$ and a set of heights $k_i$, $i=1,2,3$, the basic slope cubes $\theta(\omega,k_i)$ are identified.  Here $\mu(\omega,k_1)=\mu(\omega,k_2)=j$ and $\mu(\omega,k_3)=j+1$. 
All vertices depicting basic slope cubes are circled.}
\end{figure}
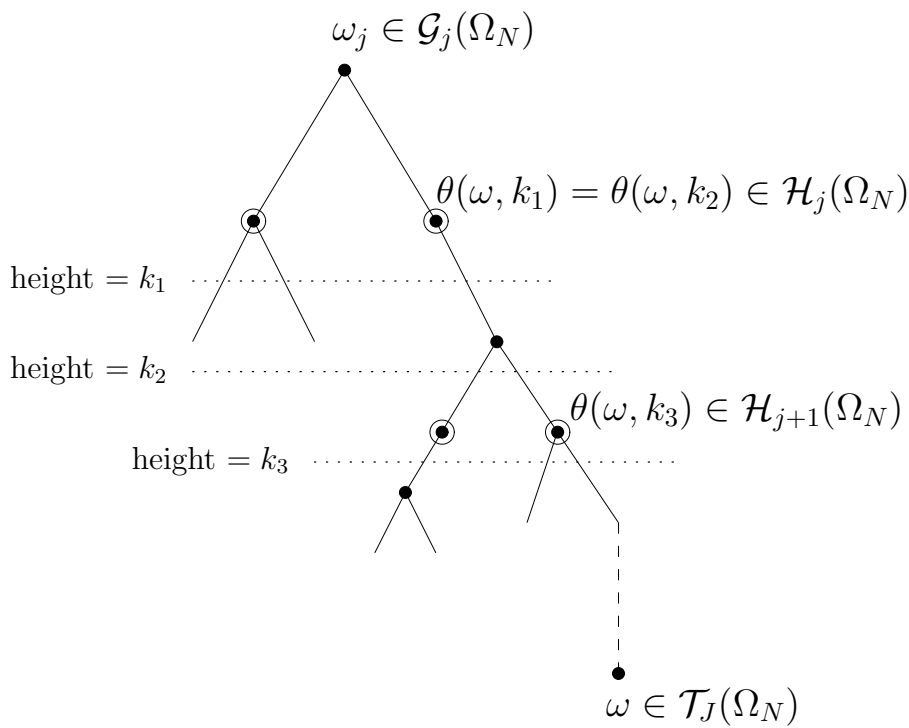

\subsubsection{The case of two roots} 
We start with the simplest case when $A$ consists of two root cubes. 
\begin{lemma}\label{two point lemma} 
Let $A = \{ t_1, t_2 \}$ be two distinct root cubes and $\Gamma_A = \{ \alpha(t_1) = v_1, \alpha(t_2) = v_2\} \subseteq \Omega_N$ be a subset of (not necessarily distinct) slopes such that $\{(t_1, v_1), (t_2, v_2)\}$ is sticky-admissible. If $u = D(t_1, t_2)$, $\omega = D(v_1, v_2)$ and $k = h(u)$, then $k \leq h(\omega)$, and 
\begin{equation}  \text{Pr}\bigl( \sigma(t_1) = v_1, \sigma(t_2) = v_2 \bigr) = \left(\frac{1}{2}\right)^{2N-\mu(\omega,k)}. \label{probability for two roots} \end{equation}  
\end{lemma} 
\begin{proof}
Since there exists a sticky map $\sigma$ such that $\sigma(t_i) = v_i$ for $i=1,2$, we see that 
\begin{equation} \label{height and stickiness} h(\omega) = h\bigl( D(v_1, v_2 )\bigr) = h \bigl(D(\sigma(t_1), \sigma(t_2)) \bigr) \geq h \bigl( D(t_1, t_2)\bigr) = h(u) = k. \end{equation} 
In order to establish \eqref{probability for two roots} we invoke Lemma \ref{general probability lemma}. The tree $\mathbb N(A;\alpha)$ consists of two rays terminating at $\Phi_N(t_1;\alpha) \cong t_1$ and $\Phi_N(t_2;\alpha) \cong t_2$ respectively,  according to the notational rule prescribed in \eqref{congruency notation}.  Letting $u_{\mathbb N} = D_{\mathbb N}(t_1, t_2)$ denote the youngest common ancestor of $t_1$ and $t_2$ in $\mathbb N(A;\alpha)$, we observe that $u_{\mathbb N} \cong Q_u[\omega, k]$, with $Q_u[\omega,k]$ defined as in \eqref{defn Q_u}. Thus $u_{\mathbb N}$ lies at height $\mu(\omega, k)$ in $\mathbb N(A;\alpha)$. This allows us to compute $n(A;\alpha)$ as follows: $n(A;\alpha) = \mu(\omega, k) + 2(N - \mu(\omega, k)) = 2N - \mu(\omega,k)$. 
\end{proof} 
\subsubsection{The case of three roots}
Next we turn to the slightly more complex event where three distinct root cubes receive prescribed slopes. Here for the first time we observe the dependence of slope assignment probabilities on configuration types of the roots. 
\begin{definition}
Let $t_1, t_2, t_2'$ be three distinct root cubes. We say that the ordered tuple $\mathbb I = \{(t_1, t_2);(t_1, t_2') \}$ with 
\begin{equation} \label{u u' height relation}
u = D(t_1, t_2), \quad u' = D(t_1, t_2'), \quad u' \subseteq u
\end{equation} 
is in type 1 configuration if exactly one of the following conditions hold: 
\begin{enumerate}[(a)]
\item $u' \subsetneq u$, or
\item $u = u' = D(t_2, t_2')$.  
\end{enumerate}
A tuple $\mathbb I$ that obeys \eqref{u u' height relation} but is not of type 1 is said to be of type 2. Thus for $\mathbb I$ of type 2, one must have $u=u'$ and additionally $t = D(t_2, t_2')$ satisfies $t \subsetneq u$. If $\mathbb I = \{(t_1, t_2); (t_1, t_2')\}$ with the same definitions of $u$ and $u'$ does not meet the containment relation required by \eqref{u u' height relation}, i.e., if $u \subsetneq u'$, then we declare $\mathbb I$ to be of the same type as $\mathbb I' = \{(t_1, t_2');(t_1, t_2) \}$. 
\end{definition} 
The different structural possibilities are shown in Figure \ref{Fig: 3 point configurations}.   
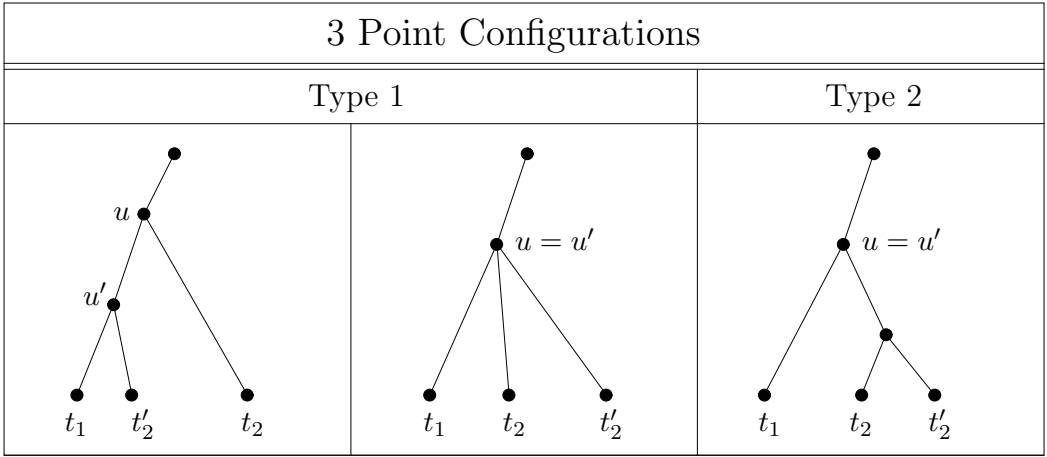
\begin{figure}[h]
\setlength{\unitlength}{0.8mm}
\begin{picture}(-50,0)(-106,35)

        \allinethickness{0.1mm}\path(-100,25)(-100,-40)(71,-40)(71,25)(-100,25)
        \path(-100,24)(71,24)
        \path(-100,15)(71,15)
        \path(-43,15)(-43,-40)
        \path(14,24)(14,-40)
        \path(-100,-40)(71,-40)
        \path(-100,25)(-100,35)(71,35)(71,25)

        \put(-47,28){\Large\shortstack{3 Point Configurations}}
        \put(-50,18){\large\shortstack{Type 1}}
        \put(35,18){\large\shortstack{Type 2}}

        \special{sh 0.99}\put(-72,10){\ellipse{2}{2}}
        \put(-82,-1){\shortstack{$u$}}
        \special{sh 0.99}\put(-77,0){\ellipse{2}{2}}
        \put(-87,-15){\shortstack{$u'$}}
        \special{sh 0.99}\put(-82,-15){\ellipse{2}{2}}
        \put(-90,-36){\shortstack{$t_1$}}
        \special{sh 0.99}\put(-88,-30){\ellipse{2}{2}}
        \put(-79,-36){\shortstack{$t_2'$}}
        \special{sh 0.99}\put(-79,-30){\ellipse{2}{2}}
        \put(-61,-36){\shortstack{$t_2$}}
        \special{sh 0.99}\put(-60,-30){\ellipse{2}{2}}

        \path(-72,10)(-77,0)
        \path(-82,-15)(-77,0)(-60,-30)
        \path(-88,-30)(-82,-15)(-79,-30)


\special{sh 0.99}\put(-14,10){\ellipse{2}{2}}
        \put(-16,-6){\shortstack{$u = u'$}}
        \special{sh 0.99}\put(-19,-5){\ellipse{2}{2}}
        \put(-2,-36){\shortstack{$t_2'$}}
        \special{sh 0.99}\put(-30,-30){\ellipse{2}{2}}
        \put(-31,-36){\shortstack{$t_1$}}
        \special{sh 0.99}\put(-17,-30){\ellipse{2}{2}}
        \put(-18,-36){\shortstack{$t_2$}}
        \special{sh 0.99}\put(-1,-30){\ellipse{2}{2}}

        \path(-14,10)(-19,-5)(-30,-30)
        \path(-17,-30)(-19,-5)(-1,-30)


        \special{sh 0.99}\put(43,10){\ellipse{2}{2}}
        \special{sh 0.99}\put(38,-5){\ellipse{2}{2}}
        \put(41,-6){\shortstack{$u = u'$}}
        \special{sh 0.99}\put(45,-20){\ellipse{2}{2}}
        \put(24,-36){\shortstack{$t_1$}}
        \special{sh 0.99}\put(25,-30){\ellipse{2}{2}}
        \put(39,-36){\shortstack{$t_2$}}
        \special{sh 0.99}\put(41,-30){\ellipse{2}{2}}
        \put(52,-36){\shortstack{$t_2'$}}
        \special{sh 0.99}\put(53,-30){\ellipse{2}{2}}

        \path(43,10)(38,-5)(25,-30)
        \path(53,-30)(45,-20)(41,-30)
        \path(38,-5)(45,-20)


\end{picture}
\vspace{6cm}
\caption{\label{Fig: 3 point configurations} All possible configurations of three distinct root cubes.}
\end{figure}

As in Lemma \ref{two point lemma}, the quantity $\mu$ defined in \eqref{defn mu} when evaluated at certain vertices of the slope tree dictated by $A = \{ t_1, t_2, t_2'\}$ provides the value of $n(A;\alpha)$ necessary for estimating the probability in \eqref{general probability equation}. 
\begin{lemma} \label{three point type 1 lemma} 
Let $A = \{ t_1, t_2, t_2' \}$ be three distinct root cubes such that the ordered tuple $\mathbb I = \{(t_1, t_2);(t_1, t_2') \}$ obeys \eqref{u u' height relation} and is of type 1. Set \[k = h(u), \quad k' = h(u').\] Suppose that $\Gamma_{A} = \{\alpha(t_1) = v_1, \; \alpha(t_2) = v_2, \; \alpha(t_2') = v_2' \}\subseteq \Omega_N$ is a subset of (not necessarily distinct) directions such that the collection $\{(t_1, v_1); (t_2, v_2); (t_2', v_2') \}$ is sticky-admissible.  Then the vertices defined by    
\[ \omega = D(v_1, v_2), \quad \omega' = D(v_1, v_2')\]
must satisfy the height relations 
\begin{equation}  k \leq h(\omega), \quad k' \leq h(\omega') \label{three point height relations} \end{equation}   
and the following equality holds:
\[
\text{Pr}\bigl( \sigma(t_1) = v_1, \; \sigma(t_2) = v_2, \; \sigma(t_2') = v_2' \bigr) \\ 
= \left( \frac{1}{2}\right)^{3N - \mu(\omega,k) - \mu(\omega', k')}. 
\]
\end{lemma} 
\begin{proof}
The inequalities in \eqref{three point height relations} are proved exactly as in Lemma \ref{two point lemma}; we omit these. The probability is again computed using Lemma \ref{general probability lemma}, via counting $n(A;\alpha)$. The tree $\mathbb N = \mathbb N(A;\alpha)$ now consists of three rays, terminating at $\Phi_N(t_1;\alpha)$, $\Phi_N(t_2;\alpha)$ and $\Phi_N(t_2';\alpha)$, which are identified with $t_1$, $t_2$ and $t_2'$ respectively. Let us recall from the proof of Lemma \ref{two point lemma} that $u_{\mathbb N} = D_{\mathbb N}(t_1, t_2)$ denotes the $M$-adic cube specifying the youngest common ancestor of $t_1$ and $t_2$ in $\mathbb N(A;\alpha)$. The vertex $u'_{\mathbb N} = D_{\mathbb N}(t_1, t_2')$ is defined similarly.  
Then using the notation \eqref{congruency notation}, \begin{equation}  u'_{\mathbb N} \cong Q_{u'}[\omega', k'] = Q_{t_1}[v_1,k'], \quad \text{ and } \quad  u_{\mathbb N} \cong Q_{u}[\omega,k] = Q_{t_1}[v_1,k]. \label{youngest ancestors 3 pt type 1} \end{equation}  Since $k \leq k'$, it follows from \eqref{youngest ancestors 3 pt type 1} above that $u'_{\mathbb N}\subseteq u_{\mathbb N}$. If $h_{\mathbb N}(\cdot)$ denotes the height of a vertex within the tree $\mathbb N(A;\alpha)$, then \eqref{youngest ancestors 3 pt type 1} also yields \[ h_{\mathbb N} \bigl(u_{\mathbb N}\bigr) = \mu(\omega,k) \quad \text{ and } \quad h_{\mathbb N}\bigl(u'_{\mathbb N}\bigr) = \mu(\omega',k'), \quad \text{ so that } \quad \mu(\omega,k) \leq \mu(\omega',k'). \] Using these relations and referring to Figure \ref{Fig: 3 point configurations}, we compute $n(A;\alpha)$ as follows,  
\begin{align*} n(A;\alpha)  &= \underset{\text{vertices on the ray of $t_2$ in $\mathbb N$}}{\underbrace{h_{\mathbb N} \bigl( u_{\mathbb N}\bigr) + \bigl[ N -   h_{\mathbb N} \bigl( u_{\mathbb N}\bigr) \bigr]}} + \underset{\text{vertices between $u_{\mathbb N}$ and $u'_{\mathbb N}$}}{\underbrace{\bigl[ h_{\mathbb N} \bigl( u'_{\mathbb N}\bigr) -h_{\mathbb N} \bigl( u_{\mathbb N}\bigr)  \bigr]}} + \underset{\text{vertices below $u'_{\mathbb N}$}}{\underbrace{2 \bigl[ N -h_{\mathbb N}(u'_{\mathbb N})\bigr]}} \\ &= \mu(\omega, k) + \bigl[ N - \mu(\omega, k) \bigl] + \bigl[ \mu(\omega',k') - \mu(\omega, k)\bigr] + 2\bigl[N - \mu(\omega',k') \bigr] \\ &= 3N - \mu(\omega,k) - \mu(\omega', k'),  
\end{align*}
which leads to the desired probability estimate by Lemma \ref{general probability lemma}.       
\end{proof}
\begin{lemma} \label{three point type 2 lemma} 
Let $A = \{ t_1, t_2, t_2'\}$ be three distinct root cubes such that the ordered tuple $\mathbb I = \{(t_1, t_2);(t_1, t_2') \}$ obeys \eqref{u u' height relation} and is of type 2. Set \[ k = h(u) = h(u'), \quad \text{ and } \quad \ell = h(t) \quad \text{ where } \quad t = D(t_2, t_2') \subsetneq u = u'.\] If $\{(t_1, v_1); (t_2, v_2); (t_2', v_2') \}$ is a sticky-admissible collection, then the vertices 
\[ \omega = D(v_1, v_2), \quad \omega' = D(v_1, v_2'), \quad \vartheta = D(v_2, v_2')\]
must satisfy the relations
\begin{equation} \label{three point type 2 height inequalities} k \leq \min\{h(\omega), h(\omega')\}, \quad \ell \leq h(\vartheta), \quad \mu(\omega, k) = \mu(\omega', k),\end{equation} 
and the following equality holds: 
\begin{equation} \label{three point type 2 probability} \text{Pr}\bigl( \sigma(t_1) = v_1, \; \sigma(t_2) = v_2, \; \sigma(t_2') = v_2' \bigr) = \left( \frac{1}{2}\right)^{3N - \mu(\omega,k) - \mu(\vartheta, \ell)}. \end{equation}  
\end{lemma}  
\begin{proof}
The first two inequalities in \eqref{three point type 2 height inequalities} are consequences of stickiness, since there exists a sticky map $\sigma$ that assigns $\sigma(t_1) = v_1$, $\sigma(t_2) = v_2$, $\sigma(t_2') = v_2'$. Thus the first inequality in \eqref{three point type 2 height inequalities} is proved as in \eqref{height and stickiness}, while the second one also follows a similar route:
\[ h(\vartheta) = h(D(v_2, v_2')) = h \bigl( D(\sigma(t_2), \sigma(t_2'))\bigr) \geq h(D(t_2, t_2')) = h(t) = \ell. \] 
For the last identity in \eqref{three point type 2 height inequalities}, we observe that both $\omega$ and $\omega'$ lie on the ray identifying  $v_1$. Thus $\theta(\omega, k) = \theta(\omega', k)$ and hence $\mu(\omega, k) = \mu(\omega', k)$ by the first inequality in \eqref{three point type 2 height inequalities}.  

We now turn to the counting of $n(A;\alpha)$, which leads to the probability estimate \eqref{three point type 2 probability}  via Lemma \ref{general probability lemma}. Using the notation introduced in the proof of Lemma \ref{three point type 1 lemma}, the pairwise youngest common ancestors of the last generation vertices in $\mathbb N(A;\alpha)$ are seen to satisfy the following:   
\begin{align*} u_{\mathbb N} &=D_{\mathbb N}(t_1, t_2) \cong Q_u[\omega, k] = Q_{t_2}[v_2, k], \\ t_{\mathbb N} &= D_{\mathbb N}(t_2, t_2') \cong Q_{t}[\vartheta, \ell] = Q_{t_2}[v_2, \ell]. \end{align*}  Since the type of $\mathbb I$ guarantees that $k < \ell$, the relations above imply \[ t_{\mathbb N} \subseteq u_{\mathbb N} \quad \text{ and hence } \quad h_{\mathbb N}(u_{\mathbb N}) = \mu(\omega, k) \leq h_{\mathbb N}(t_{\mathbb N}) = \mu(\vartheta, \ell).\]
This enables us to compute, with the aid of Figure \ref{Fig: 3 point configurations},  
\begin{align*}
n(A;\alpha) &= \underset{\text{vertices on the ray of $t_1$ in $\mathbb N$}}{\underbrace{\mu(\omega, k) + \bigl[ N - \mu(\omega, k)\bigr]}} + \underset{\text{vertices between $u_{\mathbb N}$ and $t_{\mathbb N}$}}{\underbrace{\bigl[ \mu(\vartheta, \ell) - \mu(\omega, k)\bigr]}} + \underset{\text{vertices below $t_{\mathbb N}$}}{\underbrace{2 \bigl[ N - \mu(\vartheta, \ell)\bigr]}} \\ &= 3N - \mu(\omega, k) - \mu(\vartheta, \ell).  
\end{align*}  
This is the exponent claimed in \eqref{three point type 2 probability}. 
\end{proof}

\subsubsection{The case of four roots} \label{four roots section}
Finally we turn our attention to four point root configurations. Depending on the relative positions of root cubes within the root tree, 
we can classify the configuration types as follows. Let $\mathbb I = \{(t_1, t_2); (t_1', t_2') \}$ be an ordered tuple of four distinct root cubes, for which 
\begin{equation} \label{u u' height relation four points} 
u = D(t_1, t_2) \text{ and } u' = D(t_1', t_2') \text{ obey } h(u) \leq h(u').  
\end{equation}  
Then exactly one of the following conditions must hold: 
\begin{align}
&u \cap u' = \emptyset, \label{type1a}\\ 
&u = u' = D(t_i, t_j') \text{ for all } i, j = 1,2, \label{type1b} \\ 
&u' \subsetneq u, \label{type2}\\ 
&u = u', \text{ and $\exists$ indices $1 \leq i, j \leq 2$ such that } D(t_i, t_j') \subsetneq u.  \label{type3}  
\end{align}
\begin{definition} \label{four point type definition}
For an ordered tuple $\mathbb I = \{(t_1, t_2);(t_1', t_2') \}$ of four distinct root cubes meeting the requirement of \eqref{u u' height relation four points}, we say that $\mathbb I$ is of 
\begin{enumerate}[(a)]
\item type 1 if exactly one of \eqref{type1a} or \eqref{type1b} holds,  
\item type 2 if \eqref{type2} holds, and
\item type 3 if \eqref{type3} holds. 
\end{enumerate}
If $\mathbb I$ does not meet the height relation in \eqref{u u' height relation four points}, then $\mathbb I' = \{(t_1', t_2');(t_1, t_2) \}$ does, and the type of $\mathbb I$ is said to be the same as that of $\mathbb I'$. 
\end{definition} 
Several different structural possibilities for the root quadruple exist within the confines of a single type, excluding permutations within and between the pairs $\{t_1, t_2 \}$ and $\{t_1', t_2' \}$. These have been listed in Figure \ref{Fig: 4 point configurations}. We note in passing that the type definition above is slightly different from that in \cite{KrocPramanik}. Here, the main motivation for the nomenclature is the classification of the unconditional probabilities of slope assignment as exemplified in \eqref{general probability equation}, whereas in \cite{KrocPramanik} a simpler analysis involving  conditional probabilities only was possible.  

\begin{figure}[h!]
\setlength{\unitlength}{0.6mm} 
\begin{picture}(0,0)(-136,20)

        \allinethickness{0.1mm}\path(-100,20)(-100,-220)(71,-220)(71,20)(-100,20)
        \path(-43,20)(-43,-220)
        \path(-44,20)(-44,-220)
        \path(14,20)(14,-42)
        \path(14,-100)(14,-220)
        \path(-100,-40)(71,-40)
        \path(-100,-41)(71,-41)
        \path(-100,-42)(71,-42)
        \path(-43,-100)(71,-100)
        \path(-100,-160)(71,-160)
        \path(-100,-161)(71,-161)
        \path(-100,-162)(71,-162)

        
        \put(-93,-5){\large\shortstack{Type 1}}
        \put(-93,-15){\large\shortstack{Configurations}}

        \special{sh 0.99}\put(-14,10){\ellipse{2}{2}}
        \special{sh 0.99}\put(-14,0){\ellipse{2}{2}}
        \put(-29,-10){\shortstack{$u$}}
        \special{sh 0.99}\put(-24,-10){\ellipse{2}{2}}
        \put(-1,-17){\shortstack{$u'$}}
        \special{sh 0.99}\put(-4,-17){\ellipse{2}{2}}
        \put(-32,-36){\shortstack{$t_1$}}
        \special{sh 0.99}\put(-30,-30){\ellipse{2}{2}}
        \put(-22,-36){\shortstack{$t_2$}}
        \special{sh 0.99}\put(-21,-30){\ellipse{2}{2}}
        \put(-6,-36){\shortstack{$t_1'$}}
        \special{sh 0.99}\put(-5,-30){\ellipse{2}{2}}
        \put(4,-36){\shortstack{$t_2'$}}
        \special{sh 0.99}\put(5,-30){\ellipse{2}{2}}

        \path(-14,10)(-14,0)
        \path(-24,-10)(-14,0)(-4,-17)
        \path(-30,-30)(-24,-10)(-21,-30)
        \path(-5,-30)(-4,-17)(5,-30)


        \special{sh 0.99}\put(43,10){\ellipse{2}{2}}
        \put(46,-11){\shortstack{$u = u'$}}
        \special{sh 0.99}\put(43,-10){\ellipse{2}{2}}
        \put(21,-36){\shortstack{$t_1$}}
        \special{sh 0.99}\put(23,-30){\ellipse{2}{2}}
        \put(34,-36){\shortstack{$t_1'$}}
        \special{sh 0.99}\put(36,-30){\ellipse{2}{2}}
        \put(48,-36){\shortstack{$t_2$}}
        \special{sh 0.99}\put(50,-30){\ellipse{2}{2}}
        \put(61,-36){\shortstack{$t_2'$}}
        \special{sh 0.99}\put(63,-30){\ellipse{2}{2}}

        \path(43,10)(43,-10)
        \path(23,-30)(43,-10)(36,-30)
        \path(50,-30)(43,-10)(63,-30)

        \put(-93,-95){\large\shortstack{Type 2}}
        \put(-93,-105){\large\shortstack{Configurations}}

        \special{sh 0.99}\put(14,-50){\ellipse{2}{2}}
        \put(8,-61){\shortstack{$u$}}
        \special{sh 0.99}\put(14,-60){\ellipse{2}{2}}
        \put(8,-76){\shortstack{$u'$}}
        \special{sh 0.99}\put(14,-75){\ellipse{2}{2}}
        \put(32,-96){\shortstack{$t_2$}}
        \special{sh 0.99}\put(-6,-90){\ellipse{2}{2}}
        \put(-8,-96){\shortstack{$t_1$}}
        \special{sh 0.99}\put(9,-90){\ellipse{2}{2}}
        \put(7,-96){\shortstack{$t_1'$}}
        \special{sh 0.99}\put(21,-90){\ellipse{2}{2}}
        \put(19,-96){\shortstack{$t_2'$}}
        \special{sh 0.99}\put(34,-90){\ellipse{2}{2}}

        \path(14,-50)(14,-60)(14,-75)
        \path(21,-90)(14,-75)(9,-90)
        \path(34,-90)(14,-60)(-6,-90)

        
        \special{sh 0.99}\put(-14,-110){\ellipse{2}{2}}
        \special{sh 0.99}\put(-14,-120){\ellipse{2}{2}}
        \put(-11,-121){\shortstack{$u$}}
        \special{sh 0.99}\put(-14,-130){\ellipse{2}{2}}
        \put(-11,-131){\shortstack{$u'$}}
        \special{sh 0.99}\put(-5,-140){\ellipse{2}{2}}
        \put(-32,-156){\shortstack{$t_1$}}
        \special{sh 0.99}\put(-30,-150){\ellipse{2}{2}}
        \put(-22,-156){\shortstack{$t_1'$}}
        \special{sh 0.99}\put(-21,-150){\ellipse{2}{2}}
        \put(-6,-156){\shortstack{$t_2'$}}
        \special{sh 0.99}\put(-5,-150){\ellipse{2}{2}}
        \put(4,-156){\shortstack{$t_2$}}
        \special{sh 0.99}\put(5,-150){\ellipse{2}{2}}

        \path(-14,-110)(-14,-130)
        \path(-14,-120)(-30,-150)
        \path(-21,-150)(-14,-130)(-5,-140)(-5,-150)
        \path(-5,-140)(5,-150)


        \special{sh 0.99}\put(43,-110){\ellipse{2}{2}}
        \special{sh 0.99}\put(43,-120){\ellipse{2}{2}}
        \put(46,-121){\shortstack{$u$}}
        \special{sh 0.99}\put(43,-130){\ellipse{2}{2}}
        \special{sh 0.99}\put(52,-140){\ellipse{2}{2}}
        \put(55,-141){\shortstack{$u'$}}
        \put(25,-156){\shortstack{$t_1$}}
        \special{sh 0.99}\put(27,-150){\ellipse{2}{2}}
        \put(35,-156){\shortstack{$t_2$}}
        \special{sh 0.99}\put(36,-150){\ellipse{2}{2}}
        \put(51,-156){\shortstack{$t_1'$}}
        \special{sh 0.99}\put(52,-150){\ellipse{2}{2}}
        \put(61,-156){\shortstack{$t_2'$}}
        \special{sh 0.99}\put(62,-150){\ellipse{2}{2}}

        \path(43,-110)(43,-130)
        \path(43,-120)(27,-150)
        \path(36,-150)(43,-130)(52,-140)(52,-150)
        \path(52,-140)(62,-150)



        \put(-93,-185){\large\shortstack{Type 3}}
        \put(-93,-195){\large\shortstack{Configurations}}

        \special{sh 0.99}\put(-14,-170){\ellipse{2}{2}}
        \special{sh 0.99}\put(-14,-180){\ellipse{2}{2}}
        \put(-11,-181){\shortstack{$u = u'$}}
        \special{sh 0.99}\put(-24,-190){\ellipse{2}{2}}
        \special{sh 0.99}\put(-5,-200){\ellipse{2}{2}}
        \put(-32,-216){\shortstack{$t_1$}}
        \special{sh 0.99}\put(-30,-210){\ellipse{2}{2}}
        \put(-22,-216){\shortstack{$t_1'$}}
        \special{sh 0.99}\put(-21,-210){\ellipse{2}{2}}
        \put(-6,-216){\shortstack{$t_2$}}
        \special{sh 0.99}\put(-5,-210){\ellipse{2}{2}}
        \put(4,-216){\shortstack{$t_2'$}}
        \special{sh 0.99}\put(5,-210){\ellipse{2}{2}}

        \path(-14,-170)(-14,-180)
        \path(-21,-210)(-24,-190)(-30,-210)
        \path(-24,-190)(-14,-180)(-5,-200)(-5,-210)
        \path(-5,-200)(5,-210)


        \special{sh 0.99}\put(43,-170){\ellipse{2}{2}}
        \special{sh 0.99}\put(43,-180){\ellipse{2}{2}}
        \put(46,-181){\shortstack{$u = u'$}}
        \special{sh 0.99}\put(43,-195){\ellipse{2}{2}}
        \put(21,-216){\shortstack{$t_1$}}
        \special{sh 0.99}\put(23,-210){\ellipse{2}{2}}
        \put(34,-216){\shortstack{$t_2$}}
        \special{sh 0.99}\put(36,-210){\ellipse{2}{2}}
        \put(48,-216){\shortstack{$t_2'$}}
        \special{sh 0.99}\put(50,-210){\ellipse{2}{2}}
        \put(61,-216){\shortstack{$t_1'$}}
        \special{sh 0.99}\put(63,-210){\ellipse{2}{2}}

        \path(43,-170)(43,-195)
        \path(50,-210)(43,-195)(36,-210)
        \path(23,-210)(43,-180)(63,-210)


\end{picture}
\vspace{15cm}
\caption{\label{Fig: 4 point configurations} Configurations of four root cubes, up to 
permutations.}
\end{figure}
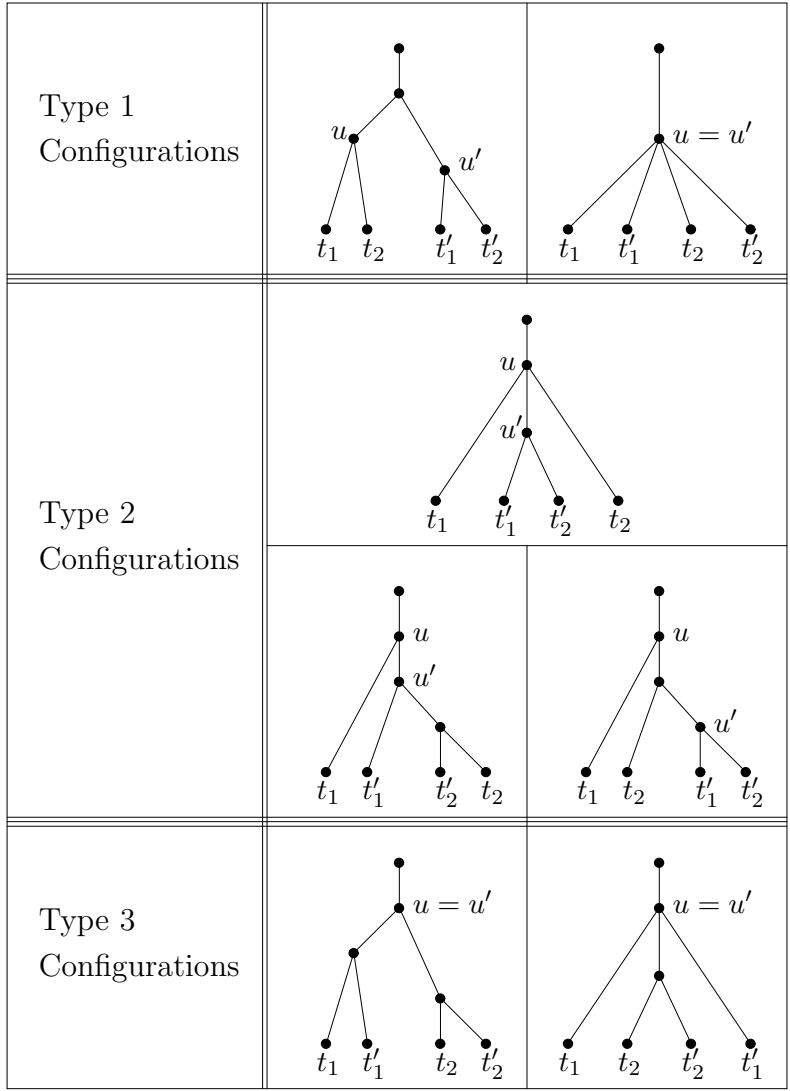

We now proceed to analyze how the configuration types affect the slope assignment probabilities. 
\begin{lemma} \label{four point type 1 lemma} 
Let $A = \{t_1, t_2, t_1', t_2'\}$ be a collection of four distinct root cubes such that  $\mathbb I = \{(t_1, t_2);(t_1', t_2') \}$ obeys \eqref{u u' height relation four points} and is of type 1. Let $\Gamma_{A} = \{v_1, v_2, v_1', v_2'\} = \{ \alpha(t_i) = v_i, \; \alpha(t_i') = v_i', \; i=1,2 \}\subseteq \Omega_N$ be a choice of slopes such that the collection $\{(t_i, v_i);(t_i', v_i'); i=1,2 \}$ is sticky-admissible. Set
\[z = D(u, u'), \quad
k = h(u), \quad k' = h(u'), \quad \ell = h(z), \] so that \[\begin{cases} u, u' \subsetneq z, \text{ and hence } \ell < k \leq k' \text{ if  \eqref{type1a} holds}, \\ u = u' = z, \text{ and hence } \ell = k = k' \text{ if \eqref{type1b} holds.} \end{cases}  \]  
Then the vertices 
\[ \omega = D(v_1, v_2), \quad  \omega' = D(v_1', v_2'), \quad  v = D(\omega, \omega'), \] must satisfy $k \leq h(\omega)$, $k' \leq h(\omega')$ and $\ell \leq h(v)$, and the following equality holds: 
\begin{equation}  \text{Pr}\bigl( \sigma(t_i) = v_i, \; \sigma(t_i') = v_i', \; i=1,2\bigr) = \left( \frac{1}{2}\right)^{4N-\mu(\omega,k) - \mu(\omega', k') - \mu(v, \ell)}.  \label{four point type 1 probability}\end{equation}  
\end{lemma}  
\begin{proof}
The proofs of the height relations may be reproduced verbatim from the previous lemmas in this section, so we focus only on the probability estimate. As before, 
\begin{align} &u_{\mathbb N} = D_{\mathbb N}(t_1, t_2) \cong Q_u[\omega, k], \quad u'_{\mathbb N} = D_{\mathbb N}(t_1', t_2') \cong Q_{u'}[\omega',k'] \nonumber \\ &z_{\mathbb N} = D_{\mathbb N}(u_{\mathbb N},u'_{\mathbb N}) \cong Q_z[v, \ell] = Q_u[w, \ell] = Q_{u'}[\omega', \ell], \text{ and hence } \label{z,u,u'} \\ &h_{\mathbb N}(u_{\mathbb N}) = \mu(\omega,k), \quad h_{\mathbb N}(u'_{\mathbb N}) = \mu(\omega', k'), \quad h_{\mathbb N}(z_{\mathbb N}) = \mu(v, \ell). \nonumber 
\end{align} 
Since $\ell \leq k \leq k'$, \eqref{z,u,u'} implies 
\[ u_{\mathbb N} \cup u'_{\mathbb N} \subseteq z_{\mathbb N}, \quad \text{ and thus } \quad \mu(v, \ell) \leq \min \bigl[ \mu(\omega,k), \mu(\omega',k') \bigr]. \] 
It is important to keep in mind that $\mathbb N(A;\alpha)$ need not inherit the same type of structure as $A$. For example, if \eqref{type1a} holds, it need not be true that $u_{\mathbb N} \cap u'_{\mathbb N} = \emptyset$; indeed the vertices $u_{\mathbb N}$, $u'_{\mathbb N}$ and $z_{\mathbb N}$ could be distinct or (partially) coincident depending on the structure of the slope tree. Nonetheless the information collected above is sufficient to compute the number of vertices in $\mathbb N(A;\alpha)$ (see Figure \ref{Fig: 4 point configurations}): 
\begin{align*}
n(A;\alpha) &= \underset{\text{vertices above $z_{\mathbb N}$}}{\underbrace{\mu(v, \ell)}} + \underset{\text{vertices between $z_{\mathbb N}$ and $u_{\mathbb N}$}}{\underbrace{\bigl[ \mu(\omega,k) - \mu(v, \ell) \bigr]}} + \underset{\text{vertices between $z_{\mathbb N}$ and $u'_{\mathbb N}$}}{\underbrace{\bigl[ \mu(\omega',k') - \mu(v, \ell) \bigr]}} \\ & \hskip1.3in + \underset{\begin{subarray}{c}\text{ancestors of $t_1$ and $t_2$} \\  \text{in $\mathbb N$ descended from $u_{\mathbb N}$} \end{subarray}}{\underbrace{2 \bigl[ N - \mu(\omega, k )\bigr]}} + \underset{\begin{subarray}{c} \text{ ancestors of $t_1'$ and $t_2'$ }\\ \text{in $\mathbb N$ descended from $u'_{\mathbb N}$}\end{subarray}}{\underbrace{2 \bigl[ N - \mu(\omega',k') \bigr]}} \\ &= 4N - \mu(\omega, k) - \mu(\omega', k') - \mu(v, \ell).  
\end{align*} 
Combined with Lemma \ref{general probability lemma}, this leads to \eqref{four point type 1 probability}. 
\end{proof} 
\begin{lemma} \label{four point type 2 lemma} 
Let $A = \{ t_i, t_i' ; i=1,2\}$ be a collection of four distinct root cubes such that $\mathbb I = \{(t_1, t_2);(t_1', t_2') \}$ obeys \eqref{u u' height relation four points} and is of type 2. Suppose that $\Gamma_{A} = \{\alpha(t_i) = v_i, \; \alpha(t_i') = v_i' \; ; i=1,2\}$ is a choice of slopes such that the collection $\{(t_i, v_i);(t_i', v_i'); i=1,2 \}$ is sticky-admissible. Set 
\[ \omega = D(v_1, v_2), \quad \omega' = D(v_1', v_2'), \quad k = h(u), \quad k' = h(u'),\] so that $k < k'$.  
Then the following inequalities hold: $k \leq h(\omega)$, $k' \leq h(\omega')$. Further, there exist permutations $\{ i_1, i_2 \}$ and $\{j_1, j_2\}$ of $\{1,2\}$ for which the quantities 
\[\vartheta = D(v_{i_2}, v_{j_2}'), \quad t = D(t_{i_2}, t_{j_2}'), \quad \ell = h(t)\]
obey the relation $\ell \leq h(\vartheta)$, and for which the probability of slope assignment can be computed as follows:
\begin{equation}  \text{Pr}\bigl( \sigma(t_i) = v_i, \sigma(t_i') = v_i' \text{ for } i= 1,2\bigr) = \left( \frac{1}{2}\right)^{4N - \mu(\omega, k) - \mu(\omega', k') - \mu(\vartheta, \ell)}. \label{four point type 2 probability} \end{equation} 
\end{lemma} 
\begin{proof} 
The definition of the configuration type dictates that $u'$ is strictly contained in $u$, but depending on other properties of the ray joining $u$ and $u'$ we are led to consider several cases. If there does not exist any vertex in the root tree that is strictly contained in $u$ and also contains $t_i$ for some $i=1,2$, then any permutation of the root pairs $\{t_1, t_2\}$ and $\{t_1', t_2'\}$ works. In particular, it suffices to choose $i_1 = j_1 = 1$, $i_2 = j_2=2$. In this case $t = u$, hence $\ell = k$. In particular this implies 
\begin{equation} \label{equality of mu} 
\theta(\omega, k) = \theta(v_2, k) = \theta(v_2, \ell) = \theta(\vartheta, \ell), \quad \text{ hence } \quad \mu(\omega,k) = \mu(\vartheta, \ell). 
\end{equation} 
Further  \begin{align*} &u'_{\mathbb N} =Q_{u'}[\omega', k'] = Q_{t_1'}[v_1',k'] \subseteq Q_{t_1'}[v_1', k] = Q_{u}[\omega, k] =  u_{\mathbb N}, \text{ and } \\  &h_{\mathbb N}(u_{\mathbb N}) =  \mu(\omega, k), \quad h_{\mathbb N}(u'_{\mathbb N}) = \mu(\omega', k'). \end{align*} 
Referring to Figure \ref{Fig: 4 point configurations} we find that  
\begin{align*} n(A;\alpha) &= \mu(\omega, k) + 2\bigl[N - \mu(\omega, k) \bigr] + \bigl[ \mu(\omega', k') - \mu(\omega, k)\bigr] + 2 \bigl[N - \mu(\omega', k') \bigr] \\ &= 4N - 2 \mu(\omega,k) - \mu(\omega', k') \\ &= 4N - \mu(\omega, k) - \mu(\omega', k') - \mu(\vartheta, \ell),\end{align*}
where the last step uses one of the equalities in \eqref{equality of mu}. 
  
Suppose next that the previous case does not hold, and also that none of the descendants of $u'$ lying on the rays of $t_1', t_2'$ is an ancestor of $t_1$ or $t_2$. Then there is a vertex, let us call it $t$, such that $u' \subseteq t \subsetneq u$, and $t$ is of maximal height in this class subject to the restriction that it is an ancestor of some $t_i$, which we call $t_{i_2}$.
Thus $t_{i_1}$ is the unique element in $\{ t_1, t_2\}$ that is not a descendant of $t$. In this case, any permutation of $\{ t_1', t_2'\}$ works, and we can keep $j_1 = 1$, $j_2=2$. Then $t = D(t_{i_2}, t_{j_2}')$, $k < \ell \leq k'$, and 
\begin{align*}
u'_{\mathbb N} &= D_{\mathbb N}(t_1', t_2') \cong Q_{u'}[\omega', k'] = Q_{u'}[v_{j_2}', k'], \\
t_{\mathbb N} &= D_{\mathbb N}(t_{i_2}, t'_{j_2}) \cong Q_t[\vartheta, \ell] = Q_{u'}[v'_{j_2}, \ell], \text{ and } \\
u_{\mathbb N} &= D_{\mathbb N}(t_1, t_2) \cong Q_u[\omega, k]= Q_u[\omega_0, k] = Q_{u'}[v'_{j_2}, k],
\end{align*}
where the last line uses the fact that $u = D(t_1, t_2, t_1', t_2')$, so that the second equality in that line holds $\omega_0 = D(v_1, v_2, v_1', v_2')$.   
These relations imply that 
\begin{equation} \label{height relations four roots type 2} u'_{\mathbb N} \subseteq t_{\mathbb N} \subseteq u_{\mathbb N}  \text{ with }  h_{\mathbb N}(u_{\mathbb N}) = \mu(\omega, k), \;  h_{\mathbb N}(u'_{\mathbb N}) = \mu(\omega', k'), \; h_{\mathbb N}(t_{\mathbb N}) = \mu(\vartheta, \ell).\end{equation} 
Using this, we compute $n(A;\alpha)$ as follows, 
\begin{align*}
n(A;\alpha) &= \underset{\text{vertices of $t_{i_1}$ in $\mathbb N$}}{\underbrace{\mu(\omega, k) + \bigl[ N - \mu(\omega,k) \bigr]}} +  \underset{\text{vertices of $t_{i_2}$ in $\mathbb N$ below $u_{\mathbb N}$}}{\underbrace{\bigl[ \mu(\vartheta, \ell) - \mu( \omega,k)\bigr] +  \bigl[ N - \mu(\vartheta, \ell)\bigr]}}  \\ &\hskip1.3in + \underset{\text{vertices between $t_{\mathbb N}$ and $u'_{\mathbb N}$}}{\underbrace{\bigl[ \mu(\omega',k') - \mu(\vartheta, \ell) \bigr]}}  + \underset{\begin{subarray}{c}\text{vertices of $t_1'$ and $t_2'$}\\ \text{in $\mathbb N$ below $u'_{\mathbb N}$}\end{subarray}}{\underbrace{2 \bigl[ N - \mu(\omega', k')\bigr]}} \\ &= 4N - \mu(\omega, k) - \mu(\omega', k') - \mu(\vartheta, \ell), 
\end{align*}   
which is the required exponent. 

The last case, complementary to the ones already considered is when there exists a pair of indices, denoted $i_2, j_2 \in \{1,2\}$ such that $t = D(t_{i_2}, t'_{j_2}) \subsetneq u'$. In this case we leave the reader to verify by the usual means that \[ t_{\mathbb N} \subseteq u'_{\mathbb N} \subseteq u_{\mathbb N}, \]
with their heights given by the same expressions as in \eqref{height relations four roots type 2}. Accordingly, 
\begin{align*}
n(A;\alpha) &= \underset{\begin{subarray}{c}\text{vertices of $t_{i_1}$} \\ \text{in $\mathbb N$} \end{subarray}}{\underbrace{N}} + \underset{\begin{subarray}{c} \text{vertices on $t'_{j_1}$}\\ \text{in $\mathbb N$ below $u_{\mathbb N}$} \end{subarray}}{\underbrace{\bigl[N - \mu(\omega, k) \bigr]}} + \underset{\text{vertices between $t_{\mathbb N}$ and $u'_{\mathbb N}$}}{\underbrace{\bigl[ \mu(\vartheta, \ell) - \mu(\omega', k')\bigr]}} + \underset{\begin{subarray}{c}\text{vertices of $t_{i_2}$ and $t_{j_2}'$} \\ \text{in $\mathbb N$ below $t_{\mathbb N}$} \end{subarray}}{\underbrace{2 \bigl[N - \mu(\vartheta, \ell) \bigr]}} \\   &= 4N - \mu(\omega,k) - \mu(\omega', k') - \mu(\vartheta, \ell). 
\end{align*}  
Thus, despite structural differences, all the cases give rise to the same value of $n(A;\alpha)$ that agrees with the exponent in \eqref{four point type 2 probability}, completing the proof.  
\end{proof}
We pause for a moment to record a few properties of the youngest common ancestors of the roots and slopes that emerged in the proof of Lemma \ref{four point type 2 lemma}. 
\begin{corollary} \label{four point type 2 structure corollary} 
Let $A$ and $\Gamma_A$ be as in Lemma \ref{four point type 2 lemma}. 
\begin{enumerate}[(i)]
\item The possibly distinct vertices $u$, $u'$ and $t$, as described in Lemma \ref{four point type 2 lemma}, are linearly ordered in terms of ancestry, i.e., there is some ray of the root tree that they all lie on. Depending on $A$, the vertex $t$ may lie above or below $u'$, but always in $u$. 
\item The splitting vertices $\omega, \omega', \vartheta$ in the slope tree also obey certain inclusions; namely, for each of the pairs $(\omega, \vartheta)$ and $(\omega', \vartheta)$, one member of the pair is contained in the other.   
\end{enumerate} 
\end{corollary} 
\begin{proof} 
Both $u'$ and $t$ lie on the ray identifying $t_{j_2}'$ by definition, and $u$ lies on the ray of $u'$ by the assumption on the type of the root configuration. This establishes the first claim. The definitions also imply that $v_{i_2} \subseteq \omega \cap \vartheta$ and $v_{j_2}' \subseteq \omega' \cap \vartheta$, hence both intersections are non-empty. The second conclusion then follows from the nesting property of $M$-adic cubes. 
\end{proof} 
\begin{lemma} \label{four point type 3 lemma} 
Let $A = \{ t_i, t_i' ; i=1,2\}$ be a collection of four distinct root cubes such that $\mathbb I = \{(t_1, t_2);(t_1', t_2') \}$ is of type 3. Suppose that $\Gamma_{A} = \{\alpha(t_i) = v_i, \; \alpha(t_i') = v_i' \; ; i=1,2\}$ is a choice of slopes such that the collection $\{(t_i, v_i);(t_i, v_i'); i=1,2 \}$ is sticky-admissible. Set 
\[ \omega = D(v_1, v_2), \quad \omega' = D(v_1', v_2'), \quad k = h(u) = h(u'). \]
Then the following relations must hold: $k \leq h(\omega)$, $k\leq h(\omega')$, $\mu(\omega, k) = \mu(\omega', k)$. 
Further, there exist permutations $\{ i_1, i_2\}$ and $\{j_1, j_2\}$ of $\{1,2\}$ such that the quantities 
\begin{align*} s_1 &= D(t_{i_1}, t'_{j_1}),  \quad s_2 = D(t_{i_2}, t'_{j_2}),   \quad \ell_1 = h(s_1), \\ \vartheta_1 &= D(v_{i_1}, v'_{j_1}), \; \; \; \vartheta_2 = D(v_{i_2}, v'_{j_2}),  \; \; \; \ell_2 = h(s_2)\end{align*}  
satisfy 
\begin{equation}  s_1 \subseteq u, \quad s_2 \subsetneq u, \quad k \leq \ell_1 \leq \ell_2, \quad \ell_i \leq h(\vartheta_i) \text{ for } i =1,2, \label{four point type 3 relations} \end{equation} 
and for which 
\begin{equation}  \text{Pr}\bigl( \sigma(t_i) = v_i, \sigma(t_i') = v_i' \text{ for } i = 1,2\bigr) = \left( \frac{1}{2}\right)^{4N - \mu(\omega, k) - \mu(\vartheta_1, \ell_1) - \mu(\vartheta_2, \ell_2)}. \label{four point type 3 probability} \end{equation}     
\end{lemma} 
\begin{proof}
Since $\mathbb I$ is of type 3, $u=u'$ is the youngest common ancestor of the four elements in $\mathbb I$. If $\omega_0$ is the youngest common ancestor of the slopes $\{ v_i, v_i' : i=1,2\}$, then $h(\omega_0) \geq h(u)=k$ by sticky admissibility. Thus $\theta(\omega, k) = \theta(\omega_0, k) = \theta(\omega',k)$, and therefore $\mu(\omega, k) = \mu(\omega',k)$, as claimed. 

We turn to \eqref{four point type 3 relations} and the probability estimate. The configuration type dictates that there exist indices $(i, j) \in \{1,2\}^2$ such that $D(t_i, t_j') \subsetneq u$. Among all such pairs $(i,j)$, we pick one for which $D(t_i, t_j')$ is of maximal height. Let us call this pair $(i_2, j_2)$, so that $h(D(t_{i_2}, t_{j_2}')) \geq h(D(t_i, t_j'))$ for all $1 \leq i,j \leq 2$. The first three relations in \eqref{four point type 3 relations} are now immediate. The last one follows from sticky admissibility and is left to the reader. 

It remains to compute $n(A;\alpha)$. The structure of $\mathbb N(A;\alpha)$ gives that 
\begin{align*}
&u_{\mathbb N} = u'_{\mathbb N} = D_{\mathbb N}(t_1, t_2) = D_{\mathbb N}(t_1', t_2') \\ &u_{\mathbb N} = Q_u[\omega, k] = Q_{u'}[\omega',k] = Q_{t_1}[v_1,k] = Q_{t_2}[v_2,k], \\ 
&s_{i\mathbb N} = D_{\mathbb N}(t_i, t'_i) = Q_{s_i}[\vartheta_i, \ell_i] = Q_{t_i}[v_i, \ell_i], \\ &s_{i\mathbb N} \subseteq u_{\mathbb N} = D_{\mathbb N}(s_{1 \mathbb N}, s_{2 \mathbb N})\; \text{ for } i = 1,2, \; \text{ so that } \\ 
&h_{\mathbb N}(u_{\mathbb N}) = \mu(\omega, k) \leq  h_{\mathbb N}(s_{i \mathbb N}) = \mu(\vartheta_i, \ell_i), \; i =1,2. 
\end{align*}   
Putting these together, the number of vertices in $\mathbb N(A;\alpha)$ is obtained as follows,
\begin{align*}
n(A;\alpha) &= \underset{\text{vertices up to $u_{\mathbb N}$}}{\underbrace{\mu(\omega, k)}} + \sum_{i=1}^2 \underset{\text{vertices between $u_{\mathbb N}$ and $s_{i \mathbb N}$}}{\underbrace{\bigl[\mu(\vartheta_i, \ell_i) - \mu(\omega, k)\bigr]}} + \sum_{i=1}^{2} \underset{\text{vertices below $s_{i \mathbb N}$}}{\underbrace{2 \bigl[ N - \mu(\vartheta_i, \ell_i)\bigr]}} \\ &= 4N - \mu(\vartheta_1, \ell_1) - \mu(\vartheta_2, \ell_2) - \mu(\omega, k).  
\end{align*}
The probability estimate claimed in \eqref{four point type 3 probability} now follows from Lemma \ref{general probability lemma}.
\end{proof}
\begin{corollary} \label{four point type 3 structure corollary} 
Let $\omega, \omega', \vartheta_1, \vartheta_2$ be as in Lemma \ref{four point type 2 lemma}. Then each of the pairs $(\omega, \vartheta_1)$,  $(\omega, \vartheta_2)$, $(\omega', \vartheta_1)$ and $(\omega', \vartheta_2)$ has the property that one member of the pair is contained in the other. 
\end{corollary}  
\begin{proof}
Since $v_{i_1} \subseteq \omega \cap \vartheta_1$, $v_{i_2} \subseteq \omega \cap \vartheta_2$, $v_{j_1}' \subseteq \omega' \cap \vartheta_1$ and $v_{j_2}' \subseteq \omega' \cap \vartheta_2$,  all four intersections are nonempty, and the desired conclusion follows from the nesting property of $M$-adic cubes. 
\end{proof} 
As the reader has noticed, the classification of probability estimates in this section is predicated on the configuration types of the roots, {\em{not}} the slopes. Of course, such definitions of type apply equally well to slope tuples $\{(v_1, v_2);(v_1', v_2')\}$. Indeed, a point worth noting is that configuration types are not preserved under sticky maps; see for example the diagram in Figure \ref{Fig: configurations not preserved} below, where a four tuple of roots of type 1 maps to a sticky image of type 3. In view of these considerations, we shall refrain for the most part from using any type properties of slopes. In the rare instances where structural properties of slopes are relevant, a case in point being Section \ref{S_{43} estimation section}, we need to consider all possible configurations.
\vskip1cm
\begin{figure}[h]
\setlength{\unitlength}{.8mm}
\begin{picture}(-50,0)(0,5)

        \special{sh 0.99}\put(40,0){\ellipse{2}{2}}
        \special{sh 0.99}\put(40,-10){\ellipse{2}{2}}
        \special{sh 0.99}\put(30,-20){\ellipse{2}{2}}
        \put(5,-20){\Large\shortstack{$D(t_1,t_2)$}}
        \special{sh 0.99}\put(50,-30){\ellipse{2}{2}}
        \put(51,-26){\Large\shortstack{$D(t_1',t_2')$}}
        \special{sh 0.99}\put(28,-30){\ellipse{2}{2}}
        \special{sh 0.99}\put(32,-30){\ellipse{2}{2}}
        \special{sh 0.99}\put(33,-40){\ellipse{2}{2}}
        \special{sh 0.99}\put(27,-40){\ellipse{2}{2}}
        \special{sh 0.99}\put(53,-40){\ellipse{2}{2}}
        \special{sh 0.99}\put(47,-40){\ellipse{2}{2}}
        \special{sh 0.99}\put(35,-60){\ellipse{2}{2}}
        \put(33,-67){\Large\shortstack{$t_2$}}
        \special{sh 0.99}\put(25,-60){\ellipse{2}{2}}
        \put(23,-67){\Large\shortstack{$t_1$}}
        \special{sh 0.99}\put(45,-60){\ellipse{2}{2}}
        \put(43,-67){\Large\shortstack{$t_1'$}}
        \special{sh 0.99}\put(55,-60){\ellipse{2}{2}}
        \put(53,-67){\Large\shortstack{$t_2'$}}

        \path(40,0)(40,-10)(50,-30)(53,-40)(55,-60)
        \path(50,-30)(47,-40)(45,-60)
        \path(40,-10)(30,-20)(32,-30)(33,-40)(35,-60)
        \path(30,-20)(28,-30)(27,-40)(25,-60)


        \special{sh 0.99}\put(135,0){\ellipse{2}{2}}
        \special{sh 0.99}\put(135,-30){\ellipse{2}{2}}
        \special{sh 0.99}\put(125,-40){\ellipse{2}{2}}
        \special{sh 0.99}\put(145,-40){\ellipse{2}{2}}
        \special{sh 0.99}\put(120,-60){\ellipse{2}{2}}
        \put(128,-68){\Large\shortstack{$\sigma(t_2)$}}
        \special{sh 0.99}\put(133,-60){\ellipse{2}{2}}
        \put(113,-68){\Large\shortstack{$\sigma(t_1')$}}
        \special{sh 0.99}\put(148,-60){\ellipse{2}{2}}
        \put(143,-68){\Large\shortstack{$\sigma(t_1)$}}
        \special{sh 0.99}\put(163,-60){\ellipse{2}{2}}
        \put(158,-68){\Large\shortstack{$\sigma(t_2')$}}

        \path(135,0)(135,-30)(145,-40)(163,-60)
        \path(145,-40)(148,-60)
        \path(135,-30)(125,-40)(133,-60)
        \path(125,-40)(120,-60)


        \dashline{3}(28,-30)(135,-30)
        \path(80,-60)(100,-60)
        \path(98,-58)(100,-60)(98,-62)
        \path(80,-62)(80,-58)
        \put(89,-58){\Large\shortstack{$\sigma$}}
\end{picture}
\vspace{6cm}
\caption{\label{Fig: configurations not preserved} An example of a four tuple of roots of type 1 mapping to a sticky image of type 3.  Notice that $D(\sigma(t_1),\sigma(t_2)) = D(\sigma(t_1'),\sigma(t_2'))$.}
\end{figure}
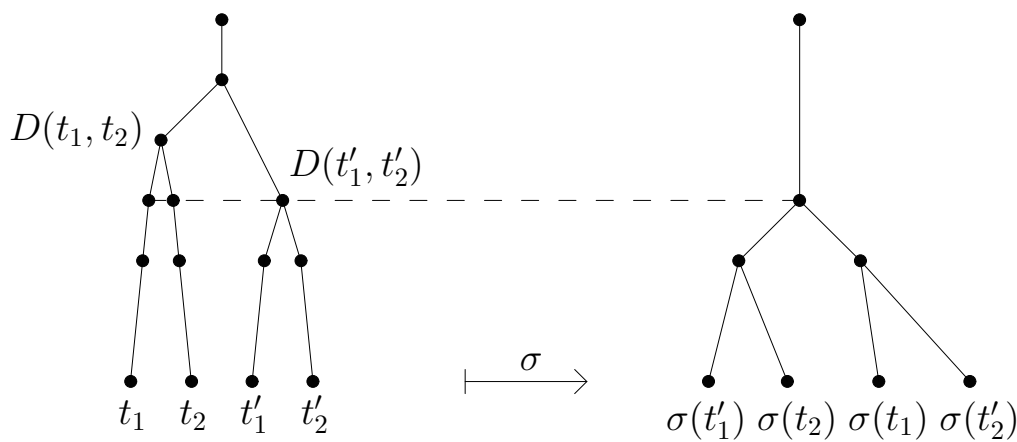
    

\section{Tube counts} \label{tube count section} 
A question of considerable import, the full significance of which will emerge in Section \ref{LowerBoundSection}, is the following: what is the maximum possible cardinality of a sticky-admissible collection of tube tuples that admit certain pairwise intersections in a pre-fixed segment of space? The answer depends, among other things, on the size and configuration type of the roots of the tubes. In this section, we discuss these size counts for collections that are simple enough in the sense that an element in the collection is either a pair, a triple or at most a quadruple of tubes, so that the configuration type of the roots has to fall in one of the categories described in Section \ref{root configurations subsection}.  

\subsection{Collections of two intersecting tubes} \label{tube count two roots section} 
Let us start with the case where the collection consists of pairs of tubes. To phrase the question above in more refined terms we define a collection of root-slope tuples $\mathcal E_2[u, \omega; \varrho]$, where $u$ is a vertex of the root tree, $\omega$ is a splitting vertex of the slope tree, and $\varrho \in [M^{-J}, 10A_0]$ is a constant that represents the (horizontal) distance from the root hyperplane to where intersection takes place. 
\begin{equation} \label{defn C_uw} 
\mathcal E_{2}[u, \omega; \varrho] := \left\{\underset{\text{sticky-admissible}}{\{(t_1, v_1), (t_2, v_2) \}} \Biggl| \begin{aligned}  &t_1, t_2 \in \mathcal Q(J), \; u = D(t_1,t_2), \; t_1 \ne t_2, \\ &v_1,v_2 \in \Omega_N, \; \omega = D(v_1,v_2), \\ &P_{t_1,v_1} \cap P_{t_2,v_2} \cap [\varrho, C_1 \varrho] \times \mathbb R^d \ne \emptyset \end{aligned}  \right\}.  
\end{equation} 
In this context the question at the beginning of this section can be restated as: what is the cardinality of $\mathcal E_{2}[u, \omega; \varrho]$? We answer this question in Lemma \ref{C count lemma} of this section, splitting the necessary work between two intermediate lemmas whose content will also be used in later counting arguments. To be specific, Lemma \ref{counting lemma 1} obtains a uniform bound on a $t_2$-slice of $\mathcal E_{2}[u, \omega; \varrho]$ for fixed $t_1, v_1$ and $v_2$. The cardinality of the projection of $\mathcal E_{2}[u, \omega; \varrho]$ onto the $t_1$ coordinate is obtained in Lemma \ref{counting lemma 2}.    
\begin{lemma} \label{counting lemma 1}
Let $\mathcal E_2[u, \omega;\varrho]$ be the collection defined in \eqref{defn C_uw}, and let $\rho_{\omega} = \sup\{|a-b| : a, b \in \Omega_N, \; D(a,b) = \omega \}$ be the quantity defined in \eqref{sup distance}.  
\begin{enumerate}[(i)]
\item \label{comparing scales} If $\mathcal E_2[u, \omega; \varrho]$ is nonempty, then $2C_1\varrho \rho_\omega \geq M^{-J}$.
\item \label{count for t' given t, v, v'} Given a constant $C_1 > 0$ used to define $\mathcal E_2[u, \omega;\varrho]$, there exists a constant $C_2 = C_2(d,M, C_0,  A_0, C_1)>0$ with the following property. For any fixed choice of $t_1 \in \mathcal Q(J)$ and $v_1, v_2 \in \Omega_N$ the following estimate holds: 
\begin{equation} \label{count for t' given t, v, v' inequality} 
\# \bigl\{t_2 \in \mathcal Q(J) \; : \; \{(t_1, v_1), (t_2, v_2)\} \in \mathcal E_2[u, \omega;\varrho] 
\bigr\} \leq C_2 \varrho \rho_\omega M^J.  
\end{equation} 
\end{enumerate} 
\end{lemma}
\begin{proof}
If $\{(t_1, v_1), (t_2, v_2)\}$ is a tuple that lies in $\mathcal E_2[u, \omega;\varrho]$, then there exists $x \in [\varrho, C_1 \varrho] \times \mathbb R^d$ such that $x$ also belongs to $P_{t_1,v_1} \cap P_{t_2,v_2}$. By Lemma  \ref{intersection criterion lemma},  an appropriate version of inequality \eqref{intersection criterion inequality} must hold, i.e., there exists $x_1 \in [\varrho, C_1 \varrho]$ such that 
\begin{equation} |\text{cen}(t_2) - \text{cen}(t_1) + x_1 (v_2 - v_1)| \leq 2c_d \sqrt{d}M^{-J}. \label{intersection inequality restated} \end{equation} 
In conjunction with Corollary \ref{which is bigger corollary}, this leads to the inequality 
\begin{align*}
M^{-J} \leq |\text{cen}(t_2) - \text{cen}(t_1)| &\leq |x_1||v_2-v_1| + 2 c_d \sqrt{d} M^{-J} \\ &\leq 2|x_1||v_2-v_1| \leq 2C_1 \varrho \rho_{\omega}, 
\end{align*}  
which is the conclusion of part (\ref{comparing scales}). The inequality \eqref{intersection inequality restated} also implies that $\text{cen}(t_2)$ is constrained to lie in a $O(M^{-J})$ neighborhood of the line segment 
\begin{equation} 
\text{cen}(t_1) - s(v_2 - v_1), \qquad \varrho \leq s \leq C_1 \varrho.
\end{equation}   
The length of this segment is at most $C_1 \varrho |v_2-v_1| \leq C_1 \varrho \rho_\omega$, since $v_1$ and $v_2$ must lie in distinct children on $\omega$. In view of part (\ref{comparing scales}), the number of possible choices for $M^{-J}$-separated points $\text{cen}(t_2)$, and hence for $t_2$, lying within this neighborhood is $O(\varrho \rho_\omega M^J)$, as claimed in part (\ref{count for t' given t, v, v'}). 
\end{proof} 
\vskip1cm
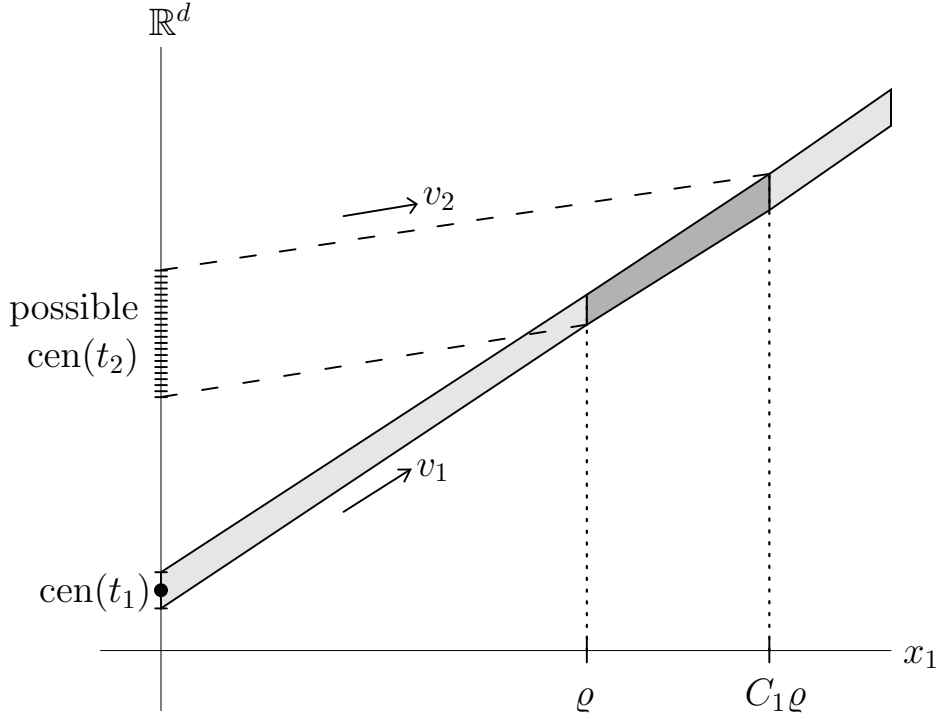
\begin{figure}[h!]
\setlength{\unitlength}{0.8mm}
\begin{picture}(-50,0)(-30,107)

        \path(-10,0)(120,0)
        \put(122,-2){\Large\shortstack{$x_1$}}
        \path(0,-10)(0,100)     
        \put(-2,102){\Large\shortstack{$\mathbb{R}^d$}}

        \allinethickness{0.254mm}\special{sh 0.1}\path(0,7)(70,54)(70,59)(0,13)(0,7)
        \allinethickness{0.254mm}\special{sh 0.1}\path(100,73)(120,87)(120,93)(100,79)(100,73)
        \allinethickness{0.254mm}\special{sh 0.3}\path(70,54)(100,73)(100,79)(70,59)(70,54)

        \special{sh 0.99}\put(0,10){\ellipse{2}{2}}
        \put(-20,8){\Large\shortstack{$\text{cen}(t_1)$}}
        \path(-1,7)(1,7)
        \path(-1,13)(1,13)

        \path(70,-2)(70,2)
        \put(68,-9){\Large\shortstack{$\varrho$}}
        \path(100,-2)(100,2)
        \put(96,-9){\Large\shortstack{$C_1\varrho$}}
        \dottedline{2}(70,0)(70,59)
        \dottedline{2}(100,0)(100,79)
        \dashline{3}(100,79)(0,63)
        \dashline{3}(70,54)(0,42)

        \path(30,23)(41,30)
        \path(39,30)(41,30)(40,28)      
        \put(42,29){\Large\shortstack{$v_1$}}

\path(30,72)(42,74)     
        \path(40.5,75)(42,74)(41,72.5)
        \put(43,74){\Large\shortstack{$v_2$}}

        \path(-1,63)(1,63)
        \path(-1,61)(1,61)
        \path(-1,59)(1,59)
        \path(-1,57)(1,57)
        \path(-1,55)(1,55)
        \path(-1,53)(1,53)
        \path(-1,51)(1,51)
        \path(-1,49)(1,49)
        \path(-1,47)(1,47)
        \path(-1,45)(1,45)
        \path(-1,43)(1,43)
        \path(-1,62)(1,62)
        \path(-1,60)(1,60)
        \path(-1,58)(1,58)
        \path(-1,56)(1,56)
        \path(-1,54)(1,54)
        \path(-1,52)(1,52)
        \path(-1,50)(1,50)
        \path(-1,48)(1,48)
        \path(-1,46)(1,46)
        \path(-1,44)(1,44)
        \path(-1,42)(1,42)
        \put(-25,55){\Large\shortstack{possible}}
        \put(-22,47){\Large\shortstack{$\text{cen}(t_2)$}}

\end{picture}
\vspace{9.5cm}
\caption{\label{Fig: counting lemma 1 figure} An illustration of the proof of Lemma~\ref{counting lemma 1}.}  
\end{figure}
\vskip1cm 
\begin{lemma}\label{counting lemma 2} 
Given $C_1 > 0$, there exists a positive constant $C_2 = C_2(d, M, A_0, C_1)$ with the following property. For any $\mathcal E_{2}[u, \omega; \varrho]$ defined as in \eqref{defn C_uw}, the following estimate holds:
\begin{equation} 
\#\left\{t_1 \in \mathcal Q(J) \Bigl| \begin{aligned} &\exists t_2 \in \mathcal Q(J) \text{ and } v_1,v_2 \in \Omega_N \text{ such } \\ 
&\text{ that } \{ (t_1,v_1);(t_2,v_2) \} \in \mathcal E_{2}[u, \omega; \varrho] \end{aligned}
\right\} \leq C_2 \varrho \rho_\omega M^{-(d-1)h(u) + dJ}, \label{t projection count} 
\end{equation}
where $\rho_\omega$ is as in \eqref{sup distance}.
\end{lemma} 
\begin{proof}
If $\{(t_1,v_1), (t_2,v_2)\} \in \mathcal E_{2}[u, \omega; \varrho]$, then there exists $x = (x_1, \cdots, x_{d+1}) \in P_{t_1,v_1} \cap P_{t_2,v_2}$ with $\varrho \leq x_1 \leq C_1 \varrho$. Combining inequality \eqref{intersection inequality restated} obtained from Lemma \ref{intersection criterion lemma} along with Corollary \ref{which is bigger corollary} as we did in Lemma \ref{counting lemma 1}, we obtain 
\begin{equation} \label{dist c(t) and c(t')}
\begin{aligned}
|\text{cen}(t_2) - \text{cen}(t_1)| &\leq |x_1||v_2-v_1| + 2 c_d \sqrt{d} M^{-J} \\ &\leq (1 + 4 c_d \sqrt{d}) |x_1||v_1-v_2| \\ &\leq C_1  (1 + 4 c_d \sqrt{d}) \varrho \rho_\omega = C \varrho \rho_\omega,
\end{aligned}  
\end{equation}
where the last step follows from the definition of $\omega$. This means that $\text{cen}(t_1)$ and $\text{cen}(t_2)$ must be within distance $C\varrho \rho_\omega$ of each other. On the other hand, it is known as part of the definition of $\mathcal E_{2}[u, \omega; \varrho]$ that $u = D(t_1,t_2)$, so $\text{cen}(t_1)$ and $\text{cen}(t_2)$ must lie in distinct children of $u$. This forces the location of $\text{cen}(t_1)$ to be within distance $C \varrho \rho_\omega$ of the boundary of some child of $u$,  to allow for the existence of a point $\text{cen}(t_2)$ contained in a different child and obeying the constraint of \eqref{dist c(t) and c(t')}.  In other words, $\text{cen}(t_1)$ belongs to the set
\begin{equation} \label{defn A_u} \mathcal A_u = \bigl\{ s \in u : \text{dist}(s, \text{bdry}(u')) \leq C \varrho \rho_\omega \text{ for some child $u'$ of $u$ }\bigr\},\end{equation}  
which is the union of at most $dM$ parallelepipeds of dimension $d$, with length $M^{-h(u)}$ in $(d-1)$ ``long'' directions and $C \varrho \rho_{\omega}$ in the remaining ``short'' direction. Note that $\rho_{\omega} \leq M^{-h(\omega)} \leq M^{-h(u)}$ by sticky-admissibility, hence $\varrho \rho_{\omega} = O(M^{-h(u)})$, which justifies this description. 

Since $C \varrho \rho_{\omega}\geq M^{-J}$ by Lemma \ref{counting lemma 1}(\ref{comparing scales}), the constituent parallelpipeds of $\mathcal A_u$ as described above are thick relative to the finest scale $M^{-J}$ in all directions. The volume of $\mathcal A_u$ is then easily computed as
\[ |\mathcal A_u| \leq C \varrho \rho_\omega M^{-(d-1)h(u)}.\] 
Therefore the number of $M^{-J}$ separated points $\text{cen}(t_1)$, and hence the number of possible root cubes $t_1$, contained in $\mathcal A_u$ is at most  $C_2 \varrho \rho_\omega M^{-(d-1)h(u) + dJ}$, as claimed.  
\end{proof} 
\vskip1cm
\begin{figure}[h]
\setlength{\unitlength}{0.8mm}
\begin{picture}(-50,0)(-95,20)

        \allinethickness{0.5mm}\path(-45,20)(-45,-70)(45,-70)(45,20)(-45,20)
        \allinethickness{0.5mm}\path(-45,-10)(45,-10)
        \path(-15,20)(-15,-70)
        \path(-45,-40)(45,-40)
        \path(15,20)(15,-70)
        \allinethickness{0.1mm}\path(-17,20)(-17,-70)
        \path(-19,20)(-19,-70)
        \path(-13,20)(-13,-70)
        \path(-11,20)(-11,-70)
        \path(-43,20)(-43,-70)
        \path(-41,20)(-41,-70)
        \path(13,20)(13,-70)
        \path(11,20)(11,-70)
        \path(17,20)(17,-70)
        \path(19,20)(19,-70)
        \path(41,20)(41,-70)
        \path(43,20)(43,-70)
        \path(-45,-42)(45,-42)
        \path(-45,-44)(45,-44)
        \path(-45,-66)(45,-66)
        \path(-45,-68)(45,-68)


\path(-45,18)(45,18)
        \path(-45,16)(45,16)
        \path(-45,-38)(45,-38)
        \path(-45,-36)(45,-36)
        \path(-45,-8)(45,-8)
        \path(-45,-6)(45,-6)
        \path(-45,-12)(45,-12)
        \path(-45,-14)(45,-14)

        \path(-39,20)(-39,16)
        \path(-37,20)(-37,16)
        \path(-35,20)(-35,16)
        \path(-33,20)(-33,16)
        \path(-31,20)(-31,16)
        \path(-29,20)(-29,16)
        \path(-27,20)(-27,16)
        \path(-25,20)(-25,16)
        \path(-23,20)(-23,16)
        \path(-21,20)(-21,16)
        \path(-19,20)(-19,16)
        \path(-17,20)(-17,16)

        \path(-39,-44)(-39,-36)
        \path(-37,-44)(-37,-36)
        \path(-35,-44)(-35,-36)
        \path(-33,-44)(-33,-36)
        \path(-31,-44)(-31,-36)
        \path(-29,-44)(-29,-36)
        \path(-27,-44)(-27,-36)
        \path(-25,-44)(-25,-36)
        \path(-23,-44)(-23,-36)
        \path(-21,-44)(-21,-36)
        \path(-19,-44)(-19,-36)
        \path(-17,-44)(-17,-36)
        \path(-39,-70)(-39,-66)
        \path(-37,-70)(-37,-66)
        \path(-35,-70)(-35,-66)
        \path(-33,-70)(-33,-66)
        \path(-31,-70)(-31,-66)
        \path(-29,-70)(-29,-66)
        \path(-27,-70)(-27,-66)
        \path(-25,-70)(-25,-66)
        \path(-23,-70)(-23,-66)
        \path(-21,-70)(-21,-66)
        \path(-19,-70)(-19,-66)
        \path(-17,-70)(-17,-66)

        \path(9,20)(9,16)
        \path(7,20)(7,16)
        \path(5,20)(5,16)
        \path(3,20)(3,16)
        \path(1,20)(1,16)
        \path(-1,20)(-1,16)
        \path(-3,20)(-3,16)
        \path(-5,20)(-5,16)
        \path(-7,20)(-7,16)
        \path(-9,20)(-9,16)
        \path(-11,20)(-11,16)
        \path(-13,20)(-13,16)

        \path(39,20)(39,16)
        \path(37,20)(37,16)
        \path(35,20)(35,16)
        \path(33,20)(33,16)
        \path(31,20)(31,16)
        \path(29,20)(29,16)
        \path(27,20)(27,16)
        \path(25,20)(25,16)
        \path(23,20)(23,16)
        \path(21,20)(21,16)

        \path(39,-6)(39,-14)
        \path(37,-6)(37,-14)
        \path(35,-6)(35,-14)
        \path(33,-6)(33,-14)
        \path(31,-6)(31,-14)
        \path(29,-6)(29,-14)
        \path(27,-6)(27,-14)
        \path(25,-6)(25,-14)
        \path(23,-6)(23,-14)
        \path(21,-6)(21,-14)

        \path(9,-44)(9,-36)
        \path(7,-44)(7,-36)
        \path(5,-44)(5,-36)
        \path(3,-44)(3,-36)
        \path(1,-44)(1,-36)
        \path(-1,-44)(-1,-36)
        \path(-3,-44)(-3,-36)
        \path(-5,-44)(-5,-36)
        \path(-7,-44)(-7,-36)
        \path(-9,-44)(-9,-36)
        \path(-11,-44)(-11,-36)
        \path(-13,-44)(-13,-36)

        \path(9,-70)(9,-66)
        \path(7,-70)(7,-66)
        \path(5,-70)(5,-66)
        \path(3,-70)(3,-66)
        \path(1,-70)(1,-66)
        \path(-1,-70)(-1,-66)
        \path(-3,-70)(-3,-66)
        \path(-5,-70)(-5,-66)
        \path(-7,-70)(-7,-66)
        \path(-9,-70)(-9,-66)
        \path(-11,-70)(-11,-66)
        \path(-13,-70)(-13,-66)

        \path(-45,14)(-41,14)
        \path(-45,12)(-41,12)
        \path(-45,10)(-41,10)
        \path(-45,8)(-41,8)
        \path(-45,6)(-41,6)
        \path(-45,4)(-41,4)
        \path(-45,2)(-41,2)
        \path(-45,0)(-41,0)
        \path(-45,-2)(-41,-2)
        \path(-45,-4)(-41,-4)

        \path(17,-44)(17,-36)
        \path(19,-44)(19,-36)
        \path(21,-44)(21,-36)
        \path(23,-44)(23,-36)
        \path(25,-44)(25,-36)
        \path(27,-44)(27,-36)
        \path(29,-44)(29,-36)
        \path(31,-44)(31,-36)
        \path(33,-44)(33,-36)
        \path(35,-44)(35,-36)
        \path(37,-44)(37,-36)
        \path(39,-44)(39,-36)

        \path(17,-70)(17,-66)
        \path(19,-70)(19,-66)
        \path(21,-70)(21,-66)
        \path(23,-70)(23,-66)
        \path(25,-70)(25,-66)
        \path(27,-70)(27,-66)
        \path(29,-70)(29,-66)
        \path(31,-70)(31,-66)
        \path(33,-70)(33,-66)
        \path(35,-70)(35,-66)
        \path(37,-70)(37,-66)
        \path(39,-70)(39,-66)

\path(-45,-16)(-41,-16)
        \path(-45,-18)(-41,-18)
        \path(-45,-20)(-41,-20)
        \path(-45,-22)(-41,-22)
        \path(-45,-24)(-41,-24)
        \path(-45,-26)(-41,-26)
        \path(-45,-28)(-41,-28)
        \path(-45,-30)(-41,-30)
        \path(-45,-32)(-41,-32)
        \path(-45,-34)(-41,-34)

        \path(-45,-46)(-41,-46)
        \path(-45,-48)(-41,-48)
        \path(-45,-50)(-41,-50)
        \path(-45,-52)(-41,-52)
        \path(-45,-54)(-41,-54)
        \path(-45,-56)(-41,-56)
        \path(-45,-58)(-41,-58)
        \path(-45,-60)(-41,-60)
        \path(-45,-62)(-41,-62)
        \path(-45,-64)(-41,-64)

        \path(19,14)(11,14)
        \path(19,12)(11,12)
        \path(19,10)(11,10)
        \path(19,8)(11,8)
        \path(19,6)(11,6)
        \path(19,4)(11,4)
        \path(19,2)(11,2)
        \path(19,0)(11,0)
        \path(19,-2)(11,-2)
        \path(19,-4)(11,-4)

        \path(45,14)(41,14)
        \path(45,12)(41,12)
        \path(45,10)(41,10)
        \path(45,8)(41,8)
        \path(45,6)(41,6)
        \path(45,4)(41,4)
        \path(45,2)(41,2)
        \path(45,0)(41,0)
        \path(45,-2)(41,-2)
        \path(45,-4)(41,-4)

        \path(19,-16)(11,-16)
        \path(19,-18)(11,-18)
        \path(19,-20)(11,-20)
        \path(19,-22)(11,-22)
        \path(19,-24)(11,-24)
        \path(19,-26)(11,-26)
        \path(19,-28)(11,-28)
        \path(19,-30)(11,-30)
        \path(19,-32)(11,-32)
        \path(19,-34)(11,-34)

        \path(45,-16)(41,-16)
        \path(45,-18)(41,-18)
        \path(45,-20)(41,-20)
        \path(45,-22)(41,-22)
        \path(45,-24)(41,-24)
        \path(45,-26)(41,-26)
        \path(45,-28)(41,-28)
        \path(45,-30)(41,-30)
        \path(45,-32)(41,-32)
        \path(45,-34)(41,-34)

        \path(19,-46)(11,-46)
        \path(19,-48)(11,-48)
        \path(19,-50)(11,-50)
        \path(19,-52)(11,-52)
        \path(19,-54)(11,-54)
        \path(19,-56)(11,-56)
        \path(19,-58)(11,-58)
        \path(19,-60)(11,-60)
        \path(19,-62)(11,-62)
        \path(19,-64)(11,-64)

        \path(45,-46)(41,-46)
        \path(45,-48)(41,-48)
        \path(45,-50)(41,-50)
        \path(45,-52)(41,-52)
        \path(45,-54)(41,-54)
        \path(45,-56)(41,-56)
        \path(45,-58)(41,-58)
        \path(45,-60)(41,-60)
        \path(45,-62)(41,-62)
        \path(45,-64)(41,-64)

        \path(-19,14)(-11,14)
        \path(-19,12)(-11,12)
        \path(-19,10)(-11,10)
        \path(-19,8)(-11,8)
        \path(-19,6)(-11,6)
        \path(-19,4)(-11,4)
        \path(-19,2)(-11,2)
        \path(-19,0)(-11,0)
        \path(-19,-2)(-11,-2)
        \path(-19,-4)(-11,-4)

\path(-19,-16)(-11,-16)
        \path(-19,-18)(-11,-18)
        \path(-19,-20)(-11,-20)
        \path(-19,-22)(-11,-22)
        \path(-19,-24)(-11,-24)
        \path(-19,-26)(-11,-26)
        \path(-19,-28)(-11,-28)
        \path(-19,-30)(-11,-30)
        \path(-19,-32)(-11,-32)
        \path(-19,-34)(-11,-34)

        \path(-19,-46)(-11,-46)
        \path(-19,-48)(-11,-48)
        \path(-19,-50)(-11,-50)
        \path(-19,-52)(-11,-52)
        \path(-19,-54)(-11,-54)
        \path(-19,-56)(-11,-56)
        \path(-19,-58)(-11,-58)
        \path(-19,-60)(-11,-60)
        \path(-19,-62)(-11,-62)
        \path(-19,-64)(-11,-64)

        \path(-39,-6)(-39,-14)
        \path(-37,-6)(-37,-14)
        \path(-35,-6)(-35,-14)
        \path(-33,-6)(-33,-14)
        \path(-31,-6)(-31,-14)
        \path(-29,-6)(-29,-14)
        \path(-27,-6)(-27,-14)
        \path(-25,-6)(-25,-14)
        \path(-23,-6)(-23,-14)
        \path(-21,-6)(-21,-14)
        \path(-19,-6)(-19,-14)
        \path(-17,-6)(-17,-14)

        \path(9,-6)(9,-14)
        \path(7,-6)(7,-14)
        \path(5,-6)(5,-14)
        \path(3,-6)(3,-14)
        \path(1,-6)(1,-14)
        \path(-1,-6)(-1,-14)
        \path(-3,-6)(-3,-14)
        \path(-5,-6)(-5,-14)
        \path(-7,-6)(-7,-14)
        \path(-9,-6)(-9,-14)
        \path(-11,-6)(-11,-14)
        \path(-13,-6)(-13,-14)


        \path(-48,20)(-50,20)(-50,-40)
        \path(-50,-40)(-50,-70)(-48,-70)
        \put(-68,-27){\shortstack{$M^{-h(u)}$}}
        \path(47,-6)(49,-6)(49,-10)(47,-10)
        \put(51,-10){\shortstack{$C\varrho\rho_{\omega}$}}

\end{picture}
\vspace{7.5cm}
\caption{\label{Fig: counting lemma 2 figure} Proof of Lemma~\ref{counting lemma 2} illustrated, with $d=2$ and $M=3$.  The outermost square is $u$, and the smallest squares depict the root cubes in $\mathcal{A}_u$.}
\end{figure}
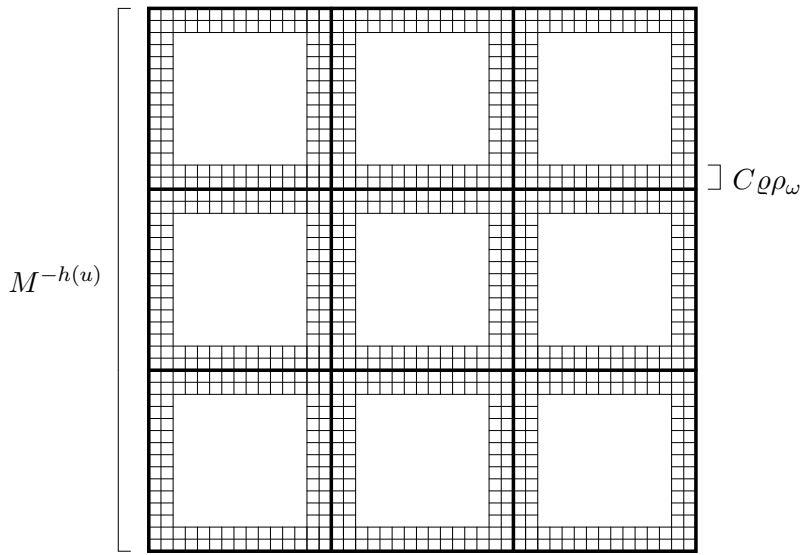
\vskip1cm
\begin{lemma} \label{C count lemma}  
Let $\mathcal E_{2}[u, \omega; \varrho]$ be the collection of pairs of tubes defined in \eqref{defn C_uw}. Then 
\[ \#(\mathcal E_{2}[u, \omega;\varrho]) \leq C \bigl(\varrho \rho_\omega \bigr)^2 2^{2(N - \nu(\omega))} M^{-(d-1)h(u) + (d+1)J}.\]
Here $\nu(\omega)$ denotes the index of the splitting vertex $\omega$, as defined in \eqref{defn nu}.   
\end{lemma} 
\begin{proof}
We combine the counts from Lemmas \ref{counting lemma 1} and \ref{counting lemma 2}. For fixed $t_1, v_1$ and $v_2$, the number of possible $t_2$ such that $\{ (t_1,v_1), (t_2,v_2)\} \in \mathcal E_{2}[u, \omega; \varrho]$ is bounded above by the quantity on the right hand side of \eqref{count for t' given t, v, v' inequality}. The number of possible $t_1$ is at most the right hand side of \eqref{t projection count}, whereas the number of possible $v_1$, hence also $v_2$, is $2^{N - \nu(\omega)}$ due to the binary nature of $\Omega_N$ as discussed in Section \ref{binary slope tree}. The claimed size estimate of $\mathcal E_{2}[u, \omega; \varrho]$ is simply the product of all the quantities mentioned above.   
\end{proof} 


\subsection{Counting slope tuples}
Variations of the arguments presented in Section \ref{tube count two roots section} also apply to more general collections. For the proof of the lower bound \eqref{random Kakeya lower bound}, we will need to estimate, in addition to the above, the sizes of collections consisting of tube triples and tube quadruples with certain pairwise intersections. The collections of tube tuples whose cardinalities are of interest are analogues of $\mathcal E_2[u, \omega;\varrho]$ of greater complexity, and their constructions share the common feature that the probability of slope assignment for any tube tuple within a collection is constant and falls into one of the categories classified in Section \ref{probability estimation section}. As we have seen in that section, the probability depends, among other things, on certain splitting vertices of the slope tree occurring as pairwise youngest common ancestors. In particular, which subset of pairwise youngest common ancestors has to be considered, whether for root or slope, is dictated by the root configuration type. An important component of tube-counting is therefore to estimate how many possible slope tuples can be generated from a given set of such splitting vertices.  Before moving on to the main counting arguments in this section presented in Sections \ref{four tube counting section}  and \ref{three tube counting section}, we observe a few facts that help in counting tuples of slopes, given some information about their ancestry. 
\begin{lemma} \label{slope vertex counting lemma} 
\begin{enumerate}[(i)]
\item \label{how many distinct pairwise ancestors} Given any $\Gamma \subseteq \Omega_N$, $\#(\Gamma) \leq 4$, there exist at most three distinct vertices $\{ \varpi_i : i=1,2,3\} \subseteq \mathcal G(\Omega_N)$ with the properties 
\begin{equation} \label{youngest common ancestor relations}
h(\varpi_1) \leq h(\varpi_2) \leq h(\varpi_3), \quad \varpi_2, \varpi_3 \subseteq \varpi_1,  
\end{equation}  
such that $D(w,w') \in \{ \varpi_i : i=1,2,3 \}$ for any $w \ne w'$, $w, w' \in \Gamma$. 
\item \label{how many slopes given ancestors} Suppose now that we are given $\{ \varpi_i : i=1,2,3 \}$, possibly distinct splitting vertices of the slope tree obeying \eqref{youngest common ancestor relations}. Define 
\begin{equation} \label{defn m}
m = m[\varpi_1, \varpi_2, \varpi_3] := \begin{cases} 2(\nu(\varpi_3) + \nu(\varpi_2))  &\text{ if } \varpi_3 \not\subseteq \varpi_2,  \\ 2 \nu(\varpi_3) + \nu(\varpi_2) + \nu(\varpi_1) &\text{ if } \varpi_3 \subseteq \varpi_2. \end{cases} 
\end{equation}  
Fix three distinct pairs of indices $\{(i_k, j_k) : \; i_k \ne j_k, \; 1\leq k \leq 3\} \subseteq \{1,2,3,4 \}^2$ with the property that $\bigcup \{i_k, j_k : k = 1,2,3 \} = \{1,2,3,4 \}$. Then
\[ \# \bigl\{(w_1, w_2, w_3, w_4) \in \Omega_N^4: D(w_{i_k}, w_{j_k}) = \varpi_k, \; 1 \leq k \leq 3 \bigr\} \leq C2^{4N -m} \]
provided the collection is nonempty.  
\end{enumerate}
\end{lemma} 
\begin{proof} 
If $\Gamma$ is given, we arrange all the pairwise youngest common ancestors of $\Gamma$, i.e,  the vertices in $\mathcal D_{\Gamma} := \{D(w, w') : w \ne w', \; w,w' \in \Gamma\}$, in increasing order of height, where distinct vertices of the same height can be arranged in any way, say according to the lexicographic ordering. We define $\varpi_3$ to be a vertex of maximal height in $\mathcal D_{\Gamma}$, and $\varpi_2$ to be a vertex of maximal height in $\mathcal D_{\Gamma} \setminus \{\varpi_3\}$. Due to maximality of height and the binary nature of the slope tree as ensured by Proposition \ref{pruning stage 1}, $\varpi_3$ has exactly two descendants in $\Gamma$, say $w_1$ and $w_2$.  

If $\varpi_3 \not\subseteq \varpi_2$, then there is no overlap among the descendants of these two vertices. Thus the two descendants $w_3$ and $w_4$ of $\varpi_2$ must be distinct from $w_1, w_2$, thus accounting for all the elements of $\Gamma$. In this case the conclusion of the lemma holds with $\varpi_1 = D(\varpi_2, \varpi_3)$. If $\varpi_3 \subsetneq \varpi_2$, then again by maximality of height $\varpi_2$ can contribute exactly one member of $\Gamma$ that is neither $w_1$ nor $w_2$. Let us call this new member $w_3$. If $\#(\Gamma) = 3$, then the proof is completed by setting $\varpi_1 = D(\varpi_2, \varpi_3) = \varpi_2$. If $\#(\Gamma) = 4$, we call the remaining child $w_4$, which is not descended from $\varpi_2$, and set $\varpi_1 = D(\varpi_2, w_4)$. This selection meets \eqref{youngest common ancestor relations}, and also accounts for all the pairwise youngest common ancestors of $\Gamma$, as required by part (\ref{how many distinct pairwise ancestors}) of the lemma.   

A very similar argument can be used to prove part (\ref{how many slopes given ancestors}). Since the total number of slopes in $\Omega_N$ generated by $\varpi_3$ is exactly $2^{N - \nu(\varpi_3)+1}$, this is the maximum number of possible choices for each of $w_{i_3}$ and $w_{j_3}$. If $\varpi_3 \not\subseteq \varpi_2$, then $\{ i_2, j_2 \} \cap \{ i_3, j_3 \} = \emptyset$. Since each of $w_{i_2}$ and $w_{j_2}$ admits at most $2^{N - \nu(\varpi_2)+1}$ possibilities by the same reasoning, the size of possible four tuples $(w_1, w_2, w_3, w_4)$ in this case is at most 2 raised to the power $2(N - \nu(\varpi_3) + 1) + 2(N - \nu(\varpi_2)+1)$, which gives the claimed estimate. If $\varpi_3 \subseteq \varpi_2 \subseteq \varpi_1$, then by our assumptions on $i_k, j_k$, there exist indices $\ell_2 \in \{i_2, j_2\} \setminus \{i_3, j_3\}$ and $\ell_1 \in \{i_1, j_1\} \setminus \{i_3, j_3, \ell_2\}$. Since $i_3, j_3, \ell_1, \ell_2$ are distinct indices and the number of possible choices of $w_{i_3}$, $w_{j_3}$, $w_{\ell_1}$ and $w_{\ell_2}$ are at most $2^{N - \nu(\varpi_3)}$, $2^{N - \nu(\varpi_3)}$, $2^{N - \nu(\varpi_1)}$ and $2^{N - \nu(\varpi_2)}$ respectively, the result follows. 
\end{proof} 
Minor modifications of the argument above yield the following analogue for slope triples. The proof is left to the reader. 
\begin{lemma} \label{three point slope counting lemma} 
\begin{enumerate}[(i)]
\item Given a collection $\Gamma \subseteq \Omega_N$, $\#(\Gamma) \leq 3$, it is possible to rearrange the collection of vertices $\{ D(w,w'); w \ne w', \; w, w' \in \Gamma\}$ as $\{ \varpi_1, \varpi_2\}$ with $\varpi_2 \subseteq \varpi_1$.  
\item Given a pair $\{ \varpi_1, \varpi_2\} \subseteq \mathcal G(\Omega_N)$ with $\varpi_2 \subseteq \varpi_1$, define 
\begin{equation} \label{defn mhat}
\widehat{m} = \widehat{m}[\varpi_1, \varpi_2] := 2 \nu(\varpi_2) + \nu(\varpi_1). 
\end{equation} 
Let $(i_1, j_1) \ne (i_2, j_2)$ be two pairs of indices such that $\{ i_1, j_1, i_2, j_2 \} = \{1, 2, 3\}$.
Then the following estimate holds:
\[ \#\{ (w_1, w_2, w_3) : D(w_{i_1}, w_{j_1}) = \varpi_1, \; D(w_{i_2}, w_{j_2}) = \varpi_2 \} \leq 2^{3N-\widehat{m}}.  \] 
\end{enumerate}
\end{lemma} 

\subsection{Collections of four tubes with at least two pairwise intersections} \label{four tube counting section}

\subsubsection{Four roots of type 1} \label{subsection four root type 1 size count}
We start with the simplest and generic situation, when the root quadruple is of type 1. Motivated by the expression of the probability obtained in \eqref{four point type 1 probability}, let us first fix two vertex triples $(u,u',z)$ and $(\omega, \omega', v)$ in the root tree and slope tree respectively that satisfy the height and containment relations prescribed in Lemma \ref{four point type 1 lemma}. For such a selection and with $\varrho \in [M^{-J}, 10A_0]$, we define a collection $\mathcal E_{41} = \mathcal E_{41}[u,u',z;\omega, \omega', v; \varrho]$ of sticky-admissible tube quadruples of the form $\{(t_1, v_1), (t_2, v_2), (t_1', v_1'), (t_2', v_2') \}$, obeying the following restrictions:
\begin{equation} \label{defn E_{41}}
\left\{ \begin{aligned} &\mathbb I = \{(t_1, t_2); (t_1', t_2')\} \text{ is of type 1,} \; t_1 \ne t_2, \; t_1' \ne t_2', \; u = D(t_1, t_2), \\  &u' = D(t_1', t_2'), \; z = D(u,u'), \; \omega = D(v_1, v_2), \; \omega' = D(v_1', v_2'), \; v = D(\omega, \omega'), \\ &P_{t_1, v_1} \cap P_{t_2, v_2} \cap [\varrho, C_1 \varrho] \times \mathbb R^d \ne \emptyset, \; P_{t_1', v_1'} \cap P_{t_2', v_2'} \cap [\varrho, C_1 \varrho] \times \mathbb R^d \ne \emptyset. \end{aligned} \right\}
\end{equation} 
The result below provides a bound on the size of $\mathcal E_{41}$.  
\begin{lemma} \label{E_{41} size lemma}
There exists a constant $C > 0$ such that
\[ \#\bigl(\mathcal E_{41} \bigr) \leq C \bigl( \varrho^2 \rho_{\omega} \rho_{\omega'}\bigr)^2 2^{4N - 2\bigl( \nu(\omega) + \nu(\omega') \bigr)} M^{-(d-1) \bigl[ h(u) + h(u') \bigr] + 2(d+1)J}.  \]  
\end{lemma} 
\begin{proof}
Since the intersection and ancestry conditions imply that 
\[ \mathcal E_{41}[u,u',z;\omega, \omega', v] \subseteq \mathcal E_2[u, \omega; \varrho] \times \mathcal E_2[u', \omega'; \varrho], \] 
the stated size bound for $\mathcal E_{41}$ is the product of the sizes of the two factors on the right. These are obtained from Lemma \ref{C count lemma} in Section \ref{tube count two roots section}, applied twice. 
\end{proof}  
\subsubsection{Four roots of type 2} \label{E_{42} size section}
The treatment of this case follows a similar route, though with certain important variations. The main distinction from Section \ref{subsection four root type 1 size count} is that the intersection and type requirements place greater constraints on the selection of the roots and slopes, and hence on the number of tube quadruples. Thus better bounds are possible, compared to the trivial ones exploited in Lemma \ref{C count lemma}. 

Let $(u,u', t)$ and $(\omega, \omega', \vartheta)$ be vertex triples in the root tree and slope tree respectively that meet the requirement of Corollary \ref{four point type 2 structure corollary}. In other words, the vertices $u,u',t$ are linearly ordered in terms of ancestry, and obey $u' \subsetneq u$, while $\omega \cap \vartheta \ne \emptyset$ and $\omega' \cap \vartheta \ne \emptyset$. Holding these fixed, define $\mathcal E_{42} = \mathcal E_{42}[u,u',t;\omega, \omega', \vartheta;\varrho]$ to be the collection of all sticky-admissible tuples of the form $\{(t_i, v_i), (t_i', v_i') : i=1,2\}$ obeying the properties:
\begin{equation} \label{defn E_{42}}
\left\{ \begin{aligned} &\mathbb I = \{(t_1, t_2); (t_1', t_2')\} \text{ is of type 2,} \; u' = D(t_1', t_2') \subsetneq u = D(t_1, t_2), \\ &t = D(t_2, t_2'),   \; \omega = D(v_1, v_2), \; \omega' = D(v_1', v_2'), \; \vartheta = D(v_2, v_2'), \\ &P_{t_1, v_1} \cap P_{t_2, v_2} \cap [\varrho, C_1 \varrho] \times \mathbb R^d \ne \emptyset, \; P_{t_1', v_1'} \cap P_{t_2', v_2'} \cap [\varrho, C_1 \varrho] \times \mathbb R^d \ne \emptyset. \end{aligned} \right\}
\end{equation} 
The vertex triple $(\omega, \omega', \vartheta)$ obeys the hypothesis of Lemma \ref{slope vertex counting lemma}(\ref{how many slopes given ancestors}), permitting the application of this lemma in the counting argument presented in Lemma \ref{E_{42} size lemma}. 
\begin{lemma} \label{applying the slope vertex counting lemma}
If the vertex pairs $(\omega, \vartheta)$ and $(\omega', \vartheta)$ both have the property that one member of the pair is contained in the other, then there exists a rearrangement of $\{\omega, \omega', \vartheta\}$ as $\{\varpi_1, \varpi_2, \varpi_3 \}$ that meets the requirement \eqref{youngest common ancestor relations}.  
\end{lemma} 
\begin{proof}
If $\omega \cap \omega' = \emptyset$, then $\vartheta$ must contain both $\omega$ and $\omega'$. In this case, we rename $\vartheta$ as $\varpi_1$ and call $\varpi_3$ the element of $\{ \omega, \omega' \}$ with greater height. If $\omega \cap \omega' \ne \emptyset$, then the inclusion requirements imply that there must be a ray which contains all three vertices. Since the vertices are linearly ordered, we rename them based on height.    
\end{proof} 
Lemma \ref{applying the slope vertex counting lemma} above allows us to define the quantity $m$ as in \eqref{defn m}, which by a slight abuse of notation we denote by $m[\omega, \omega', \vartheta]$. We are now in a position to state the main result of this subsection, namely the size estimate for $\mathcal E_{42}$. The location of $t$ relative to $u, u'$ affects the size estimate of $\mathcal E_{42}$, even though we have seen that the probability estimate in \eqref{four point type 2 probability} remains unchanged with respect to this property. 
\begin{lemma}\label{E_{42} size lemma}
The following conclusions hold: 
\begin{enumerate}[(i)]
\item \label{u' smaller} If $u' \subseteq t \subseteq u$, then  $\mathcal E_{42}$ is non-empty only if dist$(t, \text{bdry}(u_{\ast})) \leq C\varrho \rho_{\omega}$. Here $u_{\ast}$ is defined to be the unique child of $u$ containing $t$ if $t \subsetneq u$ and is set to be equal to $u$ if $t=u$. In either case,  
\begin{align} 
\#\bigl( \mathcal E_{42}) &\leq C \bigl(\varrho^3 \rho_{\omega'}^2 \rho_{\omega} \bigr)\min \bigl[ \varrho \rho_{\omega}, M^{-h(t)}\bigr]  2^{4N - m[\omega, \omega', \vartheta]} \nonumber \\ &\hskip1.5in \times M^{-(d-1) \bigl( h(t) + h(u') \bigr) + 2(d+1)J}.  \nonumber 
\end{align} 
\item \label{t smaller} If $t \subsetneq u' \subsetneq u$, then $\mathcal E_{42}$ is non-empty only if \[ \text{dist}(t, \text{bdry}(u_{\ast})) \leq C \varrho \rho_{\omega} \quad \text{ and } \quad \text{dist}(t, \text{bdry}(u'_{\ast})) \leq C \varrho \rho_{\omega'},\]
where $u_{\ast}, u'_{\ast}$ are the children of $u, u'$ respectively that contain $t$. In this case,
\begin{equation*} 
\begin{aligned}
\#\bigl( \mathcal E_{42}\bigr) &\leq C \bigl(\varrho^2 \rho_{\omega} \rho_{\omega'} \bigr) 2^{4N - m[\omega, \omega', \vartheta]} \min\left[ \varrho \rho_{\omega}, M^{-h(t)}\right] \\ &\hskip1.3in \times \min\left[ \varrho \rho_{\omega'}, M^{-h(t)} \right]  M^{-2(d-1)h(t) + 2(d+1)J}.
\end{aligned} 
\end{equation*} 
\end{enumerate}
\end{lemma}
\begin{proof}
Both statements in the lemma involve similar arguments. We only prove part (\ref{u' smaller}) in detail, and leave a brief sketch for the other part. The argument here follows the basic structure of Lemma \ref{C count lemma}, since we still have the trivial containment 
\begin{equation} \label{trivial containment}  \mathcal E_{42}[u,u',t;\omega, \omega', \vartheta; \varrho] \subseteq \mathcal E_{2}[u, \omega; \varrho] \times\mathcal E_2[u', \omega';\varrho], \end{equation}  
but with a few modifications resulting from the more refined information about the roots and slopes available from $t$ and $\vartheta$. For instance, combining the defining assumptions that $t_2 \subseteq t$ and $u = D(t_1, t_2)$ with the intersection inequality $|\text{cen}(t_2) - \text{cen}(t_1)| \leq 2C_1 \varrho \rho_{\omega}$ derived from \eqref{intersection inequality restated} in Lemma \ref{counting lemma 1}, we deduce that $t$ has to lie within distance $2C_1 \varrho \rho_{\omega}$ of the boundary of $u_{\ast}$. This is the first conclusion of part (\ref{u' smaller}). For the size bound, we reason as follows. By Lemma \ref{counting lemma 1}(\ref{count for t' given t, v, v'}), the number of $t_1$ and $t_1'$, if everything else is held fixed, is $\leq C (\varrho \rho_{\omega} M^J) (\varrho \rho_{\omega'} M^J) \leq C \varrho^2 \rho_{\omega} \rho_{\omega'} M^{2J}$. Turning to slope counts, we apply Lemma \ref{slope vertex counting lemma}(\ref{how many slopes given ancestors}), the use of which has already been justified in Lemma \ref{applying the slope vertex counting lemma}, to deduce that the number of possible slope quadruples $(v_1, v_2, v_1', v_2')$ is $2^{4N - m}$. It remains to compute the size of the $t_2$ and $t_2'$ projections of $\mathcal E_{42}$. In view of \eqref{trivial containment}, a bound on the size of the $t_2'$ projection is given by the  right hand side of \eqref{t projection count} with $u$ replaced by $u'$. On the other hand, $t_2$ is restricted to lie within $t$ and within distance $2C_1 \varrho \rho_{\omega}$ from the boundary of $t$ if $t \subsetneq u$. This places a nontrivial spatial restriction on $t_2$ only if $2C_1 \varrho \rho_{\omega} < M^{-h(t)}$. If $t = u$, the argument leading up to \eqref{t projection count} shows that $t_2$ lies in $\mathcal A_u$ defined in \eqref{defn A_u}. In either event the volume of the region where $t_2$ can range is at most $C\min(\varrho \rho_{\omega}, M^{-h(t)}) M^{-(d-1)h(t)}$, hence the cardinality of the $t_2$ projection is at most $M^{dJ}$ times this quantity (see Figure~\ref{Fig: E_{42} size lemma}). Combining all these counts yields the bound on the size of $\mathcal E_{42}$ given in part (\ref{u' smaller}).   

\begin{figure}[h!]
\setlength{\unitlength}{0.6mm}
\begin{picture}(-50,0)(-60,-20)

	\special{sh 0.3}\path(30,-30)(30,-60)(34,-60)(34,-30)(30,-30)
	\special{sh 0.3}\path(60,-30)(60,-60)(56,-60)(56,-30)(60,-30)
	\special{sh 0.3}\path(34,-30)(34,-34)(56,-34)(56,-30)(34,-30)
	\special{sh 0.3}\path(34,-56)(34,-60)(56,-60)(56,-56)(34,-56)

        \allinethickness{0.5mm}\path(10,-20)(10,-110)(100,-110)(100,-20)(10,-20)
	\path(30,-30)(60,-30)(60,-60)(30,-60)(30,-30)
	\path(36,-36)(44,-36)(44,-44)(36,-44)(36,-36)

	\allinethickness{0.25mm}\path(26,-28)(64,-28)
	\path(26,-26)(64,-26)
	\path(26,-62)(64,-62)
	\path(26,-64)(64,-64)
	\path(28,-26)(28,-64)
	\path(26,-26)(26,-64)
	\path(62,-26)(62,-64)
	\path(64,-26)(64,-64)
	\path(26,-30)(64,-30)
	\path(26,-60)(64,-60)
	\path(30,-26)(30,-64)
	\path(60,-26)(60,-64)

	\path(32,-26)(32,-64)
	\path(34,-26)(34,-64)
	\path(58,-26)(58,-64)
	\path(56,-26)(56,-64)
	\path(26,-32)(64,-32)
	\path(26,-34)(64,-34)
	\path(26,-56)(64,-56)
	\path(26,-58)(64,-58)

	\path(26,-36)(34,-36)
	\path(26,-38)(34,-38)
	\path(26,-40)(34,-40)
	\path(26,-42)(34,-42)
	\path(26,-44)(34,-44)
	\path(26,-46)(34,-46)
	\path(26,-48)(34,-48)
	\path(26,-50)(34,-50)
	\path(26,-52)(34,-52)
	\path(26,-54)(34,-54)

	\path(56,-36)(64,-36)
	\path(56,-38)(64,-38)
	\path(56,-40)(64,-40)
	\path(56,-42)(64,-42)
	\path(56,-44)(64,-44)
	\path(56,-46)(64,-46)
	\path(56,-48)(64,-48)
	\path(56,-50)(64,-50)
	\path(56,-52)(64,-52)
	\path(56,-54)(64,-54)

	\path(36,-26)(36,-34)
	\path(38,-26)(38,-34)
	\path(40,-26)(40,-34)
	\path(42,-26)(42,-34)
	\path(44,-26)(44,-34)
	\path(46,-26)(46,-34)
	\path(48,-26)(48,-34)
	\path(50,-26)(50,-34)
	\path(52,-26)(52,-34)
	\path(54,-26)(54,-34)

	\path(36,-56)(36,-64)
	\path(38,-56)(38,-64)
	\path(40,-56)(40,-64)
	\path(42,-56)(42,-64)
	\path(44,-56)(44,-64)
	\path(46,-56)(46,-64)
	\path(48,-56)(48,-64)
	\path(50,-56)(50,-64)
	\path(52,-56)(52,-64)
	\path(54,-56)(54,-64)


	\put(92,-106){\Large\shortstack{$u$}}
	\put(50,-54){\Large\shortstack{$t$}}
	\put(37,-42){\shortstack{$u'$}}

	\path(30,-66)(30,-68)(60,-68)(60,-66)
	\put(38,-77){\shortstack{$M^{-h(t)}$}}
	\path(66,-26)(68,-26)(68,-30)(66,-30)
	\put(70,-30){\shortstack{$2C_1\varrho\rho_{\omega}$}}

\end{picture}
\vspace{5.5cm}
\caption{\label{Fig: E_{42} size lemma} Illustration of the spatial restriction on $t_2$ imposed by the conditions $u'\subset t\subset u$, $t_2\subset t$, $\text{dist}(t_1, t_2) \leq 2C_1\varrho\rho_{\omega} < M^{-h(t)}$.  Here, $t_2$ must lie within the shaded region along the boundary of $t$, with $t_1$ falling just outside this boundary in the unshaded thatched region.}
\end{figure}
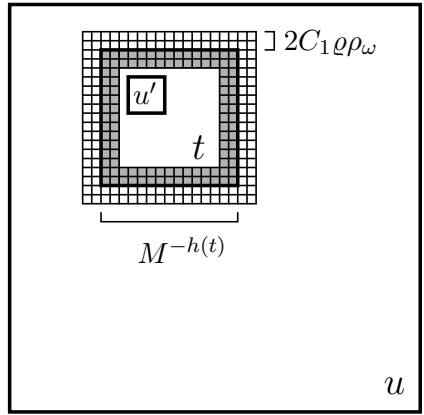 

For part (\ref{t smaller}), the size estimate of $\mathcal E_{42}$ is a product of a number of factors analogous to the ones already considered, the origins of which are indicated below. 
\begin{align*}
&\begin{aligned} &\#(t_1 \text{ given } v_1, v_2, t_2) \leq C \varrho \rho_{\omega} M^J, \\ &\#(t_1' \text{ given } v_1', v_2', t_2') \leq C \varrho \rho_{\omega'} M^J,  \end{aligned} \Biggr\}  \quad \text{(Lemma \ref{counting lemma 1}(\ref{count for t' given t, v, v'}))}\\
&\begin{aligned} &\#(t_2) \leq C \min \bigl[\varrho \rho_{\omega}, M^{-h(t)} \bigr] M^{-(d-1)h(t) + dJ}, \\ &\#(t_2') \leq C \min \bigl[\varrho \rho_{\omega'}, M^{-h(t)} \bigr] M^{-(d-1)h(t) + dJ}, \end{aligned} \Biggr\} \quad \text{(arguments similar to part (\ref{u' smaller}))}, \\ 
&\#(v_1, v_2, v_1', v_2') \leq 2^{4N - m[\omega, \omega', \vartheta]} \quad \text{(from Lemma \ref{slope vertex counting lemma}(\ref{how many slopes given ancestors}))}.
\end{align*} 
We omit the details.  
\end{proof} 

\subsubsection{Four roots of type 3} 
To complete the discussion of size for collections consisting of intersecting tube quadruples, it remains to consider the case where the root configuration is of type 3. Motivated by the conclusions of Lemma \ref{four point type 3 lemma} and Corollary \ref{four point type 3 structure corollary}, we fix two vertex tuples $(u, s_1, s_2)$ and $(\omega, \omega', \vartheta_1, \vartheta_2)$ in the root tree and the slope tree respectively, with the properties that 
$s_1, s_2 \subseteq u$, $h(u) \leq h(s_1) \leq h(s_2)$, $\omega \cap \vartheta_i \ne \emptyset$, and $\omega' \cap \vartheta_i \ne \emptyset$ for $i=1,2$. For such a selection, we define $\mathcal E_{43}[u, s_1, s_2; \omega, \omega', \vartheta_1, \vartheta_2; \varrho]$ to be the collection of all sticky-admissible tuples $\{(t_i, v_i), (t_i', v_i') : i=1,2 \}$ that satisfy the list of conditions below:    
\begin{equation} \label{defn E_{43}}
\left\{ \begin{aligned} &\mathbb I = \{(t_1, t_2); (t_1', t_2')\} \text{ is of type 3}, \; u = D(t_1, t_2) = D(t_1', t_2'), \\ & \omega' = D(v_1', v_2') \subseteq \omega = D(v_1, v_2) , \; s_i = D(t_i, t_i'),   \; \vartheta_i = D(v_i, v_i'), \; i=1,2, \\ &P_{t_1, v_1} \cap P_{t_2, v_2} \cap [\varrho, C_1 \varrho] \times \mathbb R^d \ne \emptyset, \; P_{t_1', v_1'} \cap P_{t_2', v_2'} \cap [\varrho, C_1 \varrho] \times \mathbb R^d \ne \emptyset. \end{aligned} \right\}
\end{equation} 
Since $\mathbb I$ is of type 3, interchanging $(t_1, t_2)$ and $(t_1', t_2')$ leaves $u$ unchanged. Hence we may assume without loss of generality that $\rho_\omega \leq \rho_{\omega'}$. Further, Lemma \ref{slope vertex counting lemma}(\ref{how many distinct pairwise ancestors}) dictates that for $\mathcal E_{43}$ to be non-empty, at most three out of the four vertices $\omega, \omega', \vartheta_1, \vartheta_2$ can be distinct. 
We leave the reader to verify that Lemma \ref{applying the slope vertex counting lemma} can be applied to any triple of these four vertices. Thus for any choice of an eligible tuple $\{\omega, \omega', \vartheta_1, \vartheta_2\}$, there exists a rearrangement of its entries as $\{\varpi_1, \varpi_2, \varpi_3\}$ obeying the hypothesis and hence the conclusion of Lemma \ref{slope vertex counting lemma}(\ref{how many slopes given ancestors}).  This permits an unambiguous definition of the quantity $m[\omega, \omega', \vartheta_1, \vartheta_2]$ as in \eqref{defn m}, which we use  in the statement of the lemma below. 
\begin{lemma} \label{E_{43} size lemma}
If $s_i \subsetneq u$, let $u_i$ denote the child of $u$ that contains $s_i$. Set $\Delta := \min[\varrho \rho_{\omega}, \varrho \rho_{\omega'}]$. 
\begin{enumerate}[(i)]
\item \label{s_1 s_2 location} The collection $\mathcal E_{43}$ is nonempty only if 
\begin{equation} \label{distance constraints} \sum_{i=1}^2\text{dist}\bigl(s_i,\text{bdry}(u_i) \bigr) \leq C \Delta, \end{equation} 
where dist$(s_1, \text{bdry}(u_1))$ is defined to be zero if $u = s_1$. 
\item \label{Delta small} If $\Delta \leq M^{-h(s_1)}$ and $\mathcal E_{43}$ is nonempty, then in addition to \eqref{distance constraints}, one of the following two conditions must hold: 
\begin{enumerate}[1.]
\item $s_2 \subsetneq s_1 = u$, in which case $s_2$ lies within distance $C \Delta$ of the boundary of some child of $s_1 = u$.   
\item $s_2 \cap s_1 = \emptyset$, in which case dist$(s_2, \text{bdry}(s_1)) \leq C \Delta$. 
\end{enumerate}   
In either case, $s_2$ is constrained to lie in the union of at most $2^d M$ slab-like parallepipeds, each with $(d-1)$ ``long'' directions of sidelength $M^{-h(s_1)}$ and one ``short'' direction of sidelength $\Delta$.  
\item \label{Delta large} If $\Delta \geq M^{-h(s_1)}$ and $\mathcal E_{43}$ is nonempty, then in addition to \eqref{distance constraints}, $s_2$ has to lie within a thin tube-like parallelepiped of length $\varrho \min(M^{-h(\omega)}, M^{-h(\omega')})$ in one ``long'' direction and thickness $CM^{-h(s_1)}$ in the remaining $(d-1)$ ``short'' directions; more precisely, both the following inequalities must hold: 
\begin{align}
&|\text{cen}(s_2) - \text{cen}(s_1) + x_1 (\text{cen}(\omega \cap \vartheta_2) - \text{cen}(\omega \cap \vartheta_1))| \leq CM^{-h(s_1)},  \text{ and } \label{cylinder1} \\
&|\text{cen}(s_2) - \text{cen}(s_1) + x_1' (\text{cen}(\omega' \cap \vartheta_2) - \text{cen}(\omega' \cap \vartheta_1))| \leq CM^{-h(s_1)} \label{cylinder2} 
\end{align} 
for some $x_1, x_1' \in [\varrho, C_1 \varrho]$. Here $\text{cen}(t)$ denotes the centre of the cube $t$. 
\item In all cases, if $\mathcal E_{43}$ is nonempty,  
\begin{align*}
\#(\mathcal E_{43}) \leq C 2^{4N - m[\omega, \omega', \vartheta_1, \vartheta_2]} &M^{-2(d-1)h(s_2) + 2(d+1)J} \\ & \times \prod_{i=1}^{2} \Bigl[\min \bigl[ \varrho \rho_{\omega}, M^{-h(s_i)}\bigr] \min \bigl[ \varrho \rho_{\omega'}, M^{-h(s_i)}\bigr] \Bigr]. 
\end{align*}  
\end{enumerate}
\end{lemma} 
\begin{proof}
Let us fix a tuple $\{(t_i, v_i); (t_i', v_i') : i =1,2\}$ in $\mathcal E_{43}$. As in previous proofs such as Lemma \ref{counting lemma 2} (applications of which have appeared in the counting arguments of Lemma \ref{C count lemma} and \ref{E_{42} size lemma}), the key elements are the inequalities 
\begin{equation} \label{intersection inequalities again} 
|\text{cen}(t_2) - \text{cen}(t_1)| \leq C \varrho \rho_{\omega} \quad \text{ and } \quad |\text{cen}(t_2') - \text{cen}(t_1')| \leq C \varrho \rho_{\omega'}.
\end{equation} 
They are proved exactly in the same way as \eqref{dist c(t) and c(t')} follows from \eqref{intersection inequality restated}, resulting from the nontrivial intersection conditions that define $\mathcal E_{43}$. Combined with the set inclusion relations $u = D(t_1, t_2) = D(t_1', t_2')$ and $t_1, t_1' \subseteq s_1$ and $t_2, t_2' \subseteq s_2$ that are guaranteed by the type assumption on the roots, this yields that 
\begin{align*}
\text{dist}(s_i, \text{bdry}(u_i)) &\leq \inf \bigl[\text{dist}(t_i, \text{bdry}(u_i)), \text{dist}(t_i', \text{bdry}(u_i)) \bigr] \\ &= \inf \bigl[ \text{dist}(t_i, u_i^c), \text{dist}(t_i', u_i^c) \bigr] \\ 
&\leq \inf \bigl[\text{dist}(t_1, t_2), \text{dist}(t_1', t_2') \bigr] \\ &\leq C \min[\varrho \rho_{\omega}, \varrho \rho_{\omega'}] =  C\Delta, 
\end{align*}
leading to the distance constraints in \eqref{distance constraints}. Incidentally, the inequalities \eqref{intersection inequalities again} also prove the relation in part (\ref{Delta small}) if $s_2 \cap s_1 = \emptyset$. On the other hand, if $s_2 \subseteq s_1$, then $s_1 = u$ and the desired inequality is simply a restatement of the one in (\ref{distance constraints}). For part (\ref{Delta large}), we refer again to the intersection inequality \eqref{intersection inequality restated}, using it to deduce that 
\begin{align*}
\bigl| &\text{cen}(s_2) - \text{cen}(s_1) + x_1 \bigl(\text{cen}(\omega \cap \vartheta_2) - \text{cen}(\omega \cap \vartheta_1) \bigr)\bigr| \\ &\leq \sum_{i=1}^{2}|\text{cen}(s_i) - \text{cen}(t_i)| + |x_1| \sum_{i=1}^{2}\bigl| \text{cen}(\omega \cap \vartheta_i) - v_i \bigr| +  |\text{cen}(t_2) - \text{cen}(t_1) + x_1(v_2 - v_1)| \\ &\leq \sqrt{d} \sum_{i=1}^{2}M^{-h(s_i)} + C_1 \varrho \sqrt{d} \sum_{i=1}^{2} M^{-h(\vartheta_i)} + 2 c_d \sqrt{d} M^{-J} \leq C M^{-h(s_1)}. 
\end{align*}   
Here we have also used the height and inclusion relations associated with the root configuration type established in Lemma \ref{four point type 3 lemma}; namely, \[ t_i \subseteq s_i, \quad v_i \in \omega \cap \vartheta_i, \quad h(s_i) \leq h(\vartheta_i), \quad h(s_1) \leq h(s_2), \quad i=1,2.\] The inequality above implies that $s_2$ has to lie within distance $O(M^{-h(s_1)})$ of a line segment of length at most $\varrho|\text{cen}(\omega \cap \vartheta_2) - \text{cen}(\omega \cap \vartheta_1)| \leq \varrho M^{-h(\omega)}$. The inequality \eqref{cylinder2} is proved in an identical manner, using $t_i', v_i', \omega'$ instead of $t_i, v_i, \omega$. The first statement in part (\ref{Delta large}) is a consequence of both these inequalities.  

The bound on the size of $\mathcal E_{43}$ uses the same machinery as in the proof of Lemma \ref{E_{42} size lemma}, so we simply indicate the breakdown of the contributions from the different sources: 
\begin{align*}
\#(\mathcal E_{43}) &\leq \underset{\text{$\#(t_1)$ with $t_2, v_1, v_2$ fixed}}{\underbrace{C\min\bigl[ \varrho \rho_{\omega}, M^{-h(s_1)}\bigr] M^{J}}} \times \underset{\text{$\#(t_1')$ with $t_2', v_1', v_2'$ fixed}}{\underbrace{C\min\bigl[ \varrho \rho_{\omega'}, M^{-h(s_1)}\bigr] M^{J}}} 
\times \underset{\begin{subarray}{c}\text{$\#(v_1, v_2, v_1', v_2')$} \\ {\text{from Lemma \ref{slope vertex counting lemma}(\ref{how many slopes given ancestors})}} \end{subarray}}{\underbrace{2^{4N - m[\omega, \omega', \vartheta_1, \vartheta_2]} }}\\
&\quad \times \underset{\text{$\#(t_2$-projection)}}{\underbrace{C\min\bigl[ \varrho \rho_{\omega}, M^{-h(s_2)}\bigr] M^{-(d-1)h(s_2) + dJ}}} \times \underset{\text{$\#(t_2'$-projection)}}{\underbrace{\min\bigl[ \varrho \rho_{\omega'}, M^{-h(s_2)}\bigr] M^{-(d-1)h(s_2) + dJ}}}, 
\end{align*}    
which leads to the stated estimate.
\end{proof} 

\subsection{Collections of three tubes with at least two pairwise intersections} \label{three tube counting section}
For the sake of completeness and book-keeping, we record in this section the cardinality of collections consisting of intersecting tube triples. No new ideas are involved in the proofs, which are in fact simpler than the ones in Section \ref{four tube counting section}. These are left to the interested reader.   

Using the notation set up in Lemmas \ref{three point type 1 lemma} and \ref{three point type 2 lemma}, we define the collections $\mathcal E_{31} = \mathcal E_{31}[u,u'; \omega, \omega'; \varrho]$ and $\mathcal E_{32} = \mathcal E_{32}[u,t; \omega, \omega', \vartheta; \varrho]$ in exactly the same way $\mathcal E_{4i}$ were defined. Namely, $\mathcal E_{3i}$ consists of all sticky-admissible tuples of the form $\{(t_1, v_1), (t_2, v_2), (t_2', v_2') \}$ such that $\mathbb I = \{t_1, t_2, t_2' \}$ is of type $i$ and 
\[ P_{t_1, v_1} \cap P_{t_2, v_2} \cap [\varrho, C_1 \varrho] \times \mathbb R^d \ne \emptyset, \quad P_{t_1, v_1} \cap P_{t_2', v_2'} \cap [\varrho, C_1 \varrho] \times \mathbb R^d \ne \emptyset.\]  
In addition, the members of $\mathcal E_{3i}$ must satisfy 
\[ u = D(t_1, t_2), \quad u' = D(t_1, t_2'), \quad \omega = D(v_1, v_2), \quad \omega' = D(v_1, v_2'),\]
with $u = u'$, $t = D(t_2, t_2')$ and $\vartheta = D(v_2, v_2')$ if $i=2$. We also define the quantities $\widehat{m}[\omega, \omega']$ for $\mathcal E_{31}$ and $\widehat{m}[\omega, \omega', \vartheta]$ for $\mathcal E_{32}$; both are expressed using the formula \eqref{defn mhat},  where $\{ \varpi_1, \varpi_2\}$ with $\varpi_2 \subseteq \varpi_1$ is a rearrangement of $\{ \omega, \omega' \}$ for $\mathcal E_{31}$ and of $\{ \omega, \omega', \vartheta \}$ for $\mathcal E_{32}$, by virtue of Lemma \ref{three point slope counting lemma}. With this notation in place, the size estimates on $\mathcal E_{3i}$ are as follows. 
\begin{lemma}
\begin{enumerate}[(i)]
\item Set $\Delta := \min(\varrho \rho_{\omega}, \varrho \rho_{\omega'})$. Then \[ \#(\mathcal E_{31}) \leq C \Delta \varrho^2 \rho_{\omega} \rho_{\omega'} 2^{3N - \widehat{m}[\omega, \omega']} M^{-(d-1)(h(u) + h(u')) + (2d+1)J}.  \] 
\item With the same definition of $\Delta$ as in part (i), 
\[ \#(\mathcal E_{32}) \leq C \Delta \min[\varrho \rho_{\omega}, M^{-h(t)} ] \min[\varrho \rho_{\omega'}, M^{-h(t)} ]  2^{3N - \widehat{m}[\omega, \omega', \vartheta]} M^{-2(d-1)h(t) + (2d+1)J}.  \] 
\end{enumerate}
\end{lemma} 
\section{Sums over root and slope vertices} \label{summation section} 
A recurrent feature of the proof of \eqref{random Kakeya lower bound}, as we will soon see in Section \ref{LowerBoundSection}, is the use of certain sums over specific subsets of vertices in the root and slope trees. We record the outcomes of these summation procedures in this section for easy reference later.  
\begin{lemma} \label{splitting vertex summation lemma} 
Fix a vertex $\varpi_0 \in \mathcal G(\Omega_N)$, i.e. $\varpi_0$ is a splitting vertex of the slope tree. Then the following estimates hold. 
\begin{enumerate}[(i)]
\item \label{splitting vertex summation 1}
For any $\alpha \in \mathbb R$, 
\[ 
\sum_{\begin{subarray}{c}\varpi \in \mathcal G(\Omega_N) \\ \varpi \subseteq \varpi_0 \end{subarray}} 2^{-\alpha \nu(\varpi)} \leq \begin{cases} C_{\alpha} 2^{-\alpha \nu(\varpi_0)} &\text{ if } \alpha > 1, \\ N 2^{-\nu(\varpi_0)} &\text{ if } \alpha =1, \\ C_{\alpha} 2^{-\alpha \nu(\varpi_0) + N(1 - \alpha)} &\text{ if } \alpha < 1. \end{cases} \]
\item \label{splitting vertex summation 2}
For $M \geq 2$, $\beta > 0$ and $\alpha \geq 1$, 
\[ \sum_{\begin{subarray}{c}\varpi \in \mathcal G(\Omega_N) \\ \varpi \subseteq \varpi_0 \end{subarray}} M^{- \beta h(\varpi)} 2^{- \alpha \nu(\varpi)} \leq C_{\alpha, \beta} M^{-\beta h(\varpi_0)} 2^{-\alpha \nu(\varpi_0)}. \]
\end{enumerate}
\end{lemma} 
\begin{proof}
By Proposition \ref{Splitting tree lemma}, the number of splitting vertices descended from $\varpi_0$ with the property that $\nu(\varpi) = \nu(\varpi_0) + j$ is $2^j$. Since $j$ can be at most $N$, we see that   
\[ \sum_{\begin{subarray}{c}\varpi \in \mathcal G(\Omega_N) \\ \varpi \subseteq \varpi_0 \end{subarray}} 2^{- \alpha \nu(\varpi)} \leq \sum_{j} 2^{-\alpha(\nu(\varpi_0) + j)} 2^j \leq 2^{-\alpha \nu(\varpi_0)} \sum_{j=1}^N 2^{j(1-\alpha)},    \]
from which part (\ref{splitting vertex summation 1}) follows. On the other hand, if $\nu(\varpi) = \nu(\varpi_0) + j$, then $h(\varpi) - h(\varpi_0) \geq \nu(\varpi) - \nu(\varpi_0) = j$. Thus, a similar computation shows that 
 \begin{align*} \sum_{\begin{subarray}{c}\varpi \in \mathcal G(\Omega_N) \\ \varpi \subseteq \varpi_0 \end{subarray}} M^{- \beta h(\varpi)} 2^{- \alpha \nu(\varpi)} &= \sum_{j} M^{- \beta \bigl(h(\varpi_0) + j \bigr)} 2^{- \alpha \bigl(\nu(\varpi_0) + j \bigr)} 2^j \\ &\leq M^{-\beta h(\varpi_0)} 2^{-\alpha \nu(\varpi_0)} \sum_{j=1}^{\infty} M^{-\beta j} 2^{-(\alpha-1)j}.\end{align*} 
The last sum in the displayed expression is convergent, establishing the desired conclusion in part (\ref{splitting vertex summation 2}). 
\end{proof} 

\begin{lemma} \label{root tree summation lemma} 
Fix a vertex $y$ in the root tree and a splitting vertex $\varpi$ in the slope tree such that $h(y) \leq h(\varpi)$. Given a constant $\beta$, one of the following estimates holds for  
\[ \mathfrak s(\beta) := \sum_{z}' M^{-\beta h(z)} 2^{\mu(\varpi, h(z))}, \] 
where the sum $\sum'$ takes place over all vertices $z$ of the root tree such that $z \subseteq y$ and $h(z) \leq h(\varpi)$. 
\begin{enumerate}[(i)]
\item \label{beta<d} If $\beta < d$, then $\mathfrak s(\beta) \leq C_{\beta} 2^{\nu(\varpi)} M^{(d - \beta)h(\varpi) - dh(y)}$.  
\item \label{beta=d} If $\beta = d$, then $\mathfrak s(d) \leq C 2^{\nu(\varpi)} h(\varpi) M^{-dh(y)}$. 
\item \label{beta>d} If $\beta > d$, then $\mathfrak s(\beta) \leq C_{\beta} 2^{\nu(\varpi)} M^{-\beta h(y)}$. 
\item \label{beta>d special M} If $\beta > d$ is large enough so that $2M^d < M^{\beta}$, then $\mathfrak s(\beta) \leq C_{\beta}M^{-dh(y)}$. 
\end{enumerate}
\end{lemma}
\begin{proof}
Since $\varpi$ is a splitting vertex of the slope tree, there exists an integer $j \in [1, N]$ such that $\varpi \in \mathcal G_j(\Omega_N)$, i.e., $\nu(\varpi) = j$. By definition, every $j$th splitting vertex is either itself a $(j-1)$th basic slope cube or is contained in one. Let $\varpi_\ell \in \mathcal H_{\ell}(\Omega_N)$ be the $\ell$th slope cube that contains $\varpi$, so that 
\[ \varpi_1 \supsetneq \varpi_2 \supsetneq \cdots \supsetneq \varpi_{j-1} \supseteq \varpi. \]
If $z$ is a vertex of the root tree such that $h(\varpi_{\ell-1}) \leq h(z) < h(\varpi_\ell)$ for some $\ell \leq j-1$, then $\mu(\varpi, h(z)) = \ell -1$; on the other hand, if $h(\varpi_{j-1}) \leq h(z) \leq h(\varpi)$, then $\mu(\varpi, h(z)) = j-1$. This suggests decomposing the sum defining $\mathfrak s(\beta)$ according to the heights of the slope cubes containing $\varpi$. Implementing this and recalling that $\#\{ z : z \subseteq y, \; h(z) = k \} = M^{dk-dh(y)}$, we obtain
\begin{align*}
\mathfrak s(\beta) &= \sum_{\ell=1}^{j-1}\sum_{k= h(\varpi_{\ell-1})}^{h(\varpi_{\ell})-1} 2^{\ell-1} \sum_{z: h(z) = k}' M^{- \beta k}  + \sum_{k = h(\varpi_{j-1})}^{h(\varpi)} 2^{j-1} \sum_{z : h(z) = k}' M^{- \beta k} \\ &\leq C \Bigl[\sum_{\ell=1}^{j}\sum_{k= h(y)}^{h(\varpi)} 2^{\ell-1} M^{-\beta k} M^{dk - dh(y)} \Bigr] \\ &\leq C M^{-dh(y)}\sum_{\ell=1}^{j} 2^{\ell-1} \sum_{k = h(y)}^{h(\varpi)}M^{(d - \beta)k}  \\ &\leq C M^{-dh(y)}2^{j} \sum_{k = h(y)}^{h(\varpi)}M^{(d - \beta)k} \\ &\leq C 2^{\nu(\varpi)} M^{-dh(y)} \begin{cases} M^{(d- \beta) h(\varpi)} &\text{ if } \beta < d, \\ h(\varpi) &\text{ if } \beta = d, \\ M^{-(\beta -d) h(y)} &\text{ if } \beta > d. \end{cases} 
\end{align*}
Upon simplification, these are the estimates claimed in parts (\ref{beta<d})-(\ref{beta>d}) of the lemma. Part (\ref{beta>d special M}) follows from the observation that $\mu(\varpi, h(z)) \leq h(z)$, hence 
\begin{align*}
\mathfrak s(\beta) \leq \sum_z' M^{-\beta h(z)} 2^{h(z)} &\leq \sum_k 2^k M^{-\beta k + d(k- h(y))} \\ &\leq M^{-dh(y)} \sum_k \left( \frac{2M^d}{M^{\beta}}\right)^k \leq C_{\beta} M^{-dh(y)}.  
\end{align*}  
\end{proof} 
In view of spatial constraints on the ancestors of root cubes as encountered in Lemmas \ref{E_{42} size lemma} and \ref{E_{43} size lemma}, occasionally the sums that we consider take place over more restricted ranges of vertices than the one in Lemma \ref{root tree summation lemma}, even though the summands may retain the same form. The next result makes this quantitatively precise. Let $\varpi$ be a splitting vertex of the slope tree, and $\mathcal R$ a fixed parallelepiped in the root hyperplane with sidelength $\beta$ in $(d-r)$ directions and $\gamma$ in the remaining $r$ directions, where $1 \leq r \leq d-1$ and $\beta \geq \gamma \geq M^{-J}$.  Given constants $\epsilon \geq M^{-h(\varpi)}$ and $\alpha \in \mathbb R$, we define 
\begin{equation} \label{defn spm}
\mathfrak s_{\pm} = \mathfrak s_{\pm}(\alpha, \epsilon, \mathcal R, \varpi) := \sum_{z \in \mathcal Z_{\pm}} M^{-\alpha h(z)} 2^{\mu(\omega, h(z))}, 
\end{equation} 
where the index sets $\mathcal Z_{\pm}$ are collections of vertices of the root tree defined as follows:
\begin{align*}
\mathcal Z &:= \{ z \subseteq \mathcal R : h(z) \leq h(\varpi), \; M^{-h(z)} \leq \epsilon \}, \\ \mathcal Z_{+} &:= \mathcal Z \cap \{z : M^{-h(z)} \geq \gamma \}, \\ \mathcal Z_{-} & := \mathcal Z \cap \{z : M^{-h(z)} \leq \gamma \}.  
\end{align*} 
\begin{lemma}  \label{finer sum lemma} 
The following estimates hold for $\mathfrak s_{\pm}$ defined in \eqref{defn spm}.
\begin{enumerate}[(i)]
\item \label{s+} If $\alpha > d-r$ and $\epsilon \geq \gamma$ then $\mathfrak s_{+} \leq  C 2^{\nu(\varpi)} \beta^{d-r} \epsilon^{\alpha - d + r}$.
\item \label{s-} If $\alpha > d$, then $\mathfrak s_{-} \leq C 2^{\nu(\varpi)} \beta^{d-r} \gamma^r \min(\epsilon, \gamma)^{\alpha -d}$. 
\end{enumerate}
\end{lemma} 
\begin{proof}
We have already established in the proof of Lemma \ref{root tree summation lemma} that $\mu(\varpi, h(z)) \leq \nu(\varpi)-1$. Further if $M^{-k} \geq \gamma$, then there can be at most a constant number of possible choices of $k$th generation $M$-adic cubes $z$ that are contained in $\mathcal R$ and intersect with a slice of $\mathcal R$ that fixes coordinates in the $(d-r)$ long directions. Thus we only need to count the number of possible $z$ in the long directions, obtaining 
\begin{equation} \label{z count} \#\{ z \in \mathcal Q(k) : z \subseteq \mathcal R \} \leq C_r \beta^{d-r} M^{(d-r)k}.  \end{equation} 
Taking this into account, we obtain
\[
\mathfrak s_{+} \leq 2^{\nu(\varpi)} \sum_{k: \gamma \leq M^{-k} \leq \epsilon} M^{-\alpha k} \beta^{d-r} M^{(d-r)k} \leq C 2^{\nu(\varpi)} \beta^{d-r} \epsilon^{\alpha - d+r}, \]
as claimed in part (\ref{s+}).  Part (\ref{s-}) follows in an identical manner; the only difference is that now all directions of $\mathcal R$ are thick relative to the scale of $z$, hence \eqref{z count} has to be replaced by 
\[ \#\{z \in \mathcal Q(k) : z \subseteq \mathcal R \} \leq C \gamma^r \beta^{d-r} M^{dk}. \] 
\end{proof} 

\section{Proof of the lower bound \eqref{random Kakeya lower bound}} \label{LowerBoundSection} 
We are now in a position to complete the proof of Proposition \ref{main theorem reformulated} by verifying the probabilistic statement on the lower bound of $K_N(\mathbb X)$ claimed in \eqref{random Kakeya lower bound}. The two propositions stated below are the main results of this section and allow passage to this final step.  
\begin{proposition} \label{first moment proposition}
Fix integers $N$ and $R$ with $N \gg M$ and $10 \leq R \leq \frac{1}{10} \log_MN$. Define \[P^{\ast}_{t, \sigma, R} := P_{t, \sigma(t)} \cap [M^{-R}, M^{-R+1}] \times \mathbb R^d,\] 
where $\sigma = \sigma_{\mathbb X}$ is the randomized sticky map described in Section \ref{random construction section}. Then there exists a constant $C = C(M,d) > 0$ such that  
\begin{equation} \label{first moment estimate} 
\mathbb E_{\mathbb X}\Bigl[ \sum_{t_1 \ne t_2} \bigl|P^{\ast}_{t_1, \sigma, R} \cap P^{\ast}_{t_2, \sigma, R} \bigr| \Bigr] \leq C NM^{-2R}. 
\end{equation} 
\end{proposition}
\begin{proposition} \label{second moment proposition} 
Under the same hypotheses as Proposition \ref{first moment proposition}, 
\begin{equation} \label{second moment estimate} 
\mathbb E_{\mathbb X} \Bigl[ \Bigl( \sum_{t_1 \ne t_2} \bigl|P^{\ast}_{t_1, \sigma, R} \cap P^{\ast}_{t_2, \sigma, R} \bigr|\Bigr)^2 \Bigr] \leq C \bigl(NM^{-2R}\bigr)^2. 
\end{equation} 
\end{proposition} 
Propositions \ref{first moment proposition} and \ref{second moment proposition} should be viewed as the direct generalizations of \cite[Propositions 8.2 and 8.3]{KrocPramanik} for arbitrary direction sets. These are proved below in Sections \ref{first moment proof section} and \ref{second moment proof section} respectively. Of the two results, Proposition \ref{second moment proposition} is of direct interest, since it leads to \eqref{random Kakeya lower bound}, as we will see momentarily in Corollary \ref{completing random Kakeya lower bound corollary}. Proposition \ref{first moment proposition}, while not strictly speaking relevant to \eqref{random Kakeya lower bound}, nevertheless provides a context for presenting the core arguments within a simpler framework.   
\begin{corollary} \label{completing random Kakeya lower bound corollary}
Proposition \ref{second moment proposition} implies \eqref{random Kakeya lower bound}.  
\end{corollary}
\begin{proof}
The argument here is identical to \cite[Corollary 8.4]{KrocPramanik}, and is briefly sketched. The estimate \eqref{second moment estimate} implies that for any fixed integer $R \in [10, \frac{1}{10} \log N]$, the event
\begin{equation} \label{event fixed R} 
\sum_{t_1 \ne t_2} \bigl| P^{\ast}_{t_1, \sigma, R} \cap P^{\ast}_{t_2, \sigma, R}\bigr| > C N \sqrt{\log N} M^{-2R}
\end{equation} 
holds with probability at most $(C\log N)^{-1}$, by Markov's inequality. Choosing a constant $c > 0$ sufficiently small, one can ensure that the probability of occurrence of the event \eqref{event fixed R} for some $R \in [c \log N, 2c \log N]$ cannot exceed $\frac{1}{10}$. Thus for an approriate choice of small but positive $c$, the revised estimate 
\[ \begin{aligned} \sum_{t_1, t_2} \bigl| P^{\ast}_{t_1, \sigma, R} \cap P^{\ast}_{t_2, \sigma, R}\bigr| &= \Bigl[\sum_{t_1 \ne t_2} + \sum_{t_1 = t_2} \Bigr] \bigl| P^{\ast}_{t_1, \sigma, R} \cap P^{\ast}_{t_2, \sigma, R}\bigr| \\ &\leq CN\sqrt{\log N} M^{-2R} + C M^{-R} M^{-dJ} M^{dJ}  \\ &\leq CN\sqrt{\log N} M^{-2R} \end{aligned} \] 
continues to hold for every $R \in [c \log N, 2c \log N]$ and with probability at least $\frac{9}{10}$. A general measure-theoretic observation, originally due to Bateman and Katz \cite[Proposition 2, p.75]{BatemanKatz} (which was subsequently applied in identical contexts in \cite[Lemma 8, p.64]{Bateman} and \cite[Lemma 8.1]{KrocPramanik}), says that an upper bound on the total size of pairwise intersections of a family of sets translates to a lower bound on the size of their union, according to the following prescription: 
\[ \bigl| \bigcup_t P^{\ast}_{t, \sigma, R}\bigr| \geq \frac{(M^{-R} M^{-dJ} M^{dJ})^2}{C N \sqrt{\log N} M^{-2R}} \geq C^{-1} \frac{1}{N \sqrt{\log N}}. \]
 For each $R$ in the specified range, the union of tubes on the left hand side of the displayed inequality above is contained in $K_N(\mathbb X) \cap [0,1]\times \mathbb R^d$. As $R$ varies, these sets are also essentially disjoint, since they lie in disjoint horizontal strips. Since these bounds are available with high probability for all $R \in [c \log N, 2c \log N]$, we can combine them to obtain
\[ \bigl| K_N(\mathbb X) \cap [0,1] \times \mathbb R^d \bigr| \geq \sum_{R = c \log N}^{2c \log N} \bigl| \bigcup_t P^{\ast}_{t, \sigma, R}\bigr| \geq C^{-1} \frac{\sqrt{\log N}}{N}, \] 
which is the statement \eqref{random Kakeya lower bound}.   
\end{proof} 
\subsection{Proof of Proposition \ref{first moment proposition}} \label{first moment proof section} 
\begin{proof}
We first recast the sum on the left hand side of \eqref{first moment estimate} in a form that brings into focus its connections with the material in Sections \ref{probability estimation section} and \ref{tube count section}. By Lemma \ref{intersection size lemma}, 
\begin{equation} \sum_{t_1 \ne t_2} \bigl| P^{\ast}_{t_1, \sigma, R}  \cap P^{\ast}_{t_2, \sigma, R}\bigr| \leq \sum_1 \frac{C_d M^{-(d+1)J}}{|\sigma(t_1) - \sigma(t_2)| + M^{-J}}, \label{intersection step 1} \end{equation} 
where $\sum_1$ denotes the sum over all root pairs $(t_1, t_2)$ such that $t_1 \ne t_2$ and $P^{\ast}_{t_1, \sigma, R} \cap P^{\ast}_{t_2, \sigma, R} \ne \emptyset$. Unravelling the implications of the intersection we find that 
\begin{align}  \bigl\{(t_1, t_2) :  t_1 \ne t_2, \; &P^{\ast}_{t_1, \sigma, R} \cap P^{\ast}_{t_2, \sigma, R} \ne \emptyset \bigr\} \nonumber \\ &\subseteq \left\{ (t_1, t_2) \Biggl| \begin{aligned} &\exists \text{ a unique pair } (v_1, v_2) \in \Omega_N^2 \text{ such that } \\  &P_{t_1, v_1} \cap P_{t_2, v_2} \cap [M^{-R}, M^{-R+1}] \times \mathbb R^d \ne \emptyset, \\ &\text{ and } \sigma(t_1) = v_1, \; \sigma(t_2) = v_2, \; t_1 \ne t_2 \end{aligned} \right\}.  \label{towards C}  \end{align}  
For a given root pair $(t_1, t_2)$, there may exist more than one slope pair $(v_1, v_2)$ that meets the intersection criterion in \eqref{towards C}. But only one pair will also satisfy, for a given $\sigma$, the requirement $\sigma(t_1) = v_1$, $\sigma(t_2) = v_2$, which explains the uniqueness claim in \eqref{towards C}. Using this, the expression on the right hand side of \eqref{intersection step 1} can be expanded as follows, 
\begin{align}
\sum_{t_1 \ne t_2} \bigl| P^{\ast}_{t_1, \sigma, R}  \cap P^{\ast}_{t_2, \sigma, R}\bigr| &\leq \sum_1  \frac{C_d M^{-(d+1)J}}{|\sigma(t_1) - \sigma(t_2)|} \nonumber 
\\ &\leq \sum_2 \frac{C_d M^{-(d+1)J}}{|v_1-v_2|} T((t_1, v_1), (t_2,v_2)) \nonumber \\ &\leq \sum_{u, \omega} \frac{C_d M^{-(d+1)J}}{\delta_{\omega}}  \sum_3 T((t_1, v_1), (t_2, v_2)), \label{intersection final step}
\end{align} 
where the notation $\sum_2$ in the second step denotes summation over the collection in \eqref{towards C}, and $T((t_1, v_1), (t_2, v_2))$ is a binary (random) counter given by 
\begin{equation} \label{defn Trv}
T((t_1, v_1), (t_2, v_2)) = \begin{cases} 1 &\text{ if } \sigma(t_1) = v_1 \text{ and } \sigma(t_2) = v_2, \\ 0 &\text{ otherwise. }\end{cases} 
\end{equation}  
In the last step \eqref{intersection final step} of the string of inequalities above, we have rearranged the sum in terms of the youngest common ancestors $u = D(t_1, t_2)$ and $\omega = D(v_1, v_2)$ in the root tree and in the slope tree respectively. The summation $\sum_3$ takes place over {\em{all}} sticky-admissible tube pairs $\{(t_1,v_1),(t_2, v_2)\}$ in the deterministic collection $\mathcal E_{2}[u, \omega; \varrho]$ defined in \eqref{defn C_uw}, with $\varrho = \varrho_R = M^{-R}$ and $C_1 = M$.  Incidentally, the requirement of sticky-admissibility restricts $u$ and $\omega$ to obey the height relation $h(u) \leq h(\omega)$. The quantity $\delta_{\omega}$ has been defined in \eqref{inf distance}, and is therefore $\leq |v_1 - v_2|$. 

With this preliminary simplification out of the way, we proceed to compute the expected value of the expression in \eqref{intersection final step}, combining the geometric facts and counting arguments from Section \ref{tube count section} with appropriate probability estimates from Section \ref{probability estimation section}. Accordingly, we get
\begin{align*}
\mathbb E_{\mathbb X} \Bigl[ \sum_{t_1 \ne t_2} \bigl| P^{\ast}_{t_1, \sigma, R}  \cap P^{\ast}_{t_2, \sigma, R}\bigr| \Bigr] &\leq \sum_{u, \omega} \frac{C_d M^{-(d+1)J}}{\delta_{\omega}}  \sum_{\mathcal E_{2}[u, \omega; \varrho]} \mathbb E_{\mathbb X} \bigl[T((t_1, v_1), (t_2, v_2)) \bigr] \\ 
&\leq \sum_{u, \omega} \frac{C_d M^{-(d+1)J}}{\underset{\text{Corollary \ref{Two distances comparable}}}{\underbrace{C_2\rho_{\omega}}}}  \sum_{\mathcal E_{2}[u, \omega;\varrho]}\text{Pr} \bigl( \sigma(t_1) = v_1, \; \sigma(t_2) = v_2 \bigr) \\ 
&\leq C M^{-(d+1)J} \sum_{u, \omega} \rho_{\omega}^{-1} \#(\mathcal E_{2}[u, \omega;\varrho]) \underset{\text{Lemma \ref{two point lemma}}}{\underbrace{\left( \frac{1}{2}\right)^{2N - \mu(\omega, h(u))}}} \\ 
&\leq C M^{-(d+1)J} \sum_{u, \omega} \rho_{\omega}^{-1} \underset{\text{Lemma \ref{C count lemma}}}{\underbrace{(\varrho \rho_{\omega})^2 2^{2(N - \nu(\omega))} M^{-(d-1)h(u) + (d+1)J}}} \\ &\hskip2.3in \times \left( \frac{1}{2}\right)^{2N - \mu(\omega, h(u))} \\
&\leq C M^{-2R} \sum_{u, \omega} M^{-h(\omega) - (d-1)h(u)} 2^{\mu(\omega, h(u)) - 2\nu(\omega)}, 
\end{align*}    
where the last step uses the fact that $\rho_{\omega} \leq \text{diam}(\omega) = \sqrt{d} M^{-h(\omega)}$. To establish the conclusion claimed in \eqref{first moment estimate}, it remains to show that the last expression in the displayed steps above is bounded by $CN$. This follows from a judicious use of the summation results proved in Section \ref{summation section}; namely,   
\begin{align*}
\sum_{u, \omega} M^{-h(\omega) - (d-1)h(u)} 2^{\mu(\omega, h(u)) - 2\nu(\omega)} &= \sum_{\omega \in \mathcal G} M^{-h(\omega)} 2^{-2 \nu(\omega)} \sum_u M^{-(d-1)h(u)} 2^{\mu(\omega, h(u))} \\ &\leq C \sum_{\omega \in \mathcal G} M^{-h(\omega)} 2^{-2 \nu(\omega)} \Bigl[ 2^{\nu(\omega)} M^{h(\omega)} \Bigr] \\ &\leq C \sum_{\omega \in \mathcal G} 2^{-\nu(\omega)} \leq  CN,   
\end{align*}  
where the second and last steps are consequences, respectively, of  Lemma \ref{root tree summation lemma}(\ref{beta<d}) with $\varpi=\omega$ and $\beta = d-1$ and of Lemma \ref{splitting vertex summation lemma}(\ref{splitting vertex summation 1}) with $\alpha = 1$. In both applications, $y$ and $\varpi_0$ have been chosen to be the unit cube, in the root tree and the slope tree respectively. 
\end{proof} 

\subsection{Proof of Proposition \ref{second moment proposition}} \label{second moment proof section} 
We are now ready to prove the main Proposition \ref{second moment proposition}. 
\begin{proof}
As in the proof of Proposition \ref{first moment proposition}, an initial processing of the sum on the left hand side of \eqref{second moment estimate} is needed before embarking on the evaluation of the expectation. Accordingly, we decompose and simplify the quantity of interest as follows, 
\begin{align*}
\Bigl[ \sum_{t_1 \ne t_2} \bigl| P^{\ast}_{t_1, \sigma, R} \cap P^{\ast}_{t_2, \sigma, R} \bigr| \Bigr]^2 &= \sum_{\begin{subarray}{c} t_1 \ne t_2 \\ t_1' \ne t_2' \end{subarray}} \bigl| P^{\ast}_{t_1, \sigma, R} \cap P^{\ast}_{t_2, \sigma, R} \bigr| \times \bigl| P^{\ast}_{t_1', \sigma, R} \cap P^{\ast}_{t_2', \sigma, R} \bigr| \\ &=\mathfrak S_2 + \mathfrak S_3 + \mathfrak S_4, 
\end{align*}
where for $i=2,3,4$,
\begin{align*}
\mathfrak S_i &:= \sum_{\mathfrak I_i} \bigl| P^{\ast}_{t_1, \sigma, R} \cap P^{\ast}_{t_2, \sigma, R} \bigr| \times \bigl| P^{\ast}_{t_1', \sigma, R} \cap P^{\ast}_{t_2', \sigma, R} \bigr|, \text{ and }  \\
\mathfrak I_i &:= \left\{ \mathbb I = \{ (t_1, t_2); (t_1', t_2') \} \Bigl| \begin{aligned} &t_1, t_2, t_1', t_2' \in \mathcal Q(J),  t_1 \ne t_2, \; t_1' \ne t_2', \\ &i = \text{number of distinct elements in $\mathbb I$} \end{aligned} \right\}.  
\end{align*} 
Without loss of generality, by interchanging the pairs $(t_1, t_2)$ and $(t_1', t_2')$ if necessary, we may assume that $h(D(t_1, t_2)) \leq h(D(t_1', t_2'))$ for all quadruples $\mathbb I \in \mathfrak I_i$. We will continue to make this assumption for the treatment of all three terms $\mathfrak S_i$.  

The claimed inequality in \eqref{second moment estimate} is a consequence of the three main estimates below:
\begin{align}
\mathbb E_{\mathbb X} (\mathfrak S_{2}) &\leq CN M^{-2R-dJ}, \label{expectation of S_2} \\ 
\mathbb E_{\mathbb X} (\mathfrak S_{3}) &\leq C NM^{-3R-J}, \text{ and } \label{expectation of S_3} \\ 
\mathbb E_{\mathbb X} (\mathfrak S_{4}) &\leq CN^2 M^{-4R}. \label{expectation of S_4}
\end{align} 
We will prove \eqref{expectation of S_4}  in full detail, since this clearly makes the primary contribution among the three terms mentioned above.  The other two estimates follow analogous and in fact simpler routes using the machinery developed in Sections \ref{probability estimation section} and \ref{tube count section}. We leave their verification to the reader. 

The configuration type of the quadruple $\mathbb I = \{(t_1, t_2); (t_1', t_2')\}$ of distinct roots, as introduced in Section \ref{four roots section}, plays a decisive role in the estimation of \eqref{expectation of S_4}. Recalling the type definitions from that section, we decompose $\mathfrak I_4$ as 
\[ \mathfrak I_4 = \bigsqcup_{i=1}^{3} \mathfrak I_{4i} \text{ where } \mathfrak I_{4i} := \left\{ \mathbb I \in \mathfrak I_4 \Bigl| \begin{aligned} & \text{ $\mathbb I$ is of type $i$ in the } \\ &\text{ sense of Definition \ref{four point type definition}}\end{aligned} \right\}.\]
This results in a corresponding decomposition of $\mathfrak S_4$:
\[ \mathfrak S_4 = \mathfrak S_{41} + \mathfrak S_{42} + \mathfrak S_{43}, \quad \text{ where } \quad \mathfrak S_{4i} = \sum_{\mathfrak I_{4i}} \bigl| P^{\ast}_{t_1, \sigma, R} \cap P^{\ast}_{t_2, \sigma, R} \bigr| \times \bigl| P^{\ast}_{t_1', \sigma, R} \cap P^{\ast}_{t_2', \sigma, R} \bigr|. \] 
We will prove in Sections \ref{S_{41} estimation section}-\ref{S_{43} estimation section} below that 
\begin{equation} \label{S_{4i} expectations}
\mathbb E_{\mathbb X} \bigl[ \mathfrak S_{4i}\bigr] \leq C N^2 M^{-4R} \quad \text{ for } i=1,2, 3.   
\end{equation} 
\end{proof}

\subsubsection{Expected value of $\mathfrak S_{41}$} \label{S_{41} estimation section}
We start with $\mathfrak S_{41}$, simplifying it initially along the same lines as in Proposition \ref{first moment proposition}. As before, a summand in $\mathfrak S_{41}$ is nonzero if and only if the tuple $\{(t_1, t_2); (t_1', t_2')\}$ lies in the set 
\begin{equation} \label{sum_1 range}
\left\{\{(t_1, t_2); (t_1', t_2')\} \in \mathfrak I_{41}\Bigl| \begin{aligned} &\; P^{\ast}_{t_1, \sigma, R} \cap P^{\ast}_{t_2, \sigma, R} \ne \emptyset \\ &\; P^{\ast}_{t_1', \sigma, R} \cap P^{\ast}_{t_2', \sigma, R} \ne \emptyset \end{aligned} \right\},  
\end{equation} 
which in turn is contained in  
\begin{equation} \label{sum_2 range}
\left\{ \{(t_1, t_2); (t_1', t_2')\} \in \mathfrak I_{41}  \; \Biggl| \; \begin{aligned} &\exists \text{ a unique tuple } (v_1, v_2, v_1', v_2') \in \Omega_N^4 \ni \\ &P_{t_1, v_1} \cap P_{t_2, v_2} \cap [M^{-R}, M^{-R+1} ] \times \mathbb R^d \ne \emptyset, \\  &P_{t_1', v_1'} \cap P_{t_2', v_2'} \cap [M^{-R}, M^{-R+1}] \times \mathbb R^d \ne \emptyset, \\ &\sigma(t_i) = v_i, \; \sigma(t_i') = v_i', \;  i=1,2\end{aligned} \right\}.  
\end{equation}
Incorporating this information into the simplification of the sum, we obtain 
\begin{align}
\mathfrak S_{41} &\leq \sum_1 \underset{\text{Lemma \ref{intersection size lemma}}}{\underbrace{\frac{C_d M^{-(d+1)J}}{|\sigma(t_1) - \sigma(t_2)|}  \times \frac{C_d M^{-(d+1)J}}{|\sigma(t_1') - \sigma(t_2')|}}} \nonumber \\ &\leq C M^{-2(d+1)J} \sum_2 \frac{T((t_1, v_1), (t_2, v_2))}{|v_1 - v_2|} \times \frac{T((t_1', v_1'), (t_2', v_2'))}{|v_1' - v_2'|} \nonumber \\ &\leq C M^{-2(d+1)J}\sum_{\begin{subarray}{c} u, u', z \\ \omega, \omega', v \end{subarray}} \frac{1}{\delta_{\omega} \delta_{\omega'}} \sum_3 T((t_1, v_1),(t_2, v_2)) T((t_1', v_1'), (t_2', v_2')), \label{S_{41} prelim} \\ &:= \overline{\mathfrak S}_{41} \nonumber 
\end{align} 
where the summations $\sum_1$ and $\sum_2$ range over the root quadruples in \eqref{sum_1 range} and \eqref{sum_2 range} respectively. The notation $T((t_1,v_1), (t_2, v_2))$ and $\delta_{\omega}$ represent the same quantities as they did in Proposition \ref{first moment proposition}, with their definitions in \eqref{defn Trv} and \eqref{inf distance} respectively. Following the same reasoning that led to \eqref{towards C}, in the last step we have stratified the sum in terms of the root vertices $u = D(t_1, t_2)$, $u' = D(t_1', t_2')$, $z = D(u,u')$ and the (splitting) slope vertices $\omega = D(v_1, v_2)$, $\omega' = D(v_1', v_2')$, $v = D(\omega, \omega')$, so that the summation $\sum_3$ takes place over the tube tuples in the collection $\mathcal E_{41} = \mathcal E_{41}[u,u', z; \omega, \omega, v; \varrho]$ defined in \eqref{defn E_{41}}, with $\varrho = M^{-R}$, $C_1=M$.  We are now in a position to compute the expected value of $\mathfrak S_{41}$.  
\begin{lemma} \label{S_{41} expectation lemma}
The estimate in \eqref{S_{4i} expectations} holds for $i=1$. 
\end{lemma} 
\begin{proof} 
Let us refer to the bound $\overline{\mathfrak S}_{41}$ on $\mathfrak S_{41}$ defined by \eqref{S_{41} prelim} that we obtained from the preliminary simplification. Assembling the various components of the estimation from the previous sections, the expected value of $\mathfrak S_{41}$ is estimated as follows, 
\begin{align}
\mathbb E_{\mathbb X} \bigl( &\mathfrak S_{41} \bigr) \leq \mathbb E_{\mathbb X} \bigl(\overline{\mathfrak S}_{41} \bigr)\nonumber \\ 
&\leq C M^{-2(d+1)J}\sum_{\begin{subarray}{c} u, u', z \\ \omega, \omega', v \end{subarray}} \underset{\text{Corollary \ref{Two distances comparable}}}{\frac{1}{{\underbrace{\rho_{\omega} \rho_{\omega'}}}}} \sum_{3} \text{Pr} \bigl( \sigma(t_i) = v_i, \; \sigma(t_i') = v_i', \; i=1,2\bigr) \nonumber \\ &\leq  CM^{-2(d+1)J}\sum_{\begin{subarray}{c} u, u' , z\\ \omega, \omega' , v\end{subarray}} \frac{\# \bigl(\mathcal E_{41}\bigr)}{\rho_{\omega} \rho_{\omega'}}  \underset{\text{\eqref{four point type 1 probability} from Lemma \ref{four point type 1 lemma}}}{\underbrace{\left( \frac{1}{2}\right)^{4N - \mu(\omega, h(u)) - \mu(\omega', h(u')) - \mu(v, h(z))}}}  \nonumber \\ &\leq CM^{-2(d+1)J} \sum_{\begin{subarray}{c} u, u' , z\\ \omega, \omega', v \end{subarray}} \underset{\text{bound on the size of $\mathcal E_{41}$ from Lemma \ref{E_{41} size lemma}}}{\underbrace{\bigl( \varrho^2 \rho_{\omega} \rho_{\omega'} \bigr)^2 2^{4N - 2(\nu(\omega) + \nu(\omega'))} M^{-(d-1)(h(u) + h(u')) + 2(d+1)J}}} \nonumber \\ &\hskip2in  \times \frac{1}{\rho_{\omega} \rho_{\omega'}} \times 2^{-4N + \mu(\omega, h(u)) + \mu(\omega', h(u')) + \mu(v, h(z))} \nonumber \\ 
&\leq C M^{-4R} \mathfrak S_{41}^{\ast}, \quad \text{ where } \nonumber \end{align}
\begin{equation}  
\begin{aligned} 
\mathfrak S_{41}^{\ast} := \sum_{\omega, \omega', v} & 2^{- 2(\nu(\omega) + \nu(\omega'))}  M^{- [h(\omega) + h(\omega')]}\sum_{u,u',z} 2^{\mu(\omega, h(u)) + \mu(\omega', h(u')) + \mu(v, h(z))} \\ &\hskip1.5in \times M^{-(d-1)[h(u) + h(u')]}.  \end{aligned} \label{S_{41} final sum}  
\end{equation} 
It remains to use the appropriate summation results in Section \ref{summation section} to show that $\mathfrak S_{41}^\ast$ is bounded above by a constant multiple of $N^2$.  We start with the inner sum. 
\begin{align}
\sum_{u, u', z} &M^{-(d-1)(h(u) + h(u'))} 2^{\mu(\omega, h(u)) + \mu(\omega', h(u')) + \mu(v, h(z))} \nonumber \\ 
&\leq \sum_z 2^{\mu(v, h(z))}  \underset{\text{apply Lemma \ref{root tree summation lemma}(\ref{beta<d}), $\beta = d-1$}}{\underbrace{\Bigl[\sum_{u \subseteq z}  M^{-(d-1)h(u)} 2^{\mu(\omega, h(u))}  \Bigr]}}  \underset{\text{apply the same lemma again}}{\underbrace{\Bigl[\sum_{u' \subseteq z}  M^{-(d-1)h(u')} 2^{\mu(\omega', h(u'))}  \Bigr]}} \nonumber \\ 
&\leq \sum_z 2^{\mu(v, h(z))} \Bigl[ M^{-dh(z) + h(\omega)} 2^{\nu(\omega)}\Bigr] \Bigl[ M^{-dh(z) + h(\omega')} 2^{\nu(\omega')}\Bigr] \nonumber \\ 
&\leq CM^{h(\omega) + h(\omega')}2^{\nu(\omega) + \nu(\omega')} \underset{\text{apply Lemma \ref{root tree summation lemma}(\ref{beta>d special M}), } h(y) = 0}{\underbrace{\Bigl[\sum_z 2^{\mu(v, h(z))} M^{-2dh(z)} \Bigr]}} \nonumber \\ 
&\leq  C M^{h(\omega) + h(\omega')} 2^{\nu(\omega) + \nu(\omega')}. \label{inner sum completed} 
\end{align} 
Note that Lemma \ref{root tree summation lemma}(\ref{beta>d special M}) applies with $\beta = 2d$ since $2M^d < M^{2d}$ for $M\geq 2$ and $d \geq 2$. Inserting the expression in \eqref{inner sum completed} into the inner sum of \eqref{S_{41} final sum}, we proceed to complete the outer sum in $\mathfrak S_{41}^{\ast}$. 
\begin{align*}
\mathfrak S_{41}^{\ast} & \leq C \sum_{\omega, \omega', v \in \mathcal G} M^{-h(\omega) - h(\omega')} 2^{-2(\nu(\omega) + \nu(\omega'))} \Bigl[  M^{h(\omega) + h(\omega')} 2^{\nu(\omega) + \nu(\omega')} \Bigr] \\ &\leq C \sum_{\omega, \omega' , v\in \mathcal G} 2^{-\nu(\omega) - \nu(\omega')} \\ 
&\leq C \sum_{v \in \mathcal G} \underset{\text{apply Lemma \ref{splitting vertex summation lemma}(\ref{splitting vertex summation 1}), $\alpha = 1$}}{\underbrace{\Bigl[ \sum_{\omega \in \mathcal G, \omega \subseteq v} 2^{-\nu(\omega)} \Bigr]}} \times \underset{\text{same lemma again}}{\underbrace{\Bigl[ \sum_{\omega' \in \mathcal G, \omega' \subseteq v} 2^{-\nu(\omega')} \Bigr]}} \\ &\leq C \sum_{v \in \mathcal G} \bigl[ N 2^{-\nu(v)}\bigr]^2 \leq CN^2 \sum_{v \in \mathcal G} 2^{-2\nu(v)} \leq CN^2,
\end{align*} 
where at the last step we have again used Lemma \ref{splitting vertex summation lemma} (\ref{splitting vertex summation 1}) with $\alpha = 2$, and $\nu(\varpi_0) = 0$. This completes the proof of the lemma. 
\end{proof} 

\subsubsection{Expected value of $\mathfrak S_{42}$} \label{S_{42} estimation section}
We turn to $\mathfrak S_{42}$ next. After the usual preliminary simplification similar to that of $\mathfrak S_{41}$, we find that $\mathfrak S_{42}$ is bounded by a sum $\overline{\mathfrak S}_{42}$ of the form \eqref{S_{41} prelim}, where 
\begin{equation} \label{bar S_{42}} 
\overline{\mathfrak S}_{42} := CM^{-2(d+1)J} \sum' \frac{1}{\delta_{\omega} \delta_{\omega'}} \sum_3 T((t_1,v_1), (t_2, v_2)) T((t_1', v_1'), (t_2', v_2')).
\end{equation} 
In view of Lemma \ref{four point type 2 lemma} we may assume, after a permutation of $(t_1, t_2)$ and of $(t_1', t_2')$ if necessary, that the outer sum $\sum'$ in \eqref{bar S_{42}} is over all vertex tuples $(u, u', t)$ and $(\omega, \omega', \vartheta)$  in the root tree and the slope tree respectively, such that $u,u',t$ lies on a single ray with $u' \subsetneq u$, while $\omega, \omega', \vartheta \in \mathcal G(\Omega_N)$, $\omega \cap \vartheta \ne \emptyset$, $\omega' \cap \vartheta \ne \emptyset$. The inner sum $\sum_3$ in $\overline{\mathfrak S}_{42}$ ranges over the collection $\mathcal E_{42} = \mathcal E_{42}[u,u',t;\omega, \omega', \vartheta;\varrho]$ defined in \eqref{defn E_{42}} with the usual $\varrho = M^{-R}$ and $C_1 = M$.
\begin{lemma}
The estimate in \eqref{S_{4i} expectations} holds for $i=2$.
\end{lemma} 
\begin{proof}
As in Lemma \ref{S_{41} expectation lemma}, the evaluation of the expectation requires a combination of the appropriate probabilistic estimate from Section \ref{root configurations subsection} (specifically Lemma \ref{four point type 2 lemma}), size estimate of $\mathcal E_{42}$ from Section \ref{E_{42} size section} (specifically Lemma \ref{E_{42} size lemma}) and the summation results from Section \ref{summation section}. Putting these together, we obtain  
\begin{align}
\mathbb E_{\mathbb X} \bigl( \mathfrak S_{42} \bigr) &\leq \mathbb E_{\mathbb X} \bigl( \overline{\mathfrak S}_{42} \bigr) \nonumber \\ 
&\leq C M^{-2(d+1)J} \sum' \frac{1}{\rho_{\omega} \rho_{\omega'}} \sum_3 \text{Pr}\bigl( \sigma(t_i) = v_i, \; \sigma(t_i') = v_i', \; i=1,2\bigr) \nonumber \\ &\leq C M^{-2(d+1)J} \sum' \frac{\# \bigl( \mathcal E_{42}\bigr)}{\rho_{\omega} \rho_{\omega'}}  \underset{\eqref{four point type 2 probability} \text{ from Lemma \ref{four point type 2 lemma}}}{\underbrace{\left( \frac{1}{2}\right)^{4N - \mu(\omega, h(u)) - \mu(\omega', h(u')) - \mu(\vartheta, h(t))}}} \nonumber \\ 
&\leq CM^{-4R} \bigl[ \mathfrak S_{42}^{\ast} + \mathfrak S_{42}^{\circ} \bigr], \nonumber 
\end{align} 
where the closed form expressions for $\mathfrak S_{42}^{\ast}$ and $\mathfrak S_{42}^{\circ}$ at the last step are obtained from the count on the size of $\mathcal E_{42}$ from Lemma \ref{E_{42} size lemma}, and reflect the two complementary cases considered therein. To be precise,  
\begin{align} 
\mathfrak S_{42}^{\ast} &:= \varrho^{-1} \sum'_{u' \subseteq t \subseteq u} \rho_{\omega'} \min \bigl[ \varrho \rho_{\omega}, M^{-h(t)} \bigr] M^{-(d-1) \bigl( h(t) + h(u') \bigr)} \label{S_{42} type2a}  \\ & \hskip1.5in \times 2^{\mu(\omega, h(u)) + \mu(\omega', h(u')) + \mu(\vartheta, h(t)) - m[\omega, \omega', \vartheta]}, \quad \text{ and } \nonumber \\
\mathfrak S_{42}^{\circ} &:= \varrho^{-2} \sum'_{t \subsetneq u' \subsetneq u} \min \bigl[\varrho \rho_{\omega}, M^{-h(t)} \bigr] \min \bigl[\varrho \rho_{\omega'}, M^{-h(t)} \bigr] M^{-2(d-1)h(t)} \label{S_{42} type2b} \\ &\hskip1.5in \times  2^{\mu(\omega, h(u)) + \mu(\omega', h(u')) + \mu(\vartheta, h(t)) - m[\omega, \omega', \vartheta]},  \nonumber 
\end{align} 
where the notation $\sum'_{\mathcal P}$ indicates the subsum of $\sum'$ subject to the additional requirement $\mathcal P$. These two quantities are estimated via the usual channels.  Lemma \ref{E_{42} size lemma} places certain restrictions on the spatial location of $t$, but for a large part of the proof the full strength of these statements will not be needed. For instance, replacing min$(\varrho \rho_{\omega}, M^{-h(t)})$ in \eqref{S_{42} type2a} by $\varrho \rho_{\omega}$, we arrive at the following bound for $\mathfrak S_{42}^{\ast}$:
\begin{align} 
\mathfrak S_{42}^{\ast} &\leq \sum'_{u' \subseteq t \subseteq u} \rho_{\omega} \rho_{\omega'} M^{-(d-1) \bigl( h(t) + h(u')\bigr)} \nonumber 2^{\mu(\omega, h(u)) + \mu(\omega', h(u')) + \mu(\vartheta, h(t)) - m[\omega, \omega', \vartheta]} \nonumber \\
&\leq \sum_{\omega, \omega', \vartheta} \rho_{\omega} \rho_{\omega'} 2^{-m[\omega, \omega', \vartheta]} \mathfrak S_{42}^{\ast}(\text{inner}),  \label{outer and inner} 
\end{align} 
where the inner expression $\mathfrak S_{42}^{\ast}(\text{inner})$ is a sequence of three summations in root vertices, the computation of each requiring a suitable form of Lemma \ref{root tree summation lemma}. Precisely, 
\begin{align}
\mathfrak S_{42}^{\ast}( \text{inner} )&:= \sum_{\begin{subarray}{c} u,u',t \\ u' \subseteq t \subseteq u\end{subarray}} M^{-(d-1) \bigl( h(t) + h(u')\bigr) } 2^{\mu(\omega, h(u)) + \mu(\omega', h(u')) + \mu(\vartheta, h(t))} \nonumber \\   
&= \sum_{\begin{subarray}{c}(t,u) \\ t \subseteq u \end{subarray}} M^{-(d-1) h(t)} 2^{\mu(\omega, h(u)) + \mu(\vartheta, h(t))} \Bigl[\sum_{ u' \subseteq t} M^{-(d-1)h(u')} 2^{\mu(\omega', h(u'))} \Bigr] \nonumber \\ 
&  \leq C \sum_{\begin{subarray}{c}(t,u) \\ t \subseteq u \end{subarray}} M^{-(d-1) h(t)} 2^{\mu(\omega, h(u)) + \mu(\vartheta, h(t))}  \underset{\text{from Lemma \ref{root tree summation lemma}(\ref{beta<d}), $\beta = d-1$}}{\underbrace{\Bigl[2^{\nu(\omega')} M^{-dh(t) + h(\omega')} \Bigr]}} \nonumber \\ 
&\leq C  2^{\nu(\omega')} M^{h(\omega')} \sum_{u} 2^{\mu(\omega, h(u))} \sum_{t: t \subseteq u} M^{-(2d-1)h(t)} 2^{\mu(\vartheta, h(t))} \nonumber \\
& \leq C  2^{\nu(\omega')} M^{h(\omega')} \sum_{u} 2^{\mu(\omega, h(u))} \underset{\text{from Lemma \ref{root tree summation lemma}(\ref{beta>d}), $\beta = 2d-1$}}{\underbrace{\Bigl[ M^{-(2d-1)h(u)} 2^{\nu(\vartheta)}\Bigr]}} \nonumber \\ 
&\leq C  2^{\nu(\omega') + \nu(\vartheta)} M^{h(\omega')} \sum_{u} 2^{\mu(\omega, h(u))} M^{-(2d-1)h(u)} \nonumber \\  
&\leq C  2^{\nu(\omega') + \nu(\vartheta) + \nu(\omega)} M^{h(\omega')}, \label{inner last} 
\end{align}    
where the summation in $u$ in the last step also follows from Lemma \ref{root tree summation lemma}(\ref{beta>d}), since $\beta = 2d-1 > d$. Inserting the estimate \eqref{inner last} of $\mathfrak S_{42}^{\ast}(\text{inner})$ into \eqref{outer and inner}, we proceed to simplify the outer sum.  Let us recall from Lemma \ref{applying the slope vertex counting lemma} that $\{ \omega, \omega', \vartheta \}$ can be rearranged as $\{\varpi_1, \varpi_2, \varpi_3 \}$ satisfying \eqref{youngest common ancestor relations}, and that $m[\omega, \omega', \vartheta]$ is defined  as in \eqref{defn m}. Since the definition of $m$ involves two possibilities, we write $\sum^{[a]}$ and $\sum^{[b]}$ to denote the sum over vertex triples $(\omega, \omega', \vartheta)$ for which $\varpi_3 \not\subseteq \varpi_2$ and $\varpi_3 \subseteq \varpi_2$ respectively. This means that
\begin{align}
\mathfrak S_{42}^{\ast} &\leq C \sum_{\omega, \omega', \vartheta} \rho_{\omega} \rho_{\omega'} 2^{-m[\omega, \omega', \vartheta]} \bigl[ 2^{\nu(\omega') + \nu(\vartheta) + \nu(\omega)} M^{h(\omega')} \bigr] \nonumber \\ 
&\leq C  \Bigl[\sum^{[a]} + \sum^{[b]} \Bigr] \rho_{\omega} \rho_{\omega'} 2^{-m[\omega, \omega', \vartheta]} \bigl[ 2^{\nu(\omega') + \nu(\vartheta) + \nu(\omega)} M^{h(\omega')} \bigr].  \nonumber 
\end{align}
Using the trivial bounds \[\rho_{\omega'} M^{h(\omega')} \leq C \quad \text{ and } \quad \rho_\omega \leq C M^{-h(\omega)} \leq CM^{-h(\varpi_1)}, \]
the estimation is completed as follows, 
\begin{align}
\sum^{[a]} \rho_{\omega} \rho_{\omega'} 2^{-m[\omega, \omega', \vartheta]} &\Bigl[ 2^{\nu(\omega') + \nu(\vartheta) + \nu(\omega)} M^{h(\omega')} \Bigr] \nonumber \\&\leq C \sum_{\varpi_1} M^{-h(\varpi_1)} 2^{\nu(\varpi_1)}\Bigl[\sum_{\begin{subarray}{c} \varpi_2 \\ \varpi_2 \subseteq \varpi_1 \end{subarray}} 2^{-\nu(\varpi_2)} \Bigr] \times \Bigl[\sum_{\begin{subarray}{c} \varpi_3 \\ \varpi_3 \subseteq \varpi_1 \end{subarray}} 2^{-\nu(\varpi_3)} \Bigr] \nonumber \\
&\leq C \sum_{\varpi_1} M^{-h(\varpi_1)} 2^{\nu(\varpi_1)}\underset{\text{Lemma \ref{splitting vertex summation lemma} (\ref{splitting vertex summation 1}) twice}}{\underbrace{\bigl(N 2^{-\nu(\varpi_1)} \bigr)^2}} \nonumber \\
&\leq C N^2 \underset{\text{apply Lemma \ref{splitting vertex summation lemma}(\ref{splitting vertex summation 2})}}{\underbrace{\sum_{\varpi_1} M^{-h(\varpi_1)} 2^{-\nu(\varpi_1)}}} \leq CN^2. \nonumber 
\end{align} 
The same bound holds for $\sum^{[b]}$, and is proved along similar lines: 
\begin{align} 
\sum^{[b]} \rho_{\omega} \rho_{\omega'} 2^{-m[\omega, \omega', \vartheta]} &\Bigl[ 2^{\nu(\omega') + \nu(\vartheta) + \nu(\omega)} M^{h(\omega')} \Bigr] \nonumber \\ &\leq C\sum_{\begin{subarray}{c} \varpi_1, \varpi_2 \\ \varpi_2 \subseteq \varpi_1 \end{subarray}}  M^{-h(\varpi_1)} \sum_{\begin{subarray}{c} \varpi_3 \\ \varpi_3 \subseteq \varpi_2 \end{subarray}} 2^{-\nu(\varpi_3)}  \leq C \sum_{\begin{subarray}{c} \varpi_1, \varpi_2 \\ \varpi_2 \subseteq \varpi_1 \end{subarray}}  M^{-h(\varpi_1)} \underset{\text{Lemma \ref{splitting vertex summation lemma} (\ref{splitting vertex summation 1})}}{\underbrace{\left[ N2^{-\nu(\varpi_2)}\right]}} \nonumber \\ 
&\leq CN \sum_{\varpi_1}  M^{-h(\varpi_1)} \sum_{\varpi_2 \subseteq \varpi_1} 2^{-\nu(\varpi_2)} \leq  CN \sum_{\varpi_1}  M^{-h(\varpi_1)}  \underset{\text{Lemma \ref{splitting vertex summation lemma} (\ref{splitting vertex summation 1})}}{\underbrace{\left[ N2^{-\nu(\varpi_1)}\right]}} \nonumber \\ 
&\leq CN^2 \underset{\text{apply Lemma \ref{splitting vertex summation lemma}(\ref{splitting vertex summation 2})}}{\underbrace{\sum_{\varpi_1} M^{-h(\varpi_1)} 2^{-\nu(\varpi_1)}}} \leq CN^2.   \nonumber
\end{align}
This completes the estimation of $\mathfrak S_{42}^{\ast}$.  

We briefly remark on the analysis of $\mathfrak S_{42}^{\circ}$. For $d \geq 3$, replacing the minima in \eqref{S_{42} type2b} by the trivial bounds $\varrho \rho_{\omega}$ and $\varrho \rho_{\omega'}$ results in an expression analogous to that of $\mathfrak S_{42}^{\ast}$: 
\begin{align*}
\mathfrak S_{42}^{\circ} &\leq \sum'_{t \subsetneq u' \subsetneq u}\rho_{\omega} \rho_{\omega'} M^{-2(d-1) h(t)} 2^{\mu(\omega, h(u)) + \mu(\omega', h(u')) + \mu(\vartheta, h(t)) - m[\omega, \omega', \vartheta]}.
\end{align*} 
This term is estimated exactly the same way as $\mathfrak S_{42}^{\ast}$, since Lemma \ref{root tree summation lemma}(\ref{beta>d}) applies as before with $\beta = 2(d-1) > d$ per our choice of $d$. The bound obtained is a constant multiple of $N$. These details are omitted to avoid repetition. We only present the case $d = 2$, where Lemma \ref{root tree summation lemma} does not give the desired consequence, and the treatment of which exhibits a slight departure from the norm so far.  For $d=2$, inserting the bound min$(\varrho \rho_{\omega}, M^{-h(t)}) \leq \varrho \rho_{\omega}$ into \eqref{S_{42} type2b} yields 
\begin{equation}  \begin{aligned} \mathfrak S_{42}^{\circ} &\leq \sum'_{t \subsetneq u' \subsetneq u} 2^{\mu(\omega, h(u)) + \mu(\omega', h(u')) + \mu(\vartheta, h(t)) - m[\omega, \omega', \vartheta]} \\ &\hskip1.5in \times \begin{cases} \rho_{\omega} \rho_{\omega'} M^{-2h(t)} &\text{ if } M^{-h(t)} \geq \varrho \rho_{\omega'},  \\ \varrho^{-1} \rho_{\omega} M^{-3h(t)} &\text{ if } M^{-h(t)} < \varrho \rho_{\omega'}.  \end{cases} \end{aligned} \label{S_{42} dim 2}\end{equation}
Further, Lemma \ref{E_{42} size lemma}(\ref{t smaller}) prescribes that $t$ cannot be arbitrarily placed inside $u'$, but must lie within the union of at most $2M$ thin rectangles of dimension $\varrho \rho_{\omega'} \times M^{-h(u')}$ each. Using this information, we sum the expression \eqref{S_{42} dim 2} in $t$ as follows: if $\sum_1$ and $\sum_2$ denote the summations in $t$ with $t \subseteq u'$ and $\mathcal E_{42} \ne \emptyset$ subject to the conditions $M^{-h(t)} \geq \varrho \rho_{\omega'}$ and $M^{-h(t)} < \varrho \rho_{\omega'}$ respectively, then 
\begin{align*}
\rho_{\omega} \rho_{\omega'} &\sum_{1}  M^{-2h(t)} 2^{\mu(\vartheta, h(t))} + \varrho^{-1} \rho_{\omega} \sum_2 M^{-3h(t)} 2^{\mu(\vartheta, h(t))} \\ 
&\leq C \rho_{\omega} \rho_{\omega'} \left[ 2^{\nu(\vartheta)} M^{-2h(u')}\right] + C \varrho^{-1} \rho_{\omega} \left[ 2^{\nu(\vartheta)} M^{-h(u')} \bigl(\varrho \rho_{\omega'} \bigr) \min \bigl[ \varrho \rho_{\omega'}, M^{-h(u')}\bigr] \right] \\ 
&\leq C \rho_{\omega} \rho_{\omega'} 2^{\nu(\vartheta)} M^{-2h(u')},   
\end{align*} 
where both sums have been evaluated using Lemma \ref{finer sum lemma} with $d = 2$, $r=1$, $\varpi = \vartheta$, $\beta = \epsilon = M^{-h(u')}$ and $\gamma = \varrho \rho_{\omega'}$. In particular, $\sum_1$ appeals to part (\ref{s+}) of this lemma with $\alpha = 2$ while $\sum_2$ uses part (\ref{s-}) with $\alpha = 3$. Incorporating this into \eqref{S_{42} dim 2}, we find that 
\begin{align} 
\mathfrak S_{42}^{\circ} &\leq \sum_{\omega, \omega', \vartheta} \rho_{\omega} \rho_{\omega'} 2^{\nu(\vartheta) - m[\omega, \omega', \vartheta]} \mathfrak S_{42}^{\circ}(\text{inner}), \quad \text{ where } \label{dim 2 S_{42} t sum done} \\ 
\mathfrak S_{42}^{\circ}(\text{inner}) &:= \sum_u 2^{\mu(\omega, h(u))} \underset{\text{apply Lemma \ref{root tree summation lemma}(\ref{beta=d})}}{\underbrace{\sum_{u'\subseteq u} M^{-2h(u')} 2^{\mu(\omega', h(u'))}}} \label{dim 2 summation 1}\\
&\leq C 2^{\nu(\varpi_2)}  h(\varpi_2) \underset{\text{apply the same lemma again}}{\underbrace{\sum_{u} M^{-2h(u)} 2^{\mu(\omega, h(u))} }} \label{dim 2 summation 2} \\ 
&\leq C 2^{\nu(\varpi_2) + \nu(\varpi_1)}h(\varpi_2) h(\varpi_1), \label{dim 2 inner sum completed}   
\end{align} 
We pause for a moment to explain these steps. In the first application of Lemma \ref{root tree summation lemma}(\ref{beta=d}) in \eqref{dim 2 summation 1} above we have used, in addition to $h(u') \leq h(\omega')$, the fact that 
\[ h(u') = h(D(t_1', t_2')) \leq h(t) = h(D(t_2, t_2')) \leq h(D(v_2,v_2')) = h(\vartheta), \] which is a consequence of stickiness. Since one of $\omega'$ and $\vartheta$ is contained in the other, this implies that $\mu(\omega', h(u')) = \mu(\vartheta, h(u'))$. Hence Lemma \ref{root tree summation lemma}(\ref{beta=d}), applied once with $\varpi = \omega'$ and again with $\varpi = \vartheta$, yields
\begin{align*}  \sum_{u'\subseteq u} M^{-2h(u')} 2^{\mu(\omega', h(u'))} &\leq C M^{-2h(u)}\min \left[ 2^{\nu(\vartheta)} h(\vartheta), 2^{\nu(\omega')} h(\omega')\right] \\ &\leq C h(\varpi_2) 2^{\nu(\varpi_2)} M^{-2h(u)}. \end{align*}
The second application of Lemma \ref{root tree summation lemma}(\ref{beta=d}) in \eqref{dim 2 summation 2} uses a similar argument relying on the fact that $h(u) \leq h(\varpi_1)$. Inserting \eqref{dim 2 inner sum completed} into \eqref{dim 2 S_{42} t sum done},  the estimation of $\mathfrak S_{42}^{\circ}$ can now be completed in the same way as for $\mathfrak S_{42}^{\ast}$:
\begin{align*}
\mathfrak S_{42}^{\circ} &\leq C \sum_{\omega, \omega', \vartheta} \rho_{\omega} \rho_{\omega'} h(\varpi_2) h(\varpi_1) 2^{\nu(\vartheta) + \nu(\varpi_1) + \nu(\varpi_2)-m[\omega, \omega', \vartheta]} \\ &\leq C \sum_{\omega, \omega', \vartheta} M^{-h(\varpi_1) - h(\varpi_2)} h(\varpi_1) h(\varpi_2) 2^{\nu(\vartheta) + \nu(\varpi_1) + \nu(\varpi_2)-m[\omega, \omega', \vartheta]} \\ &\leq C \sum^{[a]} M^{-\frac{1}{2}h(\varpi_1) - \frac{1}{2}h(\varpi_2)}  2^{-\nu(\varpi_3) - \nu(\varpi_2) + \nu(\varpi_1)}+ \sum^{[b]}  M^{-\frac{1}{2}h(\varpi_1) - \frac{1}{2}h(\varpi_2)} 2^{-\nu(\varpi_3)} \\ &\leq CN,   
\end{align*} 
where the symbols $\sum^{[a]}$ and $\sum^{[b]}$ carry the same meaning as they did in the estimation of $\mathfrak S_{42}^{\ast}$ and the last step involves several summations all of which have used appropriate parts of Lemma \ref{splitting vertex summation lemma}. The estimation of $\mathfrak S_{42}$ is complete.
\end{proof} 

\subsubsection{Expected value of $\mathfrak S_{43}$} \label{S_{43} estimation section}
\begin{lemma}
The estimate in \eqref{S_{4i} expectations} holds for $i=3$. 
\end{lemma} 
\begin{proof}
After the usual initial processing of $\mathfrak S_{43}$ which we omit, we reduce to the following estimate:
\begin{align*}
\mathbb E_{\mathbb X} \bigl(\mathfrak S_{43} \bigr) &\leq C M^{-2(d+1)J} \sum' \frac{\#(\mathcal E_{43})}{\rho_{\omega} \rho_{\omega'}} \left( \frac{1}{2}\right)^{4N - \mu(\omega, h(u)) - \mu(\vartheta, h(s_1)) - \mu(\vartheta_2, h(s_2))} \\
&\leq C \sum' \bigl(\rho_{\omega} \rho_{\omega'})^{-1} \bigl( \varrho \rho_{\omega'}\bigr)^2 M^{-2(d-1)h(s_2)} \prod_{i=1}^{2} \Bigl[ \min[ \varrho \rho_{\omega}, M^{-h(s_i)}]\Bigr] \\ &\hskip1.2in \times 2^{-m[\omega, \omega', \vartheta_1, \vartheta_2] + \mu(\omega, h(u)) + \mu(\vartheta_1, h(s_1)) + \mu(\vartheta_2, h(s_2))} \\ 
&\leq C M^{-4R} \bigl[ \mathfrak S_{43}^{\ast} + \mathfrak S_{43}^{\circ} \bigr],
\end{align*} 
where $\sum'$ denotes the sum over all tuples $(u,s_1, s_2)$ in the root tree and $(\omega, \omega', \vartheta_1, \vartheta_2)$ in the slope tree such that $s_1, s_2 \subseteq u$, $h(u) \leq h(s_1) \leq h(s_2)$, $\rho_\omega \leq \rho_{\omega'}$ and for which $\mathcal E_{43}$ is nonempty. The second inequality displayed above uses the estimate on $\#(\mathcal E_{43})$ obtained in Lemma \ref{E_{43} size lemma}, with an additional simplification resulting from min$(\varrho \rho_{\omega'}, M^{-h(s_i)}) \leq \varrho \rho_{\omega'}$. The quantities $\mathfrak S_{43}^{\ast}$ and $\mathfrak S_{43}^{\circ}$ refer to the subsum of $\sum'$ under the additional constraints of $M^{-h(s_1)} \geq \varrho \rho_{\omega}$ and $M^{-h(s_1)} < \varrho \rho_{\omega}$ respectively. Thus
\begin{align}
\mathfrak S_{43}^{\ast} &= \varrho^{-1} \sum'_{M^{-h(s_1)} \geq \varrho \rho_{\omega}} \rho_{\omega'} \min[\varrho \rho_{\omega}, M^{-h(s_2)}] M^{-2(d-1)h(s_2)} \nonumber \\ &\hskip1.2in \times 2^{-m[\omega, \omega', \vartheta_1, \vartheta_2] + \mu(\omega, h(u)) + \mu(\vartheta_1, h(s_1)) + \mu(\vartheta_2, h(s_2))} \nonumber \\ &=: \varrho^{-1} \sum_{\omega, \omega', \vartheta_1, \vartheta_2} \rho_{\omega'} 2^{-m[\omega, \omega', \vartheta_1, \vartheta_2]} \mathfrak S_{43}^{\ast}(\text{inner}), \quad \text{ and } \label{S_{43} star}\\
\mathfrak S_{43}^{\circ} &= \varrho^{-2} \sum'_{M^{-h(s_1)} < \varrho \rho_{\omega}} \frac{\rho_{\omega'}}{\rho_{\omega}} \min[\varrho \rho_{\omega}, M^{-h(s_2)}] M^{-2(d-1)h(s_2) - h(s_1)} \nonumber \\ &\hskip1.2in \times 2^{-m[\omega, \omega', \vartheta_1, \vartheta_2] + \mu(\omega, h(u)) + \mu(\vartheta_1, h(s_1)) + \mu(\vartheta_2, h(s_2))} \nonumber\\  &=: \varrho^{-2} \sum_{\omega, \omega', \vartheta_1, \vartheta_2} \frac{\rho_{\omega'}}{\rho_{\omega}} 2^{-m[\omega, \omega', \vartheta_1, \vartheta_2]} \mathfrak S_{43}^{\circ}(\text{inner}).  \label{S_{43} circle}
\end{align}  
For the purpose of simplifying $\mathfrak S_{43}^{\ast}(\text{inner})$, we recall from Lemma \ref{E_{43} size lemma}(\ref{Delta small}) that $s_2 \subsetneq u$ has sidelength no more than $M^{-h(s_1)}$, and moreover, is constrained to lie in the union of at most $2^dM$ parallelepipeds with $(d-1)$ long directions and one short direction, of dimensions $M^{-h(s_1)}$ and $\varrho \rho_{\omega}$ respectively. Denoting by $\sum_{s_2}^{\ast}$ the summation over all such cubes $s_2$, we find that 
\begin{align}
\sum_{s_2}^{\ast} &2^{\mu(\vartheta_2, h(s_2))} M^{-2(d-1)h(s_2)} \min \bigl[ \varrho \rho_{\omega}, M^{-h(s_2)}\bigr] \nonumber \\ &\leq \varrho \rho_{\omega} \sum^{\ast}_{M^{-h(s_2)} \geq \varrho \rho_{\omega}} M^{-2(d-1)h(s_2)}2^{\mu(\vartheta_2, h(s_2))} + \sum^{\ast}_{M^{- h(s_2)} < \varrho \rho_{\omega}} M^{-(2d-1)h(s_2)} 2^{\mu(\vartheta_2, h(s_2))} \nonumber \\ &\leq \varrho \rho_{\omega} \mathfrak s_{+} + \mathfrak s_{-} \nonumber \\ &\leq C \Bigl[ \varrho \rho_{\omega} 2^{\nu(\vartheta_2)} M^{-2(d-1)h(s_1)} +  2^{\nu(\vartheta_2)}\bigl(\varrho \rho_{\omega} \bigr)^d M^{-(d-1)h(s_1)}\Bigr] \nonumber \\ &\leq C \varrho \rho_{\omega} 2^{\nu(\vartheta_2)} M^{-2(d-1)h(s_1)}, \label{S_{43} inner sum 1}
\end{align}  
where $\mathfrak s_{\pm}$ are defined as in \eqref{defn spm}, and estimated according to Lemma \ref{finer sum lemma}, with the parameters being set at $\epsilon = \beta = M^{-h(s_1)}$, $\gamma = \varrho \rho_{\omega}$, $\varpi = \vartheta_2$ for both. The value of $\alpha$ is $2(d-1)$ for $\mathfrak s_{+}$ and $(2d-1)$ for $\mathfrak s_{-}$.  A similar argument applies for the summation in $s_1$ with $M^{-h(s_1)} \geq \varrho \rho_{\omega}$. According to Lemma \ref{E_{43} size lemma}(\ref{s_1 s_2 location}), $s_1$ has to lie in $u$ and within a distance at most $C\Delta$ from the boundary of some child of $u$. Hence the range of $s_1$ lies within the union of at most $dM$ parallelepipeds, each of dimension $M^{-h(u)}$ in $(d-1)$ directions and $C\Delta$ in the remaining one. Denoting by $\sum_{s_1}^{\ast}$ the relevant sum, and applying Lemma \ref{finer sum lemma} again with $\alpha = 2(d-1)$, $r=1$, $\epsilon = \beta = M^{-h(u)}$, $\gamma = \varrho \rho_{\omega}$, $\varpi = \vartheta_1$, 
\begin{equation} \label{S_{43} inner sum 2}
\sum_{s_1}^{\ast} M^{-2(d-1)h(s_1)} 2^{\mu(\vartheta_1, h(s_1))} \leq \mathfrak s_{+} \leq 2^{\nu(\vartheta_1)} M^{-2(d-1)h(u)}. 
 \end{equation}  
Inserting the estimates \eqref{S_{43} inner sum 1} and \eqref{S_{43} inner sum 2}, we arrive at the following bound on $\mathfrak S_{43}^{\ast}(\text{inner})$: 
\begin{align}
\mathfrak S_{43}^{\ast}(\text{inner}) &= \sum_{u} \sum_{s_1}^{\ast} 2^{\mu(\omega, h(u)) + \mu(\vartheta_1, h(s_1))} \nonumber \\ &\hskip1.5in \times \Bigl[\sum_{s_2}^{\ast} 2^{\mu(\vartheta_2, h(s_2))} M^{-2(d-1)h(s_2)} \min\bigl[ \varrho \rho_{\omega}, M^{-h(s_2)} \bigr] \Bigr] \nonumber \\
&\leq C \sum_{u, s_1} 2^{\mu(\omega, h(u)) + \mu(\vartheta_1, h(s_1))} \Bigl[ 2^{\nu(\vartheta_2)} \varrho \rho_{\omega} M^{-2(d-1)h(s_1)}\Bigr] \nonumber  \\
&\leq \varrho \rho_{\omega} 2^{\nu(\vartheta_2)} \sum_u 2^{\mu(\omega, h(u))} \sum_{s_1}^{\ast} M^{-2(d-1)h(s_1)} 2^{\mu(\vartheta_1, h(s_1))} \nonumber \\
&\leq  \varrho \rho_{\omega} 2^{\nu(\vartheta_2)} \sum_u 2^{\mu(\omega, h(u))} \Bigl[ M^{-2(d-1)h(u)} 2^{\nu(\vartheta_1)} \Bigr] \nonumber \\
&\leq \varrho \rho_{\omega} 2^{\nu(\vartheta_2) + \nu(\vartheta_1)} \sum_u 2^{\mu(\omega, h(u))} M^{-2(d-1)h(u)} \nonumber \\ 
&\leq  \varrho \rho_{\omega} 2^{\nu(\vartheta_2) + \nu(\vartheta_1) + \nu(\varpi_1)} h(\varpi_1), \label{S_{43} inner final}
\end{align} 
where $\varpi_1$ is the youngest common ancestor of $\omega, \omega', \vartheta_1, \vartheta_2$, and hence $h(\varpi_1) \geq h(u)$. The last estimate follows from Lemma \ref{root tree summation lemma}, invoking part (\ref{beta>d}) if $d \geq 3$ and part(\ref{beta<d}) if $d = 2$. An analogous sequence of steps, the details of which are left to the reader, can be executed to estimate $\mathfrak S_{43}^{\circ}(\text{inner})$, the only distinction being that the space restrictions are now dictated by Lemma \ref{E_{43} size lemma}(\ref{Delta large}), so that the summation in $s_2$ invokes Lemma \ref{finer sum lemma} with $r = d-1$, $\beta = \varrho \min(M^{-h(\omega)}, M^{-h(\omega')})$, $\gamma = M^{-h(s_1)}$. The outcome of this is that 
\begin{equation} \label{S_{43} inner circle final}
\mathfrak S_{43}^{\circ}(\text{inner}) \leq \varrho^2 \rho_{\omega} \min(M^{-h(\omega)}, M^{-h(\omega')})  2^{\nu(\vartheta_2) + \nu(\vartheta_1) + \nu(\varpi_1)} h(\varpi_1). 
\end{equation}  
Substituting \eqref{S_{43} inner final} into \eqref{S_{43} star} and \eqref{S_{43} inner circle final} into \eqref{S_{43} circle} leads to the following simpler sum over slope vertices:
\[ \mathfrak S_{43}^{\ast} + \mathfrak S_{43}^{\circ} \leq C \sum_{\omega, \omega', \vartheta_1, \vartheta_2} M^{-h(\omega) - h(\omega')} 2^{-m[\omega, \omega', \vartheta_1, \vartheta_2] + \nu(\varpi_1) + \nu(\vartheta_1) + \nu(\vartheta_2)}. \]
In order to complete the summation, let us recall that the sum, ostensibly over four parameters, in fact ranges over at most three vertices $\{ \varpi_1, \varpi_2, \varpi_3 \}$, which is a rearrangement of the quadruple $\{\omega, \omega', \vartheta_1, \vartheta_2 \}$ satisfying \eqref{youngest common ancestor relations}. However, it is not apriori possible to assign a unique correspondence between these two sets of vertices. Indeed, as already indicated in the last paragraph of Section \ref{probability estimation section}, the configuration type of the slopes (which does not in general mimic the configuration type of the roots) dictates which vertex or vertices of the quadruple $\{\omega, \omega', \vartheta_1, \vartheta_2 \}$ represents $\varpi_i$ after the rearrangement. A careful analysis of the possible structures of $\omega, \omega', \vartheta_1, \vartheta_2$, as depicted in Figure~\ref{Fig: slope configurations}, shows that 
\begin{equation*} 
\begin{aligned}
M^{-h(\omega) - h(\omega')} h(\varpi_1) &2^{-m[\omega, \omega', \vartheta_1, \vartheta_2] + \nu(\varpi_1) + \nu(\vartheta_1) + \nu(\vartheta_2)} \\ 
&\leq M^{-2h(\varpi_1)} h(\varpi_1) \times \begin{cases} 2^{-\nu(\varpi_3) - \nu(\varpi_2) + \nu(\varpi_1)} &\text{ if } \varpi_3 \not\subseteq \varpi_2 \\ 2^{-\nu(\varpi_3)} &\text{ if } \varpi_3 \subseteq \varpi_2. \end{cases} 
\end{aligned} 
\end{equation*} 
The expression on the right hand side is of the type already considered in the estimation of $\mathfrak S_{42}^{\ast}$ and $\mathfrak S_{42}^{\circ}$. In particular, it is summable in $\varpi_1, \varpi_2, \varpi_3$ using repeated applications of Lemma \ref{splitting vertex summation lemma} and yields the desired bound of $CN^2$. 

\begin{figure}[h!]
\setlength{\unitlength}{0.7mm}
\begin{picture}(-55,-10)(-118,-38)

        \allinethickness{0.1mm}\path(-100,-40)(-100,-100)(71,-100)(71,-40)(-100,-40)
	\path(-43,-40)(-43,-100)
	\path(14,-40)(14,-100)
	\path(-100,-40)(71,-40)

	\special{sh 0.99}\put(-71,-50){\ellipse{2}{2}}
	\put(-68,-60){\tiny\shortstack{$\vartheta_1=\vartheta_2 = \varpi_1$}}
	\special{sh 0.99}\put(-71,-60){\ellipse{2}{2}}
	\put(-97,-70){\tiny\shortstack{$\omega = \varpi_2$}}
	\special{sh 0.99}\put(-81,-70){\ellipse{2}{2}}
	\put(-58,-77){\tiny\shortstack{$\omega' = \varpi_3$}}
	\special{sh 0.99}\put(-61,-77){\ellipse{2}{2}}
	\put(-89,-96){\tiny\shortstack{$v_1$}}
	\special{sh 0.99}\put(-87,-90){\ellipse{2}{2}}
	\put(-79,-96){\tiny\shortstack{$v_2$}}
	\special{sh 0.99}\put(-78,-90){\ellipse{2}{2}}
	\put(-64,-96){\tiny\shortstack{$v_1'$}}
	\special{sh 0.99}\put(-62,-90){\ellipse{2}{2}}
	\put(-54,-96){\tiny\shortstack{$v_2'$}}
	\special{sh 0.99}\put(-52,-90){\ellipse{2}{2}}

	\path(-71,-50)(-71,-60)
	\path(-81,-70)(-71,-60)(-61,-77)
	\path(-87,-90)(-81,-70)(-78,-90)
	\path(-62,-90)(-61,-77)(-52,-90)
	\put(-97,-47){\shortstack{(1)}}


	\special{sh 0.99}\put(43,-50){\ellipse{2}{2}}
	\special{sh 0.99}\put(43,-60){\ellipse{2}{2}}
	\put(46,-61){\tiny\shortstack{$\omega = \vartheta_1 = \varpi_1$}}
	\special{sh 0.99}\put(43,-70){\ellipse{2}{2}}
	\put(46,-71){\tiny\shortstack{$\omega' = \varpi_2$}}
	\special{sh 0.99}\put(52,-80){\ellipse{2}{2}}
	\put(55,-81){\tiny\shortstack{$\vartheta_2 = \varpi_3$}}
	\put(25,-96){\tiny\shortstack{$v_1$}}
	\special{sh 0.99}\put(27,-90){\ellipse{2}{2}}
	\put(35,-96){\tiny\shortstack{$v_1'$}}
	\special{sh 0.99}\put(36,-90){\ellipse{2}{2}}
	\put(51,-96){\tiny\shortstack{$v_2'$}}
	\special{sh 0.99}\put(52,-90){\ellipse{2}{2}}
	\put(61,-96){\tiny\shortstack{$v_2$}}
	\special{sh 0.99}\put(62,-90){\ellipse{2}{2}}

	\path(43,-50)(43,-70)
	\path(43,-60)(27,-90)
	\path(36,-90)(43,-70)(52,-80)(52,-90)
	\path(52,-80)(62,-90)
	\put(17,-47){\shortstack{(3)}}


	\special{sh 0.99}\put(-14,-50){\ellipse{2}{2}}
	\special{sh 0.99}\put(-14,-60){\ellipse{2}{2}}
	\put(-11,-61){\tiny\shortstack{$\omega=\vartheta_1 = \varpi_1$}}
	\special{sh 0.99}\put(-14,-70){\ellipse{2}{2}}
	\put(-11,-71){\tiny\shortstack{$\vartheta_2 = \varpi_2$}}
	\special{sh 0.99}\put(-5,-80){\ellipse{2}{2}}
	\put(-2,-81){\tiny\shortstack{$\omega' = \varpi_3$}}
	\put(-31,-96){\tiny\shortstack{$v_1$}}
	\special{sh 0.99}\put(-29,-90){\ellipse{2}{2}}
	\put(-23,-96){\tiny\shortstack{$v_2$}}
	\special{sh 0.99}\put(-21,-90){\ellipse{2}{2}}
	\put(-6,-96){\tiny\shortstack{$v_1'$}}
	\special{sh 0.99}\put(-5,-90){\ellipse{2}{2}}
	\put(4,-96){\tiny\shortstack{$v_2'$}}
	\special{sh 0.99}\put(5,-90){\ellipse{2}{2}}

	\path(-14,-50)(-14,-70)
	\path(-14,-60)(-30,-90)
	\path(-21,-90)(-14,-70)(-5,-80)(-5,-90)
	\path(-5,-80)(5,-90)
	\put(-40,-47){\shortstack{(2)}}

\end{picture}
\vspace{4.5cm}
\caption{\label{Fig: slope configurations} A partial list of 4-slope configurations for 4 roots of type 3, with distinct $\{\varpi_1, \varpi_2, \varpi_3 \}$. Other configurations (where partial coincidences may arise) are possible after permutation of $\{ v_1, v_1', v_2, v_2' \}$ in these diagrams.} 
\end{figure}
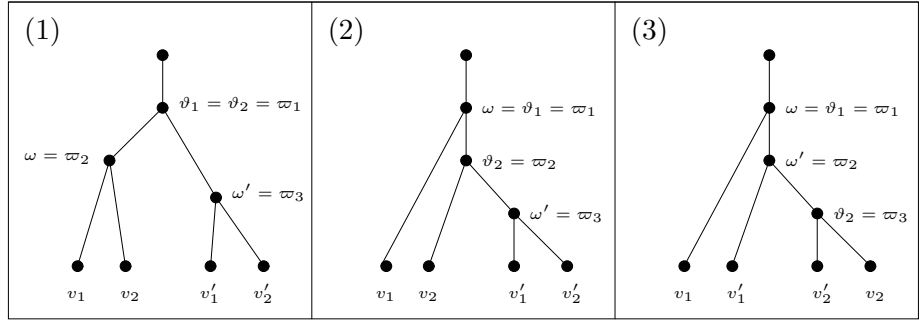

\end{proof} 

\section{Appendix: Percolation on trees} \label{percolation on trees section}  

As in \cite{{BatemanKatz}, {Bateman}, {KrocPramanik}}, the argument of Section \ref{UpperBoundSection} requires the use of a special probabilistic process on certain trees called a (bond) percolation.  More precisely, given some tree $\mathcal{T}$ with edge set $\mathcal{E}$, we define an \textit{edge-dependent Bernoulli (bond) percolation process} to be a collection of independent random variables $\{X_e : e\in\mathcal{E}\}$, where $X_e$ is Bernoulli$(p_e)$ with $p_e<1$.  If the random variables $\{X_e : e\in\mathcal{E}\}$ are mutually independent and identically distributed Bernoulli$(p)$ random variables, with a constant $p < 1$ independent of the edge $e$, then the process is called a \textit{standard Bernoulli$(p)$ percolation}.  We are concerned with only standard Bernoulli$(\frac 12)$ percolation in this paper.  The interested reader may consult~\cite{Grimmett} for a discussion of percolation processes in much greater generality.

Given a percolation process on a tree $\mathcal{T}$, we think of the event $\{X_e=0\}$ as the event that we \textit{remove} the edge $e$ from the edge set $\mathcal{E}$, and the event $\{X_e=1\}$ as the event that we \textit{retain} this edge.  Thus, for a given edge $e\in\mathcal{E}$, we think of $p = \text{Pr}(X_e=1)$ as the probability that we retain this edge after percolation.  Survival of the tree is defined to be the event that at least one ray remains from the root of the tree to its bottommost level. The probability of this event is referred to as the \textit{survival probability} of the corresponding percolation process.  This probability can be estimated by visualizing percolation on a tree as a certain electrical network, as first imagined by Lyons in \cite{Lyons1}.  The natural electrical network is defined as follows: we truncate the tree $\mathcal{T}$ at height $N$ and place the positive node of a battery at the root of $\mathcal{T}_N$.  Then, for every ray in $\partial\mathcal{T}_N$, there is a unique terminating vertex; we connect each of these vertices to the negative node of the battery.  A resistor is placed on every edge $e$ of $\mathcal{T}_N$ with resistance $R_e$ defined by
\begin{equation}\label{resistance}
	\frac 1{R_e} = \frac 1{1-p_e}\prod_{\begin{subarray}{c} e' \in \mathcal E \\ v(e)\subseteq v(e')\end{subarray}} p_{e'},
\end{equation}
where $v(e)$ is the vertex in $\mathcal{T}$ at which $e$ terminates.  Notice that the resistance for the edge $e$ is essentially the reciprocal of the probability that a path remains from the root of the tree to the vertex $v(e)$ after percolation.  For standard Bernoulli$(\frac 12)$ percolation, we have 
\begin{equation}\label{resistances}
	R_e = 2^{h(v(e))-1}.
\end{equation}
A seminal result of Lyons~\cite[Theorem 2.1]{Lyons2}, says that for quite general trees the total resistance of an electrical network is comparable to the inverse of the survival probability of the associated Bernoulli percolation process.  For our purposes, a reasonable upper bound on the survival probability of Bernoulli$(\frac{1}{2})$ percolation on a rooted labelled subtree of the $M$-adic tree suffices.  We state the necessary result in a form convenient for our usage. 

\begin{proposition}[Lyons \cite{Lyons2}]\label{survival prob}
	Let $M\geq 2$ and let $\mathcal{T}_N$ be a subtree of height $N$ of the full $M$-adic tree of the same height in dimension $d$.  For a Bernoulli$(\frac{1}{2})$ percolation process defined on $\mathcal T$, let $R(\mathcal{T}_N)$ be the total resistance of the electrical network on $\mathcal{T}_N$ defined via \eqref{resistance}. If $\text{Pr}(\mathcal{T}_N)$ denotes the survival probability after percolation on $\mathcal{T}_N$, then 
\begin{equation}\label{survival}
	\text{Pr}(\mathcal{T}_N) \leq \frac 2{1+R(\mathcal{T}_N)}.
\end{equation}
\end{proposition}

See~\cite{Bateman} or~\cite[Proposition 5.3]{KrocPramanik} for a proof of this result.  In light of Proposition~\ref{survival prob}, we see that to bound the survival probability after Bernoulli percolation it is sufficient to bound the resistance of the network from below.  To accomplish this, we need the useful fact that connecting any two vertices at a given height by an ideal conductor (i.e. one with zero resistance) only decreases the overall resistance of the circuit.  
\begin{proposition}\label{resistance prop}
	Let $\mathcal{T}_N$ be a truncated tree of height $N$ with corresponding electrical network generated by a standard Bernoulli$(\frac 12)$ percolation process.  Suppose at height $k<N$ we connect two vertices by a conductor with zero resistance.  Then the resulting electrical network has a total resistance no greater than that of the original network.
\end{proposition}
For a proof of this fact, see~\cite[Proposition 5.1]{KrocPramanik}.  The main consequence of this observation that we draw upon  in Lemma~\ref{computing survival probability} is given by the following corollary.

\begin{corollary} \label{survival probability reduced}
	Given a subtree $\mathcal{T}_N$ of height $N$ contained in the full $d$-dimensional $M$-adic tree, let $R(\mathcal T_N)$ denote the total resistance of the electrical network that corresponds to standard Bernoulli$(\frac{1}{2})$ percolation on this tree, in the sense of the theorem of Lyons as given in Proposition \ref{survival prob}.   Then
\begin{equation}\label{ResistBound}
	R(\mathcal{T}_N) \geq \sum_{k=1}^N\frac {2^{k-1}}{n_k},
\end{equation}
where $n_k$ denote the number of its $k$th generation vertices in $\mathcal T_N$.
\end{corollary} 

\begin{proof}
To show this, we construct an auxiliary electrical network from the one naturally associated to our tree $\mathcal{T}_N$, as follows.  For every $k\geq 1$, we connect all vertices at height $k$ by an ideal conductor to make one node $V_k$.  Call this new circuit $E$.  The resistance of $E$ cannot be greater than the resistance of the original circuit, by Proposition~\ref{resistance prop}. 

Fix $k$, $1\leq k\leq N$, and let $R_k$ denote the resistance in $E$ between $V_{k-1}$ and $V_k$.  The number of edges between $V_{k-1}$ and $V_k$ is equal to the number $n_k$ of $k$th generation vertices in $\mathcal {T}_N$, and each edge is endowed with resistance $2^{k-1}$ by \eqref{resistances}.  Since these resistors are in parallel, we obtain 
\begin{equation}
	\frac 1{R_k} = \sum_{1}^{n_k} \frac{1}{2^{k-1}} = \frac{n_k}{2^{k-1}}.\nonumber
\end{equation}
This holds for every $1\leq k\leq N$.  Since the resistors $\{R_k\}_{k=1}^N$ are in series, $R(\mathcal{T}_N) \geq R(E) = \sum_{k=1}^{N} R_k$, establishing inequality \eqref{ResistBound}.
\end{proof}


\noindent \author{\textsc{Edward Kroc}}\\
University of British Columbia, Vancouver, Canada. \\
Electronic address: \texttt{ekroc@math.ubc.ca}
\vskip0.2in 
\noindent \author{\textsc{Malabika Pramanik}}\\
University of British Columbia, Vancouver, Canada. \\
Electronic address: \texttt{malabika@math.ubc.ca}

}


\newpage
\begin{thebibliography}{99}

\bibitem{Alfonseca}
  A.~Alfonseca.
  \newblock {\em Strong type inequalities and an almost-orthogonality principle for families of maximal operators along directions in $\mathbb{R}^2$}.
  \newblock J. London Math. Soc. (2) {\bf{67}}, No. 1, 208-218 (2003).

\bibitem{AlfonsecaSoriaVargas}
  A.~Alfonseca, F.~Soria, A.~Vargas.
  \newblock {\em A remark on maximal operators along directions in $\mathbb{R}^2$}.
  \newblock Math. Res. Lett. {\bf{10}}, No. 1, 41-49 (2003).

\bibitem{Bateman}
  M.~Bateman.
  \newblock {\em Kakeya sets and directional maximal operators in the plane}.
  \newblock Duke Math. J. {\bf{147}}, No. 1, 55-77 (2009).

\bibitem{BatemanKatz}
  M.~Bateman, N.H.~Katz.
  \newblock {\em Kakeya sets in Cantor directions}.
  \newblock Math. Res. Lett. {\bf{15}}, No. 1, 73-81 (2008).

\bibitem{Besicovitch}
  A.S.~Besicovitch.
  \newblock {\em On Kakeya's problem and a similar one}.
  \newblock Mat. Zeitschrift, {\bf 27}, No. 1, 312-320 (1928).
  
\bibitem{Carbery}
  A.~Carbery.
  \newblock {\em Differentiation in lacunary directions and an extension of the Marcinkiewicz multiplier theorem}.
  \newblock Ann. Inst. Four. {\bf{38}}, No. 1, 157-168 (1988).

\bibitem{Cordoba}
  A.~C\'ordoba.
  \newblock {\em The Kakeya maximal function and spherical summation multipliers}.
  \newblock Amer. J. Math. {\bf{99}}, No. 1, 1-22 (1977).

\bibitem{DuoandikoetxeaVargas}
  J.~Duoandikoetxea, A.~Vargas.
  \newblock {\em Directional operators and radial functions on the plane}.
  \newblock Ark. Mat. {\bf 33}, 281-291 (1995).

\bibitem{Falconer}
  K.~Falconer.
  \newblock {\em Fractal Geometry: mathematical foundations and applications}, 2nd edition.
  \newblock John Wiley \& Sons (2003).

\bibitem{Fefferman}
  C.~Fefferman.
  \newblock {\em The multiplier problem for the ball}.
  \newblock Ann. Math. (2) {\bf{94}}, 330-336 (1971).

\bibitem{Grafakos}
L.~Grafakos.
 \newblock{\em Modern Fourier analysis}.
 \newblock Springer, 2008. 

\bibitem{Grimmett}
  G.~Grimmett.
  \newblock {\em Percolation}, 2nd edition.
  \newblock Grundlehren der math. Wissenschaften, Vol. 321 (1999).

\bibitem{Katz}
  N.H.~Katz.
  \newblock {\em A counterexample for maximal operators over a Cantor set of directions}.
  \newblock Mat. Res. Lett. {\bf 3}, 527-536 (1996).

\bibitem{KatzLabaTao}
  N.H.~Katz, I.~\L aba, T.~Tao.
  \newblock {\em An improved bound on the Minkowski dimension of Besicovitch sets in $\mathbb{R}^3$}.
  \newblock Ann. of Math. (2) {\bf 152}, 383-446 (2000).

\bibitem{KatzTao}
  N.H.~Katz, T.~Tao.
  \newblock {\em New bounds for Kakeya problems}.
  \newblock J. Anal. Mat. {\bf 87}, 231-263 (2002).

\bibitem{KrocPramanik}
  E.~Kroc, M.~Pramanik.
  \newblock {\em Kakeya-type sets over Cantor sets of directions in $\mathbb{R}^{d+1}$}.
  \newblock [preprint available at http://arxiv.org/abs/1404.6235]

\bibitem{Lyons1}
  R.~Lyons.
  \newblock {\em Random walks and percolation on trees}.
  \newblock Ann. Prob., {\bf 18}, 931-958 (1990).

\bibitem{Lyons2}
  R.~Lyons.
  \newblock {\em Random walks, capacity, and percolation on trees}.
  \newblock Ann. of Prob. {\bf{20}}, (1992).

\bibitem{LyonsPeres}
  R.~Lyons, Y.~Peres.
  \newblock {\em Probability on Trees and Networks}.
  \newblock Cambridge Univ. Press, mypage.iu.edu/$\sim$rdlyons/prbtree/prbtree.html (in preparation).

\bibitem{NagelSteinWainger}
  A.~Nagel, E.M.~Stein, S.~Wainger.
  \newblock{\em Differentiation in lacunary directions}.
  \newblock Proc. Natl. Acad. Sci. US, {\bf{75}} (3), 1060-1062 (1978).

\bibitem{ParcetRogers}
  J.~Parcet, K.M.~Rogers.
  \newblock {\em Differentiation of integrals in higher dimensions}.
  \newblock Proc. Natl. Acad. Sci. US, {\bf{110}} (13), 4941-4944 (2013).

\bibitem{SjogrenSjolin}
  P.~Sj\"ogren, P.~Sj\"olin.
  \newblock {\em Littlewood-Paley decompositions and Fourier multipliers with singularities on certain sets}.
  \newblock Ann. Inst. Four., {\bf 31}, 157-175 (1981).

\bibitem{SteinHA}
E.~Stein. 
\newblock{\em Harmonic Analysis: real-variable methods, orthogonality and oscillatory integrals}.
\newblock Princeton mathematical series; 43. Princeton University Press, 1995. 


\bibitem{Vargas}
  A.~Vargas.
  \newblock {\em A remark on a maximal function over a Cantor set of directions}.
  \newblock Rend. Circ. Mat. Palermo {\bf 44}, 273-282 (1995).

\bibitem{Wolff}
  T.~Wolff.
  \newblock {\em An improved bound for Kakeya type maximal functions}.
  \newblock Rev. Math. Iberoamericana {\bf{11}}, No. 3, 651-674 (1995).
  
\end{thebibliography}
\end{document}